\documentclass[a4paper,10pt]{book}
\author{Thomas Cottrell}
\title{Comparing algebraic and non-algebraic foundations of $n$-category theory}

\usepackage{fancyhdr}
\usepackage{calc}
\fancypagestyle{intro}{
\fancyhf{}
\fancyhead[LE,RO]{\thepage}
\fancyhead[LO]{\itshape Introduction}
\fancyhead[RE]{\itshape Introduction}
}
\fancypagestyle{bookstyle}{
\fancyhf{}
\fancyhead[LE,RO]{\thepage}
\fancyhead[LO]{\itshape\rightmark}
\fancyhead[RE]{\itshape\leftmark}
}

\usepackage{url}
\usepackage[newenum]{paralist}
\usepackage[all,2cell]{xy} \UseAllTwocells
\usepackage{xcolor,graphicx}
\usepackage{amsmath}
\usepackage{amssymb}
\usepackage{amsthm}
\xyoption{tips}
\SelectTips{cm}{}

\newdir{|>}{!/4.5pt/@{|}*:(1,-.2)@^{>}*:(1,+.2)@_{>}}

\theoremstyle{definition}  \newtheorem{defn}{Definition}[section]
\theoremstyle{definition}  
\theoremstyle{plain}  \newtheorem{prop}[defn]{Proposition}
\theoremstyle{plain}  \newtheorem{thm}[defn]{Theorem}
\theoremstyle{plain}  \newtheorem{lemma}[defn]{Lemma}
\theoremstyle{plain}  \newtheorem{corol}[defn]{Corollary}
\newtheorem{conjecture}[defn]{Conjecture}

\newcommand{\Set}{\mathbf{Set}}
\newcommand{\Cat}{\mathbf{Cat}}
\newcommand{\CAT}{\mathbf{CAT}}
\newcommand{\SSet}{\mathbf{SSet}}
\newcommand{\nSSet}{n\text{-}\mathbf{SSet}}
\newcommand{\GSet}{\mathbf{GSet}}
\newcommand{\nGSet}{n\text{-}\mathbf{GSet}}
\newcommand{\kGSet}{k\text{-}\mathbf{GSet}}
\newcommand{\twoGSet}{2\text{-}\mathbf{GSet}}
\newcommand{\Alg}{\text{-}\mathbf{Alg}}

\newcommand{\nColl}{n\text{-}\mathbf{Coll}}
\newcommand{\kColl}{k\text{-}\mathbf{Coll}}
\newcommand{\TColl}{T\text{-}\mathbf{Coll}}

\newcommand{\Bicat}{\mathbf{Bicat}}

\newcommand{\Mag}{\mathbf{Mag}}
\newcommand{\nMag}{n\text{-}\mathbf{Mag}}
\newcommand{\nCat}{n\text{-}\mathbf{Cat}}
\newcommand{\Rn}{\mathcal{R}}
\newcommand{\Qn}{\mathcal{Q}}
\newcommand{\Rtwo}{\mathcal{R}}
\newcommand{\Qtwo}{\mathcal{Q}}
\newcommand{\OCS}{\mathbf{OCS}}
\newcommand{\OUC}{\mathbf{OUC}}
\newcommand{\Tr}{\mathrm{Tr}}
\newcommand{\nerve}{\mathcal{N}}

\newcommand{\NB}{\nerve \mathcal{B}}
\newcommand{\comp}{\circ}
\newcommand{\newcomp}{\odot}
\newcommand{\id}{\text{id}}  %or should I use 1 for identities?
\DeclareMathOperator*{\colim}{colim}
\newcommand{\op}{\text{op}}
\newcommand{\ladj}{\dashv}

\newcommand{\iso}{\cong}

\DeclareMathOperator*{\sumcomptwo}{\bigcirc}
\newcommand{\sumcomp}[2]{\sumcomptwo^{#1,#2}}
\newcommand{\gt}{\text{ }|\text{ }}
\newcommand{\ContrR}{\mathbf{Contr}}
\newcommand{\MagR}{\mathbf{Mag}}
\newcommand{\Contr}{\mathbf{Contr}}
\newcommand{\SoC}{\mathbf{SoC}}
\newcommand{\kOpd}{k\text{-}\mathbf{Opd}}
\newcommand{\vectj}{\mathbf{j}}
\newcommand{\vectk}{\mathbf{k}}

\newcommand{\vectp}{\mathbf{p}}
\newcommand{\vectone}{\mathbf{1}}
\newcommand{\vectzero}{\mathbf{0}}
\newcommand{\hatX}{\widehat{X}}

\newcommand{\dotblank}{\cdot}
\newcommand{\pushj}{\amalg j}
\newcommand{\pushk}{\amalg k}

\begin{document}

\begin{titlepage}
\vspace*{\stretch{1}}
\begin{center}\bf
{\LARGE Comparing algebraic and non-algebraic foundations of $n$-category theory}\\
\vspace{2cm}
{\Large by}\\
\vspace{15mm}
{\LARGE Thomas Peter Cottrell.}\\
\vspace{5cm} {\Large A thesis submitted for the degree of\\ Doctor of Philosophy.}\\
\vspace{4cm}
{\large School of Mathematics and Statistics\\ The University of Sheffield}\\
\bigskip
{\large February 2014.} \vspace{1cm}
\end{center}
\end{titlepage}

\newpage\hbox{}\thispagestyle{empty}\newpage

\pagenumbering{roman}
\chapter*{Abstract}

  Many definitions of weak $n$-category have been proposed. It has been widely observed that each of these definitions is of one of two types:  algebraic definitions, in which composites and coherence cells are explicitly specified, and non-algebraic definitions, in which a coherent choice of composites and constraint cells is merely required to exist.  Relatively few comparisons have been made between definitions, and most of those that have concern the relationship between definitions of just one type.  The aim of this thesis is to establish more comparisons, including a comparison between an algebraic definition and a non-algebraic definition.

  The thesis is divided into two parts.  Part~\ref{Part1} concerns the relationships between three algebraic definitions of weak $n$-category: those of Penon and Batanin, and Leinster's variant of Batanin's definition.  A correspondence between the structures used to define composition and coherence in the definitions of Batanin and Leinster has long been suspected, and we make this precise for the first time.  We use this correspondence to prove several coherence theorems that apply to all three definitions, and also to take the first steps towards describing the relationship between the weak $n$-categories of Batanin and Leinster.

  In Part~\ref{Part2} we take the first step towards a comparison between Penon's definition of weak $n$-category and a non-algebraic definition, Simpson's variant of Tamsamani's definition, in the form of a nerve construction.  As a prototype for this nerve construction, we recall a nerve construction for bicategories proposed by Leinster, and prove that the nerve of a bicategory given by this construction is a Tamsamani--Simpson weak $2$-category.  We then define our nerve functor for Penon weak $n$-categories.  We prove that the nerve of a Penon weak $2$-category is a Tamsamani--Simpson weak $2$-category, and conjecture that this result holds for higher $n$.

\tableofcontents

\mainmatter
\chapter*{Introduction}
\addcontentsline{toc}{chapter}{Introduction}

\pagestyle{intro}

  An $n$-category is a higher-dimensional categorical structure in which, as well as objects and morphisms between those objects, we have morphisms between morphisms (``$2$-morphisms''), morphisms between $2$-morphisms (``$3$-morphisms''), and so on up to $n$-morphisms for some fixed natural number $n$.  Such structures arise in areas as diverse as homotopy theory, computer science, and theoretical physics, as well as category theory itself.  The case of strict $n$-categories, in which composition of morphisms is strictly associative and unital, is well-understood, but for many applications it is not sufficiently general, since many naturally occurring ``composition-like'' operations satisfy associativity and unitality only up to some kind of higher-dimensional isomorphism or equivalence.  Thus, a notion of weak $n$-category is required.  The theory of weak $n$-categories has grown rapidly over the past two decades; many different definitions of weak $n$-category have been proposed, using a wide variety of approaches, but the relationships between these definitions are not yet well understood, with few comparisons having been made.  It has been widely observed (see \cite{Lei02}) that each of these definitions belongs to one of two groups, called ``algebraic'' and ``non-algebraic''.

  The distinction between algebraic and non-algebraic definitions lies in the way in which composites and coherence cells are treated, and is often described as follows: in an algebraic definition composites and coherence cells are explicitly specified; in a non-algebraic definition a suitable choice of composites and coherence cells is required to exist, but is not specified and is not necessarily unique.  However, the difference is more deeply ingrained in the approaches used than this description would suggest.  Algebraic definitions draw upon techniques from universal algebra, such as the theories of monads and operads, whereas non-algebraic definitions use topological techniques, such as homotopy theory and model category theory, and are closely related to the more algebraic notions of topological space, such as Kan complexes.  Thus when making comparisons between definitions that belong to just one group, there are pre-existing techniques that can be used, but making a comparison between an algebraic definition and a non-algebraic definition is more of a challenge.  The way in which the algebraic and non-algebraic approaches fit into the bigger picture of the relationship between algebra and topology is well-illustrated in the following diagram, by Leinster~\cite{Lei10}:
    \begin{center}
    \includegraphics[width=0.9\textwidth]{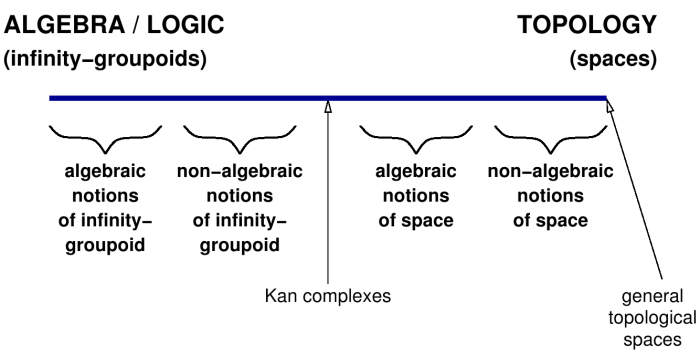}
    \end{center}
  Leinster used this diagram to illustrate the homotopy hypothesis of Gro\-then\-dieck, an important application of the theory of weak $n$-categories.  Very roughly, this hypothesis states that ``$\omega$-groupoids should be the same as spaces'' (and, in the $n$-dimensional case, ``$n$-groupoids should be the same as $n$-types'').  To state it formally we need to choose a notion of weak $n$-category and a notion of space, with the ``strongest'' statement of the hypothesis arising when we use an algebraic definition of weak $n$-category and a non-algebraic definition of space.  A statement of the hypothesis using a non-algebraic notion of weak $n$-category or an algebraic notion of space is less strong since it connects concepts that are more similar to one another, but it should be easier to prove for the same reason.  Understanding the relationship between algebraic and non-algebraic definitions of weak $n$-category would thus represent a significant step towards proving the strongest version of the homotopy hypothesis.  The case of weak $n$-groupoids, weak $n$-categories in which all morphisms (including higher morphisms) are invertible up to some higher cell, is of particular interest not only in the case of the homotopy hypothesis, but also in the study of Homotopy Type Theory \cite{UFP13}.

  We now discuss the various definitions of weak $n$-category that have been proposed.  The earliest algebraic definitions were made for specific values of $n$, and took a very direct approach; these include the classical definitions of bicategory \cite{Ben67} and tricategory (first defined in \cite{GPS95}, and made fully algebraic in \cite{Gur06}), and Trimble's definition of tetracategory~\cite{Tri95}.  The definition of tetracategory is long, and highlights the fact that it is not practical to use the classical approach for higher values of $n$.  Thus, in algebraic definitions of weak $n$-category for a general value of $n$, more abstract methods are used.  Algebraic definitions for a general value of $n$ include Penon's definition \cite{Pen99}, which uses a weakened version of the free strict $n$-category monad, Batanin's definition (\cite{Bat98}, with variants \cite{Lei00, Ber02, Lei02, Cis07, Gar10, vdBG11, Che11}), which uses globular operads, and the definitions of Trimble (\cite{Tri99}, first published in \cite{Lei02}, with variants \cite{CG07, Che11}) and May \cite{May01}, which use weakened forms of enrichment.

  Non-algebraic definitions take the form ``a weak $n$-category consists of some underlying data (usually a presheaf) satisfying a certain condition''.  This condition ensures that a coherent choice of composition structure can be made from the underlying data, but the exact composition structure is not explicitly specified.  Various types of underlying data are used in non-algebraic definitions, including simplicial sets \cite{Str87}, multisimplicial sets \cite{Tam99, Sim97}, cellular sets \cite{Joy97}, opetopic and multitopic sets \cite{BD98, HMP00, HMP01, HMP02, Lei98}, and $\omega$-hypergraphs \cite{HMT99, MT00}.  Many non-algebraic definitions deal with the case of $(\infty, n)$-categories, a type of higher category in which there is no maximum dimension of cell, but in which all cells of dimension greater than $n$ are equivalences. These definitions include quasi-categories (also known as weak Kan complexes) \cite{BV73, Vog73, Joy02, Lur09}, Segal $n$-categories \cite{HS98, Sim12}, and complete Segal spaces \cite{Rez01}.  % 
  Although there is a large number of different definitions, relatively few comparisons have been made between them, and most of the comparisons that have been made are either exclusively between algebraic definitions, or exclusively between non-algebraic definitions.  In the case of algebraic definitions, Batanin has made a comparison between his definition and that of Penon \cite{Bat02}, showing that his definition is weaker than Penon's and conjecturing some sort of weak equivalence between the two; Cheng has shown that a generalisation of Trimble's definition is an instance of a variant of Batanin's \cite{Che11}.  In the case of non-algebraic definitions, Cheng has proved an equivalence between the opetopic and multitopic definitions of Baez--Dolan, Hermida--Makkai--Power, and Leinster~\cite{Che04a, Che04b}; Bergner has proved equivalences between various definitions of $(\infty, 0)$- and $(\infty, 1)$-categories~\cite{Berg08}; Joyal and Tierney have proved equivalences between quasi-categories, complete Segal spaces, and Segal $n$-categories \cite{JT07}; Barwick and Schommer-Pries have shown the definitions of $(\infty, n)$-category of Joyal, Lurie, Rezk, and Simpson all satisfy a certain axiomatisation \cite{BSP12}.  Leinster has compared various definitions, both algebraic and non-algebraic, with the classical definitions of category and bicategory in the cases $n = 1$ and $n = 2$ \cite{Lei02}.  In the non-algebraic cases these are comparisons of an algebraic definition with a non-algebraic definition, although not for a general value of $n$.  Similarly, Gurski has also proved an equivalence between bicategories and a $2$-dimensional reformulation of Street's definition of weak $\omega$-category \cite{Gur09}, building on work of Duskin~\cite{Dus02}.

  It may appear from this list of comparisons that the most progress has been made with non-algebraic definitions;  this is because in the non-algebraic setting the lack of specified composites and coherence cells means we do not need to be as careful about keeping track of things as we do in the algebraic setting.  Very little progress has been made in comparing algebraic and non-algebraic definitions, with the only existing comparisons being restricted to the case $n = 2$.  Moving between the algebraic and non-algebraic settings is difficult; it is not simply a case of taking a non-algebraic definition and making choices of composites and coherence cells, or of taking an algebraic definition and just asking for existence in place of specified structure.  There are situations in which it is possible to make changes like this, but the resulting definitions are not far removed from the original ones.  For example, Batanin's definition~\cite{Bat98}, which is algebraic, uses a globular operad with a specified contraction; we can make this ``less algebraic'' by replacing this with a contractible operad \cite[Definition~B2]{Lei02}, but the resulting notion of weak $n$-category is still an algebraic one.  Similarly, the Tamsamani--Simpson definition, which is non-algebraic, asks for certain maps to be contractible; we can make this ``more algebraic'' by asking for specified contractions~\cite{Pel08}, but the resulting notion of weak $n$-category is still non-algebraic.

  One established method of moving between the algebraic and non-algebraic settings is the idea of a ``nerve construction''.  This idea arose from the well-known nerve construction for categories, which allows us to express a category as a simplicial set satisfying a ``nerve condition''.  Roughly speaking, a nerve construction takes an algebraic object, and produces from it a particular kind of presheaf, so a nerve construction can be seen as a way of passing from an algebraic setting to a non-algebraic setting.  Various authors have given nerve constructions for algebraic definitions of weak $n$-category \cite{Web07, Mel10, BMW12}, but these have focussed on extracting a canonical nerve from a given algebraic notion of $n$-category, rather than making connections with existing non-algebraic definitions.  This can be seen as creating a new non-algebraic definition corresponding to the given algebraic definition; the presheaves this approach gives are therefore specific to the chosen algebraic definition, and are unlikely to be presheaves on a category that arises naturally elsewhere.  One exception to this is the case of strict $\omega$-categories;  Berger has shown that, in this case, the canonical nerve is a presheaf on a category that arises naturally as a wreath product of the simplex category $\Delta$ \cite{Ber02, Ber07}.  These nerve constructions illustrate one reason that the subject has grown; various authors have tried to make connections between definitions, and ended up inventing new definitions (see, for example, \cite{BD98, Lei00}).  The proliferation of definitions of weak $n$-category has led to a disjointed, disparate subject, and this highlights the importance of making comparisons between existing definitions.  The aim of this thesis is to make the first comparison between an algebraic definition and a non-algebraic definition of weak $n$-category.

  The thesis is structured as follows: it is divided into two parts;  broadly speaking, the first part concerns the relationships between various algebraic definitions of weak $n$-category, and second part describes the first steps towards a comparison between an algebraic definition and a non-algebraic definition.

  The first part begins with Chapter~\ref{chap:Penon}, in which we recall the definition of Penon weak $n$-category \cite{Pen99}.  This is the central definition of the thesis, in the sense that it is the only definition of weak $n$-category to be used in both parts.  The idea of the definition of Penon weak $n$-category is to weaken the well-understood notion of strict $n$-category by means of a ``contraction''.  This is inspired by the topological notion of a contraction; and Penon uses it to ensure that any axiom that holds in a strict $n$-category holds ``up to homotopy'' in a Penon weak $n$-category.

  Penon weak $n$-categories are defined as algebras for a monad induced by a certain adjunction.  Penon described this adjunction in his original paper, but we give a new construction that we will use later, in Chapters~\ref{chap:Penonop} and \ref{chap:nerveconstr}.  The left adjoint in this adjunction freely adds two types of structure: a binary composition structure and a contraction structure.  In our construction we add these structures alternately, dimension by dimension, using an interleaving construction based on that of Cheng~\cite{Che10} (see also \cite{HDM06}).

  In Chapter~\ref{chap:opdefns} we discuss definitions of weak $n$-categories as algebras for globular operads.  Globular operads were introduced by Batanin~\cite{Bat98} as a tool for defining weak $n$-categories; they are a type of higher operad in which the operations have as their arities globular pasting diagrams.  We recall the definitions of globular operads and their algebras, then discuss Batanin's approach to identifying which globular operads give a ``sensible'' definition of weak $n$-category.  We recall the definitions of a system of compositions and a contraction on a globular operad, and the definition of Batanin weak $n$-categories as algebras for the initial globular operad with a contraction and system of compositions.

  Batanin's definition can be seen as a whole family of definitions, with many authors using variants \cite{Ber02, Lei02, Cis07, Gar10, vdBG11, Che11}.  One variant of particular note is that of Leinster~\cite{Lei00}; in place of a contraction and system of compositions, Leinster uses a notion called an ``unbiased contraction'' on an operad, which simultaneously ensures that we have unbiased composition operations and coherence operations.  We recall the definition of unbiased contraction, and Leinster's variant of Batanin's definition, in which he defines weak $n$-categories to be the algebras for the initial operad with an unbiased contraction; we refer to these algebras as ``Leinster weak $n$-categories''.

  In this chapter, we make the following results precise for the first time:
    \begin{itemize}
      \item The existence of the initial operad with a contraction and system of compositions; this has previously been (very reasonably) assumed by other authors \cite{Bat98, Lei02}.
      \item The correspondence between operads with contractions and systems of compositions, and operads with unbiased contractions; specifically, we prove a conjecture of Leinster~\cite[Section~10.1]{Lei04} stating that any operad with a contraction and system of compositions can be equipped with an unbiased contraction (the converse is already known \cite[Examples~10.1.2 and 10.1.4]{Lei04}).
      \item Three coherence theorems for algebras for globular operads; these results are not surprising, but have not previously been proved.  These theorems hold for the algebras for any globular operad equipped either with a contraction and system of compositions, or with an unbiased contraction.  By the previous result, for each theorem we can pick whichever notion is most convenient for the purposes of the proof.
    \end{itemize}

  Chapter~\ref{chap:Penonop} concerns comparisons between various operadic definitions of weak $n$-category.  It is a result of Batanin~\cite{Bat02} that Penon weak $n$-categories can be defined as algebras for a certain globular operad, and that this operad can be equipped with a canonical choice of contraction and system of compositions.  We give a new proof of this, using our construction of the monad for Penon weak $n$-categories from Chapter~\ref{chap:Penon}.  Our proof is more direct than Batanin's, and gives a different point of view, elucidating the structure of the operad.  This implies that the coherence theorems from the previous chapter also hold for Penon weak $n$-categories; Batanin has already observed~\cite{Bat02} that this result also gives a canonical comparison map from the operad for Batanin weak $n$-categories to the operad for Penon weak $n$-categories, and has used this to prove that Batanin weak $n$-categories are the weaker of the two for $n \geq 3$.

  We then take several steps towards a comparison between Batanin weak $n$-categories and Leinster weak $n$-categories.  It has been widely believed that these definitions are in some sense equivalent (see \cite[end of Section~4.5]{Lei00}), but no attempt to formalise this statement has been made.  We derive comparison functors between the categories of Batanin weak $n$-categories and Leinster weak $n$-categories, and discuss how close these functors are to being equivalences of categories.  We investigate what happens when we take a Leinster weak $n$-category and apply first the comparison functor to the category of Batanin weak $n$-categories, then the comparison functor back to the category of Leinster weak $n$-categories.  We believe that the Leinster weak $n$-category we obtain is in some sense equivalent to the one with which we started, and take a preliminary step towards formalising this statement.

  In the second part we describe a new nerve construction for Penon weak $n$-categories.  As mentioned earlier, nerves give a non-algebraic approach to the study of $n$-categories.  A nerve construction takes the form of a functor from a category of ``algebraic objects'' of some kind (e.g.~algebras for a certain monad) to a category of ``non-algebraic objects'' (e.g.~presheaves on a certain category); essentially, a nerve construction gives us a way of comparing algebraic things with non-algebraic things.  The nerve given by our nerve construction for Penon weak $n$-categories is an $n$-simplicial set, a presheaf of the same kind used in the definition of Tamsamani--Simpson weak $n$-category; thus this allows for a comparison to be made between the two definitions.  While various authors have given nerve constructions for algebraic definitions of weak $n$-category in the past \cite{Ber02, Web07, Mel10, BMW12}, our nerve construction is the first to provide a comparison between an algebraic definition of weak $n$-category and a pre-existing non-algebraic definition.

  We begin Chapter~\ref{chap:TamSim} by recalling the nerve construction for categories, and explaining how this leads to an alternative equivalent definition of a category as a simplicial set satisfying a condition called the ``nerve condition''.  We then recall Simpson's variant of Tamsamani's definition of weak $n$-category \cite{Tam99, Sim97}.  Instead of simplicial sets, i.e.~functors $\Delta^{\op} \rightarrow \Set$, this definition uses $n$-simplicial sets, functors $(\Delta^n)^{\op} \rightarrow \Set$; these must satisfy a generalised nerve condition, named the Segal condition in analogy with a similar condition arising in the study of Segal categories \cite{Seg74, DKS89}.  The Segal condition ensures that coherent composition exists in a Tamsamani--Simpson weak $n$-category;  note that there may be many different choices of coherent composition, and since this is a non-algebraic definition, no specific choice is made.

  In Chapter~\ref{chap:nerveconstr} we describe our nerve construction for Penon weak $n$-categories in the case $n = 2$.  We treat the case $n = 2$ separately because it is simpler, both notationally and conceptually, and also because we are able to prove that the nerve of a Penon weak $2$-category satisfies the Segal condition.  This represents strict progress towards a comparison of the two definitions for $n = 2$, since it tells us that the image of the nerve functor is contained in the category of Tamsamani--Simpson weak $2$-categories.

  In \cite{Lei02}, Leinster proposed a nerve construction for bicategories in which the nerve of a bicategory is a bisimplicial set; we use this as the prototype for our nerve construction.  Leinster defined the action of the nerve formally only on the objects of $\Delta^2$;  we complete this definition, and extend it to a definition of a nerve functor for bicategories.  We then prove for the first time that the nerve of a bicategory satisfies the Segal condition, and is therefore a Tamsamani--Simpson weak $2$-category.  Note that various other nerve constructions for bicategories have been proposed \cite{Dus02, LP08, Gur09}, but these are less suitable for generalisation to nerve constructions for an algebraic definition of weak $n$-category, since they do not make a distinction between the dimensions of cells in the same way.  Tamsamani--Simpson weak $n$-categories are well-suited to comparison with algebraic definitions since cells of different dimensions are kept separate in the underlying data, as they are in the algebraic definitions.  It also appears that a generalisation of one of these other nerve constructions for bicategories would require a definition of lax maps of Penon weak $n$-categories, and such a definition does not exist.

  We then define our nerve functor for Penon weak $2$-categories.  To do so, we use our construction of the monad for Penon weak $n$-categories from Chapter~\ref{chap:Penon}.  It is here that the construction shows its utility; it allows us to describe a certain type of Penon weak $n$-category which we can think of as being ``partially free'', and Penon weak $n$-categories of this type are used to describe the shapes of cells we require in our nerve.  The full necessity of the interleaving construction in developing this nerve is somewhat hidden, however; describing the shapes of cells in the nerve required great care, and our construction of Penon's monad gave us the precise control needed to do this correctly, allowing us to tweak the construction to get it just right.

  At the end of this chapter we prove that the nerve of a Penon weak $2$-category satisfies the Segal condition, and is therefore a Tamsamani--Simpson weak $2$-category.  The proof is unavoidably technical, and is also in some parts elementary, and we apologise for this;  both Penon weak $n$-categories and Tamsamani--Simpson weak $n$-categories are naturally arising in their own contexts, but these contexts are very different, and it is inevitable that any comparison will be technically complicated.  In this proof we use the notation for the cells of a Penon weak $n$-category given by our construction of Penon's monad from Chapter~\ref{chap:Penon}.

  In Chapter~\ref{chap:nerveconstrn} we generalise our nerve construction for Penon weak $2$-categories to the case of general $n$.  We conclude with a discussion of further results still to be proved that our nerve construction now makes it possible to state precisely, including a conjecture that the nerve of a Penon weak $n$-category is a Tamsamani--Simpson weak $n$-category; proofs of these results are beyond the scope of this thesis.  Comparison functors are rare in the study of weak $n$-categories, as discussed in the background section, and it is an achievement to have obtained a comparison functor allowing us to compare the definitions of Penon and Tamsamani--Simpson, even though proving that this functor satisfies the desired properties is still out of reach.  We hope that this work will lead to a full comparison between these two definitions, as well as paving the way for more comparisons between algebraic and non-algebraic definitions of weak $n$-category.

\section*{Notation and terminology}
\addcontentsline{toc}{section}{Notation and terminology}

  Throughout this thesis, the letter $n$ always denotes a fixed natural number, which is assumed to be the highest dimension of cell in the definition(s) of weak $n$-category being discussed.  Where a specific value of $n$ is used, this is noted.  We write $\mathbb{N}$ for the set of natural numbers; note that we take this to include $0$.

  All of our algebraic definitions of weak $n$-category use $n$-globular sets as their underlying data.  An $n$-globular set is a presheaf on the \emph{$n$-globe category} $\mathbb{G}$, which is defined as the category with
      \begin{itemize}
        \item objects: natural numbers $0$, $1$, $\dotsc$, $n - 1$, $n$;
        \item morphisms generated by, for each $1 \leq m \leq n$, morphisms
          \[
            \sigma_m, \tau_m \colon (m - 1) \rightarrow m
          \]
        such that $\sigma_{m + 1} \sigma_m = \tau_{m + 1} \sigma_m$ and $\sigma_{m + 1} \tau_m = \tau_{m + 1} \tau_m$ for $m \geq 2$ (called the ``globularity conditions'').
      \end{itemize}
  For an $n$-globular set $X \colon \mathbb{G}^{\op} \rightarrow \Set$, we write $s$ for $X(\sigma_m)$, and $t$ for $X(\tau_m)$, regardless of the value of $m$, and refer to them as the source and target maps respectively.  We denote the set $X(m)$ by $X_m$.  We say that two $m$-cells $x$, $y \in X_m$ are \emph{parallel} if $s(x) = s(y)$ and $t(x) = t(y)$; note that all $0$-cells are considered to be parallel.  We write $\nGSet$ for the category of $n$-globular sets $[\mathbb{G}^{\op}, \Set]$.

  We write $T$ for the free strict $n$-category monad on $\nGSet$.  This is the monad induced by the adjunction
        \[
          \xy
            % POINTS
            (0, 0)*+{\nGSet}="nG";
            (24, 0)*+{\nCat ,}="nC";
            % ARROWS
            {\ar@<1ex>_-*!/u1pt/{\labelstyle \bot} "nG" ; "nC"};
            {\ar@<1ex> "nC" ; "nG"};
          \endxy
       \]
  where $\nCat$ is the category of strict $n$-categories, and the right adjoint is the forgetful functor sending a strict $n$-category to its underlying $n$-globular set.  In certain circumstances (for results in which a greater degree of generality is possible) we write $T$ to denote an arbitrary monad, possibly satisfying certain conditions.  In such cases, $T$ can always be taken to be the free strict $n$-category monad, and this will be the particular example in which we are interested.  We write $\eta^T \colon 1 \Rightarrow T$ for the unit of the monad $T$, and $\mu^T \colon T^2 \Rightarrow T$ for its multiplication.  Similarly, for any monad $K$, we denote its unit by $\eta^K \colon 1 \Rightarrow K$ and its multiplication by $\mu^K \colon K^2 \Rightarrow K$.

  We write $1$ for the terminal $n$-globular set, which has precisely one $m$-cell for each $0 \leq m \leq n$.  The $n$-globular set $T1$ appears frequently throughout the thesis, specifically in Chapters~\ref{chap:opdefns} and \ref{chap:Penonop}.  Applying the monad $T$ freely generates all possible formal composites; since there is only one $m$-cell in $1$ for each $0 \leq m \leq n$, that $m$-cell can be composed with itself any number of times along all dimensions of boundary.  Thus an element of $T1_m$ is a pasting diagram made up entirely of globular cells (possibly including identity cells).  There is only one such pasting diagram at dimension $0$, since $0$-cells cannot be composed, so $T1_0$ has only one element.  At dimension $1$ such a diagram consists of a finite string of $1$-cells composed end to end, so for each natural number $k$ there is an element of $T1_1$ which should be visualised as
      \[
        \underbrace{
        \xy
          % POINTS
          (0, 0)*+{\bullet}="1";
          (10, 0)*+{\bullet}="2";
          (20, 0)*+{\bullet}="3";
          (30, 0)*+{\bullet}="4";
          (25, 0)*+{\dotsc};
          (40, 0)*+{\bullet}="5";
          (42, -1)*+{,};
          % ARROWS
          {\ar "1" ; "2"};
          {\ar "2" ; "3"};
          {\ar "4" ; "5"};
        \endxy
        }_{k \text{ } 1 \text{-cells}}
      \]
  where each arrow represents the unique element of $1_1$.  Note that this includes the degenerate case $k = 0$.  At higher dimensions the diagrams become more complicated, since cells can be composed in more ways; a typical element of $T1_2$ looks like
      \[
        \xy
          % POINTS
          (-16, 0)*+{\bullet}="0";
          (0, 0)*+{\bullet}="1";
          (16, 0)*+{\bullet}="2";
          (32, 0)*+{\bullet}="3";
          (34, -1)*+{,};
          (8, 3.5)*+{\Downarrow};
          (8, -3.5)*+{\Downarrow};
          (24, 0)*+{\Downarrow};
          % ARROWS
          {\ar "0" ; "1"};
          {\ar@/^1.75pc/ "1" ; "2"};
          {\ar "1" ; "2"};
          {\ar@/_1.75pc/ "1" ; "2"};
          {\ar@/^1.25pc/ "2" ; "3"};
          {\ar@/_1.25pc/ "2" ; "3"};
        \endxy
      \]
  where each double arrow represents the unique element of $1_2$, and the single arrow on the left-hand end is a degenerate $2$-cell (i.e. an identity $2$-cell on the unique element of $1_1$).  We call a cell in $T1$ a \emph{globular pasting diagram}; for a fixed $m$ we call a cell in $T1_m$ an \emph{$m$-globular pasting diagram}.

\newpage

\section*{Acknowledgements}

  First and foremost I would like to thank my supervisor Eugenia Cheng, whose guidance and support has been invaluable at every stage of this project.  I am also grateful to Nick Gurski for fulfilling the role of supervisor during Eugenia's occasional absences.  I would like to thank Roald Koudenburg, Jonathan Elliott, Tom Athorne, Alex Corner and Ben Fuller for many useful and enlightening conversations, and Freya Massey for her encouragement and emotional support.  Finally, I thank the University of Sheffield for their financial support of my Ph.D.

\part{Algebraic definitions of weak $n$-category} \label{Part1}

\pagestyle{bookstyle}

\chapter{Penon weak $n$-categories}  \label{chap:Penon}

  This chapter concerns the definition of Penon weak $n$-category.  This was originally given in \cite{Pen99}, but we use a variant given in \cite{Bat02, CM08}.  Penon defined weak $n$-categories as the algebras for a monad on the category of \emph{reflexive} globular sets (globular sets in which each cell has a putative identity cell at the dimension above).  In \cite{CM08} Cheng and Makkai observed that, in the finite dimensional case, Penon's definition did not encompass certain well-understood examples of weak $n$-categories, such as braided monoidal categories, but that this could be remedied by using globular sets instead of reflexive globular sets.  Note that Penon originally gave his definition in the case $n = \omega$, whereas we take $n$ to be finite (this modification of the definition for finite $n$ is standard, see \cite{Lei02, CM08}).  The use of a finite value of $n$ allows us to prove coherence theorems that hold for Penon weak $n$-categories in Chapter~\ref{chap:opdefns}, most of which do not hold in the $\omega$-dimensional case, and in Chapters~\ref{chap:nerveconstr} and \ref{chap:nerveconstrn} it allows us to make a comparison between Penon weak $n$-categories and Tamsamani--Simpson weak $n$-categories, the latter only being defined for finite $n$.

  The reason for choosing to use Penon weak $n$-categories over another algebraic definition is that we are able to give an explicit description of Penon's monad, and thus of a free Penon weak $n$-category.  This was very useful when devising the nerve constructions in Chapters~\ref{chap:nerveconstr} and \ref{chap:nerveconstrn}; these constructions involve algebras that are almost free, and the construction of Penon's monad in this chapter made it possible to modify the free algebra construction in a way that would not be possible with the definitions of Batanin and Leinster discussed in Chapter~\ref{chap:opdefns}.  In spite of its unusual construction, Penon's monad is known to arise from an $n$-globular operad with contraction and system of compositions (see \cite{Bat02}), so this definition belongs to a commonly studied family of definitions of weak $n$-category.  This is discussed in more detail in Section~\ref{sect:Penonop}.

  The monad for Penon weak $n$-categories is induced by a certain adjunction; in this chapter we recall the definition, then give a new construction of the left adjoint of the adjunction.

\section{Definition of Penon weak $n$-categories}  \label{sect:Penondefn}

  In this section we recall the non-reflexive variant of Penon's definition of weak $n$-category \cite{Pen99, Bat02, CM08}.  The idea of Penon's definition is to weaken the well-understood notion of strict $n$-category by means of a ``contraction''.  To do this Penon considers ``$n$-magmas'': $n$-globular sets equipped with binary composition operations that are not required to satisfy any axioms (apart from the usual source and target conditions). He then asks when an $n$-magma is ``coherent enough'' to be considered a weak $n$-category.  To answer this question he uses the fact that every strict $n$-category has an underlying $n$-magma to compare $n$-magmas with strict $n$-categories by considering maps
    \[
      \xy
        % POINTS
        (0, 0)*+{X}="0";
        (16, 0)*+{S,}="1";
        % ARROW
        {\ar^f "0" ; "1"};
      \endxy
    \]
  where $X$ is an $n$-magma, $S$ is the underlying $n$-magma of a strict $n$-category, and $f$ preserves the $n$-magma structure.  Penon defines a notion of a contraction on such a map, which lifts identities in $S$ to equivalences in $X$, ensuring that the axioms that hold in $S$ hold up to equivalence in $X$; by analogy with contractions in the topological sense, we can think of the axioms as holding ``up to homotopy'' in $X$.

  Penon then defines a category whose objects are maps $f \colon X \rightarrow S$ as above equipped with contractions; we denote this category by $\Qn$, following the notation of Leinster~\cite{Lei02}.  An object of $\Qn$ can be thought of as consisting of an $n$-magma $X$ that can be contracted down to a strict $n$-category $S$.  There is a forgetful functor $\Qn \rightarrow \nGSet$ sending an object of $\Qn$ to the underlying $n$-globular set of its magma part.  This functor has a left adjoint, which induces a monad on $\nGSet$, and a Penon weak $n$-category is defined to be an algebra for this monad.

  We begin by recalling the definition of an $n$-magma.

  \begin{defn}  \label{defn:magma}
    An \emph{$n$-magma} (or simply \emph{magma}, when $n$ is fixed) consists of an $n$-globular set $X$ equipped with, for each $m$, $p$, with $0 \leq p < m \leq n$, a binary composition function
      \[
        \comp^m_p \colon X_m \times_{X_p} X_m \rightarrow X_m,
      \]
    where $X_m \times_{X_p} X_m$ denotes the pullback
      \[
        \xy
          % POINTS
          (0, 0)*+{X_m \times_{X_p} X_m}="XX";
          (20, 0)*+{X_m}="Xr";
          (0, -16)*+{X_m}="Xb";
          (20, -16)*+{X_p}="Xp";
          % ARROWS
          {\ar "XX" ; "Xr"};
          {\ar "XX" ; "Xb"};
          {\ar_-{t} "Xb" ; "Xp"};
          {\ar^{s} "Xr" ; "Xp"};
          % PULLBACK STUFF
          (6,-2)*{}; (6,-6)*{} **\dir{-};
          (2,-6)*{}; (6,-6)*{} **\dir{-};
        \endxy
     \]
    in $\Set$; these composition functions must satisfy the following source and target conditions:
      \begin{itemize}
        \item if $p = m - 1$, given $(a, b) \in X_m \times_{X_p} X_m$,
          \[
            s(b \comp^m_p a) = s(a), \; t(b \comp^m_p a) = t(b);
          \]
        \item if $p < m - 1$, given $(a, b) \in X_m \times_{X_p} X_m$,
          \[
            s(b \comp^m_p a) = s(b) \comp^{m - 1}_p s(a), \; t(b \comp^m_p a) = t(b) \comp^{m - 1}_p t(a).
          \]
      \end{itemize}
    A map of $n$-magmas $f \colon X \rightarrow Y$ is a map of the underlying $n$-globular sets such that, for all $m$, $p$, with $0 \leq p < m \leq n$, and for all $(a, b) \in X_m \times_{X_p} X_m$,
      \[
        f(b \comp^m_p a) = f(b) \comp^m_p f(a).
      \]
    We write $\nMag$ for the category whose objects are $n$-magmas and whose morphisms are maps of $n$-magmas.
  \end{defn}

  Observe that every strict $n$-category has an underlying $n$-magma, and we have a forgetful functor
    \[
      \nCat \longrightarrow \nMag.
    \]

  We now recall the definition of a contraction on a map of $n$-globular sets $f \colon X \rightarrow S$, where $S$ is the underlying $n$-globular set of a strict $n$-category.  Note that this definition does not require a magma structure on $X$.  We must treat dimension $n$ slightly differently, since there is no dimension $n + 1$;  to do so, we define a notion of a ``tame'' map of $n$-globular sets (the terminology is due to Leinster \cite[Definition~9.3.1]{Lei04}), which ensures that we have equalities between $n$-cells where we would normally expect contraction $(n + 1)$-cells.

  It is common to express the definition of contraction in terms of lifting conditions \cite{Bat02, Ber02, Cis07}; however, we express the definition using pullbacks of sets since this approach allows for a straightforward construction of free contractions, which we describe in the next section.

  In the following definition, $X^c_{m + 1}$ is the set of all pairs of $m$-cells requiring a contraction $(m + 1)$-cell, i.e. the set of all pairs of parallel $m$-cells on $X_m$ which are mapped by $f$ to the same $m$-cell in $S_m$.  For any $(a, a) \in X^c_{m + 1}$, we write $\gamma_m(a, a) = 1_a$, since it is these contraction cells that give us the identities in a Penon weak $n$-category.

  \begin{defn} \label{defn:contr}
    Let $f \colon X \rightarrow S$ be a map of $n$-globular sets, where $S$ is the underlying $n$-globular set of a strict $n$-category.  The map $f$ is said to be \emph{tame} if, given $a$, $b \in X_n$, if $s(a) = s(b)$, $t(a) = t(b)$, and $f_n(a) = f_n(b)$, then $a = b$.

    For each $0 \leq m < n$, define a set $X^c_{m + 1}$ by the following pullback:
      \[
        \xy
          % POINTS
          (0, 0)*+{X^c_{m + 1}}="Xc";
          (32, 0)*+{X_m}="Xr";
          (0, -16)*+{X_m}="Xb";
          (32, -16)*+{X_{m - 1} \times X_{m - 1} \times S_m.}="XXS";
          % ARROWS
          {\ar "Xc" ; "Xr"};
          {\ar "Xc" ; "Xb"};
          {\ar_-{(s, t, f_m)} "Xb" ; "XXS"};
          {\ar^{(s, t, f_m)} "Xr" ; "XXS"};
          % PULLBACK STUFF
          (6,-1)*{}; (6,-5)*{} **\dir{-};
          (2,-5)*{}; (6,-5)*{} **\dir{-};
        \endxy
     \]
    Note that when $m = 0$, we take $X_{m - 1}$ to be the terminal set.

    A \emph{contraction} $\gamma$ on a tame map $f \colon X \rightarrow S$ consists of, for each $0 \leq m < n$, a map
      \[
        \gamma_{m + 1} \colon X^c_{m + 1} \rightarrow X_{m + 1}
      \]
    such that, for all $(a, b) \in X^c_{m + 1}$,
      \begin{itemize}
        \item $s(\gamma_{m + 1}(a, b)) = a$;
        \item $t(\gamma_{m + 1}(a, b)) = b$;
        \item $f_{m + 1}(\gamma_{m + 1}(a, b)) = 1_{f_m(a)} = 1_{f_m(b)}$.
      \end{itemize}
  \end{defn}

  Note that we only ever speak of a contraction on a tame map; thus, whenever we state that a map is equipped with a contraction, the map is automatically assumed to be tame.  One way to think about this is to say that we do require a contraction $(n + 1)$-cell for each pair of $n$-cells in $X^c_n$, and the only $(n + 1)$-cells in $X$ are equalities.

  Penon does not use the term ``contraction''; instead, he uses the word ``stretching'' (``\'etirement'').  This may appear somewhat counterintuitive, since the two words seem antonymous.  However, Penon's terminology comes from viewing the same situation from a different point of view; rather than seeing $S$ as a contracted version of $X$, Penon sees $X$ as a stretched-out version of $S$.  In the case in which $X$ has a magma structure, Penon refers to a such a map as a ``categorical stretching'' (``\'etirement cat\'egorique'').  Categorical stretchings form a category $\Qn$, which we now define.

  \begin{defn}
    The \emph{category of $n$-categorical stretchings} $\Qn$ is the category with
      \begin{itemize}
        \item objects: an object of $\Qn$ consists of an $n$-magma $X$, a strict $n$-category $S$, and a map of $n$-magmas
      \[
        \xy
          % POINTS
          (0, 0)*+{X}="X";
          (0, -16)*+{S}="S";
          % ARROWS
          {\ar_f "X" ; "S"};
        \endxy
     \]
          equipped with a contraction $\gamma$;
        \item morphisms: a morphism in $\Qn$ is a commuting square
     \[
        \xy
          % POINTS
          (0, 0)*+{X}="X";
          (16, 0)*+{Y}="Y";
          (0, -16)*+{S}="S";
          (16, -16)*+{R}="R";
          % ARROWS
          {\ar^u "X" ; "Y"};
          {\ar_f "X" ; "S"};
          {\ar^g "Y" ; "R"};
          {\ar_v "S" ; "R"};
        \endxy
     \]
     in $\nMag$ such that
              \begin{itemize}
                \item $v$ is a map of strict $n$-categories;
                \item writing $\gamma$ for the contraction on the map $f$ and $\delta$ for the contraction on the map $g$, for all $0 \leq m < n$, and $(a, b) \in X^c_{m + 1}$, we have
                  \[
                    u(\gamma_m(a, b)) = \delta_m(u(a), u(b)).
                  \]
              \end{itemize}
     We denote such a morphism by $(u, v)$.
      \end{itemize}
  \end{defn}

  For an object
      \[
        \xy
          % POINTS
          (0, 0)*+{X}="X";
          (0, -16)*+{S}="S";
          % ARROWS
          {\ar_f "X" ; "S"};
        \endxy
     \]
  of $\Qn$, we refer to $X$ as its \emph{magma part} and $S$ as its \emph{strict $n$-category part}.  There is a forgetful functor
      \[
        \xy
          % POINTS
            (0, 0)*+{U \colon \Qn}="Qn";
            (30, 0)*+{\nGSet}="nGSet";
            (4, -8)*+{X}="X";
            (30, -16)*+{X}="Xr";
            (4, -24)*+{S}="S";
          % ARROWS
            {\ar "Qn" ; "nGSet"};
            {\ar_x "X" ; "S"};
            {\ar@{|->} (10, -16) ; (24, -16)};
        \endxy
      \]
  and this functor has a left adjoint $F \colon \nGSet \rightarrow \Qn$.  Penon gives a construction of this left adjoint in the second part of \cite{Pen99}.

  \begin{defn}  \label{defn:Penonncat}
    Let $P$ be the monad on $\nGSet$ induced by the adjunction $F \ladj U$.  A Penon weak $n$-category is defined to be an algebra for the monad $P$, and $P\Alg$ is the category of Penon weak $n$-categories.
  \end{defn}

\section{Construction of Penon's left adjoint}  \label{sect:ladj}

  In~\cite{Pen99} Penon gave a construction of the left adjoint $F$, mentioned above, using computads (which he called ``polygraphs'', terminology due to Burroni \cite{Bur93}).  In this section we give a new, alternative construction of the functor $F$, using a monad interleaving construction similar to that used by Cheng to construct the operad for Leinster weak $\omega$-categories~\cite{Che10} (see also \cite{HDM06}, which describes a more general interleaving argument).  There are two reasons for giving this alternative construction:  in Section~\ref{sect:Penonop} we use it to prove that there is an $n$-globular operad whose algebras are Penon weak $n$-categories, and it also gives us notation for all the individual cells in a Penon weak $n$-category, a fact which we use in Section~\ref{sect:Segal}.

  The first step of our construction is the same as that of Penon.  There is a forgetful functor $U_T \colon \nCat \rightarrow \nGSet$ (the notation $U_T$ is used because $\nCat = T\Alg$), and we write $\Rn$ for the comma category
    \[
      \nGSet \downarrow U_T.
    \]
  Thus an object of $\Rn$ is a map of $n$-globular sets $f \colon X \rightarrow S$, where $S$ is underlying $n$-globular set of a strict $n$-category.  Recall that an object of $\Qn$ consists of an object $f \colon X \rightarrow S$ of $\Rn$ equipped with a magma structure on $X$ and a contraction on $f$, so we can factorise the forgetful functor $U \colon \Qn \rightarrow \nGSet$ as
      \[
        \xy
          % POINTS
          (0, 0)*+{\Qn}="Qn";
          (24, 0)*+{\nGSet}="nG";
          (12, -10)*+{\Rn}="Rn";
          % ARROWS
          {\ar^-U "Qn" ; "nG"};
          {\ar_-W "Qn" ; "Rn"};
          {\ar_-V "Rn" ; "nG"};
        \endxy
     \]
  where $W$ forgets the magma and contraction structures, and $V$ sends an object $f \colon X \rightarrow S$ of $\Rn$ to its $n$-globular set part $X$.  To construct a left adjoint to $U$ we construct left adjoints to $V$ and $W$ separately.  Constructing a left adjoint to $V$ is straightforward: it sends an $n$-globular set $X$ to $\eta^T_X \colon X \rightarrow TX$.

  We now explain the interleaving argument, which is used to construct the left adjoint to $W$; this is where our construction differs from that of Penon.  In an object of $\Qn$ the magma structure and contraction structure exist independently of one another, and there are no axioms governing their interaction.  Thus, we can define categories
    \begin{itemize}
      \item $\MagR_n$, whose objects are objects $f \colon X \rightarrow S$ of $\Rn$, together with a magma structure on $X$, which is respected by $f$;
      \item $\ContrR_n$, whose objects are objects $f \colon X \rightarrow S$ of $\Rn$, together with a contraction.
    \end{itemize}
  The maps in these categories are required to respect the magma and contraction structures respectively.  We can write the category $\Qn$ as the pullback
    \[
      \xy
        % POINTS
        (0, 0)*+{\Qn}="0,0";
        (16, 0)*+{\MagR_n}="1,0";
        (0, -16)*+{\ContrR_n}="0,1";
        (16, -16)*+{\Rn,}="1,1";
        % ARROWS
        {\ar "0,0" ; "1,0"};
        {\ar "0,0" ; "0,1"};
        {\ar^-N "1,0" ; "1,1"};
        {\ar_-D "0,1" ; "1,1"};
        % PULLBACK STUFF
        (6,-1)*{}; (6,-5)*{} **\dir{-};
        (2,-5)*{}; (6,-5)*{} **\dir{-};
      \endxy
    \]
  where $N$ and $D$ are the forgetful functors that forget the magma and contraction structures respectively.  The functor $N$ has a left adjoint $M$, which freely adds binary composites, and the functor $D$ has a left adjoint $C$, which freely adds contraction cells.  We wish to combine these left adjoints to obtain a left adjoint to $W \colon \Qn \rightarrow \Rn$, which adds both the magma and contraction structures freely.  However, we can't just add all of one structure, then all of the other, since with this approach we do not end up with enough cells.  If we add a contraction structure first, followed by a magma structure, we do not get any contraction cells whose sources or targets are composites, such as unitors and associators.  If we add a magma structure first, followed by a contraction structure, we do not get any composites involving contraction cells.

  We therefore ``interleave'' the structures, one dimension at a time.  To do so, we make the following observations:
    \begin{itemize}
      \item when we add contraction cells freely, the contraction $m$-cells depend only on cells at dimension $m - 1$;
      \item when we add composites freely, the composites of $m$-cells depend only on cells at dimensions $m$ and below.
    \end{itemize}
  This means that we can add the contraction cells and composites one dimension at a time; starting with dimension $1$, we first add contraction cells freely, then add composites freely; we then move up to the next dimension and repeat the process.

  To formalise this, we give separate dimension-by-dimension constructions of both the free contraction structure and the free magma structure, then interleave these constructions by lifting them to the case in which we have both a magma structure and a contraction structure.  Thus we obtain a left adjoint to the forgetful functor $W \colon \Qn \rightarrow \Rn$; by composing this with the left adjoint to the functor $V \colon \Rn \rightarrow \nGSet$ we obtain the left adjoint $F$ to $U \colon \Qn \rightarrow \nGSet$.

  Owing to the length of this construction, this section is divided into four subsections.  In Subsection~\ref{subsect:HladjV} we construct the left adjoint to $V$.  In Subsections~\ref{subsect:contr} and ~\ref{subsect:magma} we give dimension-by-dimension constructions of the left adjoints to $D$ and $N$ respectively; these describe the free contraction structure and free magma structure.  Finally, in Subsection~\ref{subsect:interleaving} we then interleave these constructions to give a left adjoint to $W$.

  \subsection{Left adjoint to $V$}  \label{subsect:HladjV}
  We begin by describing $\Rn$ explicitly, in order to establish some terminology, and to make clear its connection with $\Qn$.

  \begin{defn}
    Write $\Rn$ to denote the comma category $\nGSet \downarrow U_T$; explicitly, $\Rn$ is the category with
      \begin{itemize}
        \item objects: an object of $\Rn$ consists of an $n$-globular set $X$, a strict $n$-category $S$, and a map of $n$-globular sets
      \[
        \xy
          % POINTS
          (0, 0)*+{X}="X";
          (0, -16)*+{S}="S";
          % ARROWS
          {\ar_f "X" ; "S"};
        \endxy
     \]
        \item morphisms: a morphism in $\Rn$ is a commuting square
     \[
        \xy
          % POINTS
          (0, 0)*+{X}="X";
          (16, 0)*+{Y}="Y";
          (0, -16)*+{S}="S";
          (16, -16)*+{R}="R";
          % ARROWS
          {\ar^u "X" ; "Y"};
          {\ar_f "X" ; "S"};
          {\ar^g "Y" ; "R"};
          {\ar_v "S" ; "R"};
        \endxy
     \]
     in $\nGSet$ such that $v$ is a map of strict $n$-categories.  We denote such a morphism by $(u, v)$.
      \end{itemize}
  \end{defn}

  As in the case of $\Qn$, for an object
      \[
        \xy
          % POINTS
          (0, 0)*+{X}="X";
          (0, -16)*+{S}="S";
          % ARROWS
          {\ar_f "X" ; "S"};
        \endxy
     \]
  we refer to $X$ as its \emph{$n$-globular set part} and $S$ as its \emph{strict $n$-category part}.

  We have a forgetful functor $W \colon \Qn \rightarrow \Rn$, which forgets the contraction and $n$-magma structures but leaves the underlying map of $n$-globular sets unchanged, and a forgetful functor $V \colon \Rn \rightarrow \nGSet$, defined by
    \[
      V(\xymatrix{ X \ar[r]^{f} & S }) = X;
    \]
  these compose to give $V \comp W = U$.  We construct left adjoints to $V$ and $W$ separately, then compose these to obtain the left adjoint to $U$.  We begin with the construction of the left adjoint to $V$; we do this in more generality than we require, since this construction is valid for any monad $T$.

  \begin{defn}
    Let $T$ be a monad on a category $\mathcal{C}$, and write $U_T \colon T\Alg \rightarrow \mathcal{C}$ for the forgetful functor that sends a $T$-algebra to its underlying object in $\mathcal{C}$.  Define a functor $H \colon \mathcal{C} \rightarrow \mathcal{C}/U_T$ as follows:
      \begin{itemize}
        \item on objects: for $X \in \nGSet$,
          \[
            H(X) = (\xymatrix{ X \ar[r]^{\eta^T_X} & TX }),
          \]
          where $TX$ has the structure of the free $T$-algebra on $X$;
        \item on morphisms: for $f \colon X \rightarrow Y$ in $\mathcal{C}$, $Hf = (f, Tf)$.
      \end{itemize}
  \end{defn}

  \begin{prop} \label{prop:HladjV}
    Write $V \colon \mathcal{C}/U_T \rightarrow \mathcal{C}$ for the forgetful functor defined by, for an object $f \colon X \rightarrow S$ of $\mathcal{C}/U_T$, where $S$ has a $T$-algebra structure $\theta: TS \rightarrow S$,
    \[
      V(\xymatrix{ X \ar[r]^{f} & S }) = X.
    \]
    Then there is an adjunction $H \ladj V$.
  \end{prop}

  \begin{proof}
    First, we define the unit $\alpha \colon 1 \Rightarrow VH$ and the counit $\beta \colon HV \Rightarrow 1$.  For $X \in \mathcal{C}$,
      \[
        \alpha_X = \id_X \colon X \longrightarrow VHX = X.
      \]
    For $f \colon X \rightarrow S$ in $\mathcal{C}/U_T$, where $S$ has a $T$-algebra denoted by $\theta \colon TS \rightarrow S$, observe that $\theta$ is a map of $T$-algebras since, by the algebra axioms, the diagram
      \[
        \xy
          % POINTS
          (0, 0)*+{T^2 S}="TTS";
          (16, 0)*+{TS}="TSr";
          (0, -16)*+{TS}="TSb";
          (16, -16)*+{S}="S";
          % ARROWS
          {\ar^{T\theta} "TTS" ; "TSr"};
          {\ar_{\mu^T_S} "TTS" ; "TSb"};
          {\ar^{\theta} "TSr" ; "S"};
          {\ar_{\theta} "TSb" ; "S"};
        \endxy
      \]
    commutes.  The component of $\beta$ at $f \colon X \rightarrow S$, denoted $\beta_f$, is given by the commuting diagram
      \[
        \xy
          % POINTS
          (0, 0)*+{X}="Xl";
          (28, 0)*+{X}="Xr";
          (14, -14)*+{S}="Sm";
          (0, -28)*+{TX}="TX";
          (14, -28)*+{TS}="TS";
          (28, -28)*+{S,}="S";
          % ARROWS
          {\ar^{\id_X} "Xl" ; "Xr"};
          {\ar_{\eta^T_X} "Xl" ; "TX"};
          {\ar^{f} "Xl" ; "Sm"};
          {\ar^f "Xr" ; "S"};
          {\ar^{\eta^T_S} "Sm" ; "TS"};
          {\ar^{\id_S} "Sm" ; "S"};
          {\ar_{Tf} "TX" ; "TS"};
          {\ar_{\theta} "TS" ; "S"};
        \endxy
     \]
    as a map in $\mathcal{C}/U_T$.  This diagram commutes since the left-hand square is a naturality square for $\eta$ and the bottom-right triangle is the unit axiom for the algebra $\theta$; the remaining square commutes trivially.

    We now show that $\alpha$ and $\beta$ satisfy the triangle identities.  First, consider
      \[
        \xymatrix{
          V \ar[rr]^{\alpha V} \ar[drr]_{1_V} && VHV \ar[d]^{V\beta}  \\
          && V.
                 }
      \]
    For $f \colon X \rightarrow S$ in $\Rn$,
      \[
        V(\xymatrix{ X \ar[r]^{f} & S }) = X = VHV(\xymatrix{ X \ar[r]^{f} & S }),
      \]
    $\alpha_{X} = 1_X$, and $U\beta_{f} = 1_X$, so this diagram commutes.

    Now consider
      \[
        \xymatrix{
          H \ar[rr]^{H\alpha} \ar[drr]_{1_H} && HVH \ar[d]^{\beta H}  \\
          && H.
                 }
      \]
    For $X \in \mathcal{C}$,
      \[
        H(X) = (\xymatrix{ X \ar[r]^{\eta^T_X} & TX }) = HVH(X),
      \]
    $H\alpha_X = H\id_X = (\id_X, \id_{TX})$, and $\beta_{HX} = (\id, \mu_X \comp T\eta_X) = (\id_X, \id_{TX})$, so this diagram commutes.
  \end{proof}

  This gives us the left adjoint to the functor $V \colon \Rn \rightarrow \nGSet$.

\subsection{Free contraction structure} \label{subsect:contr}

  We now construct the free contraction on an object of $\Rn$.  In order to be able to use the construction in the interleaving argument in Section~\ref{subsect:interleaving}, we give the construction one dimension at a time.  In order to do so, we define, for each $0 \leq k \leq n + 1$, a category $\ContrR_k$, an object of which consists of an object of $\Rn$ equipped with a contraction up to dimension $k$. (Observe that $\ContrR_0 = \Rn$, and note that a ``contraction at dimension $n + 1$'' refers to the tameness condition at dimension $n$.)  We then have, for each $0 < k \leq n + 1$, a forgetful functor
    \[
      D_k \colon \ContrR_k \rightarrow \ContrR_{k - 1}.
    \]
  We construct a left adjoint to each $D_k$, which freely adds a contraction structure at dimension $k$, leaving all other dimensions unchanged.

  \begin{defn}  \label{defn:kcontr}
    Let $f \colon X \rightarrow S$ be a map of $n$-globular sets, where $S$ is the underlying $n$-globular set of a strict $n$-category, and let $0 \leq k \leq n$.  Recall from Definition~\ref{defn:contr} that, for each $0 \leq m < n$, the set $X^c_{m + 1}$ is defined by the pullback
      \[
        \xy
          % POINTS
          (0, 0)*+{X^c_{m + 1}}="Xc";
          (32, 0)*+{X_m}="Xr";
          (0, -16)*+{X_m}="Xb";
          (32, -16)*+{X_{m - 1} \times X_{m - 1} \times S_m.}="XXS";
          % ARROWS
          {\ar "Xc" ; "Xr"};
          {\ar "Xc" ; "Xb"};
          {\ar_-{(s, t, f_m)} "Xb" ; "XXS"};
          {\ar^{(s, t, f_m)} "Xr" ; "XXS"};
          % PULLBACK STUFF
          (6,-1)*{}; (6,-5)*{} **\dir{-};
          (2,-5)*{}; (6,-5)*{} **\dir{-};
        \endxy
     \]
    where, when $m = 0$, we take $X_{m - 1}$ to be the terminal set.

    A $k$-contraction $\gamma$ on the map $f$ consists of, for each $0 \leq m < k$, a map
      \[
        \gamma_{m + 1} \colon X^c_{m + 1} \rightarrow X_{m + 1}
      \]
    such that, for $(a, b) \in X^c_m$,
      \[
        s(\gamma_{m + 1}(a, b)) = a,
      \]
      \[
        t(\gamma_{m + 1}(a, b)) = b,
      \]
      \[
        f_{m + 1}(\gamma_{m + 1}(a, b)) = \id_{f(a)}.
      \]
  \end{defn}

  Note that having an $n$-contraction on a map $f$ is not the same as having contraction on $f$;  for a contraction on $f$, we also require that $f_n$ satisfies the condition that, for all $a$, $b \in X_n$, if $s(a) = s(b)$, $t(a) = t(b)$, and $f_n(a) = f_n(b)$, then $a = b$.  This condition can be thought as having contraction cells at dimension $n + 1$, but the only $(n + 1)$-cells are equalities.

  \begin{defn}
    For each $0 \leq k \leq n$, define a category $\ContrR_k$, with
      \begin{itemize}
        \item objects: an object of $\ContrR_k$ consists of an $n$-globular set $X$, a strict $n$-category $S$, and a map of $n$-globular sets
     \[
        \xy
          % POINTS
          (0, 0)*+{X}="X";
          (0, -16)*+{S}="S";
          % ARROWS
          {\ar_f "X" ; "S"};
        \endxy
     \]
            equipped with a $k$-contraction $\gamma$;
        \item morphisms: a morphism in $\ContrR_k$ is a commuting square
     \[
        \xy
          % POINTS
          (0, 0)*+{X}="X";
          (16, 0)*+{Y}="Y";
          (0, -16)*+{S}="S";
          (16, -16)*+{R}="R";
          % ARROWS
          {\ar^u "X" ; "Y"};
          {\ar_f "X" ; "S"};
          {\ar^g "Y" ; "R"};
          {\ar_v "S" ; "R"};
        \endxy
     \]
     in $\nGSet$ such that
              \begin{itemize}
                \item $v$ is a map of strict $n$-categories;
                \item writing $\gamma$ for the contraction on the map $f$ and $\delta$ for the contraction on the map $g$, for all $0 < m \leq k$, and $(a, b) \in X^c_{m}$, we have
                  \[
                    u(\gamma_m(a, b)) = \delta_m(u(a), u(b)).
                  \]
              \end{itemize}
      \end{itemize}
    Define a category $\ContrR_{n + 1}$, with
      \begin{itemize}
        \item objects: an object of $\ContrR_{n + 1}$ consists of an $n$-magma $X$, a strict $n$-category $S$, and a map of $n$-magmas
     \[
        \xy
          % POINTS
          (0, 0)*+{X}="X";
          (0, -16)*+{S}="S";
          % ARROWS
          {\ar_f "X" ; "S"};
        \endxy
     \]
            equipped with a contraction $\gamma$;
        \item morphisms: a morphism in $\ContrR_{n + 1}$ is a commuting square
     \[
        \xy
          % POINTS
          (0, 0)*+{X}="X";
          (16, 0)*+{Y}="Y";
          (0, -16)*+{S}="S";
          (16, -16)*+{R}="R";
          % ARROWS
          {\ar^u "X" ; "Y"};
          {\ar_f "X" ; "S"};
          {\ar^g "Y" ; "R"};
          {\ar_v "S" ; "R"};
        \endxy
     \]
     in $\nGSet$ such that
              \begin{itemize}
                \item $v$ is a map of strict $n$-categories;
                \item writing $\gamma$ for the contraction on the map $f$ and $\delta$ for the contraction on the map $g$, for all $0 < m \leq n$, and $(a, b) \in X^c_{m}$, we have
                  \[
                    u(\gamma_m(a, b)) = \delta_m(u(a), u(b)).
                  \]
              \end{itemize}
      \end{itemize}
  \end{defn}

  For all $0 < k \leq n + 1$, we have a forgetful functor
    \[
      D_k \colon \ContrR_k \rightarrow \ContrR_{k - 1};
    \]
  for $0 < k \leq n$, this functor forgets the contraction at dimension $k$, and for \hbox{$k = n + 1$} it is the inclusion functor of the subcategory $\ContrR_{n + 1}$ into $\ContrR_n$.

  We now define a putative left adjoint $C_k$ to the functor $D_k$;  we will then prove that this functor is left adjoint to $D_k$ in Proposition~\ref{prop:CladjD}.

  \begin{defn}  \label{defn:freekcontr}
    For each $k$, $0 < k \leq n$, we define a functor
      \[
        C_k \colon \ContrR_{k - 1} \rightarrow \ContrR_k.
      \]
    We begin by giving the action of $C_k$ on objects.  Let
       \[
        \xy
          % POINTS
          (0, 0)*+{X}="X";
          (0, -16)*+{S}="S";
          % ARROWS
          {\ar_f "X" ; "S"};
        \endxy
       \]
    be an object of $\ContrR_{k - 1}$, and write $\gamma$ for its $(k - 1)$-contraction (assuming $k > 1$; if $k = 1$, we have $\ContrR_{k - 1} = \ContrR_{0} = \Rn$, so there is no contraction on $f$).  We define an object
       \[
        \xy
          % POINTS
          (0, 0)*+{\tilde{X}}="X";
          (0, -16)*+{S}="S";
          % ARROWS
          {\ar_{\tilde{f}} "X" ; "S"};
        \endxy
       \]
    of $\ContrR_k$, with $k$-contraction $\tilde{\gamma}$.  The $n$-globular set $\tilde{X}$ is defined by:
      \begin{itemize}
        \item $\tilde{X}_j = X_j$ for all $j \neq k$;
        \item $\tilde{X}_k = X_k \amalg X^c_{k}$,
        \item for $(x, y) \in X^c_{k} \subseteq \tilde{X}_k$, $s(x, y) = x$, $t(x, y) = y$,
        \item for all other cells, sources and targets are inherited from $X$.
      \end{itemize}
    The map $\tilde{f} \colon \tilde{X} \rightarrow S$ is defined by
      \begin{itemize}
        \item $\tilde{f}_j = f_j$ for all $j \neq k$;
        \item $\tilde{f}_k \colon \tilde{X}_k \rightarrow S_k$ is defined by
          \begin{itemize}
            \item $\tilde{f}_k(\alpha) = f_k(\alpha)$ for $\alpha \in X_k \subseteq \tilde{X}_k$;
            \item $\tilde{f}_k(x, y) = 1_{f_{k - 1}(x)}$ for $(x, y) \in X^c_{k} \subseteq \tilde{X}_k$.
          \end{itemize}
      \end{itemize}
    The $k$-contraction $\tilde{\gamma}$ on $\tilde{f}$ is defined by
      \begin{itemize}
        \item $\tilde{\gamma}_m = \gamma^{k - 1}_m$ for all $m < k - 1$;
        \item $\tilde{\gamma}_{k - 1} \colon X^c_{k} \rightarrow \tilde{X}_k$ is the inclusion into the coproduct $\tilde{X}_k = X_k \amalg X^c_{k}$.
      \end{itemize}
    This defines the action of $C_k$ on objects.

    We now give the action of $C_k$ on morphisms.  Let
      \[
        \xy
          % POINTS
          (0, 0)*+{X}="X";
          (16, 0)*+{Y}="Y";
          (0, -16)*+{S}="S";
          (16, -16)*+{R}="R";
          % ARROWS
          {\ar^u "X" ; "Y"};
          {\ar_f "X" ; "S"};
          {\ar^g "Y" ; "R"};
          {\ar_v "S" ; "R"};
        \endxy
      \]
    be a morphism in $\ContrR_{k - 1}$.  Define a morphism
      \[
        \xy
          % POINTS
          (0, 0)*+{\tilde{X}}="X";
          (16, 0)*+{\tilde{Y}}="Y";
          (0, -16)*+{S}="S";
          (16, -16)*+{R}="R";
          % ARROWS
          {\ar^{\tilde{u}} "X" ; "Y"};
          {\ar_{\tilde{f}} "X" ; "S"};
          {\ar^{\tilde{g}} "Y" ; "R"};
          {\ar_{v} "S" ; "R"};
        \endxy
      \]
    in $\ContrR_k$, where $\tilde{u}$ is defined by
      \begin{itemize}
        \item $\tilde{u}_j = u_j$ for all $j \neq k$;
        \item $\tilde{u}_k \colon \tilde{X}_k \rightarrow \tilde{Y}_k$ is given by
          \begin{itemize}
            \item $\tilde{u}_k(\alpha) = u_k(\alpha)$ for $\alpha \in X_k \subseteq \tilde{X}_k$;
            \item $\tilde{u}_k(x, y) = (u_{k - 1}(x), u_{k - 1}(y))$ for $(x, y) \in X^c_{k} \subseteq \tilde{X}_k$.
          \end{itemize}
      \end{itemize}
    This defines the action of $C_k$ on morphisms.
  \end{defn}

  \begin{prop}  \label{prop:CladjD}
    For all $0 < k \leq n$, there is an adjunction $C_k \ladj D_k$.
  \end{prop}

  \begin{proof}
    We first define the unit $\eta \colon 1 \Rightarrow D_kC_k$, and counit $\epsilon: C_kD_k \Rightarrow 1$.

    Let
       \[
        \xy
          % POINTS
          (0, 0)*+{X}="X";
          (0, -16)*+{S}="S";
          % ARROWS
          {\ar_f "X" ; "S"};
        \endxy
       \]
    be an object of $\ContrR_{k - 1}$, with $(k - 1)$-contraction $\gamma$ (assuming $k > 1$; if $k = 1$, we have $\ContrR_{k - 1} = \ContrR_{0} = \Rn$, so there is no contraction on $f$).  Applying $D_kC_k$ gives
       \[
        \xy
          % POINTS
          (0, 0)*+{\tilde{X}}="X";
          (0, -16)*+{S}="S";
          % ARROWS
          {\ar_{\tilde{f}} "X" ; "S"};
        \endxy
       \]
    in $\ContrR_{k - 1}$ with the same $(k - 1)$-contraction.  The corresponding component of the unit $\eta$ is given by the map
      \[
        \xy
          % POINTS
          (0, 0)*+{X}="X";
          (16, 0)*+{\tilde{X}}="Xr";
          (0, -16)*+{S}="S";
          (16, -16)*+{S}="Sr";
          % ARROWS
          {\ar^{\eta_X} "X" ; "Xr"};
          {\ar_{f} "X" ; "S"};
          {\ar^{\tilde{f}} "Xr" ; "Sr"};
          {\ar_{\id_S} "S" ; "Sr"};
        \endxy
      \]
    where $\eta_X$ is defined by
      \[
        (\eta_X)_j =
            \left\{
              \begin{array}{ll}
              1_{X_j} & \text{if } j \neq k, \\
              \text{the inclusion } X_j \hookrightarrow X_j \amalg X^c_{j} & \text{if } j = k.
              \end{array}
            \right.
      \]

    Now let
       \[
        \xy
          % POINTS
          (0, 0)*+{X}="X";
          (0, -16)*+{S}="S";
          % ARROWS
          {\ar_f "X" ; "S"};
        \endxy
       \]
    be an object of $\ContrR_{k}$, with $k$-contraction $\gamma$.  Applying $C_kD_k$ gives
       \[
        \xy
          % POINTS
          (0, 0)*+{\tilde{X}}="X";
          (0, -16)*+{S}="S";
          % ARROWS
          {\ar_{\tilde{f}} "X" ; "S"};
        \endxy
       \]
    in $\ContrR_k$ with $k$-contraction $\tilde{\gamma}$, which is equal to $\gamma$ at all dimensions except $k$.  The corresponding component of the counit $\epsilon$ is given by the map
      \[
        \xy
          % POINTS
          (0, 0)*+{\tilde{X}}="X";
          (16, 0)*+{X}="Xr";
          (0, -16)*+{S}="S";
          (16, -16)*+{S}="Sr";
          % ARROWS
          {\ar^{\epsilon_X} "X" ; "Xr"};
          {\ar_{\tilde{f}} "X" ; "S"};
          {\ar^{f} "Xr" ; "Sr"};
          {\ar_{\id_S} "S" ; "Sr"};
        \endxy
      \]
    where $\epsilon_X$ is defined by
      \begin{itemize}
        \item $(\epsilon_X)_j = 1_{X_j}$ for all $j \neq k$;
        \item $(\epsilon_X)_k \colon \tilde{X}_k \rightarrow X_k$ is given by
          \begin{itemize}
            \item $(\epsilon_X)_k(\alpha) = \alpha$ for $\alpha \in X_k \subseteq \tilde{X}_k$;
            \item $(\epsilon_X)_k(x, y) = \tilde{\gamma}_{k}(x, y)$ for $(x, y) \in X^c_{k} \subseteq \tilde{X}_k$.
          \end{itemize}
      \end{itemize}

      We now check that the triangle identities hold; consider the diagrams
        \[
          \xymatrix{
            D_k \ar[rr]^{\eta D_k} \ar[rrd]_1 && D_kC_kD_k \ar[d]^{D_k\epsilon} & C_k \ar[rr]^{C_k\eta} \ar[rrd]_1 && C_kD_kC_k \ar[d]^{\epsilon C_k}  \\
            && D_k, & && C_k.
                   }
        \]
      In all of the natural transformations in these diagrams, the components on the strict $n$-category parts are all identities, so to show that the diagrams commute we need only consider the components on the $n$-globular set parts.  Since the components of the maps of $n$-globular sets are identities at every dimension except dimension $k$, we need only check that the corresponding diagrams of maps of sets of $k$-cells commute.

      First, we must show that, given
       \[
        \xy
          % POINTS
          (0, 0)*+{X}="X";
          (0, -16)*+{S}="S";
          % ARROWS
          {\ar_f "X" ; "S"};
        \endxy
       \]
      in $\ContrR_k$, the diagram
        \[
          \xymatrix{
            X_k \ar[rr]^-{(\eta_X)_k} \ar[rrd]_{1_{X_k}} && \tilde{X}_k = X_k \amalg X^c_{k - 1} \ar[d]^{(\epsilon_X)_k}  \\
            && X_k
                   }
        \]
      commutes; this is true, since given $\alpha \in X_k$, we have
        \[
          (\epsilon_X)_k \comp (\eta_X)_k(\alpha) = (\epsilon_X)_k(\alpha) = \alpha.
        \]

      Secondly, we must show that, given
       \[
        \xy
          % POINTS
          (0, 0)*+{X}="X";
          (0, -16)*+{S}="S";
          % ARROWS
          {\ar_f "X" ; "S"};
        \endxy
       \]
      in $\ContrR_{k - 1}$ with $(k - 1)$-contraction $\gamma$, the diagram
        \[
          \xymatrix{
            \tilde{X}_k \ar[rr]^-{(\tilde{\eta}_X)_k} \ar[rrd]_{1_{\tilde{X}_k}} && \tilde{X}_k \amalg \tilde{X}_{k}^c \ar[d]^{(\epsilon_{\tilde{X}})_k}  \\
            && \tilde{X}_k
                   }
        \]
      commutes.  We have two kinds of freely added contraction cells in $\tilde{X}_k \amalg \tilde{X}_{k}^c$; we write $(x, y)$ for the contraction cells in $X^c_k$, and $[x, y]$ for those in $\tilde{X}^c_k$ (the latter being the specified contraction cells in this case).  Given $\alpha \in X_k \subseteq \tilde{X}_k$,
        \[
          (\epsilon_{\tilde{X}})_k\comp (\tilde{\eta}_{X})_k(\alpha) = (\epsilon_{\tilde{X}})_k(\alpha) = \alpha;
        \]
      given $(x, y) \in X^c_k \subseteq \tilde{X}_k$,
        \[
          (\epsilon_{\tilde{X}})_k\comp (\tilde{\eta}_{X})_k(x, y) = (\epsilon_{\tilde{X}})_k[x, y] = (x, y);
        \]
      hence the diagram commutes.

      Thus the triangle identities hold, and we have an adjunction $C_k \ladj D_k$, with unit $\eta$ and counit $\epsilon$.
  \end{proof}

  We must also define $C_{n + 1}$ separately, since ``adding contraction $(n + 1)$-cells'' consists of identifying certain $n$-cells rather than actually adding cells; we can think of this as adding equality $(n + 1)$-cells between pairs of $n$-cells that would usually require a contraction cell between them.

  \begin{defn}  \label{defn:freenplusonecontr}
    We define a functor
      \[
        C_{n + 1} \colon \ContrR_n \rightarrow \ContrR_{n + 1}.
      \]
    We begin by giving the effect of $C_{n + 1}$ on objects.  Let
       \[
        \xy
          % POINTS
          (0, 0)*+{X}="X";
          (0, -16)*+{S}="S";
          % ARROWS
          {\ar_f "X" ; "S"};
        \endxy
       \]
    be an object of $\ContrR_n$, with $n$-contraction $\gamma$.  Define a set $X^c_{n + 1}$ and maps $\pi_1$, $\pi_2 \colon X^c_{n + 1} \rightarrow X_n$ by the following pullback:
      \[
        \xy
          % POINTS
          (0, 0)*+{X^c_{n + 1}}="Xc";
          (32, 0)*+{X_n}="Xr";
          (0, -16)*+{X_n}="Xb";
          (32, -16)*+{X_{n - 1} \times X_{n - 1} \times S_n.}="XXS";
          % ARROWS
          {\ar^{\pi_1} "Xc" ; "Xr"};
          {\ar_{\pi_2} "Xc" ; "Xb"};
          {\ar_-{(s, t, f_n)} "Xb" ; "XXS"};
          {\ar^{(s, t, f_n)} "Xr" ; "XXS"};
          % PULLBACK STUFF
          (6,-1)*{}; (6,-5)*{} **\dir{-};
          (2,-5)*{}; (6,-5)*{} **\dir{-};
        \endxy
     \]
    The set $X^c_{n + 1}$ can be thought of as the set of pairs of $n$-cells to be identified, but note that there is some redundancy: for all $a \in X_n$, $(a, a) \in X^c_{n + 1}$, and if we have $(a, b) \in X^c_{n + 1}$ we also have $(b, a) \in X^c_{n + 1}$.

    We now define an object
       \[
        \xy
          % POINTS
          (0, 0)*+{\tilde{X}}="X";
          (0, -16)*+{S}="S";
          % ARROWS
          {\ar_{\tilde{f}} "X" ; "S"};
        \endxy
       \]
    of $\ContrR_{n + 1}$ with contraction $\tilde{\gamma}$.  For $0 \leq m < n$, define
      \[
        \tilde{X}_m = X_m,
      \]
    and define $\tilde{X}_n$ to be the coequaliser of the diagram
      \[
        \xymatrix{
          X^c_{n + 1} \ar@<1ex>[r]^-{\pi_1} \ar@<-1ex>[r]_-{\pi_2} & X_n.
                 }
      \]
    For $0 \leq m < n$, define
      \[
        \tilde{f}_m = f_m,
      \]
     and define $\tilde{f}_n \colon \tilde{X}_n \rightarrow S_n$ to be the unique map such that
      \[
        \xymatrix{
          X^c_{n + 1} \ar@<1ex>[r]^-{\pi_1} \ar@<-1ex>[r]_-{\pi_2} & X_n \ar[r]^q \ar[dr]_{f_n} & \tilde{X}_n \ar@{-->}[d]^{\tilde{f}_n}  \\
          && S_n
                 }
      \]
    commutes, where $q \colon X_n \rightarrow \tilde{X}_n$ is the coequaliser map.  Finally, define $\tilde{\gamma}$ to be the $n$-contraction defined by
          \[
            \tilde{\gamma}_m =
                \left\{
                  \begin{array}{ll}
                  \gamma_m & \text{if } m < n, \\
                  q \comp \gamma_n & \text{if } m = n.
                  \end{array}
                \right.
          \]
    This defines the action of $C_{n + 1}$ on objects.

    We now give the action of $C_{n + 1}$ on morphisms.  Let
      \[
        \xy
          % POINTS
          (0, 0)*+{X}="X";
          (16, 0)*+{Y}="Y";
          (0, -16)*+{S}="S";
          (16, -16)*+{R}="R";
          % ARROWS
          {\ar^u "X" ; "Y"};
          {\ar_f "X" ; "S"};
          {\ar^g "Y" ; "R"};
          {\ar_v "S" ; "R"};
        \endxy
      \]
    be a morphism in $\ContrR_{n}$.  Define a morphism
      \[
        \xy
          % POINTS
          (0, 0)*+{\tilde{X}}="X";
          (16, 0)*+{\tilde{Y}}="Y";
          (0, -16)*+{S}="S";
          (16, -16)*+{R}="R";
          % ARROWS
          {\ar^{\tilde{u}} "X" ; "Y"};
          {\ar_{\tilde{f}} "X" ; "S"};
          {\ar^{\tilde{g}} "Y" ; "R"};
          {\ar_{v} "S" ; "R"};
        \endxy
      \]
    in $\ContrR_{n + 1}$, where, for $0 \leq m < n$,
      \[
        \tilde{u}_m = u_m,
      \]
    and $\tilde{u}_n \colon \tilde{X}_n \rightarrow \tilde{Y}_n$ is defined to be the unique map such that the diagram
      \[
        \xymatrix{
          X^c_{n + 1} \ar@<1ex>[r]^-{\pi_1} \ar@<-1ex>[r]_-{\pi_2} & X_n \ar[r]^q \ar[d]_{u_n} & \tilde{X}_n \ar@{-->}[d]^{\tilde{u}_n}  \\
          & Y_n \ar[r]_-p & \tilde{Y}_n
                 }
      \]
    commutes, where $p$ is the coequaliser map for $\tilde{Y}_n$.  This defines the action of $C_{n + 1}$ on morphisms.
  \end{defn}

  \begin{prop}
    There is an adjunction $C_{n + 1} \ladj D_{n + 1}$.
  \end{prop}

  \begin{proof}
    We first define the unit $\eta \colon 1 \Rightarrow D_{n + 1}C_{n + 1}$, and counit $\epsilon: C_{n + 1}D_{n + 1} \Rightarrow 1$.  Let
       \[
        \xy
          % POINTS
          (0, 0)*+{X}="X";
          (0, -16)*+{S}="S";
          % ARROWS
          {\ar_f "X" ; "S"};
        \endxy
       \]
    be an object of $\ContrR_{n}$.  Applying $D_{n + 1}C_{n + 1}$ gives
       \[
        \xy
          % POINTS
          (0, 0)*+{\tilde{X}}="X";
          (0, -16)*+{S}="S";
          % ARROWS
          {\ar_{\tilde{f}} "X" ; "S"};
        \endxy
       \]
    in $\ContrR_{n + 1}$.  The corresponding component of the unit $\eta$ is given by the map
      \[
        \xy
          % POINTS
          (0, 0)*+{X}="X";
          (16, 0)*+{\tilde{X}}="Xr";
          (0, -16)*+{S}="S";
          (16, -16)*+{S}="Sr";
          % ARROWS
          {\ar^{\eta_X} "X" ; "Xr"};
          {\ar_{f} "X" ; "S"};
          {\ar^{\tilde{f}} "Xr" ; "Sr"};
          {\ar_{\id_S} "S" ; "Sr"};
        \endxy
      \]
    where $\eta_X$ is defined by
      \[
        (\eta_X)_j =
            \left\{
              \begin{array}{ll}
              1_{X_j} & \text{if } j < n, \\
              \text{the coequaliser map } q \colon X_n \rightarrow \tilde{X}_n & \text{if } j = n.
              \end{array}
            \right.
      \]
    Observe that $C_{n + 1}D_{n + 1} = \id$, and if
       \[
        \xy
          % POINTS
          (0, 0)*+{X}="X";
          (0, -16)*+{S}="S";
          % ARROWS
          {\ar_f "X" ; "S"};
        \endxy
       \]
    is in the image of $D_{n + 1}$, $\tilde{X} = X$ and $q = \id_{X}$, so $\eta = \id$.  Furthermore, the counit
      \[
        \epsilon: C_{n + 1}D_{n + 1} \Rightarrow 1
      \]
    is also the identity.  Thus all maps appearing in the diagrams for the triangle identities are identity maps, so both diagrams commute.  Hence there is an adjunction $C_{n + 1} \ladj D_{n + 1}$.
  \end{proof}

  Thus Definitions~\ref{defn:freekcontr} and \ref{defn:freenplusonecontr} give us a dimension-by-dimension construction of the free contraction on an object of $\Rn$.

\subsection{Free magma structure}  \label{subsect:magma}

  We now construct the free $n$-magma on the source of an object of $\Rn$.  As with the construction of the free contraction in the previous subsection, in order to be able to use the construction in the interleaving argument in Subsection~\ref{subsect:interleaving}, we give the construction one dimension at a time.  To do so we define, for each $0 \leq j \leq n$, a category $\MagR_j$, an object of which consists of an object of $\Rn$ in which the source is equipped with a magma structure up to dimension $j$. (Observe that $\MagR_0 = \Rn$.)  We then have, for each $0 < j \leq n$, a forgetful functor
    \[
      N_j \colon \MagR_j \rightarrow \MagR_{j - 1}.
    \]
  We construct a left adjoint to each $N_j$, which freely adds a magma structure at dimension $j$, leaving all other dimensions unchanged.

  In order to define what it means for an $n$-globular set to have a $j$-magma structure, we use the $j$-truncation functor
    \[
      \Tr_j \colon \nGSet \longrightarrow j\text{-}\mathbf{GSet},
    \]
  which forgets the sets of $m$-cells for all $m > j$, and, for $m \leq j$, leaves the sets of $m$-cells and their source and target maps unchanged; the action on maps is defined similarly.

  \begin{defn}
    Define a category $\MagR_j$, with
      \begin{itemize}
        \item objects: an object of $\MagR_j$ consists of an object
       \[
        \xy
          % POINTS
          (0, 0)*+{X}="X";
          (0, -16)*+{S}="S";
          % ARROWS
          {\ar_f "X" ; "S"};
        \endxy
       \]
            in $\Rn$ such that $\Tr_j X$ is a $j$-magma, and $\Tr_j f$ is a map of $j$-magmas.
        \item morphisms: a morphism in $\MagR_j$ is a morphism
     \[
        \xy
          % POINTS
          (0, 0)*+{X}="X";
          (16, 0)*+{Y}="Y";
          (0, -16)*+{S}="S";
          (16, -16)*+{R}="R";
          % ARROWS
          {\ar^u "X" ; "Y"};
          {\ar_f "X" ; "S"};
          {\ar^g "Y" ; "R"};
          {\ar_v "S" ; "R"};
        \endxy
     \]
    in $\Rn$ such that $\Tr_j u$ is a map of $j$-magmas.
      \end{itemize}
  \end{defn}

  We can express the category $\MagR_j$ as a pullback.  For any $m \in \mathbb{N}$ we have a commuting triangle of forgetful functors
    \[
      \xy
        % POINTS
        (0, 0)*+{m\text{-}\Cat}="mCat";
        (32, 0)*+{m\text{-}\mathbf{Mag}}="mMag";
        (16, -12)*+{m\text{-}\mathbf{GSet}}="mGSet";
        % ARROWS
        {\ar^G "mCat" ; "mMag"};
        {\ar_{U_{T_m}} "mCat" ; "mGSet"};
        {\ar^-E "mMag" ; "mGSet"};
      \endxy
    \]
  in $\CAT$.  We can then write $\Mag_j$ as the pullback
    \[
      \xy
        % POINTS
        (0, 0)*+{\MagR_j}="Magj";
        (28, 0)*+{\nGSet \downarrow U_{T}}="Rn";
        (0, -16)*+{j\text{-}\mathbf{Mag} \downarrow G}="jMag";
        (28, -16)*+{j\text{-}\mathbf{GSet} \downarrow U_{T_j}}="Rj";
        % ARROWS
        {\ar "Magj" ; "Rn"};
        {\ar "Magj" ; "jMag"};
        {\ar^{\Tr_j} "Rn" ; "Rj"};
        {\ar_-E "jMag" ; "Rj"};
        % PULLBACK STUFF
        (6,-1)*{}; (6,-5)*{} **\dir{-};
        (2,-5)*{}; (6,-5)*{} **\dir{-};
      \endxy
    \]

  For all $0 < j \leq n$, we have a forgetful functor
    \[
      N_j \colon \MagR_j \rightarrow \MagR_{j - 1},
    \]
  which forgets the composition maps for $j$-cells.  We will define, for each $0 < j \leq n$, a functor
    \[
      M_j \colon \MagR_{j - 1} \rightarrow \MagR_j
    \]
  which freely adds binary composites at dimension $j$, taking an $n$-globular set equipped with a $(j - 1)$-magma structure and adding a magma structure at dimension $j$ to give an $n$-globular set equipped with a $j$-magma structure.  We will then show that the functor $M_j$ is left adjoint to the forgetful functor $N_j$.

  Before defining $M_j$, we first fix some notation that will be used in the construction of the free binary composites.  Let $X$ be an $n$-globular set equipped with a $(j - 1)$-magma structure.  For each $0 \leq p < j$, we can form the set of pairs of $j$-cells that are composable along $p$-cells using the following pullback:
     \[
        \xy
          % POINTS
          (0, 0)*+{X_j \times_{X_p} X_j}="XjXj";
          (20, 0)*+{X_j}="Xr";
          (0, -16)*+{X_j}="Xb";
          (20, -16)*+{X_p.}="Xp";
          % ARROWS
          {\ar "XjXj" ; "Xr"};
          {\ar "XjXj" ; "Xb"};
          {\ar^s "Xr" ; "Xp"};
          {\ar_t "Xb" ; "Xp"};
          % PULLBACK STUFF
          (6,-2)*{}; (6,-6)*{} **\dir{-};
          (2,-6)*{}; (6,-6)*{} **\dir{-};
        \endxy
     \]
  We view $X_j \times_{X_p} X_j$ as the set of freely generated binary composites of $j$-cells along $p$-cells.  We can form the set of freely generated binary composites of $j$-cells along boundaries of all dimensions by taking the coproduct of these sets over $p$.  As the notation will become somewhat complicated in the definition of the left adjoint to $N_j$, we use the following shorthand:
    \[
      X^2_j := \coprod_{0 \leq p < j} X_j \times_{X_p} X_j.
    \]
  This set comes equipped with source and target maps into $X_{j - 1}$, which are defined in analogy with the sources and targets of composites in a magma structure from Definition~\ref{defn:magma}, as follows:
      \begin{itemize}
        \item if $p = m - 1$, given $(a, b) \in X_m \times_{X_p} X_m$,
          \[
            s(a, b) = s(a),
          \]
          \[
            t(a, b) = t(b)
          \]
        \item if $p < m - 1$, given $(a, b) \in X_m \times_{X_p} X_m$,
          \[
            s(a, b) = s(b) \comp^{m - 1}_p s(a),
          \]
          \[
            t(a, b) = t(b) \comp^{m - 1}_p t(a).
          \]
      \end{itemize}

  The set $X^2_j$ contains only binary composites of ``depth $1$''; that is, it contains binary composites of pairs of $j$-cells in $X$, but it does not contain binary composites of binary composites, binary composites of binary composites of binary composites, etc.  In order to obtain these composites of greater ``depth'', which we require in the free magma structure, we must iterate this process.  To do so we define, for each $k \geq 0$, a set $X^{(k)}_j$ of composites of depth at most $k$.  We have inclusion maps
    \[
      X^{(k)}_j \hookrightarrow X^{(k + 1)}_j,
    \]
  so this gives a sequence of sets; we take the colimit of this sequence to obtain the set of freely generated binary composites of all depths.  We now describe and illustrate this iterative process for low depths of composite ($k \leq 2$).

    When $k = 0$, we define
      \[
        X^{(0)}_j = X_j,
      \]
    with source and target maps $s$, $t \colon X^{(0)}_j \rightarrow X_{j - 1}$ given by those in $X$.

    When $k = 1$, we define
      \[
        X_j^{(1)} := X_j + \left(X_j^{(0)}\right)^2 = X_j + X_j^2,
      \]
    where the notation $X_j^2$ is shorthand, as described earlier.  The set $X_j^2$ inherits source and target maps from $X_j^{(0)}$, so we have source and target maps
      \[
        s, t \colon X^{(1)}_j \longrightarrow X_{j - 1}
      \]
    inherited from those for $X_j$ and $X_j^2$.  To see how this gives the set of composites of depth at most $1$, we consider the case $j = 2$.  By ``expanding out'' $X_2^2$, we see that $X_2^{(1)}$ contains the following shapes of composites:
      \[
        \xy
          % EQUATION
          (0, 0.5)*+{X_2^{(1)} =};
          (16, 0)*+{X_2};
          (28, 0)*+{+};
          (46, 0)*+{X_2 \times_{X_0} X_2};
          (64, 0)*+{+};
          (82, 0)*+{X_2 \times_{X_1} X_2};
          % DIAGRAM POINTS
          % GLOBULAR 2-CELL:
          (10,-10)*+{\bullet}="0";
          (16,-10)*+{\Downarrow};
          (22,-10)*+{\bullet}="1";
          % 0-COMP:
          (34,-10)*+{\bullet}="0,0";
          (40,-10)*+{\Downarrow};
          (46,-10)*+{\bullet}="0,1";
          (52,-10)*+{\Downarrow};
          (58,-10)*+{\bullet}="0,2";
          % 1-COMP:
          (76,-10)*+{\bullet}="1,0";
          (82,-13)*+{\Downarrow};
          (82, -7)*+{\Downarrow};
          (88,-10)*+{\bullet}="1,1";
          % ARROWS
          {\ar@/^1pc/ "0" ; "1"};
          {\ar@/_1pc/ "0" ; "1"};
          {\ar@/^1pc/ "0,0" ; "0,1"};
          {\ar@/_1pc/ "0,0" ; "0,1"};
          {\ar@/^1pc/ "0,1" ; "0,2"};
          {\ar@/_1pc/ "0,1" ; "0,2"};
          {\ar@/^1.5pc/ "1,0" ; "1,1"};
          {\ar "1,0" ; "1,1"};
          {\ar@/_1.5pc/ "1,0" ; "1,1"};
        \endxy
      \]

    When $k = 2$, we define
      \[
        X_j^{(2)} := X_j + \left(X_j^{(1)}\right)^2.
      \]
    As in the case $k = 1$, this comes equipped with source and target maps.  In the case $j = 2$, ``expanding out'' $\left(X_j^{(1)}\right)^2$ gives
      \begin{align*}
        X_2^{(2)} & = \;\; X_2 \;\; + \;\; X_2^2 \;\; + \;\; X_2 \times_{X_0} X^2_2 \;\; + \;\; X_2 \times_{X_1} X^2_2 \;\; + \;\;  X^2_2 \times_{X_0} X_2  \\
        & + \;\; X^2_2 \times_{X_1} X_2 \;\; + \;\; X^2_2 \times_{X_0} X^2_2 \;\; + \;\; X^2_2 \times_{X_1} X^2_2.
      \end{align*}
    Thus $X_2^{(2)}$ contains the same shapes of composites that appear in $X_2^{(1)}$, as well as those composites of depth $2$: in $X_2 \times_{X_0} X^2_2$ we have composites of the following shapes:
      \[
        \xy
          % 0-COMP 0-COMP
          (0, 0)*+{\bullet}="2,0";
          (6, 0)*+{\Downarrow};
          (12, 0)*+{\bullet}="2,1";
          (14, 0)*+{\phantom{\bullet}}="2,1p";
          (14, 0)*+{\Bigg(};
          (20, 0)*+{\Downarrow};
          (26, 0)*+{\bullet}="2,2";
          (32, 0)*+{\Downarrow};
          (38, 0)*+{\phantom{\bullet}}="2,3p";
          (38, 0)*+{\Bigg)};
          (40, 0)*+{\bullet}="2,3";
          (46, 0)*+{\text{and}};
          % 1-COMP 0-COMP
          (52, 0)*+{\bullet}="3,0";
          (58, 0)*+{\Downarrow};
          (64, 0)*+{\bullet}="3,1";
          (66, 0)*+{\phantom{\bullet}}="3,1p";
          (66, 0)*+{\Bigg(};
          (72, -3)*+{\Downarrow};
          (72, 3)*+{\Downarrow};
          (78, 0)*+{\phantom{\bullet}}="3,2p";
          (78, 0)*+{\Bigg)};
          (80, 0)*+{\bullet}="3,2";
          (81.5, 0)*+{;};
          % ARROWS
          {\ar@/^1pc/ "2,0" ; "2,1"};
          {\ar@/_1pc/ "2,0" ; "2,1"};
          {\ar@/^1pc/ "2,1p" ; "2,2"};
          {\ar@/_1pc/ "2,1p" ; "2,2"};
          {\ar@/^1pc/ "2,2" ; "2,3p"};
          {\ar@/_1pc/ "2,2" ; "2,3p"};
          {\ar@/^1pc/ "3,0" ; "3,1"};
          {\ar@/_1pc/ "3,0" ; "3,1"};
          {\ar@/^1.5pc/ "3,1p" ; "3,2p"};
          {\ar "3,1p" ; "3,2p"};
          {\ar@/_1.5pc/ "3,1p" ; "3,2p"};
        \endxy
      \]
    in $X_2 \times_{X_1} X^2_2$ we have composites of the following shape:
      \[
        \xy
          % POINTS
          (0, 0)*+{\bullet}="0,0";
          (6, 0)*+{\Downarrow};
          (12, 0)*+{\bullet}="1,0";
          (6, -7)*+{\comp};
          (0,-16)*+{\bullet}="0,1";
          (6,-19)*+{\Downarrow};
          (6, -13)*+{\Downarrow};
          (12,-16)*+{\bullet}="1,1";
          (13.5,-16)*+{;};
          % ARROWS
          {\ar@/^1pc/ "0,0" ; "1,0"};
          {\ar@/_1pc/ "0,0" ; "1,0"};
          {\ar@/^1.5pc/ "0,1" ; "1,1"};
          {\ar "0,1" ; "1,1"};
          {\ar@/_1.5pc/ "0,1" ; "1,1"};
        \endxy
      \]
    the shapes of composites in $X^2_2 \times_{X_0} X_2$ and $X^2_2 \times_{X_1} X_2$ are similar to those above; in $X^2_2 \times_{X_0} X^2_2$ we have composites of the following shapes:
      \[
        \xy
          % 0-COMP 0-COMP
          (0, 0)*+{\bullet}="2,-1";
          (2, 0)*+{\phantom{\bullet}}="2,-1p";
          (2, 0)*+{\Bigg(};
          (8, 0)*+{\Downarrow};
          (14, 0)*+{\bullet}="2,0";
          (20, 0)*+{\Downarrow};
          (26, 0)*+{\phantom{\bullet}}="2,1q";
          (26, 0)*+{\Bigg)};
          (28, 0)*+{\bullet}="2,1";
          (30, 0)*+{\phantom{\bullet}}="2,1p";
          (30, 0)*+{\Bigg(};
          (36, 0)*+{\Downarrow};
          (42, 0)*+{\bullet}="2,2";
          (48, 0)*+{\Downarrow};
          (54, 0)*+{\phantom{\bullet}}="2,3p";
          (54, 0)*+{\Bigg)};
          (56, 0)*+{\bullet}="2,3";
          % ARROWS
          {\ar@/^1pc/ "2,-1p" ; "2,0"};
          {\ar@/_1pc/ "2,-1p" ; "2,0"};
          {\ar@/^1pc/ "2,0" ; "2,1q"};
          {\ar@/_1pc/ "2,0" ; "2,1q"};
          {\ar@/^1pc/ "2,1p" ; "2,2"};
          {\ar@/_1pc/ "2,1p" ; "2,2"};
          {\ar@/^1pc/ "2,2" ; "2,3p"};
          {\ar@/_1pc/ "2,2" ; "2,3p"};
        \endxy
      \]
    and also
      \[
        \xy
          % POINTS
          (0, 0)*+{\bullet}="2,-1";
          (2, 0)*+{\phantom{\bullet}}="2,-1p";
          (2, 0)*+{\Bigg(};
          (8, 0)*+{\Downarrow};
          (14, 0)*+{\bullet}="2,0";
          (20, 0)*+{\Downarrow};
          (26, 0)*+{\phantom{\bullet}}="2,1q";
          (26, 0)*+{\Bigg)};
          (28, 0)*+{\bullet}="2,1";
          (34, 3)*+{\Downarrow};
          (34, -3)*+{\Downarrow};
          (40, 0)*+{\bullet}="2,2";
          (46, 0)*+{\text{and}};
          (52, 0)*+{\bullet}="3,0";
          (58, 3)*+{\Downarrow};
          (58, -3)*+{\Downarrow};
          (64, 0)*+{\bullet}="3,1";
          (66, 0)*+{\phantom{\bullet}}="3,1p";
          (66, 0)*+{\Bigg(};
          (72, 0)*+{\Downarrow};
          (78, 0)*+{\bullet}="3,2";
          (84, 0)*+{\Downarrow};
          (90, 0)*+{\phantom{\bullet}}="3,3q";
          (90, 0)*+{\Bigg)};
          (92, 0)*+{\bullet}="3,3";
          (93.5, 0)*+{;};
          % ARROWS
          {\ar@/^1pc/ "2,-1p" ; "2,0"};
          {\ar@/_1pc/ "2,-1p" ; "2,0"};
          {\ar@/^1pc/ "2,0" ; "2,1q"};
          {\ar@/_1pc/ "2,0" ; "2,1q"};
          {\ar@/^1.5pc/ "2,1" ; "2,2"};
          {\ar "2,1" ; "2,2"};
          {\ar@/_1.5pc/ "2,1" ; "2,2"};
          {\ar@/^1.5pc/ "3,0" ; "3,1"};
          {\ar "3,0" ; "3,1"};
          {\ar@/_1.5pc/ "3,0" ; "3,1"};
          {\ar@/^1pc/ "3,1p" ; "3,2"};
          {\ar@/_1pc/ "3,1p" ; "3,2"};
          {\ar@/^1pc/ "3,2" ; "3,3q"};
          {\ar@/_1pc/ "3,2" ; "3,3q"};
        \endxy
      \]
    and finally, in $X^2_2 \times_{X_1} X^2_2$ we have composites of the following shapes:
      \[
        \xy
          % POINTS
          (0, 0)*+{\bullet}="0,0";
          (6, 3)*+{\Downarrow};
          (6, -3)*+{\Downarrow};
          (12, 0)*+{\bullet}="1,0";
          (6, -8)*+{\comp};
          (0,-16)*+{\bullet}="0,1";
          (6,-19)*+{\Downarrow};
          (6, -13)*+{\Downarrow};
          (12,-16)*+{\bullet}="1,1";
          (18, -8)*+{\text{and}};
          (24, -4)*+{\bullet}="2,0";
          (24, -12)*+{\bullet}="2,1";
          (30, -1)*+{\Downarrow};
          (30, -15)*+{\Downarrow};
          (36, -4)*+{\bullet}="3,0";
          (36, -8)*+{\comp};
          (36, -12)*+{\bullet}="3,1";
          (42, -1)*+{\Downarrow};
          (42, -15)*+{\Downarrow};
          (48, -4)*+{\bullet}="4,0";
          (48, -12)*+{\bullet}="4,1";
          (49.5, -13)*+{.};
          % ARROWS
          {\ar@/^1.5pc/ "0,0" ; "1,0"};
          {\ar "0,0" ; "1,0"};
          {\ar@/_1.5pc/ "0,0" ; "1,0"};
          {\ar@/^1.5pc/ "0,1" ; "1,1"};
          {\ar "0,1" ; "1,1"};
          {\ar@/_1.5pc/ "0,1" ; "1,1"};
          {\ar@/^1.5pc/ "2,0" ; "3,0"};
          {\ar "2,0" ; "3,0"};
          {\ar@/_1.5pc/ "2,1" ; "3,1"};
          {\ar "2,1" ; "3,1"};
          {\ar@/^1.5pc/ "3,0" ; "4,0"};
          {\ar "3,0" ; "4,0"};
          {\ar@/_1.5pc/ "3,1" ; "4,1"};
          {\ar "3,1" ; "4,1"};
        \endxy
      \]
    Thus $X_2^{(2)}$ contains all binary composites of $2$-cells of depth at most $2$.

  Since the construction of the the free $j$-magma structure consists of taking pullbacks and filtered colimits of sets, in order to define the composition maps at dimension $j$ we require the following lemma due to Mac Lane \cite[Theorem~IX.2.1]{Mac98}, which states that finite limits commute with filtered colimits in $\Set$.  Note that this theorem still holds if $\Set$ is replaced by any locally finitely presentable category; see \cite[Proposition 1.59]{AR94}.

  \begin{lemma}[Mac Lane]  \label{lem:maclane}
    Let $\mathbb{I}$ be a finite category, and let $\mathbb{J}$ be a small, filtered category.  Then for any bifunctor
      \[
        F \colon \mathbb{I} \times \mathbb{J} \rightarrow \Set
      \]
    the canonical arrow
      \[
        \colim_{j \in \mathbb{J}} \lim_{i \in \mathbb{I}} F(i, j) \longrightarrow \lim_{i \in \mathbb{I}} \colim_{j \in \mathbb{J}} F(i, j)
      \]
    is an isomorphism.
  \end{lemma}

  We now define a putative left adjoint $M_j$ to the functor $N_j$; we will then prove that this functor is left adjoint to $N_j$ in Proposition~\ref{prop:MadjN}.

  \begin{defn}  \label{defn:Mj}
    For each $0 < j \leq n$, we define a functor
      \[
        M_j \colon \MagR_{j - 1} \rightarrow \MagR_{j}.
      \]
    We begin by giving the action of $M_j$ on objects.  Let
       \[
        \xy
          % POINTS
          (0, 0)*+{X}="X";
          (0, -16)*+{S}="S";
          % ARROWS
          {\ar_f "X" ; "S"};
        \endxy
       \]
    be an object of $\MagR_{j - 1}$.  We will define an object
       \[
        \xy
          % POINTS
          (0, 0)*+{\hatX}="X";
          (0, -16)*+{S}="S";
          % ARROWS
          {\ar_{\hat{f}} "X" ; "S"};
        \endxy
       \]
    of $\MagR_{j}$, where $\hatX$ differs from $X$ only at dimension $j$.  The set $\hatX_j$ of $j$-cells of $\hatX$ is the set of freely generated binary composites of $j$-cells of $X$.  We define this as the colimit of a sequence of sets $X^{(k)}_j$, where $X^{(k)}_j$ is the set of freely generated binary composites of $j$-cells of $X$ of depth at most $k$.  We define $X^{(k)}_j$ by induction over $k$, as follows: when $k = 0$, define
      \[
        X^{(0)}_j = X_j,
      \]
    with source and target maps $s$, $t \colon X^{(0)}_j \rightarrow X_{j - 1}$ given by those in $X$.  Now suppose that $k > 0$ and we have defined $X^{(k - 1)}_j$, equipped with source and target maps
      \[
        s, t \colon X^{(k - 1)}_j \longrightarrow X_{j - 1}.
      \]
    We define $X^{(k)}_j$ by
      \[
        X_j^{(k)} := X_j + \left(X_j^{(k - 1)}\right)^2.
      \]
    Recall that the notation used above is shorthand, defined by
      \[
        \left(X_j^{(k - 1)}\right)^2 := \coprod_{0 \leq p < j} X_j^{(k - 1)} \times_{X_{p}} X_j^{(k - 1)},
      \]
    and that this set inherits source and target maps from $X_j^{(k - 1)}$.  Thus we have source and target maps
      \[
        s, t \colon X^{(k)}_j \longrightarrow X_{j - 1}
      \]
    inherited from those for $X_j$ and $\left(X_j^{(k - 1)}\right)^2$.

    For each $k \geq 0$, we define a map
      \[
        i^{(k)} \colon X_j^{(k)} \rightarrow X_j^{(k + 1)},
      \]
    which includes the freely generated composites in $X_j^{(k)}$ (those of depth at most $k$) into the set $X_j^{(k + 1)}$ (which contains composites of depth at most $k + 1$), and leaves the generating cells unchanged.  The maps $i^{(k)}$ are defined by induction over $k$, as follows:

      \begin{itemize}
        \item for $k = 0$, $i^{(0)}$ is the coprojection map
          \[
            i^{(0)} \colon X_j \rightarrow X_j + X_j^2;
          \]
        \item for $k \geq 1$, suppose we have defined $i^{(k - 1)} \colon X_j^{(k - 1)} \rightarrow X_j^{(k)}$.  We define $i^{(k)}$ to be the map
              \[
                i^{(k)} := 1_{X_j} + \coprod_{0 \leq p < j} \left( i^{(k - 1)}, i^{(k - 1)} \right) \colon X_j + \left(X_j^{(k - 1)}\right)^2 \rightarrow X_j + \left(X_j^{(k)}\right)^2.
              \]
      \end{itemize}
    These sets and maps give us a diagram
      \[
        \xymatrix{
          X_j^{(0)} \ar[r]^{i^{(0)}} & X_j^{(1)} \ar[r]^{i^{(1)}} & X_j^{(2)} \ar[r]^{i^{(2)}} & X_j^{(3)} \ar[r]^{i^{(3)}} & \dotsc
                 }
      \]
    in $\Set$; we then define
      \[
        \hatX_j := \colim_{k \geq 0} X_j^{(k)}.
      \]
    For $m \neq j$, we define
      \[
        \hatX_m := X_m.
      \]
    For $m \neq j$, $j + 1$, the source and target maps
      \[
        s, t \colon \hatX_m \rightarrow \hatX_{m - 1}
      \]
    are those inherited from $X$.  Now write $c^{(k)}_j \colon X^{(k)}_j \rightarrow \hatX_j$ for the coprojection maps.  The source and target maps for $m = j + 1$ are given by the composites
      \[
        \xy
          % POINTS
          (0, 0)*+{\hatX_{j + 1} = X_{j + 1}}="XX";
          (26, 0)*+{X_{j}}="Xj";
          (42, 0)*+{\hatX_{j}}="hatXj";
          % ARROWS
          {\ar^-{s} "XX" ; "Xj"};
          {\ar^-{c^{(0)}_j} "Xj" ; "hatXj"};
        \endxy
      \]
    and
      \[
        \xy
          % POINTS
          (0, 0)*+{\hatX_{j + 1} = X_{j + 1}}="XX";
          (26, 0)*+{X_{j}}="Xj";
          (42, 0)*+{\hatX_{j}}="hatXj";
          % ARROWS
          {\ar^-{t} "XX" ; "Xj"};
          {\ar^-{c^{(0)}_j} "Xj" ; "hatXj"};
        \endxy
      \]
    respectively.  To define the source and target maps for $m = j$, recall that, for each $k$, we have source and target maps $s$, $t \colon X_j^{(k)} \rightarrow X_{j - 1}$; we define $s$, $t \colon \hatX_j \rightarrow X_{j - 1}$ to be the unique maps induced by the colimit defining $\hatX_j$ that make, for all $k \geq 1$, the diagrams
      \[
        \xy
          % POINTS
          (0, 0)*+{X^{(k)}_j}="Xkjl";
          (16, 0)*+{\hatX_j}="hatXjl";
          (16, -16)*+{X_{j - 1}}="Xl";
          (40, 0)*+{X^{(k)}_j}="Xkjr";
          (56, 0)*+{\hatX_j}="hatXjr";
          (56, -16)*+{X_{j - 1}}="Xr";
          % ARROWS
          {\ar^-{c^{(k)}_j} "Xkjl" ; "hatXjl"};
          {\ar_s "Xkjl" ; "Xl"};
          {\ar@{-->}^s "hatXjl" ; "Xl"};
          {\ar^-{c^{(k)}_j} "Xkjr" ; "hatXjr"};
          {\ar_t "Xkjr" ; "Xr"};
          {\ar@{-->}^t "hatXjr" ; "Xr"};
        \endxy
      \]
    commute respectively.

    We now define the $j$-magma structure on $\hatX$.  For all $m < j$, and for all $0 \leq p < m$, the composition map
      \[
        \comp^m_p \colon \hatX_m \times_{\hatX_p} \hatX_m = X_m \times_{X_p} X_m \rightarrow X_m
      \]
    is the corresponding composition map from the $(j - 1)$-magma structure on $X$.  To define the composition map $\comp^j_p$ for $0 \leq p < j$, we begin by observing that, by Lemma~\ref{lem:maclane}, we have an isomorphism
      \[
        \underset{k, l \geq 0}\colim\Big(X_j^{(k)} \times_{X_p} X_j^{(l)}\Big) \iso \Big(\underset{k \geq 0}\colim X_j^{(k)} \Big) \times_{X_p} \Big(\underset{l \geq 0}\colim X_j^{(l)} \Big) = \hatX_j \times_{X_p} \hatX_j.
      \]
    Thus, to define the composition maps at dimension $j$, we define, for each $k$, $l > 0$, $0 \leq p < j$, a map
      \[
        \comp^j_p \colon X_j^{(k)} \times_{X_p} X_j^{(l)} \rightarrow \hatX_j.
      \]
    To do so, observe that, in the case $k = l$, the source of the composition map above includes in $X^{(k + 1)}_j$, which in turn includes in $\hatX_j$; thus in this case we define the composition map to be the composite:
      \[
        \xymatrix{
           X_j^{(k)} \times_{X_p} X_j^{(k)} \ar@{^(->}[r] & X_j^{(k + 1)} \ar[r]^-{c_j^{(k + 1)}} & \hatX_j.
                 }
      \]
      % Replace with xy?
    Now suppose that $k < l$; in this case we first include the source of the composition map in
      \[
        X_j^{(l)} \times_{X_p} X_j^{(l)},
      \]
    and we can then follow the same method as for $k = l$.  Write
      \[
        i^{(k, l)} := i^{(l)} \comp i^{(l - 1)} \comp \dotsb \comp i^{(k)} \colon X_j^{(k)} \longrightarrow X_j^{(l)},
      \]
     and define $\comp^j_p$ to be the composite
      \[
        \xy
          % POINTS
          (0, 0)*+{X_j^{(k)} \times_{X_p} X_j^{(l)}}="XkXl";
          (38, 0)*+{X_j^{(l)} \times_{X_p} X_j^{(l)}}="XlXl";
          (64, 0)*+{X_j^{(l + 1)}}="Xl";
          (82, 0)*+{\hatX_j,}="Xj";
          % ARROWS
          {\ar^-{(i^{(k, l)}, \id)} "XkXl" ; "XlXl"};
          {\ar@{^(->} "XlXl" ; "Xl"};
          {\ar^-{c^{(l + 1)}_j} "Xl" ; "Xj"};
        \endxy
      \]
    where the second map is the coprojection into the coproduct defining $X_j^{(l + 1)}$.  Similarly, for $l > k$, we define $\comp^j_p$ to be the composite
      \[
        \xy
          % POINTS
          (0, 0)*+{X_j^{(k)} \times_{X_p} X_j^{(l)}}="XkXl";
          (38, 0)*+{X_j^{(k)} \times_{X_p} X_j^{(k)}}="XkXk";
          (64, 0)*+{X_j^{(k + 1)}}="Xk";
          (82, 0)*+{\hatX_j,}="Xj";
          % ARROWS
          {\ar^-{(\id, i^{(k, l)})} "XkXl" ; "XkXk"};
          {\ar@{^(->} "XkXk" ; "Xk"};
          {\ar^-{c^{(k + 1)}_j} "Xk" ; "Xj"};
        \endxy
      \]
    Then $\comp^j_p \colon \hatX_j \times_{X_p} \hatX_j \rightarrow \hatX_j$ is defined to be the unique map induced by universal property of
      \[
        \hatX_j \times_{X_p} \hatX_j
      \]
    as a colimit (using Lemma~\ref{lem:maclane}) such that, for all $k$, $l > 0$, the diagram
      \[
        \xy
          % POINTS
          (0, 0)*+{X_j^{(k)} \times_{X_p} X_j^{(l)}}="XkXl";
          (38, 0)*+{\hatX_j \times_{X_p} \hatX_j}="hatXX";
          (38, -16)*+{\hatX_j}="hatX";
          % ARROWS
          {\ar^-{\left(c_j^{(k)}, c_j^{(l)}\right)} "XkXl" ; "hatXX"};
          {\ar_-{\comp^j_p} "XkXl" ; "hatX"};
          {\ar@{-->}^-{\comp^j_p} "hatXX" ; "hatX"};
        \endxy
      \]
    commutes.  This defines a $j$-magma structure on $\hatX$.

    We now define the map $\hat{f} \colon \hatX \rightarrow S$.  At dimension $j$, $\hat{f}$ acts on a freely generated composite in $\hatX_j$ by first applying $f$ to each individual generating $j$-cell in the composite, then evaluating this composite via the magma structure on $S$; at all other dimensions it is the same as the map $f$.

    For $m \neq j$, define
      \[
        \hat{f}_m = f_m \colon \hatX_m = X_m \rightarrow S_m.
      \]
    To define $\hat{f}_m$ for $m = j$, we first define, for each $k \geq 0$, a map
      \[
        f_j^{(k)} \colon X_j^{(k)} \rightarrow S_j.
      \]
    When $k = 0$, define
      \[
        f_j^{(k)} = f_j \colon X_j^{(0)} \rightarrow S_j.
      \]
    Now let $k \geq 1$ and suppose we have defined the map
      \[
        f_j^{(k - 1)} \colon X_{j}^{(k - 1)} \rightarrow S_j;
      \]
    we define the map
      \[
        f_j^{(k)} \colon X_{j}^{(k)} \rightarrow S_j
      \]
    as follows:  for each $0 \leq p < j$ there is a map
      \[
        \left(f_j^{(k - 1)}, f_j^{(k - 1)}\right) \colon X_j^{(k - 1)} \times_{X_p} X_j^{(k - 1)} \rightarrow S_j \times_{S_p} S_j
      \]
    induced by the universal property of $S_j \times_{S_p} S_j$.  We compose each of these with the composition map $\comp^j_p$, and define $f_j^{(k)} \colon X_j^{(k)} \rightarrow S_j$ to be a coproduct of these composites, as follows:
      \begin{align*}
        f_j^{(k)} := & f_j^{(0)} + \coprod_{0 \leq p < j} \left((\comp^j_p) \comp \left(f_j^{(k - 1)}, f_j^{(k - 1)}\right)\right) \colon \\
        & X_j^{(k)} = X_j + \coprod_{0 \leq p < j} X_j^{(k - 1)} \times_{X_p} X_j^{(k - 1)} \rightarrow S_j.
      \end{align*}
    We then define $\hat{f}_j$ to be the unique map such that, for all $k \geq 1$, the diagram
      \[
        \xy
          % POINTS
          (0, 0)*+{X^{(k)}_j}="Xkjl";
          (16, 0)*+{\hatX_j}="hatXjl";
          (16, -16)*+{S_{j}}="Sl";
          % ARROWS
          {\ar^-{c^{(k)}_j} "Xkjl" ; "hatXjl"};
          {\ar_{f^{(k)}_j} "Xkjl" ; "Sl"};
          {\ar@{-->}^{\hat{f}_j} "hatXjl" ; "Sl"};
        \endxy
      \]
    commutes.

    Thus we have defined an object
       \[
        \xy
          % POINTS
          (0, 0)*+{\hatX}="X";
          (0, -16)*+{S}="S";
          % ARROWS
          {\ar_{\hat{f}} "X" ; "S"};
        \endxy
       \]
    of $\MagR_{j}$; this gives the action of $M_j$ on objects.

    We now give the action of $M_j$ on morphisms.  Let
      \[
        \xy
          % POINTS
          (0, 0)*+{X}="X";
          (16, 0)*+{Y}="Y";
          (0, -16)*+{S}="S";
          (16, -16)*+{R}="R";
          % ARROWS
          {\ar^u "X" ; "Y"};
          {\ar_f "X" ; "S"};
          {\ar^g "Y" ; "R"};
          {\ar_v "S" ; "R"};
        \endxy
      \]
    be a morphism in $\MagR_{j - 1}$.  We define a morphism
      \[
        \xy
          % POINTS
          (0, 0)*+{\hat{X}}="X";
          (16, 0)*+{\hat{Y}}="Y";
          (0, -16)*+{S}="S";
          (16, -16)*+{R}="R";
          % ARROWS
          {\ar^{\hat{u}} "X" ; "Y"};
          {\ar_{\hat{f}} "X" ; "S"};
          {\ar^{\hat{g}} "Y" ; "R"};
          {\ar_{v} "S" ; "R"};
        \endxy
      \]
    in $\MagR_{j}$.  At dimension $j$, the map $\hat{u}$ acts on a freely generated composite in $\hat{X}$ by applying $u$ to each individual generating $j$-cell in the composite, thus giving a freely generated composite of $j$-cells in $\hat{Y}_j$; at all other dimensions it is the same as the map $u$.

    For $m \neq j$, we define $\hat{u}_m = u_m$.  To define $\hat{u}_m$ for $m = j$, first we define, for each $k \geq 1$, a map
      \[
        u^{(k)}_j \colon X^{(k)}_j \rightarrow Y^{(k)}_j.
      \]
    When $k = 0$, define
      \[
        u^{(1)}_j := u_j \colon X^{(1)}_j \rightarrow Y^{(1)}_j.
      \]
    Now let $k \geq 1$ and suppose we have defined
      \[
        u^{(k - 1)}_j := u_j \colon X^{(k - 1)}_j \rightarrow Y^{(k - 1)}_j;
      \]
    we define $u^{(k)}_j$ as follows:
      \[
        u^{(k)}_j := u^{(k - 1)}_j + \coprod_{0 \leq p < j} \left(u^{(k - 1)}_j, u^{(k - 1)}_j\right) \colon X^{(k)}_j \rightarrow Y^{(k)}_j.
      \]
    We then define $\hat{u}_j$ to be the unique map such that, for all $k \geq 1$, the diagram
      \[
        \xy
          % POINTS
          (0, 0)*+{X^{(k)}_j}="Xkj";
          (16, 0)*+{\hatX_j}="Xj";
          (0, -16)*+{Y^{(k)}_j}="Ykj";
          (16, -16)*+{\hat{Y}_j}="Yj";
          % ARROWS
          {\ar^{c^{(k)}_j} "Xkj" ; "Xj"};
          {\ar_{u^{(k)}_j} "Xkj" ; "Ykj"};
          {\ar@{-->}^{\hat{u}_j} "Xj" ; "Yj"};
          {\ar_{c^{(k)}_j} "Ykj" ; "Yj"};
        \endxy
      \]
    commutes.  This gives the action of the functor $M_j$ on morphisms.
  \end{defn}

  \begin{prop}  \label{prop:MadjN}
    For all $0 < j \leq n$, there is an adjunction $M_j \ladj N_j$.
  \end{prop}

  \begin{proof}
    We first define the unit $\eta \colon 1 \Rightarrow N_j M_j$ and counit $\epsilon \colon M_j N_j \Rightarrow 1$.

    Let
       \[
        \xy
          % POINTS
          (0, 0)*+{X}="X";
          (0, -16)*+{S}="S";
          % ARROWS
          {\ar_f "X" ; "S"};
        \endxy
       \]
    be an object in $\MagR_{j - 1}$.  Then the corresponding component of the unit map $\eta$ is
       \[
        \xy
          % POINTS
          (0, 0)*+{X}="X";
          (16, 0)*+{\hatX}="hatX";
          (0, -16)*+{S}="S";
          (16, -16)*+{S,}="Sr";
          % ARROWS
          {\ar_f "X" ; "S"};
          {\ar^{\eta_X} "X" ; "hatX"};
          {\ar^{\hat{f}} "hatX" ; "Sr"};
          {\ar_{\id_S} "S" ; "Sr"};
        \endxy
       \]
    where $\eta_X$ is defined by
      \[
        (\eta_X)_k =
            \left\{
              \begin{array}{ll}
              \id_{X_k} & \text{if } k \neq j, \\
              \text{the coprojection map } c^{(0)}_j \colon X_j \rightarrow \hatX_j & \text{if } k = j.
              \end{array}
            \right.
      \]
    Naturality of $\eta$ is immediate at dimensions $k \neq j$, and follows from the definition of the action of $M_j$ on maps when $k = j$.

    Now let
       \[
        \xy
          % POINTS
          (0, 0)*+{X}="X";
          (0, -16)*+{S}="S";
          % ARROWS
          {\ar_f "X" ; "S"};
        \endxy
       \]
    be an object in $\MagR_{j}$.  The corresponding component of the counit map $\epsilon$ should be a map of the form
       \[
        \xy
          % POINTS
          (0, 0)*+{\hatX}="X";
          (16, 0)*+{X}="hatX";
          (0, -16)*+{S}="S";
          (16, -16)*+{S.}="Sr";
          % ARROWS
          {\ar_{\hat{f}} "X" ; "S"};
          {\ar^{\eta_X} "X" ; "hatX"};
          {\ar^{f} "hatX" ; "Sr"};
          {\ar_{\id_S} "S" ; "Sr"};
        \endxy
       \]
    To define the map $\epsilon_X$, recall that
      \[
        \hatX_j := \colim_{k \geq 0} X^{(k)}_j;
      \]
    thus for each $k \geq 0$, we define a map
      \[
        \epsilon^{(k)}_X \colon X^{(k)}_j \rightarrow X_j,
      \]
    by induction over $k$.

    When $k = 0$, $X^{(k)}_j = X_j$, and we define
      \[
        \epsilon^{(0)}_X := \id_{X_j}.
      \]
    Now suppose we have defined $\epsilon^{(k)}_X$ for some $k = l$.  Recall that
      \[
        X^{(l + 1)}_j := X_j + \coprod_{0 \leq p < j} X^{(l)}_j \times_{X_p} X^{(l)}_j.
      \]
    We define $\epsilon^{(l + 1)}_X$ by
      \[
        \epsilon^{(l + 1)}_X := \id_{X_j} + \coprod_{0 \leq p < j} \left( (\comp^j_p) \comp \big( \epsilon^{(l)}_X, \epsilon^{(l)}_X \big) \right) \colon X^{(l + 1)}_j \longrightarrow X_j,
      \]
    where $\comp^j_p$ is the composition map from the $j$-magma structure on $X$.  We then define $(\epsilon_X)_j \colon \hatX_j \rightarrow X_j$ to be the unique map such that, for all $k \geq 0$, the diagram
      \[
        \xy
          % POINTS
          (0, 0)*+{X^{(k)}_j}="Xkjl";
          (16, 0)*+{\hatX_j}="hatXjl";
          (16, -16)*+{X_{j}}="Xl";
          % ARROWS
          {\ar^-{c^{(k)}_j} "Xkjl" ; "hatXjl"};
          {\ar_{\epsilon^{(k)}_X} "Xkjl" ; "Xl"};
          {\ar@{-->}^{(\epsilon_X)_j} "hatXjl" ; "Xl"};
        \endxy
      \]
    commutes.  This defines the counit $\epsilon \colon M_j N_j \Rightarrow 1$.  We now check naturality of $\epsilon$.  Let
      \[
        \xy
          % POINTS
          (0, 0)*+{X}="X";
          (16, 0)*+{Y}="Y";
          (0, -16)*+{S}="S";
          (16, -16)*+{R}="R";
          % ARROWS
          {\ar^u "X" ; "Y"};
          {\ar_f "X" ; "S"};
          {\ar^g "Y" ; "R"};
          {\ar_v "S" ; "R"};
        \endxy
      \]
    be a morphism in $\MagR_{j}$; since the components of $\epsilon$ are identities on strict $n$-category parts, and at all dimensions other than dimension $j$, to show that $\epsilon$ is natural we need only show that the diagram
      \[
        \xy
          % POINTS
          (0, 0)*+{\hatX_j}="hatX";
          (16, 0)*+{\hat{Y}_j}="hatY";
          (0, -16)*+{X_j}="X";
          (16, -16)*+{Y_j}="Y";
          % ARROWS
          {\ar^{\hat{u}_j} "hatX" ; "hatY"};
          {\ar_{(\epsilon_X)_j} "hatX" ; "X"};
          {\ar^{(\epsilon_Y)_j} "hatY" ; "Y"};
          {\ar_{u_j} "X" ; "Y"};
        \endxy
      \]
    commutes.  By definition of $\hatX$ as a colimit, this diagram commutes if, for each $k \geq 0$, the diagram
      \[
        \xy
          % POINTS
          (0, 0)*+{X^{(k)}_j}="hatX";
          (16, 0)*+{Y^{(k)}_j}="hatY";
          (0, -16)*+{X_j}="X";
          (16, -16)*+{Y_j}="Y";
          % ARROWS
          {\ar^{u^{(k)}_j} "hatX" ; "hatY"};
          {\ar_{\epsilon^{(k)}_X} "hatX" ; "X"};
          {\ar^{\epsilon^{(k)}_Y} "hatY" ; "Y"};
          {\ar_{u_j} "X" ; "Y"};
        \endxy
      \]
    commutes; we prove this by induction over $k$.  It is immediate when $k = 0$, since $\epsilon^{(0)}_X = \id_{X_j}$ and $\epsilon^{(0)}_Y = \id_{Y_j}$.  Now suppose we have shown that the diagram commutes for some $k = l$; then we have
      \begin{align*}
        u \comp \epsilon^{(l + 1)}_X
        & = u_j + \coprod_{0 \leq p < j} \left(u_j \comp (\comp^j_p) \comp \big( \epsilon^{(l)}_X, \epsilon^{(l)}_X \big) \right)  \\
        & = u_j + \coprod_{0 \leq p < j} \left((\comp^j_p) \comp \big( u_j \epsilon^{(l)}_X, u_j \epsilon^{(l)}_X \big) \right)  \\
        & = u_j + \coprod_{0 \leq p < j} \left((\comp^j_p) \comp \big( \epsilon^{(l)}_X u^{(l)}_j, \epsilon^{(l)}_X u^{(l)}_j \big) \right) = \epsilon^{(l + 1)}_j u^{(l + 1)}_j,
      \end{align*}
    so the diagram commutes for $k = l + 1$.  Thus, by induction, the diagram commutes for all $k \geq 0$.  Hence $\epsilon$ is natural.

    We now check that $\eta$ and $\epsilon$ satisfy the triangle identities, i.e. that the diagrams
        \[
          \xymatrix{
            N_k \ar[rr]^{\eta N_k} \ar[rrd]_1 && N_kM_kN_k \ar[d]^{N_k\epsilon} & M_k \ar[rr]^{M_k\eta} \ar[rrd]_1 && M_kN_kM_k \ar[d]^{\epsilon M_k}  \\
            && N_k, & && M_k
                   }
        \]
    commute.  In all of the natural transformations in these diagrams, the components on strict $n$-category parts are all identities, so to show that the diagrams commute we need only consider the components on underlying $n$-globular sets.  Since the components of the maps of $n$-globular sets are identities at every dimension except dimension $j$, we need only check that the corresponding diagrams of maps of sets of $j$-cells commute.

    For the first triangle identity, let
       \[
        \xy
          % POINTS
          (0, 0)*+{X}="X";
          (0, -16)*+{S}="S";
          % ARROWS
          {\ar_f "X" ; "S"};
        \endxy
       \]
    be an object of $\MagR_j$.  Then the diagram
      \[
        \xy
          % POINTS
          (0, 0)*+{X_j}="Xkjl";
          (24, 0)*+{\hatX_j}="hatXjl";
          (24, -16)*+{X_{j}}="Xl";
          % ARROWS
          {\ar^-{(\eta_X)_j = c^{(0)}_j} "Xkjl" ; "hatXjl"};
          {\ar_{\epsilon^{(0)}_j} "Xkjl" ; "Xl"};
          {\ar^{(\epsilon_X)_j} "hatXjl" ; "Xl"};
        \endxy
      \]
    commutes by the universal property of $(\epsilon_X)_j$, so this triangle identity is satisfied.

    Similarly, for the second triangle identity, let
       \[
        \xy
          % POINTS
          (0, 0)*+{X}="X";
          (0, -16)*+{S}="S";
          % ARROWS
          {\ar_f "X" ; "S"};
        \endxy
       \]
    be an object of $\MagR_{j - 1}$.  Then the diagram
      \[
        \xy
          % POINTS
          (0, 0)*+{\hatX_j}="Xkjl";
          (24, 0)*+{\skew{2.9}\widehat{\widehat{X}}_j}="hatXjl";
          (24, -16)*+{\hatX_{j}}="Xl";
          % ARROWS
          {\ar^-{(\eta_{\hatX})_j = c^{(0)}_j} "Xkjl" ; "hatXjl"};
          {\ar_{\epsilon^{(0)}_j} "Xkjl" ; "Xl"};
          {\ar^{(\epsilon_{\hatX})_j} "hatXjl" ; "Xl"};
        \endxy
      \]
    commutes by the universal property of $(\epsilon_{\hatX})_j$, so this triangle identity is satisfied.

    Thus we have an adjunction $M_j \ladj N_j$, as required.
  \end{proof}

\subsection{Interleaving the contraction and magma structures}  \label{subsect:interleaving}

  We now explain the interleaving argument and show that we can interleave the constructions of Subsections~\ref{subsect:contr} and \ref{subsect:magma} to give a construction of the left adjoint to the functor
    \[
      W \colon \Qn \rightarrow \Rn.
    \]
  To do so we add the contraction and magma structures one dimension at a time, starting with dimension $1$ and working upwards.  At dimension $m$ we first add free contraction cells using the functor $C_m$, then add free composites using the functor $M_m$, and then move up to the next dimension.  Finally, we add ``contraction $(n + 1)$-cells'' using the functor $C_{n + 1}$, which identifies the appropriate cells at dimension $n$.  Note that the method we use very closely follows the method used by Cheng in \cite{Che10}.

  This construction is possible because of the dimensional dependencies of the functors $C_k$ and $M_j$ defined in Subsections~\ref{subsect:contr} and \ref{subsect:contr};  the contraction $k$-cells added by $C_k$ only depend on the $(k - 1)$-cells, and the composites added by the $M_j$ only depend on the $j$-cells.

  In order to describe this interleaving process formally, we define, for each $0 \leq j, k \leq n$, a category whose objects are objects of $\Rn$ equipped with both a $j$-magma structure and a $k$-contraction.

  \begin{defn}
    For each $0 \leq j \leq n$, $0 \leq k \leq n + 1$, define a category $\mathcal{R}_{j, k}$ with
      \begin{itemize}
        \item objects: an object of $\mathcal{R}_{j, k}$ consists of an $n$-globular set $X$ equipped with a $j$-magma structure, a strict $n$-category $S$, and a map of $n$-globular sets
     \[
        \xy
          % POINTS
          (0, 0)*+{X}="X";
          (0, -16)*+{S}="S";
          % ARROWS
          {\ar_f "X" ; "S"};
        \endxy
     \]
            that preserves the $j$-magma structure of $X$, equipped with a $k$-contraction $\gamma$;
        \item morphisms: a morphism in $\mathcal{R}_{j, k}$ is a commuting square
     \[
        \xy
          % POINTS
          (0, 0)*+{X}="X";
          (16, 0)*+{Y}="Y";
          (0, -16)*+{S}="S";
          (16, -16)*+{R}="R";
          % ARROWS
          {\ar^u "X" ; "Y"};
          {\ar_f "X" ; "S"};
          {\ar^g "Y" ; "R"};
          {\ar_v "S" ; "R"};
        \endxy
     \]
     in $\nGSet$ such that
              \begin{itemize}
                \item $v$ is a map of strict $n$-categories;
                \item $u$ preserves the $j$-magma structure of $X$;
                \item writing $\gamma$ for the contraction on the map $f$ and $\delta$ for the contraction on the map $g$, for all $0 < m \leq n$, and $(a, b) \in X^c_{m}$, we have
                  \[
                    u(\gamma_m(a, b)) = \delta_m(u(a), u(b)).
                  \]
              \end{itemize}
      \end{itemize}
  \end{defn}

  For $0 < j \leq n$, $0 < k \leq n + 1$, we have forgetful functors
    \[
      D_{j, k} \colon \mathcal{R}_{j, k} \rightarrow \mathcal{R}_{j, k - 1},
    \]
  which forgets the contraction structure at dimension $k$, and
    \[
      N_{j, k} \colon \mathcal{R}_{j, k} \rightarrow \mathcal{R}_{j - 1, k},
    \]
  which forgets the magma structure at dimension $j$.  Thus we can write the forgetful functor
    \[
      W \colon \Qn \rightarrow \Rn
    \]
  as the composite
    \[
      \xymatrix{
        \Qn = \mathcal{R}_{n, n + 1} \ar[r]^-{D_{n, n + 1}} & \mathcal{R}_{n, n} \ar[r]^-{N_{n, n}} & \mathcal{R}_{n - 1, n} \ar[r]^-{D_{n - 1, n}} & \dotsb \ar[r]^-{N_{1, 1}} & \mathcal{R}_{0, 1} \ar[r]^-{D_{0, 1}} & \mathcal{R}_{0, 0} = \Rn.
               }
    \]
  In order to construct the left adjoint to $W$, we construct left adjoint to each of the factors in the composite above, by lifting the constructions of $C_k$ and $M_j$ from Subsections~\ref{subsect:contr} and \ref{subsect:magma} in a way that interacts properly with the forgetful functors
    \[
      \mathcal{R}_{j, k} \rightarrow \MagR_j,
    \]
  which forget the $k$-contraction structure entirely, and
    \[
      \mathcal{R}_{j, k} \rightarrow \ContrR_j,
    \]
  which forget the $j$-magma structure entirely.

  \begin{lemma}  \label{lem:magstab}
    For all $0 < k \leq n + 1$, the adjunction
      \[
        \xy
          % POINTS
          (0, 0)*+{\ContrR_{k - 1}}="Cl";
          (30, 0)*+{\ContrR_{k}}="Cr";
          % ARROWS
          {\ar@<1ex>^-{C_{k}}_-*!/u1pt/{\labelstyle \bot} "Cl" ; "Cr"};
          {\ar@<1ex>^-{D_{k}} "Cr" ; "Cl"};
        \endxy
      \]
    lifts to an adjunction
      \[
        \xy
          % POINTS
          (0, 0)*+{\mathcal{R}_{k - 1, k - 1}}="Rl";
          (30, 0)*+{\mathcal{R}_{k - 1, k}}="Rr";
          % ARROWS
          {\ar@<1ex>^-{C_{k - 1, k}}_-*!/u1pt/{\labelstyle \bot} "Rl" ; "Rr"};
          {\ar@<1ex>^-{D_{k - 1, k}} "Rr" ; "Rl"};
        \endxy
      \]
    making the diagram
      \[
        \xy
          % POINTS
          (0, 0)*+{\mathcal{R}_{k - 1, k - 1}}="Rl";
          (30, 0)*+{\mathcal{R}_{k - 1, k}}="Rr";
          (0, -16)*+{\ContrR_{k - 1}}="Cl";
          (30, -16)*+{\ContrR_{k}}="Cr";
          % ARROWS
          {\ar@<1ex>^-{C_{k - 1, k}}_-*!/u1pt/{\labelstyle \bot} "Rl" ; "Rr"};
          {\ar@<1ex>^-{D_{k - 1, k}} "Rr" ; "Rl"};
          {\ar@<1ex>^-{C_{k}}_-*!/u1pt/{\labelstyle \bot} "Cl" ; "Cr"};
          {\ar@<1ex>^-{D_{k}} "Cr" ; "Cl"};
          {\ar "Rl" ; "Cl"};
          {\ar "Rr" ; "Cr"};
        \endxy
      \]
    commute serially.
  \end{lemma}

  \begin{proof}
    We need to show that, given $(X, f, S, \gamma) \in \ContrR_{k - 1}$, if $X$ is equipped with a $(k - 1)$-magma structure, then this $(k - 1)$-magma structure is ``stable'' under $C_k$;  this is immediate since, by construction, $C_k$ adds only $k$-cells to $X$, so the underlying $(k - 1)$-globular set of $X$ remains stable under $C_k$.
  \end{proof}

  \begin{lemma}  \label{lem:contrstab}
    For all $0 < j \leq n $, the adjunction
      \[
        \xy
          % POINTS
          (0, 0)*+{\MagR_{j - 1}}="Ml";
          (30, 0)*+{\MagR_{j}}="Mr";
          % ARROWS
          {\ar@<1ex>^-{M_{j}}_-*!/u1pt/{\labelstyle \bot} "Ml" ; "Mr"};
          {\ar@<1ex>^-{N_{j}} "Mr" ; "Ml"};
        \endxy
      \]
    lifts to an adjunction
      \[
        \xy
          % POINTS
          (0, 0)*+{\mathcal{R}_{j - 1, j}}="Rl";
          (30, 0)*+{\mathcal{R}_{j, j}}="Rr";
          % ARROWS
          {\ar@<1ex>^-{M_{j, j}}_-*!/u1pt/{\labelstyle \bot} "Rl" ; "Rr"};
          {\ar@<1ex>^-{N_{j, j}} "Rr" ; "Rl"};
        \endxy
      \]
    making the diagram
      \[
        \xy
          % POINTS
          (0, 0)*+{\mathcal{R}_{j - 1, j}}="Rl";
          (30, 0)*+{\mathcal{R}_{j, j}}="Rr";
          (0, -16)*+{\MagR_{j - 1}}="Ml";
          (30, -16)*+{\MagR_{j}}="Mr";
          % ARROWS
          {\ar@<1ex>^-{M_{j, j}}_-*!/u1pt/{\labelstyle \bot} "Rl" ; "Rr"};
          {\ar@<1ex>^-{N_{j, j}} "Rr" ; "Rl"};
          {\ar@<1ex>^-{M_{j}}_-*!/u1pt/{\labelstyle \bot} "Ml" ; "Mr"};
          {\ar@<1ex>^-{N_{j}} "Mr" ; "Ml"};
          {\ar "Rl" ; "Ml"};
          {\ar "Rr" ; "Mr"};
        \endxy
      \]
    commute serially.
  \end{lemma}

  \begin{proof}
    We need to show that, given $(X, f, S) \in \MagR_{j - 1}$, if $(X, f, S)$ is equipped with a $j$-contraction $\gamma$, this $j$-contraction structure is ``stable'' under $M_j$.  By construction, $M_j$ adds only $j$-cells to $X$, so the underlying $(j - 1)$-globular set of $X$ remains stable under $M_j$.  The required contraction $j$-cells depend only on the $(j - 1)$-cells of $\hatX$, and we have $\hatX^c_{j - 1} = X^c_{j - 1}$, so the contraction $j$-cells in $\hatX$ are given by
      \[
        \xymatrix{
          X^c_{j - 1} \ar[r]^{\gamma_{j - 1}} & X_j \ar[r]^{c_j} & \hatX_j.
                 }
      \]
    For $m < j $, we have $\hatX^c_{m - 1} = X^c_{m - 1}$, $\hatX_m = X_m$, and the contraction $m$-cells are given by $\gamma_{m - 1} \colon X^c_{m - 1} \rightarrow X_m$.  Hence the $j$-contraction structure is stable under $M_j$.
  \end{proof}

  Combining Lemmas~\ref{lem:magstab} and \ref{lem:contrstab}, we obtain a chain of adjunctions
    \[
      \xymatrix{
        \Rn = \mathcal{R}_{0, 0} \ar@<1ex>[r]^-{C_{0, 1}}_-*!/u1pt/{\labelstyle \bot} & \mathcal{R}_{0, 1} \ar@<1ex>[l]^-{D_{0, 1}} \ar@<1ex>[r]^-{M_{1, 1}}_-*!/u1pt/{\labelstyle \bot} & \dotsc \ar@<1ex>[l]^-{N_{1, 1}} \ar@<1ex>[r]^-{C_{n - 1, n}}_-*!/u1pt/{\labelstyle \bot} & \mathcal{R}_{n - 1, n} \ar@<1ex>[l]^-{D_{n - 1, n}} \ar@<1ex>[r]^-{M_{n, n}}_-*!/u1pt/{\labelstyle \bot}  & \mathcal{R}_{n, n} \ar@<1ex>[l]^-{N_{n, n}} \ar@<1ex>[r]^-{C_{n, n + 1}}_-*!/u1pt/{\labelstyle \bot} & \mathcal{R}_{n, n + 1} = \Qn \ar@<1ex>[l]^-{D_{n, n + 1}}.
               }
    \]
  Composing these, we obtain an adjunction
    \[
      \xymatrix{
        \Rn \ar@<1ex>[r]^-{J}_-*!/u1pt/{\labelstyle \bot} & \Qn \ar@<1ex>[l]^-{W},
               }
    \]
  where $J = C_{n, n + 1} \comp M_{n, n} \comp C_{n - 1, n} \comp \dotsb \comp M_{1, 1} \comp C_{0, 1}$.  We then have
    \[
      \xymatrix{
        \nGSet \ar@<1ex>[r]^-{F}_-*!/u1pt/{\labelstyle \bot} & \Qn \ar@<1ex>[l]^-{U},
               }
    \]
  where $F = J \comp H$.  Thus $U$ has a left adjoint, so Penon weak $n$-categories are indeed well-defined, and moreover we have an explicit description of this left adjoint.

  \chapter{Operadic definitions of weak $n$-category}  \label{chap:opdefns}

  In this chapter we recall and compare two operadic definitions of weak $n$-category, those of Batanin~\cite{Bat98} and Leinster~\cite{Lei98}.  We give a new proof of the existence of the initial $n$-globular operad with a contraction and system of compositions, a fact that has previously only been assumed~\cite{Bat98, Lei02}.  A correspondence between the contractions and systems of compositions used in Batanin's definition, and the unbiased contractions used in Leinster's definition, has long been suspected, and we prove a conjecture of Leinster~\cite[Section~10.1]{Lei04} that shows that the two notions are in some sense equivalent.  We then prove several new coherence theorems which apply to algebras for any operad with an unbiased contraction or with a contraction and system of compositions; these coherence theorems thus apply to both Batanin weak $n$-categories and Leinster weak $n$-categories.

  All definitions in this chapter are of the $n$-dimensional case, but it is straightforward and well-established how to modify the definitions to the $\omega$-dimensional case \cite{Bat98, Lei98}.  The results in Section~\ref{sect:globopcoh} are mostly not applicable in the $\omega$-dimensional case, since most of the coherence theorems concern behaviour of cells at dimension $n$ (for example, stating that certain diagrams of $n$-cells commute).

  As noted in the introduction, throughout this chapter we write $1$ to denote the terminal object of the category in which we are working; with the exception of Section~\ref{sect:globopalgs}, we work exclusively in $\nGSet$.  In this case, $1$ is the $n$-globular set in which the set $1_m$ of $m$-cells is a one-element set for every $0 \leq m \leq n$. We also write $1$ for the unique element of each set $1_m$.  Applying $T$ to the terminal $n$-globular set freely generates composites, giving the $n$-globular set $T1$ of globular pasting diagrams.  The cells of $T1$ are the arities of the operations in an $n$-globular operad.

  \section{Globular operads}  \label{sect:globop}  \label{sect:globopalgs}
  
  In this section we recall the definitions of generalised operads and their algebras, with particular emphasis placed on the case of $n$-globular operads.  None of the material in this section is new; it originates in \cite{Lei04}, with the special case of $n$-globular operads originating in \cite{Bat98}.

  A classical operad has a set of operations, each equipped with an arity: a natural number which is to be thought of as the number of inputs that the operation has.  In the definition of generalised operad, we replace $\Set$ with any category $\mathcal{C}$ that has all pullbacks and a terminal object, denoted $1$.  The arities of the generalised operad are then generated by applying a suitably well-behaved monad $T$ to the terminal object, giving an ``object of arities'' $T1$ in $\mathcal{C}$; hence such a generalised operad is called a ``$T$-operad''.  To retrieve the definition of classical non-symmetric operad, we take $\mathcal{C} = \Set$, and take $T$ to be the free monoid monad on $\Set$.  The terminal object in $\Set$ is the one-element set $1$, and applying $T$ gives $T1 \cong \mathbb{N}$, so in this case a $T$-operad has natural numbers as arities.  We will return to this example throughout the section.

  We will be particularly interested in the case of $n$-globular operads, in which $\mathcal{C} = \nGSet$ and $T$ is the free strict $n$-category monad.  As explained in the introduction, in this case $T1$ is the $n$-globular set whose elements are globular pasting diagrams.  In an $n$-globular operad, each operation has such a diagram as its arity, and should be thought of as a way of composing a diagram of cells of that shape.

  Before giving the definition of $T$-operad, we must first state formally what it means for $T$ to be ``suitably well-behaved''.

  \begin{defn}  \label{defn:cartesian}
    A category is said to be \emph{cartesian} if it has all pullbacks.  A functor is said to be \emph{cartesian} if it preserves pullbacks.  A natural transformation is said to be \emph{cartesian} if all of its naturality squares are pullback squares.  A map of monads is said to be \emph{cartesian} if its underlying natural transformation is cartesian.  A monad is said to be \emph{cartesian} if its functor part is a cartesian functor and its unit and counit are cartesian natural transformations.
  \end{defn}

  We now explain what it means for a monad to be cartesian with reference to the example of the free monoid monad on $\Set$.  Let $f \colon X \rightarrow Y$ be a map of sets, and consider the corresponding naturality squares for the free monoid monad $T$, i.e.
      \[
        \xy
          % POINTS
          % UNIT
          (0, 0)*+{X}="X";
          (16, 0)*+{Y}="Y";
          (0, -16)*+{TX}="TX";
          (16, -16)*+{TY}="TY";
          % AND
          (32, -8)*+{\text{and}};
          % MULTIPLICATION
          (48, 0)*+{T^2X}="TTX";
          (64, 0)*+{T^2Y}="TTY";
          (48, -16)*+{TX}="TXr";
          (64, -16)*+{TY.}="TYr";
          % ARROWS
          {\ar^-{f} "X" ; "Y"};
          {\ar_{\eta^T_X} "X" ; "TX"};
          {\ar_-{Tf} "TX" ; "TY"};
          {\ar^{\eta^T_Y} "Y" ; "TY"};
          {\ar^-{T^2 f} "TTX" ; "TTY"};
          {\ar_{\mu^T_X} "TTX" ; "TXr"};
          {\ar_-{Tf} "TXr" ; "TYr"};
          {\ar^{\mu^T_Y} "TTY" ; "TYr"};
          % PULLBACK STUFF
          (6,-1)*{}; (6,-5)*{} **\dir{-};
          (2,-5)*{}; (6,-5)*{} **\dir{-};
          (54,-1)*{}; (54,-5)*{} **\dir{-};
          (50,-5)*{}; (54,-5)*{} **\dir{-};
        \endxy
     \]
  Recall that an element of $TX$ is a ``word'' in $X$; that is, a finite string of elements of $X$.  The fact that the naturality square for the unit is a pullback square tells us that every element of $X$ is uniquely determined by its image under $f$ and the corresponding word of length $1$ in $TX$.  This means that we have no equations of the form $x = t$ where $x$ is a word of length $1$ but $t$ is a word of some other length, and also tells us that any element of $TX$ whose image under $Tf$ is a word of length $1$ must itself be a word of length $1$.

  Similarly, the fact that the naturality square for the multiplication is a pullback square tells us that any element of $T^2 X$ (a word of words in $X$) is uniquely determined by an element of $TX$, which tells us which elements of $X$ appear in this word of words, and an element of $T^2 Y$, which tells us how these elements of $X$ are divided into words.  This means that if two words in $TX$ are equal, they must be of the same length, and consist of the same elements of $X$ in the same order.

  So the fact that $T$ is cartesian tells us that, for a set $X$, each element of $TX$ has a fixed length and order, and that this length and order will be preserved by any map in the image of $T$.  More generally, for any cartesian monad $T$ on $\Set$ or a presheaf category, we can think of this as telling us that, for an object $X$, each element of $TX$ has a fixed ``shape'' that is preserved by any map in the image of $T$.  This goes some way towards explaining why cartesian monads are appropriate for generating the arities of generalised operads; arities are supposed to be the shapes that the inputs of an operation can take, so we need a monad for which the elements of $T1$ are fixed shapes.

  We now recall the definition of $T$-collections, the underlying data for $T$-operads.

  \begin{defn}  \label{defn:globcoll}
    Let $\mathcal{C}$ be a cartesian category with a terminal object $1$, and let $T$ be a cartesian monad on $\mathcal{C}$.  The \emph{category of $T$-collections} is the slice category $\mathcal{C}/T1$.  Explicitly, this is the category with:
      \begin{itemize}
        \item objects: an object of $\mathcal{C}/T1$, called an \emph{$T$-collection} (or simply, a \emph{collection}), consists of an object $K$ of $\mathcal{C}$, and a map
      \[
        \xymatrix{
          K \ar[d]^-{k} \\
          T1
                 }
      \]
  in $\mathcal{C}$;
    \item morphisms: a morphism of $f \colon K \rightarrow K'$, called a \emph{map of collections}, is a map of the underlying objects of $\mathcal{C}$ such that
      \[
        \xy
          % POINTS
          (-10, 14)*+{K}="K";
          (10, 14)*+{K'}="K'";
          (0, 0)*+{T1}="T";
          % ARROWS
          {\ar^{f} "K" ; "K'"};
          {\ar_{k} "K" ; "T"};
          {\ar^{k'} "K'" ; "T"}
        \endxy
      \]
  commutes.
      \end{itemize}

  We obtain from $\mathcal{C}/T1$ the monoidal category of collections $\TColl$ by equipping it with a tensor product.  Let $k \colon K \rightarrow T1$, $k' \colon K' \rightarrow T1$ be collections.  Then their tensor product is defined to be the composite along the top of the diagram
    \[
      \xy
        % POINTS
        (0, 0)*+{K \otimes K'}="KK'";
        (20, 0)*+{TK'}="TK'";
        (40, 0)*+{T^2 1}="T21";
        (60, 0)*+{T1}="T";
        (0, -16)*+{K}="K";
        (20, -16)*+{T1}="Tb";
        % PULLBACK STUFF
        (7,-2)*{}; (7,-6)*{} **\dir{-};
        (3,-6)*{}; (7,-6)*{} **\dir{-};
        % ARROWS
        {\ar "KK'" ; "TK'"};
        {\ar^{Tk'} "TK'" ; "T21"};
        {\ar^{\mu^T_1} "T21" ; "T"};
        {\ar "KK'" ; "K"};
        {\ar^{T!} "TK'" ; "Tb"};
        {\ar_{k} "K" ; "Tb"};
      \endxy
    \]
  where $!$ is the unique map $K' \rightarrow 1$ in $\mathcal{C}$ (since $1$ is terminal).  The unit for this tensor is the collection
      \[
        \xymatrix{
          1 \ar[d]^-{\eta^T_1} \\
          T1.
                 }
      \]

  In the case in which $\mathcal{C} = \nGSet$, and $T$ is the free strict $n$-category monad, a $T$-collection is called an \emph{$n$-globular collection}.  We write $\nColl$ for the monoidal category of $n$-globular collections.
  \end{defn}

  To understand this definition, we return to the example of classical non-symmetric operads, in which $\mathcal{C} = \Set$ and $T$ is the free monoid monad.  In this case we have $T1 \cong \mathbb{N}$, so a collection consists of a set $K$ of operations and a map $k \colon K \rightarrow \mathbb{N}$ that assigns an arity in $\mathbb{N}$ to each operation in $K$.  The tensor product will be used to define composition in a $T$-operad, via a multiplication map $\mu^K \colon K \otimes K \rightarrow K$.  In this case, an element $K \otimes K$ consists of an operation of arity $n$ for some $n \in \mathbb{N}$, together with a string of $n$ operations to be composed into each of the first operation's $n$ inputs; the pullback checks that the arity of the first operation is the same as the length of this string.  Thus $\mu^K$ takes this information an gives a single operation of $K$, with the appropriate arity.

  We now explain the role the tensor product of collections will play in the definition of $n$-globular operad, i.e. in the case $\mathcal{C} = \nGSet$, and $T$ is the free strict $n$-category monad.  As in the case above, in an $n$-globular operad with underlying collection $\xymatrix{ K \ar[r]^k & T1 }$, composition of operations will be defined as a map of collections $\mu^K \colon K \otimes K \rightarrow K$, so our explanation will focus on $K \otimes K$.

  We think of a typical element of $K \otimes K$ as looking like:
      \[
        \xy
          % POINTS
          % THETA
          (0, 0)*+{\bullet}="1";
          (16, 0)*+{\bullet}="2";
          (32, 0)*+{\bullet}="3";
          (8, 3.5)*+{\Downarrow};
          (8, -3.5)*+{\Downarrow};
          (24, 0)*+{\Downarrow};
          (-5, 0)*+{\theta =};
          % THETA_1
          (-34, 12)*+{\bullet}="11";
          (-24, 12)*+{\bullet}="12";
          (-14, 12)*+{\bullet}="13";
          (-29, 12)*+{\Downarrow};
          (-19, 15)*+{\Downarrow};
          (-19, 9)*+{\Downarrow};
          (-40, 12)*+{\theta_1 =};
          % THETA_2
          (-34, -12)*+{\bullet}="21";
          (-24, -12)*+{\bullet}="22";
          (-14, -12)*+{\bullet}="23";
          (-19, -12)*+{\Downarrow};
          (-40, -12)*+{\theta_2 =};
          % THETA_3
          (44, 12)*+{\bullet}="31";
          (54, 12)*+{\bullet}="32";
          (64, 12)*+{\bullet}="33";
          (49, 12)*+{\Downarrow};
          (70, 12)*+{= \theta_3};
          % ARROWS
          % THETA
          {\ar@/^1.75pc/ "1" ; "2"};
          {\ar "1" ; "2"};
          {\ar@/_1.75pc/ "1" ; "2"};
          {\ar@/^1.25pc/ "2" ; "3"};
          {\ar@/_1.25pc/ "2" ; "3"};
          % THETA_1
          {\ar@/^1.25pc/ "11" ; "12"};
          {\ar@/_1.25pc/ "11" ; "12"};
          {\ar@/^1.75pc/ "12" ; "13"};
          {\ar "12" ; "13"};
          {\ar@/_1.75pc/ "12" ; "13"};
          % THETA_2
          {\ar "21" ; "22"};
          {\ar@/^1.25pc/ "22" ; "23"};
          {\ar@/_1.25pc/ "22" ; "23"};
          % THETA_3
          {\ar@/^1.25pc/ "31" ; "32"};
          {\ar@/_1.25pc/ "31" ; "32"};
          {\ar "32" ; "33"};
          % Composition arrows
          \SelectTips{lu}{12}
          {\ar@/^1pc/@{.>} (-14, 16) ; (4, 8)}
          {\ar@/_1pc/@{.>} (-15, -15) ; (4, -8)}
          {\ar@/_1pc/@{.>} (42, 16) ; (26, 6)}
        \endxy
      \]
  where $\theta$, $\theta_1$, $\theta_2$, $\theta_3$ are $m$-cell in $K$ (with $\theta_1$, $\theta_2$, $\theta_3$ the ``labels'' of the $m$-cell in $TK$).  Applying the composite
    \[
      \xy
        % POINTS
        (0, 0)*+{K \otimes K}="KK";
        (16, 0)*+{TK}="TK";
        (30, 0)*+{T^2 1}="TT";
        (43, 0)*+{T1}="T";
        % ARROWS
        {\ar "KK" ; "TK"};
        {\ar^{Tk} "TK" ; "TT"};
        {\ar^{\mu^T_1} "TT" ; "T"};
      \endxy
    \]
  combines the arities of these operations to give:
    \[
      \xy
        % POINTS
          (0, 0)*+{\bullet}="1";
          (16, 0)*+{\bullet}="2";
          (32, 0)*+{\bullet}="3";
          (48, 0)*+{\bullet}="4";
          (64, 0)*+{\bullet}="5";
          (8, 0)*+{\Downarrow};
          (24, 0)*+{\Downarrow};
          (24, 8)*+{\Downarrow};
          (24, -8.5)*+{\Downarrow};
          (40, 0)*+{\Downarrow};
        % ARROWS
          {\ar@/^1.25pc/ "1" ; "2"};
          {\ar@/_1.25pc/ "1" ; "2"};
          {\ar@/^1.25pc/ "2" ; "3"};
          {\ar@/_1.25pc/ "2" ; "3"};
          {\ar@/^2.75pc/ "2" ; "3"};
          {\ar@/_2.75pc/ "2" ; "3"};
          {\ar@/^1.25pc/ "3" ; "4"};
          {\ar@/_1.25pc/ "3" ; "4"};
          {\ar "4" ; "5"};
      \endxy
    \]
  which is the arity of the operation we obtain by applying the multiplication map $\mu^K \colon K \otimes K \rightarrow K$, since $\mu^K$ is a map of collections.

  We now give the definition of a $T$-operad.

  \begin{defn}  \label{defn:globop}
    Let $\mathcal{C}$ be a cartesian category with a terminal object $1$, and let $T$ be a cartesian monad on $\mathcal{C}$.  A \emph{$T$-operad} is a monoid in the monoidal category $\TColl$.  Thus, an $T$-operad consists of:
      \begin{itemize}
        \item a collection
          \[
            \xymatrix{
              K \ar[d]^-k \\
              T1;
                     }
          \]
        \item a unit map $\eta^K \colon 1 \rightarrow K$ in $\mathcal{C}$ such that
          \[
            \xy
              % POINTS
              (-10, 14)*+{1}="1";
              (10, 14)*+{K}="K";
              (0, 0)*+{T1}="T";
              % ARROWS
              {\ar^{\eta^K} "1" ; "K"};
              {\ar_{\eta^T_1} "1" ; "T"};
              {\ar^{k} "K" ; "T"};
            \endxy
          \]
          commutes;
        \item a multiplication map $\mu^K \colon K \otimes K \rightarrow K$ such that the triangle in the diagram
          \[
            \xy
            % POINTS
              (0, 0)*+{K \otimes K}="KK";
              (-10, -10)*+{K}="K";
              (10, -10)*+{TK}="TK";
              (0, -20)*+{T1}="Tl";
              (20, -20)*+{T^2 1}="TT";
              (30, -30)*+{T1}="Tr";
              (30, 0)*+{K}="Kr";
            % ARROWS
              {\ar "KK" ; "K"};
              {\ar "KK" ; "TK"};
              {\ar_k "K" ; "Tl"};
              {\ar^{T!} "TK" ;"Tl"};
              {\ar_{Tk} "TK" ; "TT"};
              {\ar_{\mu^T_1} "TT" ; "Tr"};
              {\ar^{\mu^K} "KK" ; "Kr"};
              {\ar^-{k} "Kr" ; "Tr"};
              % PULLBACK STUFF
              (-3, -5)*{}; (0,-8)*{} **\dir{-};
              (3, -5)*{}; (0,-8)*{} **\dir{-};
            \endxy
          \]
          commutes.
      \end{itemize}
    These must satisfy the usual monoid axioms.  Note that we usually refer to such an operad as simply ``an operad $K$''.

    A \emph{map of $T$-operads} $f \colon K \rightarrow K'$  is a map of monoids.  This consists of a map $f \colon K \rightarrow K'$ of underlying collections such that the diagrams
      \[
        \xy
          % POINTS
          (0, 0)*+{1}="1";
          (-10, -14)*+{K}="K";
          (10, -14)*+{K'}="K'";
          % ARROWS
          {\ar_{\eta^K} "1" ; "K"};
          {\ar^{\eta^{K'}} "1" ; "K'"};
          {\ar_{f} "K" ; "K'"};
        \endxy
      \]
    and
      \[
        \xy
          % POINTS
          (-10, 0)*+{K \otimes K}="K2";
          (10, 0)*+{K' \otimes K'}="K'2";
          (-10, -16)*+{K}="K";
          (10, -16)*+{K'}="K'";
          % ARROWS
          {\ar^-{f \otimes f} "K2" ; "K'2"};
          {\ar_{f} "K" ; "K'"};
          {\ar_{\mu^K} "K2" ; "K"};
          {\ar^{\mu^{K'}} "K'2" ; "K'"};
        \endxy
      \]
    commute.

    In the case in which $\mathcal{C} = \nGSet$, and $T$ is the free strict $n$-category monad, a $T$-operad is called an \emph{$n$-globular operad}.  Since $n$-globular operads are the only kind of operads used in this thesis, we will often refer to them simply as ``operads''.
  \end{defn}

  To see that this is a generalisation of the definition of classical non-symmetric operad, we once again return to the case in which $\mathcal{C} = \Set$ and $T$ is the free monoid monad.  In a $T$-operad with underlying collection $k \colon K \rightarrow \mathbb{N}$, the composition of operations is given by the multiplication map $\mu^K \colon K \otimes K \rightarrow K$. This map takes an element of the tensor product, that is an operation of arity $n$ and a string of $n$ operations of arities $i_1, \dotsc, i_n$, and composes to give a single operation in $K$.  The commuting triangle in the diagram defining $\mu^K$ ensures that the arity of the composite operation is $i_1 + \dotsc + i_n$.  Identities come from the unit map $\eta^K \colon 1 \rightarrow K$, which picks out a single operation, and the commuting triangle ensures that the arity of this operation is $1$.

  The algebras for a $T$-operad are the algebras for a particular induced monad, which we now define.

  \begin{defn}  \label{defn:indmonad}
    Let $\mathcal{C}$ be a cartesian category with a terminal object $1$, let $T$ be a cartesian monad on $\mathcal{C}$ and let $K$ be a $T$-operad.  Then there is an induced monad on $\mathcal{C}$, which by abuse of notation we denote $(K, \eta^K, \mu^K)$ (so the endofunctor part of the monad is denoted by the same letter as the underlying $n$-globular set of the operad, and we use the same notation for the unit and multiplication of the monad as we do for those of the operad).  The endofunctor
      \[
        K \colon \mathcal{C} \rightarrow \mathcal{C}
      \]
    is defined as follows: on objects, given an object $X$ in $\mathcal{C}$, $KX$ is defined by the pullback:
      \[
        \xy
          % POINTS
          (0, 0)*+{KX}="KX";
          (16, 0)*+{K}="K";
          (0, -16)*+{TX}="TX";
          (16, -16)*+{T1,}="T";
          % ARROWS
          {\ar^-{K!} "KX" ; "K"};
          {\ar_{k_X} "KX" ; "TX"};
          {\ar_-{T!} "TX" ; "T"};
          {\ar^{k} "K" ; "T"};
          % PULLBACK STUFF
          (6,-1)*{}; (6,-5)*{} **\dir{-};
          (2,-5)*{}; (6,-5)*{} **\dir{-};
        \endxy
     \]
    where $!$ is the unique morphism $X \rightarrow 1$ in $\mathcal{C}$; on morphisms, given a morphism $u \colon X \rightarrow Y$ in $\mathcal{C}$, $Ku$ is defined to be the unique map induced by the universal property of the pullback defining $KY$ such that the diagram
      \[
        \xy
          % POINTS
          (-16, 0)*+{KX}="KX";
          (0, 0)*+{KY}="KY";
          (16, 0)*+{K}="K";
          (-16, -16)*+{TX}="TX";
          (0, -16)*+{TY}="TY";
          (16, -16)*+{T1}="T";
          % ARROWS
          {\ar@{-->}^-{Ku} "KX" ; "KY"};
          {\ar^-{K!} "KY" ; "K"};
          {\ar_{k_X} "KX" ; "TX"};
          {\ar_{k_Y} "KY" ; "TY"};
          {\ar_-{Tu} "TX" ; "TY"};
          {\ar_-{T!} "TY" ; "T"};
          {\ar^{k} "K" ; "T"};
          {\ar@/^1.5pc/^{K!} "KX" ; "K"};
          {\ar@/_1.5pc/_{T!} "TX" ; "T"};
          % PULLBACK STUFF
          (6,-1)*{}; (6,-5)*{} **\dir{-};
          (2,-5)*{}; (6,-5)*{} **\dir{-};
          (-10,-1)*{}; (-10,-5)*{} **\dir{-};
          (-14,-5)*{}; (-10,-5)*{} **\dir{-};
        \endxy
     \]
    commutes.  Observe that commutativity of the left-hand square in the diagram above shows that $k$ is a natural transformation $K \Rightarrow T$; the fact that this square is a pullback square shows that this natural transformation is cartesian.

    The unit map $\eta^K \colon 1 \Rightarrow K$ for the monad $K$ has, for each $X \in \mathcal{C}$, a component $\eta^K_X \colon X \rightarrow KX$ which is the unique map such that the diagram
      \[
        \xy
          % POINTS
          (-12, 12)*+{X}="X";
          (16, 12)*+{1}="1";
          (0, 0)*+{KX}="KX";
          (16, 0)*+{K}="K";
          (0, -16)*+{TX}="TX";
          (16, -16)*+{T1,}="T";
          % ARROWS
          {\ar^-{K!} "KX" ; "K"};
          {\ar_{k_X} "KX" ; "TX"};
          {\ar_-{T!} "TX" ; "T"};
          {\ar^{k} "K" ; "T"};
          {\ar^{!} "X" ; "1"};
          {\ar@{-->}^{\eta^K_X} "X" ; "KX"};
          {\ar^{\epsilon} "1" ; "K"};
          {\ar@/_1pc/_{\eta^T_X} "X" ; "TX"};
          {\ar@/^1.5pc/^{\eta^T_1} "1" ; "T"};
          % PULLBACK STUFF
          (6,-1)*{}; (6,-5)*{} **\dir{-};
          (2,-5)*{}; (6,-5)*{} **\dir{-};
        \endxy
     \]
    commutes.

    The multiplication map $\mu^K \colon K^2 \Rightarrow K$ for the monad $K$ has, for each object $X$ in $\mathcal{C}$, a component $\mu^K_A \colon K^2 X \rightarrow K X$ which is the defined to be unique map such that the diagram
      \[
        \xy
          % POINTS
            (0, 0)*+{K^2 X}="KKX";
            (-12, -12)*+{K \otimes K}="KK";
            (12, -12)*+{TKX}="TKX";
            (-24, -24)*+{K}="Kt";
            (0, -24)*+{TK}="TK";
            (24, -24)*+{T^2 X}="TTX";
            (-12, -36)*+{T1}="Tt";
            (12, -36)*+{T^2 1}="TT";
            (0, -48)*+{KX}="KX";
            (-12, -60)*+{K}="Kb";
            (12, -60)*+{TX}="TX";
            (0, -72)*+{T1}="Tb";
          % ARROWS
            {\ar "KKX" ; "KK"};
            {\ar "KKX" ; "TKX"};
            {\ar "KK" ; "Kt"};
            {\ar "KK" ; "TK"};
            {\ar "TKX" ; "TK"};
            {\ar "TKX" ; "TTX"};
            {\ar_k "Kt" ; "Tt"};
            {\ar^{T!} "TK" ; "Tt"};
            {\ar_{Tk} "TK" ; "TT"};
            {\ar^{T^2 !} "TTX" ; "TT"};
            {\ar "KX" ; "Kb"};
            {\ar "KX" ; "TX"};
            {\ar_k "Kb" ; "Tb"};
            {\ar^{T!} "TX" ; "Tb"};
            {\ar@{-->}@/^1.5pc/^{\mu^K_X} "KKX" ; "KX"};
            {\ar@/_5pc/_{\mu^K} "KK" ; "Kb"};
            {\ar@/^1.7pc/^{\mu^T_X} "TTX" ; "TX"};
          % PULLBACK STUFF
            (-3, -5)*{}; (0,-8)*{} **\dir{-};
            (3, -5)*{}; (0,-8)*{} **\dir{-};
            (-15, -17)*{}; (-12,-20)*{} **\dir{-};
            (-9, -17)*{}; (-12,-20)*{} **\dir{-};
            (9, -17)*{}; (12,-20)*{} **\dir{-};
            (15, -17)*{}; (12,-20)*{} **\dir{-};
            (-3, -53)*{}; (0,-56)*{} **\dir{-};
            (3, -53)*{}; (0,-56)*{} **\dir{-};
        \endxy
      \]
    commutes.
  \end{defn}

  \begin{defn}
    Let $\mathcal{C}$ be a cartesian category with a terminal object $1$, let $T$ be a cartesian monad on $\mathcal{C}$ and let $K$ be a $T$-operad.  An \emph{algebra} for the operad $K$, referred to as a \emph{$K$-algebra}, is defined to be an algebra for the induced monad $(K, \eta^K, \mu^K)$.  Similarly, a \emph{map of algebras for the operad $K$} is a map of algebras for the induced monad, and the category of algebras for the operad $K$ is $K\Alg$, the category of algebras for the induced monad.
  \end{defn}

  Batanin and Leinster each define weak $n$-categories to be the algebras for a particular $n$-globular operad; we recall the definitions of these in Sections~\ref{sect:Batanin} and \ref{sect:Leinster} respectively.

  \section{Batanin weak $n$-categories}  \label{sect:Batanin}

  In this section we recall the definition of Batanin weak $n$-category, which was originally given by Batanin in \cite{Bat98}.  Batanin weak $n$-categories and Leinster weak $n$-categories are defined to be the algebras for particular $n$-globular operads.  In order to identify an appropriate operad to use, Batanin's approach is to define two pieces of extra structure on an operad:
    \begin{itemize}
      \item a system of compositions: this picks out binary composition operations at each dimension;
      \item a contraction on the underlying collection: this ensures that we have contraction operations which give the constraint cells in algebras for the operad; it also ensures that composition is strict at dimension $n$.
    \end{itemize}
  Operads equipped with contractions and systems of compositions form a category, and this category has an initial object; a Batanin weak $n$-category is defined to be an algebra for this initial operad.

  In fact, the approach described here is slightly different from that of \cite{Bat98}, in which Batanin uses contractible operads rather than operads equipped with a specified contraction.  Since contractibility is non-algebraic, there is no initial object in the category of contractible operads with systems of compositions, so Batanin explicitly constructs an operad that is weakly initial in this category.  He claims without proof that, if we use specified contractions, this operad is initial \cite[Section 8, Remark 2]{Bat98}, so the operad we describe is the same as Batanin's, even though the approach is slightly different.

  We begin by defining what it means for an operad to be equipped with a system of compositions.  To do this, we define a collection
    \[
      \xy
        % POINTS
        (0, 0)*+{S}="S";
        (0, -16)*+{T1}="T";
        % ARROWS
        {\ar^s "S" ; "T"};
      \endxy
    \]
  that contains precisely one binary composition operation for each dimension of cell and boundary;  in order for the sources and targets of these operations to be well-defined, $S$ also contains a unary operation (i.e. one whose arity is a single globular cell) at each dimension, but otherwise contains no other operations.  The collection $S$ comes equipped with a unit map, which picks out the unary operation at each dimension.  Note that it is not possible to equip $S$ with an operad structure, since it does not have operations of all the arities we would require in order to define a multiplication map on $S$.  For example, in $S \otimes S$ we have $1$-cells such as
    \[
      \xy
        % POINTS
        (0, 0)*+{\bullet}="1";
        (16, 0)*+{\bullet}="2";
        (32, 0)*+{\bullet}="3";
        (-22, 12)*+{\bullet}="1t";
        (-10, 12)*+{\bullet}="2t";
        (2, 12)*+{\bullet}="3t";
        % ARROWS
        {\ar "1" ; "2"};
        {\ar "2" ; "3"};
        {\ar "1t" ; "2t"};
        {\ar "2t" ; "3t"};
        % DOTTED ARROWS
        \SelectTips{lu}{12}
        {\ar@/^0.25pc/@{.>} (-8, 11) ; (7, 1)}
      \endxy
    \]
  but there is no way to define the action of the multiplication on this cell since there is no operation of arity
    \[
      \xy
        % POINTS
        (0, 0)*+{\bullet}="1";
        (16, 0)*+{\bullet}="2";
        (32, 0)*+{\bullet}="3";
        (48, 0)*+{\bullet}="4";
        % ARROWS
        {\ar "1" ; "2"};
        {\ar "2" ; "3"};
        {\ar "3" ; "4"};
      \endxy
    \]
  in $S$.  Once we have defined $S$, we define a system of compositions on an operad $K$ to be a map of collection $S \rightarrow K$, which picks out the desired binary composition operations in $K$; this map is required to interact properly with the unit maps for $S$ and $K$.

  \begin{defn}  \label{defn:SoC}
    Let $0 \leq m \leq n$, and write $\eta_m := \eta^T_m(1)$, the single $m$-cell in the image of the unit map $\eta^T \colon 1 \rightarrow T1$.  Define, for $0 \leq p \leq m \leq n$,
      \[
        \beta^m_p =
            \left\{
              \begin{array}{ll}
              \eta_m & \text{if } p = m, \\
              \eta_m \comp^m_p \eta_m & \text{if } p < m.
              \end{array}
            \right.
      \]
    Define an $n$-globular collection $\xymatrix{ S \ar[r]^s & T1 }$, in which
      \[
        S_m := \{ \beta^m_p \gt 0 \leq p \leq m \leq n \} \subseteq T1_m,
      \]
    and define the unit map $\eta^S \colon 1 \rightarrow S$ by $\eta^S_m(1) = \beta^m_m$.

    Let $\xymatrix{ K \ar[r]^k & T1 }$ be an $n$-globular operad.  A \emph{system of compositions} on $K$ consists of a map of collections
      \[
        \xy
          % POINTS
          (-10, 14)*+{S}="S";
          (10, 14)*+{K}="K";
          (0, 0)*+{T1}="T";
          % ARROWS
          {\ar^{\sigma} "S" ; "K"};
          {\ar_{s} "S" ; "T"};
          {\ar^{k} "K" ; "T"}
        \endxy
      \]
    such that the diagram
      \[
        \xy
        % POINTS
        (-15, 0)*+{1}="1";
        (0, 0)*+{S}="S";
        (15, 0)*+{K}="K";
        % ARROWS
        {\ar^{\eta^S} "1" ; "S"};
        {\ar^{\sigma} "S" ; "K"};
        {\ar@/_1.5pc/_{\eta^K} "1" ; "K"}
        \endxy
      \]
    commutes.
  \end{defn}

  The notion of contraction on a collection used to define Batanin weak $n$-categories is the same as the notion of contraction, from Definition~\ref{defn:contr}, on a map of $n$-globular sets
    \[
      f \colon X \rightarrow R,
    \]
  where $R$ is the underlying $n$-globular set of a strict $n$-category.  In the case of a contraction on a collection, this strict $n$-category is always $T1$, the free strict $n$-category on $1$.  We will restate the definition of contraction in this case in an alternative equivalent way;  the reason for doing this is that it allows for easier comparison between contractions and the unbiased contractions of Leinster, which we recall in Section~\ref{sect:Leinster}.  Before giving this alternative definition of contraction, we establish some notation that will be used in the definition.  This notation is more general than is necessary at this stage, but will be used in its full generality in the definition of unbiased contractions.

  Let $\xymatrix{ K \ar[r]^k & T1 }$ be an $n$-globular collection.  We will define, for each globular pasting diagram $\pi$, a set $C_K(\pi)$ whose elements are parallel pairs of cells in $K$, the first of which maps to the source of $\pi$ under $k$, and the second of which maps to the target of $\pi$ under $k$.  When $\pi = \id_{\alpha}$ for some $\alpha \in T1$, we can think of $C_K(\pi)$ as a set of contraction cells living over $\pi$, since every such pair requires a contraction cell for there to be a contraction on the map $k$.  To modify the definition of contraction to a definition of unbiased contraction in Section~\ref{sect:Leinster}, we use all pasting diagrams $\pi$ in $T1$, not just those of the form $\pi = \id_{\alpha}$ for some $\alpha \in T1$.

  To define $C_K(\pi)$, we first define, for all $0 \leq m \leq n$, $x \in T1_m$, a set
    \[
      K(x) = \{ a \in K_m \gt k(a) = x \};
    \]
  that is, the preimage of $x$ under $k$.  Then, for all $1 \leq m \leq n$, $\pi \in T1_m$, we define
    \[
      C_K(\pi) =
          \left\{
            \begin{array}{ll}
            K(s(\pi)) \times K(t(\pi)) & \text{if } m = 1, \\
            \{ (a, b) \in K(s(\pi)) \times K(t(\pi)) \gt s(a) = s(b), t(a) = t(b) \} & \text{if } m > 1.
            \end{array}
          \right.
    \]

  \begin{defn}  \label{defn:Bcontr}
    A \emph{contraction} $\gamma$ on an $n$-globular collection $\xymatrix{ K \ar[r]^k & T1 }$ consists of, for all $1 \leq m \leq n$, and for each $\alpha \in (T1)_{m - 1}$, a function
      \[
        \gamma_{\id_{\alpha}} \colon C_K(\id_{\alpha}) \rightarrow K(\id_{\alpha})
      \]
    such that, for all $(a, b) \in C_K(\id_{\alpha})$,
      \[
        s\gamma_{\id_{\alpha}}(a, b) = a, \; t\gamma_{\id_{\alpha}}(a, b) = b
        % I'm not sure how acceptable this spacing is.
      \]
    We also require the following tameness condition, as in Definition~\ref{defn:contr}: for $\alpha$, $\beta \in K_n$, if
      \[
        s(\alpha) = s(\beta), \; t(\alpha) = t(\beta), \; k(\alpha) = k(\beta),
      \]
    then $\alpha = \beta$.
  \end{defn}

  Operads with contractions and systems of compositions form a category, which we now define.

  \begin{defn}  \label{defn:OCS}
    Define $\OCS$ to be the category with
      \begin{itemize}
        \item objects: an object of $\OCS$ is an operad
          \[
            \xy
              % POINTS
              (0, 0)*+{K}="K";
              (0, -16)*+{T1}="T";
              % ARROWS
              {\ar^k "K" ; "T"};
            \endxy
          \]
        equipped with a contraction $\gamma$ and a system of compositions $\sigma \colon S \rightarrow K$;
        \item morphisms: for operads $\xymatrix{ K \ar[r]^k & T1 }$, $\xymatrix{ K' \ar[r]^{k'} & T1 }$, respectively equipped with contraction $\gamma$, $\gamma'$, and systems of compositions $\sigma$, $\sigma'$, a morphism $u \colon K \rightarrow K'$ consists of a map $u$ of the underlying operads such that
              \begin{itemize}
                \item the diagram
                  \[
                    \xy
                      % POINTS
                      (0, 0)*+{S}="S";
                      (-10, -15)*+{K}="K";
                      (10, -15)*+{K'}="K'";
                      % ARROWS
                      {\ar_{\sigma} "S" ; "K"};
                      {\ar^{\sigma'} "S" ; "K'"};
                      {\ar_{u} "K" ; "K'"};
                    \endxy
                  \]
                commutes;
                \item for all $1 \leq m \leq n$, $\alpha \in T1_{m - 1}$, $(a, b) \in C_K(\id_{\alpha})$,
                  \[
                    u_m(\gamma_{\id_{\alpha}}(a, b)) = \gamma'_{\id_{\alpha}}(u_{m - 1}(a), u_{m - 1}(b)).
                  \]
              \end{itemize}
      \end{itemize}
  \end{defn}

  We often refer to an operad with a contraction and system of compositions simply as a \emph{Batanin operad}.  The category $\OCS$ has an initial object
          \[
            \xy
              % POINTS
              (0, 0)*+{B}="B";
              (0, -16)*+{T1,}="T";
              % ARROWS
              {\ar^b "B" ; "T"};
            \endxy
          \]
  the existence of which we prove in Section~\ref{sect:initial}.  This initial object is in some sense the ``simplest'' operad in $\OCS$.  It has precisely the operations required to have a system of compositions, a contraction, and an operad structure, and no more; furthermore, it has no spurious relations between these operations.

  \begin{defn}
    A Batanin weak $n$-category is an algebra for the $n$-globular operad $\xymatrix{ B \ar[r]^b & T1 }$.  The category of Batanin weak $n$-categories is $B\Alg$.
  \end{defn}

  Note that the presence of a system of compositions and a contraction on an operad does not affect the category of algebras for that operad.  The algebras depend only on the operad itself; systems of compositions and contractions are used purely as a tool for selecting an appropriate choice of operad.

  \section{Initial object in $\OCS$}  \label{sect:initial}

  We now prove that the category $\OCS$ has an initial object.  This has been believed for some time \cite{Bat98, Lei02}, but has not previously been proved.  Our proof is based on a proof by Leinster \cite[Appendix G]{Lei04} of the existence of the operad for Leinster weak $n$-categories, which is defined as the initial operad in a different, but similar, category of operads, as we shall see in the next section.

  The idea of this proof is as follows:  the category $\nColl$ has an initial object
    \[
      \xy
        % POINTS
        (0, 0)*+{\emptyset}="0";
        (16, 0)*+{T1.}="1";
        % ARROW
        {\ar^-{!} "0" ; "1"};
      \endxy
    \]
  There is a forgetful functor
    \[
      \OCS \longrightarrow \nColl,
    \]
  which sends an operad to its underlying collection, and this forgetful functor has a left adjoint.  The initial collection is the colimit of the empty diagram in $\nColl$, and left adjoints preserve colimits, so applying the left adjoint to the initial collection gives the initial object in $\OCS$.

  Thus we can prove the existence of the initial Batanin operad $B$ by proving the existence of this left adjoint.  To do so we use the following monadicity result, due to Kelly~\cite[27.1]{Kel80} (which appears in this form in \cite[Appendix G]{Lei04}):

  \begin{prop}  \label{prop:Kelly}
    Let
      \[
        \xy
          % POINTS
          (0, 0)*+{\mathcal{D}}="D";
          (16, 0)*+{\mathcal{C}}="C";
          (0, -16)*+{\mathcal{B}}="B";
          (16, -16)*+{\mathcal{A}}="A";
          % ARROWS
          {\ar "D" ; "C"};
          {\ar "D" ; "B"};
          {\ar_-{U} "B" ; "A"};
          {\ar^{V} "C" ; "A"};
          % PULLBACK STUFF
          (6,-1)*{}; (6,-5)*{} **\dir{-};
          (2,-5)*{}; (6,-5)*{} **\dir{-};
        \endxy
     \]
    be a pullback square in $\mathbf{CAT}$.  If $\mathcal{A}$ is locally finitely presentable and each of $U$ and $V$ is finitary and monadic, then the functor $\mathcal{D} \rightarrow \mathcal{A}$ is monadic.
  \end{prop}

  To apply this result to our situation, we take $\mathcal{A} = \nColl$, $\mathcal{D} = \OCS$; to see what $\mathcal{B}$ and $\mathcal{C}$ should be, observe that in a Batanin operad the contraction structure exists independently of the operad structure (though note that the system of compositions cannot exist without the unit of the operad structure).  Thus we have categories
    \begin{itemize}
      \item $\Contr$ of collections equipped with contractions;
      \item $\SoC$ of operads equipped with systems of compositions (``$\SoC$'' stands for ``system of compositions'');
    \end{itemize}
  and we can write $\OCS$ as the pullback
    \[
      \xy
        % POINTS
        (0, 0)*+{\OCS}="0,0";
        (16, 0)*+{\Contr}="1,0";
        (0, -16)*+{\SoC}="0,1";
        (16, -16)*+{\nColl.}="1,1";
        % ARROWS
        {\ar "0,0" ; "1,0"};
        {\ar "0,0" ; "0,1"};
        {\ar^V "1,0" ; "1,1"};
        {\ar_-U "0,1" ; "1,1"};
        % PULLBACK STUFF
        (6,-1)*{}; (6,-5)*{} **\dir{-};
        (2,-5)*{}; (6,-5)*{} **\dir{-};
      \endxy
    \]
  The composite (of either side, since this diagram commutes) is the forgetful functor $\OCS \rightarrow \nColl$.  Thus we take $\mathcal{B} = \SoC$ and $\mathcal{C} = \Contr$.

  Note that this is not the same as an interleaving construction;  Proposition~\ref{prop:Kelly} does not require us to decompose the left adjoints to $U$ and $V$ dimension by dimension (or even construct them explicitly), and it does not give an explicit description of the left adjoint $\mathcal{A} \rightarrow \mathcal{D}$.

  We now define the categories $\Contr$ and $\SoC$ formally.

  \begin{defn}
    Define $\Contr$ to be the category with
      \begin{itemize}
        \item objects: an object of $\Contr$ consists of a collection $\xymatrix{ X \ar[r]^x & T1 }$ equipped with a contraction $\gamma$;
        \item morphisms: for collections $\xymatrix{ X \ar[r]^x & T1 }$, $\xymatrix{ X' \ar[r]^{x'} & T1 }$, respectively equipped with contractions $\gamma$, $\gamma'$, a morphism $u \colon X \rightarrow X'$ consists of a map $u$ of the underlying collections such that, for all $1 \leq m \leq n$, $\alpha \in (T1)_{m - 1}$, $(a, b) \in C_X(\id_{\alpha})$,
                  \[
                    u_m(\gamma_{\id_{\alpha}}(a, b)) = \gamma'_{\id_{\alpha}}(u_{m - 1}(a), u_{m - 1}(b)).
                  \]
      \end{itemize}

    Define $\SoC$ to be the category with
      \begin{itemize}
        \item objects: an object consists of an operad $\xymatrix{ K \ar[r]^k & T1 }$ equipped with a system of compositions $\sigma$;
        \item morphisms: for operads $\xymatrix{ K \ar[r]^k & T1 }$, $\xymatrix{ K' \ar[r]^{k'} & T1 }$, respectively equipped with contraction $\gamma$, $\gamma'$, and systems of compositions $\sigma$, $\sigma'$, a morphism $u \colon K \rightarrow K'$ consists of a map $u$ of the underlying operads such that the diagram
                  \[
                    \xy
                      % POINTS
                      (0, 0)*+{S}="S";
                      (-10, -15)*+{K}="K";
                      (10, -15)*+{K'}="K'";
                      % ARROWS
                      {\ar_{\sigma} "S" ; "K"};
                      {\ar^{\sigma'} "S" ; "K'"};
                      {\ar_{u} "K" ; "K'"};
                    \endxy
                  \]
                commutes.
      \end{itemize}
  \end{defn}

  To show that the conditions of Proposition~\ref{prop:Kelly} hold in our case, we must prove that the forgetful functors
    \[
      U \colon \SoC \longrightarrow \nColl,
    \]
    \[
      V \colon \Contr \longrightarrow \nColl,
    \]
  which send objects to their underlying collections, are monadic.  To do so, we use Beck's monadicity theorem~\cite[Theorem 4.4.4]{Bor94}:

  \begin{thm}[Beck's monadicity theorem]  \label{thm:Beck}
    A functor $U \colon \mathcal{D} \rightarrow \mathcal{C}$ is monadic if and only if
      \begin{itemize}
        \item $U$ has a left adjoint;
        \item $U$ reflects isomorphisms;
        \item given a pair of morphisms
      \[
        \xy
          % POINTS
          (0, 0)*+{X}="X";
          (16, 0)*+{Y}="Y";
          % ARROWS
          {\ar@<1ex>^{f} "X" ; "Y"};
          {\ar@<-1ex>_{g} "X" ; "Y"};
        \endxy
      \]
        in $\mathcal{D}$ such that
      \[
        \xy
          % POINTS
          (0, 0)*+{U(X)}="X";
          (16, 0)*+{U(Y)}="Y";
          % ARROWS
          {\ar@<1ex>^{Uf} "X" ; "Y"};
          {\ar@<-1ex>_{Ug} "X" ; "Y"};
        \endxy
      \]
        has a split coequaliser in $\mathcal{C}$, then $(f, g)$ has a coequaliser in $\mathcal{D}$ which is preserved by $U$.
      \end{itemize}
  \end{thm}

  \begin{lemma}  \label{lem:Vmonadic}
    The functor
      \[
        V \colon \Contr \rightarrow \nColl
      \]
    is monadic.
  \end{lemma}

  \begin{proof}
    We show that $V$ is monadic by checking that it satisfies the conditions in Beck's monadicity theorem (Theorem~\ref{thm:Beck}).  The functor $V$ has a left adjoint, which can be constructed using exactly the same method as was used to construct the free contractions in Definitions~\ref{defn:freekcontr} and \ref{defn:freenplusonecontr}; we just restrict to objects of $\Rn$ and $\Qn$ with $T1$ as their strict $n$-category parts.  Since $V$ leaves the underlying maps of collections unchanged, it reflects isomorphisms.  Thus we only need to check the condition regarding coequalisers.

    Consider a pair of maps
      \[
        \xy
          % POINTS
          (0, 0)*+{X}="X";
          (20, 0)*+{Y}="Y";
          (10, -14)*+{T1}="T";
          % ARROWS
          {\ar@<1ex>^{f} "X" ; "Y"};
          {\ar@<-1ex>_{g} "X" ; "Y"};
          {\ar_x "X" ; "T"};
          {\ar^y "Y" ; "T"};
        \endxy
      \]
    in $\Contr$ such that its image under $V$ has a split coequaliser.  We have
      \[
        \xy
          % POINTS
          (0, 0)*+{X}="X";
          (20, 0)*+{Y}="Y";
          (40, 0)*+{Z}="Z";
          (20, -14)*+{T1}="T";
          % ARROWS
          {\ar@<1ex>^{f} "X" ; "Y"};
          {\ar@<-1ex>_{g} "X" ; "Y"};
          {\ar@/_1.5pc/_q "Y" ; "X"};
          {\ar@/_1.5pc/_p "Z" ; "Y"};
          {\ar^e "Y" ; "Z"};
          {\ar_x "X" ; "T"};
          {\ar^y "Y" ; "T"};
          {\ar^z "Z" ; "T"};
        \endxy
      \]
    in $\nColl$, where $\xymatrix{ Z \ar[r]^{z} & T1 }$ is the coequaliser of $(f, g)$, and $p$ and $q$ satisfy the following equations:
      \[
        ep = \id_Z, \; fq = \id_Y, \; gq = pe.
      \]
    To show that $(f, g)$ also has a coequaliser in $\Contr$, we need to show that we can equip $Z$ with a contraction in such a way that $e$ and any maps induced by the universal property preserve contractions.

    Write $\gamma$ for the contraction on $\xymatrix{ Y \ar[r]^{y} & T1 }$, so for all $0 < m \leq n$, and for all $\alpha \in T1_{m - 1}$, we have a function
      \[
        \gamma_{\id_{\alpha}} \colon C_Y(\id_{\alpha}) \rightarrow Y(\id_{\alpha}).
      \]
    Suppose we have $0 < m \leq n$ and $\alpha \in T1_{m - 1}$.  We define a function
      \[
        \delta_{\id_{\alpha}} \colon C_Z(\id_{\alpha}) \rightarrow Z(\id_{\alpha})
      \]
    as follows: given $(a, b) \in C_Z(\id_{\alpha})$,
      \[
        \delta_{\id_{\alpha}}(a, b) = \gamma_{\id_{\alpha}}(p(a), p(b)).
      \]
    We need to check that this defines a contraction $\delta$ on $\xymatrix{ Z \ar[r]^{z} & T1 }$.  Since $ep = \id_{\alpha}$, $\delta_{\id_{\alpha}}(a, b)$ has the correct source and target, and since $y = ze$, we have
      \[
        ze\delta_{\id_{\alpha}}(a, b) = y\gamma_{\id_{\alpha}}(p(a), p(b)) = \id_{\alpha}.
      \]
    Thus $\delta$ is a contraction on $\xymatrix{ Z \ar[r]^{z} & T1 }$, and by definition of $\delta$, $e$ preserves the contraction structure.

    Now suppose we have a collection with contraction $\xymatrix{ W \ar[r]^{w} & T1 }$ and a map $r \colon Y \rightarrow W$ in $\Contr$ such that $wf = wg$. There is a unique map $u \colon Z \rightarrow W$ in $\nColl$ such that $ue = r$. For all $0 < m \leq n$, $\alpha \in T1_{m - 1}$, and $(a, b) \in C_Z(\id_{\alpha})$, we have
      \[
        u\delta_{\id_{\alpha}}(a, b) = ue\gamma_{\id_{\alpha}}(p(a), p(b)) = r\gamma_{\id_{\alpha}}(p(a), p(b)),
      \]
    so since $r$ preserves contraction cells, so does $u$; hence $u$ is a map in $\Contr$.  Thus $\xymatrix{ Z \ar[r]^{z} & T1 }$ is the coequaliser of $(f, g)$ in $\Contr$.

    Hence $V$ is monadic, as required.
  \end{proof}

  To prove that the functor $U \colon \SoC \rightarrow \nColl$ is monadic, we first prove that it has a left adjoint. Observe that, if we did not require a system of compositions, we could use the free monoid construction of Kelly~\cite{Kel80}, since an operad is a monoid in $\nColl$.  However, we cannot simply add a system of compositions to our generating data and then apply the free monoid construction, because in order to define the sources and targets of the operations in a system of compositions, we require a unit operation at the dimensions below.  Thus we construct the left adjoint to $U$ via an interleaving-style construction, similar to that used by Cheng in \cite{Che10}.  At each dimension we freely add the binary composition operations required for a system of compositions, then apply the free monoid construction at that dimension to generate the operad structure freely; we then move up to the next dimension and repeat the process.  Note that this is not a true interleaving since the system of compositions cannot exist independently of the operad structure.

  \begin{lemma}  \label{lem:SoCladj}
    The functor
      \[
        U \colon \SoC \longrightarrow \nColl
      \]
    has a left adjoint.
  \end{lemma}

  \begin{proof}
    To describe the interleaving-style construction we must define what it means for a collection to have an operad structure up to dimension $k$ for some $k \leq n$.  To do this, we use a truncation functor, defined as follows: for each $k \leq n$, we have
      \[
        \Tr_k \colon \nGSet \longrightarrow \kGSet
      \]
    which sends an $n$-globular set to its underlying $k$-globular set, which has the same set of $m$-cells for all $0 \leq m \leq k$, and the same source and target maps.  Since $\Tr_k(T1) = T(\Tr_k 1)$, this induces a functor
      \[
        \xy
          % POINTS
            (0, 0)*+{\Tr_k \colon \nColl}="nColl";
            (30, 0)*+{\kColl}="kColl";
            (4, -8)*+{X}="X";
            (30, -8)*+{\Tr_k X}="TrX";
            (4, -24)*+{T1}="T1";
            (30, -24)*+{T \Tr_k 1.}="TTr";
          % ARROWS
            {\ar "nColl" ; "kColl"};
            {\ar_x "X" ; "T1"};
            {\ar^{\Tr_k x} "TrX" ; "TTr"};
            {\ar@{|->} (10, -16) ; (24, -16)};
        \endxy
      \]
    Note that we denote both functors by $\Tr_k$; it will be clear from the context which we are using.

    For each $0 \leq k \leq n$, define $\kOpd$ to be the category with:
      \begin{itemize}
        \item objects: collections $\xymatrix{ X \ar[r]^{x} & T1 }$ such that $\Tr_k(\xymatrix{ X \ar[r]^{x} & T1 })$ has the structure of a $k$-operad;
        \item morphisms: maps of collections that preserve the $k$-operad structure.
      \end{itemize}
    We can equip an object of $\kOpd$ with a system of compositions at every dimension up to $(k + 1)$.  Since we have no operad structure at dimension $(k + 1)$, this system of compositions cannot pick out the unit operation at this dimension.

    For each $0 \leq j \leq n$, define an $n$-collection $\xymatrix{ S^{(j)} \ar[r]^{s^{(j)}} & T1 }$ by
    \[
      S^{(j)}_m :=
          \left\{
            \begin{array}{ll}
            S_m & \text{if } m < j, \\
            \{ \beta^m_p \gt 0 \leq p < m \} & \text{if } m = j, \\
            \emptyset & \text{if } m > j; \\
            \end{array}
          \right.
    \]
    with $s^{(j)}$ the inclusion into $T1$.  Given an object $\xymatrix{ X \ar[r]^{x} & T1 }$ of $\kOpd$, where $k \geq j - 1$, a $j$-system of compositions on $X$ consists of a map of collections
      \[
        \xy
          % POINTS
          (-10, 14)*+{S^{(j)}}="S";
          (10, 14)*+{X}="X";
          (0, 0)*+{T1}="T";
          % ARROWS
          {\ar^{\sigma^{(j)}} "S" ; "X"};
          {\ar_{s^{(j)}} "S" ; "T"};
          {\ar^{x} "X" ; "T"}
        \endxy
      \]
    such that the diagram
      \[
        \xy
        % POINTS
        (-20, 0)*+{1}="1";
        (0, 0)*+{\Tr_k S^{(j)}}="S";
        (20, 0)*+{\Tr_k X}="X";
        % ARROWS
        {\ar^{\eta^S} "1" ; "S"};
        {\ar^{\sigma^{(j)}} "S" ; "X"};
        {\ar@/_1.5pc/_{\eta^X} "1" ; "X"}
        \endxy
      \]
    commutes.

    Let $k \geq j - 1$, and define $(k, j)\text{-}\SoC$ to be the category with:
      \begin{itemize}
        \item objects:  an object of $(k, j)\text{-}\SoC$ is a collection $\xymatrix{ X \ar[r]^{x} & T1 }$ such that $\Tr_k( \xymatrix{ X \ar[r]^{x} & T1 })$ has the structure of a $k$-operad, equipped with a $j$-system of compositions;
        \item morphisms:  a morphism in $(k, j)\text{-}\SoC$ is a map of the underlying collections that preserves both the $k$-operad structure and the $j$-system of compositions.
      \end{itemize}
    For each $0 \leq j < n$, we have an inclusion $S^{(j)} \hookrightarrow S^{(j + 1)}$, giving a forgetful functor
      \[
        B_{k, j} \colon (k, j)\text{-}\SoC \longrightarrow (k, j - 1)\text{-}\SoC
      \]
    which forgets the system of compositions at dimension $j$.  When $j < k + 1$, we also have a forgetful functor
      \[
        D_{k, j} \colon (k, j)\text{-}\SoC \longrightarrow (k - 1, j)\text{-}\SoC,
      \]
    which forgets the operad structure at dimension $k$.  We also have a forgetful functor
      \[
        D_{0, 0} \colon (0, 0)\text{-}\SoC = 0\text{-}\mathbf{Opd} \longrightarrow \nColl.
      \]
    Thus we can write the functor $U \colon \SoC \rightarrow \nColl$ as the composite
      \[
        \xymatrix{
          \SoC \iso (n, n)\text{-}\SoC \ar[r]^-{D_{n, n}} & (n - 1, n)\text{-}\SoC \ar[r]^-{B_{n - 1, n}} & (n - 1, n - 1)\text{-}\SoC \ar[r]^-(0.65){D_{n - 1, n - 1}} & \dotsb
                 }
      \]
      \[
        \xymatrix{
          \dotsb \ar[r]^-{B_{1, 2}} & (1, 1)\text{-}\SoC \ar[r]^-{D_{1, 1}} & (0, 1)\text{-}\SoC \ar[r]^-{B_{0, 1}} & (0, 0)\text{-}\SoC \ar[r]^-{D_{0, 0}} & \nColl.
                 }
      \]
    We show that $U$ has a left adjoint by showing that each of its factors has a left adjoint.

    To show that each functor $D_{k, j}$ has a left adjoint, we observe that, for each $0 \leq k \leq n$, there is a forgetful functor
      \[
        D_k \colon \kOpd \longrightarrow (k - 1)\text{-}\mathbf{Opd}
      \]
    that forgets the operad structure at dimension $k$ (when $k = 0$, we take $(k - 1)\text{-}\mathbf{Opd} = \nColl$, and $D_{0} = D_{0, 0}$).  We recall from \cite[Proposition 2.1]{Che10} that each $D_k$ has a left adjoint, which we denote by $C_k$; this is constructed using a dimension-by-dimension decomposition of Kelly's free monoid construction from \cite{Kel80}.  To show that this lifts to a functor
      \[
        C_{k, j} \colon (k - 1, j)\text{-}\SoC \longrightarrow (k, j)\text{-}\SoC,
      \]
    we must check that, when we apply $C_k$ to a $(k - 1)$-operad in $(k - 1, j)\text{-}\SoC$, it retains its $j$-system of compositions.  This is true when $j < k$, since $C_k$ leaves dimensions below dimension $k$ unchanged.  When $j = k$, given an object $X$ of $(k - 1, k)\text{-}\SoC$, with $k$-system of compositions $\sigma^{(k)}$, we have an inclusion map
      \[
        X_k \hookrightarrow C_k X_k
      \]
    given by the component of the unit of the adjunction $C_k \ladj D_k$ at dimension $k$.  Thus we can equip $C_k X$ with a $k$-system of compositions, given by $\sigma^{(k)}$ at all dimension less than $k$, and at dimension $k$ given by
      \[
        \xy
          % POINTS
          (0, 0)*+{S^{(k)}_k}="0";
          (16, 0)*+{X_k}="1";
          (34, 0)*+{C_k X_k.}="2";
          % ARROWS
          {\ar^-{\sigma^{(k)}_k} "0" ; "1"};
          {\ar@{^(->} "1" ; "2"};
        \endxy
      \]
    Thus for each $0 \leq j \leq k \leq n$ we have an adjunction $C_{k, j} \ladj D_{k, j}$.

    Let $0 \leq j \leq k \leq n$.  We now construct a putative left adjoint
      \[
        A_{k, j} \colon (k, j - 1)\text{-}\SoC \longrightarrow (k, j)\text{-}\SoC
      \]
    to the functor $B_{k, j}$.  We first describe the action on objects.  Let $\xymatrix{ X \ar[r]^{x} & T1 }$ be an object of $(k, j - 1)\text{-}\SoC$, and write $\sigma \colon S^{(j - 1)} \rightarrow \Tr_{j - 1} X$ for its $(j - 1)$-system of compositions.  We define
      \[
        A_{k, j}(\xymatrix{ X \ar[r]^{x} & T1 }) = (\xymatrix{ \tilde{X} \ar[r]^{\tilde{x}} & T1 }),
      \]
    where
      \begin{itemize}
        \item $\tilde{X}$ is defined by
    \[
      \tilde{X}_m :=
          \left\{
            \begin{array}{ll}
            X_m & \text{if } m \neq j, \\
            X_m \amalg S^{(j)}_m & \text{if } m = j, \\
            \end{array}
          \right.
    \]
      with source and target maps $s$, $t \colon \tilde{X}_j \rightarrow \tilde{X}_{j - 1} = X_{j - 1}$ given by
    \[
      s(\beta^m_p) = t(\beta^m_p) :=
          \left\{
            \begin{array}{ll}
            \mu^X(\beta^{m - 1}_p, \eta_{m - 1} \comp^{m - 1}_p \eta_{m - 1}) & \text{if } m - 1 \neq p, \\
            \eta^X_{m - 1}(1) & \text{if } m - 1 = p; \\
            \end{array}
          \right.
    \]
      \item $\tilde{x}$ is defined by
    \[
      \tilde{x}_m :=
          \left\{
            \begin{array}{ll}
            x_m & \text{if } m \neq j, \\
            x_m \amalg s^{(j)}_m & \text{if } m = j; \\
            \end{array}
          \right.
    \]
      \item the $j$-system of compositions $\sigma^{(j)} \colon S^{(j)} \rightarrow \Tr_j \tilde{X}$ is given by $\sigma^{(j)}_m = \sigma^{(j - 1)}_m$ for $m < j$, and
            \[
              \sigma^{(j)}_j \colon S^{(j)}_j \hookrightarrow \tilde{X}_j = X_j \amalg S^{(j)}_j
            \]
          is given by the coprojection into the coproduct.
      \end{itemize}

    For the action on morphisms, given a map $f \colon X \rightarrow Y$ in $(k, j - 1)\text{-}\SoC$, we define $A_{k, j}(f) = \tilde{f}$, where
    \[
      \tilde{f}_m :=
          \left\{
            \begin{array}{ll}
            f_m & \text{if } m \neq j, \\
            f_m \amalg \id_{S_m} & \text{if } m = j. \\
            \end{array}
          \right.
    \]

    We now show that $A_{k, j} \ladj B_{k, j}$.  Define a natural transformation $\alpha \colon 1 \Rightarrow B_{k, j}A_{k, j}$ whose component at $\xymatrix{ X \ar[r]^{x} & T1 }$ in $(k, j - 1)\text{-}\SoC$ is given by
    \[
      (\alpha_X)_m :=
          \left\{
            \begin{array}{ll}
            \id_{X_m} & \text{if } m \neq j, \\
            X_m \hookrightarrow X_m \amalg S^{(j)}_m & \text{if } m = j. \\
            \end{array}
          \right.
    \]
    Define a natural transformation $\beta \colon A_{k, j}B_{k, j} \Rightarrow 1$ whose component at $\xymatrix{ X \ar[r]^{x} & T1 }$ in $(k, j)\text{-}\SoC$ is given by
    \[
      (\beta_X)_m :=
          \left\{
            \begin{array}{ll}
            \id_{X_m} & \text{if } m \neq j, \\
            \id_{X_m} \amalg \sigma^{(j)}_m & \text{if } m = j; \\
            \end{array}
          \right.
    \]
    where $\sigma^{(j)} \colon S^{(j)} \rightarrow X$ is the $j$-system of compositions on $X$.  We now check the triangle identities to show that $A_{k, j} \ladj B_{k, j}$ with unit $\alpha$ and counit $\beta$.  Since the components of $\alpha$ and $\beta$ are equal to the identity at all dimensions other than $j$, we need only check that the triangle identities hold at dimension $j$.  For $\xymatrix{ X \ar[r]^{x} & T1 }$ in $(k, j - 1)\text{-}\SoC$, the diagram
      \[
        \xy
          % POINTS
          (0, 0)*+{X_j \amalg S^{(j)}_j}="XSl";
          (40, 0)*+{(X_j \amalg S^{(j)}_j) \amalg S^{(j)}_j}="XSS";
          (40, -14)*+{X_j \amalg S^{(j)}_j}="XSb";
          % ARROWS
          {\ar^-{(\alpha_{X \amalg S^{(j)}})_j} "XSl" ; "XSS"};
          {\ar^{(\beta_{(X \amalg S^{(j)}) \amalg S^{(j)}})_j} "XSS" ; "XSb"};
          {\ar_{\id} "XSl" ; "XSb"};
        \endxy
      \]
    commutes.  For $\xymatrix{ X \ar[r]^{x} & T1 }$ in $(k, j)\text{-}\SoC$, the diagram
      \[
        \xy
          % POINTS
          (0, 0)*+{X_j}="Xl";
          (24, 0)*+{X_j \amalg S^{(j)}_j}="XS";
          (24, -14)*+{X_j}="Xb";
          % ARROWS
          {\ar^-{(\alpha_X)_j} "Xl" ; "XS"};
          {\ar^{(\beta_{X \amalg S^{(j)}})_j} "XS" ; "Xb"};
          {\ar_{\id} "Xl" ; "Xb"};
        \endxy
      \]
    commutes.  Thus, $A_{k, j} \ladj B_{k, j}$.

    Hence, using the decomposition of $U$ described above, $U$ has a left adjoint, as required.
  \end{proof}

  \begin{lemma}  \label{lem:Umonadic}
    The functor
      \[
        U \colon \SoC \rightarrow \nColl
      \]
    is monadic.
  \end{lemma}

  \begin{proof}
    We show that $U$ is monadic by checking that it satisfies the conditions in Beck's monadicity theorem (Theorem~\ref{thm:Beck}).  By Lemma~\ref{lem:SoCladj}, $U$ has a left adjoint, and since $U$ is the identity on maps, it reflects isomorphisms.  Thus we only need to check the condition regarding coequalisers.

    Consider a pair of maps
      \[
        \xy
          % POINTS
          (0, 0)*+{X}="X";
          (20, 0)*+{Y}="Y";
          (10, -14)*+{T1}="T";
          % ARROWS
          {\ar@<1ex>^{f} "X" ; "Y"};
          {\ar@<-1ex>_{g} "X" ; "Y"};
          {\ar_x "X" ; "T"};
          {\ar^y "Y" ; "T"};
        \endxy
      \]
    in $\SoC$ such that its image under $U$ has a split coequaliser.  We have
      \[
        \xy
          % POINTS
          (0, 0)*+{X}="X";
          (20, 0)*+{Y}="Y";
          (40, 0)*+{Z}="Z";
          (20, -14)*+{T1}="T";
          % ARROWS
          {\ar@<1ex>^{f} "X" ; "Y"};
          {\ar@<-1ex>_{g} "X" ; "Y"};
          {\ar@/_1.5pc/_q "Y" ; "X"};
          {\ar@/_1.5pc/_p "Z" ; "Y"};
          {\ar^e "Y" ; "Z"};
          {\ar_x "X" ; "T"};
          {\ar^y "Y" ; "T"};
          {\ar^z "Z" ; "T"};
        \endxy
      \]
    in $\nColl$, where $\xymatrix{ Z \ar[r]^{z} & T1 }$ is the coequaliser of $(f, g)$, and $p$ and $q$ satisfy the following equations:
      \[
        ep = \id_Z, \; fq = \id_Y, \; gq = pe.
      \]
    To show that $(f, g)$ also has a coequaliser in $\SoC$, we need to show that we can equip $Z$ with an operad structure and a system of compositions, in such a way that $e$ and any maps induced by the universal property of the coequaliser preserve both the operad structure and the system of compositions.

    For the operad structure, define the unit map $\eta^Z$ to be the composite
      \[
        \xy
          % POINTS
          (0, 0)*+{1}="1";
          (14, 0)*+{Y}="Y";
          (28, 0)*+{Z}="Z";
          (14, -14)*+{T1}="T";
          % ARROWS
          {\ar_{\eta^T_1} "1" ; "T"};
          {\ar^y "Y" ; "T"};
          {\ar^z "Z" ; "T"};
          {\ar^{\eta^Y} "1" ; "Y"};
          {\ar^{e} "Y" ; "Z"};
        \endxy
      \]
    and define the multiplication map $\mu^Z$ to be the composite along the top of
          \[
            \xy
            % POINTS
              (0, 0)*+{Z \otimes Z}="ZZ";
              (20, 0)*+{Y\otimes Y}="YY";
              (36, 0)*+{Y}="Y";
              (-10, -10)*+{Z}="Z";
              (10, -10)*+{TZ}="TZ";
              (30, -10)*+{TY}="TY";
              (0, -20)*+{T1}="Tl";
              (30, -20)*+{T^2 1}="TT";
              (50, -30)*+{T1.}="Tr";
              (50, 0)*+{Z}="Zr";
            % ARROWS
              {\ar "ZZ" ; "Z"};
              {\ar^-{p \otimes p} "ZZ" ; "YY"};
              {\ar^-{\mu^Y} "YY" ; "Y"};
              {\ar "YY" ; "TY"};
              {\ar^-{e} "Y" ; "Zr"};
              {\ar_{Tp} "TZ" ; "TY"};
              {\ar^{Ty} "TY" ; "TT"};
              {\ar "ZZ" ; "TZ"};
              {\ar_{z} "Z" ; "Tl"};
              {\ar^{T!} "TZ" ;"Tl"};
              {\ar_{Tz} "TZ" ; "TT"};
              {\ar_{\mu^T_1} "TT" ; "Tr"};
              {\ar^-{z} "Zr" ; "Tr"};
              {\ar^-{y} "Y" ; "Tr"};
              % PULLBACK STUFF
              (-3, -5)*{}; (0,-8)*{} **\dir{-};
              (3, -5)*{}; (0,-8)*{} **\dir{-};
            \endxy
          \]
    The diagrams above show that $\eta^Z$ and $\mu^Z$ are both maps in $\nColl$.

    For the system of compositions, define $\sigma^Z \colon S \rightarrow Z$ to be the composite
      \[
        \xy
          % POINTS
          (0, 0)*+{S}="S";
          (14, 0)*+{Y}="Y";
          (28, 0)*+{Z}="Z";
          (14, -14)*+{T1.}="T";
          % ARROWS
          {\ar_{s} "S" ; "T"};
          {\ar^y "Y" ; "T"};
          {\ar^z "Z" ; "T"};
          {\ar^{\sigma^Y} "S" ; "Y"};
          {\ar^{e} "Y" ; "Z"};
        \endxy
      \]
    The diagram above shows that $\sigma^Z$ is a map in $\nColl$.

    We now check that $e$ preserves the operad structure and the system of compositions.  That $e$ preserves the unit and system of compositions is immediate from the definitions of $\eta^Z$ and $\sigma^Z$; for the multiplication, the diagram
      \[
        \xy
          % POINTS
          (0, 0)*+{Y \otimes Y}="YY";
          (40, 0)*+{Z \otimes Z}="ZZ";
          (20, -14)*+{X \otimes X}="XX";
          (40, -14)*+{Y \otimes Y}="YYr";
          (20, -28)*+{X}="X";
          (40, -28)*+{Y}="Yr";
          (0, -42)*+{Y \otimes Y}="YYb";
          (20, -42)*+{Y}="Yb";
          (40, -42)*+{Z}="Zb";
          % ARROWS
          {\ar^{e \otimes e} "YY" ; "ZZ"};
          {\ar_{\id_{Y \otimes Y}} "YY" ; "YYb"};
          {\ar^{q \otimes q} "YY" ; "XX"};
          {\ar^{p \otimes p} "ZZ" ; "YYr"};
          {\ar^{g \otimes g} "XX" ; "YYr"};
          {\ar_{f \otimes f} "XX" ; "YYb"};
          {\ar^{\mu^X} "XX" ; "X"};
          {\ar^{\mu^Y} "YYr" ; "Yr"};
          {\ar^{g} "X" ; "Yr"};
          {\ar_{f} "X" ; "Yb"};
          {\ar^{e} "Yr" ; "Zb"};
          {\ar_{\mu^Y} "YYb" ; "Yb"};
          {\ar_e "Yb" ; "Zb"};
        \endxy
      \]
    commutes.  Hence $e$ is a map in $\SoC$.

    We now check that the maps induced by the universal property of $Z$ also preserve the operad structure and the system of compositions.  Suppose we have an operad $K$ in $\SoC$, and a map $h \colon Y \rightarrow K$ in $\SoC$.  The universal property of $Z$ in $\nColl$ gives us a unique map $u \colon Z \rightarrow K$ making the diagram
      \[
        \xy
          % POINTS
          (0, 0)*+{X}="X";
          (14, 0)*+{Y}="Y";
          (28, 0)*+{Z}="Z";
          (42, -10)*+{K}="K";
          (14, -20)*+{T1}="T";
          % ARROWS
          {\ar@<1ex>^{f} "X" ; "Y"};
          {\ar@<-1ex>_{g} "X" ; "Y"};
          {\ar_x "X" ; "T"};
          {\ar^y "Y" ; "T"};
          {\ar^z "Z" ; "T"};
          {\ar^{e} "Y" ; "Z"};
          {\ar_{h} "Y" ; "K"};
          {\ar@{-->}^{u} "Z" ; "K"};
          {\ar^k "K" ; "T"};
        \endxy
      \]
    commute in $\nColl$.  The diagrams
      \[
        \xy
          % POINTS
          (0, 0)*+{1}="1";
          (0, -10)*+{Y}="Y";
          (0, -20)*+{Z}="Z";
          (20, -20)*+{K}="K";
          % ARROWS
          {\ar^{\eta^Y} "1" ; "Y"};
          {\ar@/_1pc/_{\eta^Z} "1" ; "Z"};
          {\ar@/^1pc/^{\eta^K} "1" ; "K"};
          {\ar^h "Y" ; "K"};
          {\ar^e "Y" ; "Z"};
          {\ar_u "Z" ; "K"};
        \endxy
      \]
    and
      \[
        \xy
          % POINTS
          (0, 0)*+{Z \otimes Z}="ZZ";
          (20, 0)*+{K \otimes K}="KK";
          (0, -14)*+{Y \otimes Y}="YY";
          (0, -28)*+{Y}="Y";
          (0, -42)*+{Z}="Z";
          (20, -42)*+{K}="K";
          % ARROWS
          {\ar^{u \otimes u} "ZZ" ; "KK"};
          {\ar_{p \otimes p} "ZZ" ; "YY"};
          {\ar_{h \otimes h} "YY" ; "KK"};
          {\ar_{\mu^Y} "YY" ; "Y"};
          {\ar^{\mu^K} "KK" ; "K"};
          {\ar_{h} "Y" ; "K"};
          {\ar_{e} "Y" ; "Z"};
          {\ar_u "Z" ; "K"};
        \endxy
      \]
    commute, so $u$ is a map of operads.  Also, the diagram
      \[
        \xy
          % POINTS
          (0, 0)*+{S}="S";
          (0, -10)*+{Y}="Y";
          (0, -20)*+{Z}="Z";
          (20, -20)*+{K}="K";
          % ARROWS
          {\ar^{\sigma^Y} "S" ; "Y"};
          {\ar@/_1pc/_{\sigma^Z} "S" ; "Z"};
          {\ar@/^1pc/^{\sigma^K} "S" ; "K"};
          {\ar^h "Y" ; "K"};
          {\ar^e "Y" ; "Z"};
          {\ar_u "Z" ; "K"};
        \endxy
      \]
    commutes, so $u$ preserves the system of compositions on $Z$.  Thus $Z$ is the coequaliser of $(f, g)$ in $\SoC$, as required, so $U$ is monadic.
  \end{proof}

  The final step we need to take in order to use Proposition~\ref{prop:Kelly} to prove that $\OCS$ has an initial object is to prove that the functors
    \[
      U \colon \SoC \longrightarrow \nColl,
    \]
  and
    \[
      V \colon \Contr \longrightarrow \nColl
    \]
  are finitary.  To do so, we first give a result describing colimits in slice categories, which gives us a description of colimits in $\nColl$.

  \begin{lemma}  \label{lem:ncollcolims}
    Let $\mathcal{C}$ be a cocomplete category, let $Z$ be an object of $\mathcal{C}$, and let $D \colon \mathbb{I} \rightarrow \mathcal{C}/Z$ be a diagram in the slice category $\mathcal{C}/Z$.  For each $i \in \mathbb{I}$, write
      \[
        \xy
          % POINTS
          (0, 0)*+{X^{(i)}}="Xi";
          (0, -16)*+{Z}="T";
          % ARROWS
          {\ar^{x^{(i)}} "Xi" ; "T"};
        \endxy
      \]
    for the object $D(i)$ in $\mathcal{C}/Z$.  Write
      \[
        X := \colim_{i \in \mathbb{I}} X^{(i)}
      \]
    for the colimit in $\mathcal{C}$, and write
      \[
        c_i \colon X^{(i)} \rightarrow X
      \]
    for the coprojections.  Then the colimit
      \[
        \colim_{i \in \mathbb{I}} D(i)
      \]
    in $\mathcal{C}/Z$ is given by
      \[
        \xy
          % POINTS
          (0, 0)*+{X}="X";
          (0, -16)*+{Z,}="T";
          % ARROWS
          {\ar^{x} "X" ; "T"};
        \endxy
      \]
    where $x$ is the unique map such that, for all $i \in \mathbb{I}$,
      \[
        \xy
          % POINTS
          (0, 0)*+{X^{(i)}}="Xi";
          (16, 0)*+{X}="X";
          (8, -14)*+{Z}="T";
          % ARROWS
          {\ar^{c_i} "Xi" ; "X"};
          {\ar_{x^{(i)}} "Xi" ; "T"};
          {\ar^{x} "X" ; "T"};
        \endxy
      \]
    commutes.
  \end{lemma}

  \begin{proof}
    Let
      \[
        \xy
          % POINTS
          (0, 0)*+{Y}="Y";
          (0, -16)*+{Z}="T";
          % ARROWS
          {\ar^{y} "Y" ; "T"};
        \endxy
      \]
    be an object of $\mathcal{C}/Z$, and let
      \[
        \xy
          % POINTS
          (0, 0)*+{X^{(i)}}="Xi";
          (16, 0)*+{Y}="Y";
          (8, -14)*+{Z}="T";
          % ARROWS
          {\ar^{f_i} "Xi" ; "Y"};
          {\ar_{x^{(i)}} "Xi" ; "T"};
          {\ar^{y} "Y" ; "T"};
        \endxy
      \]
    be a cocone under the diagram $D$ in $\mathcal{C}/Z$.  The universal property of $X$ in $\mathcal{C}$ induces a unique map $f \colon X \rightarrow Y$ such that, for every $i \in \mathbb{I}$, the diagram
      \[
        \xy
          % POINTS
          (0, 0)*+{X^{(i)}}="Xi";
          (-8, -14)*+{X}="X";
          (8, -14)*+{Y}="Y";
          % ARROWS
          {\ar_{c_i} "Xi" ; "X"};
          {\ar^{f_i} "Xi" ; "Y"};
          {\ar_{f} "X" ; "Y"};
        \endxy
      \]
    commutes.  We need to check that $f$ is a map in $\mathcal{C}/Z$, i.e. that the diagram
      \[
        \xy
          % POINTS
          (0, 0)*+{X}="X";
          (16, 0)*+{Y}="Y";
          (8, -14)*+{Z}="T";
          % ARROWS
          {\ar^{f} "X" ; "Y"};
          {\ar_{x} "X" ; "T"};
          {\ar^{y} "Y" ; "T"};
        \endxy
      \]
    commutes.  The maps
      \[
        \xy
          % POINTS
          (0, 0)*+{X^{(i)}}="Xi";
          (16, 0)*+{Y}="Y";
          (32, 0)*+{Z}="T";
          % ARROWS
          {\ar^-{f_i} "Xi" ; "Y"};
          {\ar^-{y} "Y" ; "T"};
        \endxy
      \]
    define a cocone under the diagram defining $X$ in $\nGSet$, so there is a unique map $u \colon X \rightarrow Z$ such that each
      \[
        \xy
          % POINTS
          (0, 0)*+{X^{(i)}}="Xi";
          (16, 0)*+{X}="X";
          (0, -16)*+{Y}="Y";
          (16, -16)*+{Z}="T";
          % ARROWS
          {\ar^-{c_i} "Xi" ; "X"};
          {\ar_{f_i} "Xi" ; "Y"};
          {\ar@{-->}^{u} "X" ; "T"};
          {\ar_{y} "Y" ; "T"};
        \endxy
      \]
    commutes.  Both $u = x$ and $u = yf$ make this diagram commute; thus $x = yf$, so $f$ is a map in $\mathcal{C}/Z$.  Thus $\colim_{i \in \mathbb{I}} D(i)$ is given by
      \[
        \xy
          % POINTS
          (0, 0)*+{X}="X";
          (0, -16)*+{Z,}="T";
          % ARROWS
          {\ar^{x} "X" ; "T"};
        \endxy
      \]
    as required.
  \end{proof}

  \begin{lemma}  \label{lem:Vfinitary}
    The functor
      \[
        V \colon \Contr \rightarrow \nColl
      \]
    is finitary.
  \end{lemma}

  \begin{proof}
    First observe that, by Lemma~\ref{lem:ncollcolims} and the fact that $\nGSet$ is a presheaf category and is thus cocomplete, $\nColl$ is cocomplete.  Thus our approach is to show that $V$ creates filtered colimits; since all filtered colimits exist in $\nColl$, this implies that $V$ preserves filtered colimits.

    Let $D \colon \mathbb{I} \rightarrow \Contr$ be a filtered diagram in $\Contr$.  Since $\nColl$ is cocomplete, $VD$ has a colimit in $\nColl$.  Write
      \[
        \xy
          % POINTS
          (0, 0)*+{X^{(i)}}="Xi";
          (0, -16)*+{T1}="T";
          % ARROWS
          {\ar^{x^{(i)}} "Xi" ; "T"};
        \endxy
      \]
    for the underlying collection of the object $D(i)$ in $\Contr$, and write $\gamma^{(i)}$ for the contraction on this collection.  Write
      \[
        \xy
          % POINTS
          (0, 0)*+{X}="X";
          (0, -16)*+{T1,}="T";
          % ARROWS
          {\ar^{x} "X" ; "T"};
        \endxy
      \]
    for the colimit $\colim_{i \in \mathbb{I}} VD(i)$ in $\nColl$, and write $c_i \colon X^{(i)} \rightarrow X$ for the coprojections.  We will show that there is a unique way to equip $X$ with a contraction in such a way that each $c_i$ preserves the contraction structure, and that this gives the colimit of $D$ in $\Contr$.

    Let $0 \leq m < n$, and let $a$, $b$ be parallel $m$-cells in $X$ with $x(a) = x(b)$.  By Lemma~\ref{lem:ncollcolims},
      \[
        X = \colim_{i \in \mathbb{I}} X^{(i)}
      \]
    in $\nGSet$, and since colimits are computed pointwise in presheaf categories,
      \[
        X_m = \colim_{i \in \mathbb{I}} X^{(i)}_m.
      \]
    Thus there exist $i$, $j \in \mathbb{I}$ and $u \in X^{(i)}_m$, $v \in X^{(j)}_m$ such that
      \[
        c_i(u) = a, \; c_j(v) = b.
      \]
    By definition of the map $x$, we have
      \[
        x^{(i)}c_i(u) = x(a) = x(b) = x^{(j)}c_j(v).
      \]
    Since $\mathbb{I}$ is filtered, we have a cocone in the diagram $D$ under $X^{(i)}_m$ and $X^{(j)}_m$, with vertex $X^{(k)}_m$ for some $k \in \mathbb{I}$ and maps $d_i \colon X^{(i)}_m \rightarrow X^{(K)}_m$, $d_j \colon X^{(j)}_m \rightarrow X^{(K)}_m$, such that $d_i(u)$ and $d_j(v)$ are parallel in $X^{(k)}$.  Thus we can define a contraction $\gamma$ on $x \colon X \rightarrow T1$ by
      \[
        \gamma(a, b) = c_k \gamma^{(k)}(d_i(u), d_j(v)).
      \]
    The fact that $\mathbb{I}$ is filtered, and commutativity of the universal cocone defining $X$, ensure that this definition is independent of the choice of $i$, $j$, $k$.  By definition of $\gamma$, each $c_i$ preserves the contraction structure, and furthermore this is the only way to equip $x \colon X \rightarrow T1$ with a contraction in a such a way that the $c_i$'s preserve contractions.

    We now check that maps induced by the universal property of $X$ also preserve contractions.  Suppose we have a collection
      \[
        \xy
          % POINTS
          (0, 0)*+{K}="K";
          (0, -16)*+{T1}="T";
          % ARROWS
          {\ar^{k} "K" ; "T"};
        \endxy
      \]
    with contraction $\delta$, and a cocone
      \[
        \xy
          % POINTS
          (0, 0)*+{X^{(i)}}="Xi";
          (16, 0)*+{K}="K";
          (8, -14)*+{T1}="T";
          % ARROWS
          {\ar^{c_i} "Xi" ; "K"};
          {\ar_{x^{(i)}} "Xi" ; "T"};
          {\ar^{k} "K" ; "T"};
        \endxy
      \]
    in $\Contr$.  Then there is a unique map of collections $u \colon X \rightarrow K$ such that each
      \[
        \xy
          % POINTS
          (0, 0)*+{X^{(i)}}="Xi";
          (16, 0)*+{X}="X";
          (8, -14)*+{K}="K";
          % ARROWS
          {\ar^{c_i} "Xi" ; "X"};
          {\ar_{r_i} "Xi" ; "K"};
          {\ar^{u} "X" ; "K"};
        \endxy
      \]
    commutes.  We must show that $u$ preserves the contraction structure.  Let $0 \leq m < n$, and let $(a, b) \in X^c_{m + 1}$.  By the definition of $\gamma$, we can pick $k \in \mathbb{I}$ such that there exist $u$, $v \in X^{(k)}_m$ with
      \[
        \gamma(a, b) = c_k\gamma^{(k)}(u, v).
      \]
    Thus
      \[
        u\gamma(a, b) = uc_k\gamma^{(k)}(u, v) = r_k\gamma^{(k)}(u, v) = \delta(r_k(u), r_k(v)),
      \]
    as required.  Hence $u$ is a map in $\Contr$.

    So $V$ creates filtered colimits; hence, since $\nColl$ is cocomplete, $V$ preserves filtered colimits.
  \end{proof}

  \begin{lemma}  \label{lem:Ufinitary}
    The functor
      \[
        U \colon \SoC \rightarrow \nColl
      \]
    is finitary.
  \end{lemma}

  \begin{proof}
    Our approach is analogous to that used in the proof of Lemma~\ref{lem:Vfinitary}; we show that $U$ creates filtered colimits, and since $\nColl$ is cocomplete (and in particular, it has all filtered colimits) this implies that $U$ preserves filtered colimits.

    Let $D \colon \mathbb{I} \rightarrow \SoC$ be a filtered diagram in $\SoC$.  Since $\nColl$ is cocomplete, $UD$ has a colimit in $\nColl$.  Write
      \[
        \xy
          % POINTS
          (0, 0)*+{X^{(i)}}="Xi";
          (0, -16)*+{T1}="T";
          % ARROWS
          {\ar^{x^{(i)}} "Xi" ; "T"};
        \endxy
      \]
    for the underlying collection of the object $D(i)$ in $\SoC$; write $\eta^{(i)}$ and $\mu^{(i)}$ for the unit and multiplication maps for the operad structure on this collection, and $\sigma^{(i)}$ for its system of compositions.  Write
      \[
        \xy
          % POINTS
          (0, 0)*+{X}="X";
          (0, -16)*+{T1,}="T";
          % ARROWS
          {\ar^{x} "X" ; "T"};
        \endxy
      \]
    for the colimit $\colim_{i \in \mathbb{I}} UD(i)$ in $\nColl$, and write $c_i \colon X^{(i)} \rightarrow X$ for the coprojections.  We will show that there is a unique way to equip $X$ with an operad structure and a system of compositions in such a way that each $c_i$ preserves the operad structure and the system of compositions, and that this gives the colimits of $D$ in $\SoC$.

    Let $i \in \mathbb{I}$ and define the unit $\eta^X$ and system of compositions $\sigma^X$ on $X$ to be given by the composites
      \[
        \xy
          % POINTS
          (0, 0)*+{1}="1";
          (16, 0)*+{X^{(i)}}="Xil";
          (32, 0)*+{X,}="Xl";
          (48, 0)*+{S}="S";
          (64, 0)*+{X^{(i)}}="Xir";
          (80, 0)*+{X.}="Xr";
          % ARROWS
          {\ar^{\eta^{(i)}} "1" ; "Xil"};
          {\ar^{c_i} "Xil" ; "Xl"};
          {\ar@/_1.5pc/_{\eta^X} "1" ; "Xl"};
          {\ar^{\sigma^{(i)}} "S" ; "Xir"};
          {\ar^{c_i} "Xir" ; "Xr"};
          {\ar@/_1.5pc/_{\sigma^X} "S" ; "Xr"};
        \endxy
      \]
    Commutativity of the universal cocone defining $X$ ensures that these maps are well-defined and independent of the choice of $i$.

    To define the multiplication map $\mu^X$ we take an elementary approach.  Let $0 \leq m \leq n$, and let $(\alpha, \beta) \in X \otimes X_m$, so $\alpha \in X_m$, $\beta \in TX_m$, and $z(\alpha) = T!(\beta)$.  Since $X$ is computed pointwise, so each $X_m$ is computed in $\Set$, we can find $i \in \mathbb{I}$ and $a \in X^{(i)}_m$ such that $c_i(a) = \alpha$.  The element $\beta$ is a freely generated composite of $m$-cells $\beta_k \in X_m$, indexed over some set $K$ with $k \in K$.  For each $\beta_k$, we can find $j_k \in \mathbb{I}$ with $b_k \in X^{(j_k)}$ such that $c_{j_k}(b_k) = \beta_k$.  Since $\mathbb{I}$ is filtered, we have a cocone in $\mathbb{I}$ under $X^{(i)}$ and the $X^{(j_k)}$'s.  Write $X^{(p)}$ for the vertex of this cocone, and write
      \[
        d_i \colon X^{(i)} \rightarrow X^{(p)}, \; d_{j_k} \colon X^{(j_k)} \rightarrow X^{(p)}
      \]
    for the coprojections.  We can then find $b \in TX^{p}_m$ with $Tc_p(b) = \beta$, given by the appropriate composite of the $d_{j_k}(b_k)$'s.

    We thus define $\mu^X$ by
      \[
        \mu^X(\alpha, \beta) := c_p\mu^{(p)}(d_i(a), b).
      \]
    Commutativity of the universal cocone defining $X$ ensures that this maps is well-defined and independent of the choices of $i$, $j_k$, and $p$.  We now check that $\mu^X$ is a map of collections.  Let $(\alpha, \beta) \in X \otimes X$.  Then we have
      \begin{align*}
        x\mu^X(\alpha, \beta) & = xc_p\mu^{(p)}(d_i(a), b) \\
          & = x^{(p)}\mu^{(p)}(d_i(a), b) \\
          & = \mu^T_1 \comp Tx^{(p)} \comp x^{(p)}_{X^{(p)}}(d_i(a), b) \\
          & = \mu^T_1 \comp Tx \comp xc_p(d_i(a), b) \\
          & = \mu^T_1 \comp Tx \comp x(\alpha, \beta),
      \end{align*}
    where $i$, $p$, $d_i$, $a$ and $b$ are as above.  Thus the triangle in the diagram
          \[
            \xy
            % POINTS
              (0, 0)*+{X \otimes X}="XX";
              (-10, -10)*+{X}="X";
              (10, -10)*+{TX}="TX";
              (0, -20)*+{T1}="Tl";
              (20, -20)*+{T^2 1}="TT";
              (30, -30)*+{T1}="Tr";
              (30, 0)*+{X}="Xr";
            % ARROWS
              {\ar "XX" ; "X"};
              {\ar "XX" ; "TX"};
              {\ar_x "X" ; "Tl"};
              {\ar^{T!} "TX" ;"Tl"};
              {\ar_{Tx} "TX" ; "TT"};
              {\ar_{\mu^T_1} "TT" ; "Tr"};
              {\ar^{\mu^X} "XX" ; "Xr"};
              {\ar^-{x} "Xr" ; "Tr"};
              % PULLBACK STUFF
              (-3, -5)*{}; (0,-8)*{} **\dir{-};
              (3, -5)*{}; (0,-8)*{} **\dir{-};
            \endxy
          \]
    commutes, so $\mu^X$ is a map of collections.

    It is immediate from the definitions of $\eta^X$, $\mu^X$, and $\sigma^X$ that, for each $i \in \mathbb{I}$, the coprojection $c_i$ preserves the operad structure and the system of compositions on $X^{(i)}$; furthermore, this is the only way to equip $X$ with an operad structure and a system of compositions such that this is true.

    We now check that maps induced by the universal property of $X$ preserve the operad structure and system of compositions.  Suppose we have an operad
      \[
        \xy
          % POINTS
          (0, 0)*+{K}="K";
          (0, -16)*+{T1}="T";
          % ARROWS
          {\ar^{k} "K" ; "T"};
        \endxy
      \]
    with unit $\eta^K$, multiplication $\mu^K$, and system of compositions $\sigma^K$, and a cocone
      \[
        \xy
          % POINTS
          (0, 0)*+{X^{(i)}}="Xi";
          (16, 0)*+{K}="K";
          (8, -14)*+{T1}="T";
          % ARROWS
          {\ar^{c_i} "Xi" ; "K"};
          {\ar_{x^{(i)}} "Xi" ; "T"};
          {\ar^{k} "K" ; "T"};
        \endxy
      \]
    in $\SoC$.  Then there is a unique map of collections $u \colon X \rightarrow K$ such that each
      \[
        \xy
          % POINTS
          (0, 0)*+{X^{(i)}}="Xi";
          (16, 0)*+{X}="X";
          (8, -14)*+{K}="K";
          % ARROWS
          {\ar^{c_i} "Xi" ; "X"};
          {\ar_{r_i} "Xi" ; "K"};
          {\ar^{u} "X" ; "K"};
        \endxy
      \]
    commutes.  We must show that $u$ preserves the operad structure and the system of compositions on $X$.  The diagrams
      \[
        \xy
          % POINTS
          (-16, 0)*+{1}="1";
          (0, 0)*+{X^{(i)}}="Xi";
          (16, 0)*+{X}="X";
          (0, -14)*+{K}="K";
          % ARROWS
          {\ar^{\eta^{(i)}} "1" ; "Xi"};
          {\ar_{\eta^k} "1" ; "K"};
          {\ar@/^1.5pc/^{\eta^X} "1" ; "X"};
          {\ar^{c_i} "Xi" ; "X"};
          {\ar_{r_i} "Xi" ; "K"};
          {\ar^{u} "X" ; "K"};
        \endxy
      \]
    and
      \[
        \xy
          % POINTS
          (-16, 0)*+{S}="S";
          (0, 0)*+{X^{(i)}}="Xi";
          (16, 0)*+{X}="X";
          (0, -14)*+{K}="K";
          % ARROWS
          {\ar^{\sigma^{(i)}} "S" ; "Xi"};
          {\ar_{\sigma^k} "S" ; "K"};
          {\ar@/^1.5pc/^{\sigma^X} "S" ; "X"};
          {\ar^{c_i} "Xi" ; "X"};
          {\ar_{r_i} "Xi" ; "K"};
          {\ar^{u} "X" ; "K"};
        \endxy
      \]
    commute, so $u$ preserves the unit and the system of compositions.  For preservation of the multiplication, we need to show that the diagram
      \[
        \xy
          % POINTS
          (0, 0)*+{X \otimes X}="XX";
          (20, 0)*+{K \otimes K}="KK";
          (0, -16)*+{X}="X";
          (20, -16)*+{K}="K";
          % ARROWS
          {\ar^-{u \otimes u} "XX" ; "KK"};
          {\ar_{\mu^Z} "XX" ; "X"};
          {\ar^{\mu^K} "KK" ; "K"};
          {\ar_{u} "X" ; "K"};
        \endxy
      \]
    commutes; this is true since, given $(\alpha, \beta) \in X \otimes X$, we have
      \begin{align*}
        \mu^K \comp u \otimes u (\alpha, \beta) & = \mu^K(u(\alpha), Tu(\beta))  \\
          & = \mu^K(uc_p d_i(a), T(ic_p)(b))  \\
          & = \mu^K(r_p d_i(a), Tr_p(b))  \\
          & = r_p\mu^{(p)}(d_i(a), b)  \\
          & = uc_p\mu^{(p)}(d_i, b)  \\
          & = u\mu^X(\alpha, \beta),
      \end{align*}
    where $i$, $p$, $d_i$, $a$ and $b$ are as in the definition of $\mu^X$.

    Thus $U$ creates filtered colimits; since $\nColl$ is cocomplete, $U$ preserves filtered colimits.
  \end{proof}

  We now have all the results required to prove the following proposition:

  \begin{prop}  \label{prop:initial}
    The category $\OCS$ has an initial object.
  \end{prop}

  \begin{proof}
    Our aim is to show that the category $\OCS$ has an initial object.  Recall that there is a forgetful functor
      \[
        \OCS \longrightarrow \nColl,
      \]
    and that our approach is to show that this forgetful functor has a left adjoint, then apply that left adjoint to the initial object
      \[
        \xy
          % POINTS
          (0, 0)*+{\emptyset}="K";
          (0, -16)*+{T1}="T";
          % ARROWS
          {\ar^{!} "K" ; "T"};
        \endxy
      \]
    in $\nColl$; since the initial object is the colimit of the empty diagram, and left adjoints preserve colimits, this will give us the initial object in $\OCS$.

    The forgetful functor $\OCS \rightarrow \nColl$ can be factorised as either of the composites in the pullback square
      \[
        \xy
          % POINTS
          (0, 0)*+{\mathcal{\OCS}}="D";
          (16, 0)*+{\mathcal{\Contr}}="C";
          (0, -16)*+{\mathcal{\SoC}}="B";
          (16, -16)*+{\nColl,}="A";
          % ARROWS
          {\ar "D" ; "C"};
          {\ar "D" ; "B"};
          {\ar_-{U} "B" ; "A"};
          {\ar^{V} "C" ; "A"};
          % PULLBACK STUFF
          (6,-1)*{}; (6,-5)*{} **\dir{-};
          (2,-5)*{}; (6,-5)*{} **\dir{-};
        \endxy
     \]
    and, by Proposition~\ref{prop:Kelly}, it has a left adjoint if $\nColl$ is locally finitely presentable and each of $U$ and $V$ is finitary and monadic.  The functor $U$ is monadic by Lemma~\ref{lem:Umonadic} and finitary by Lemma~\ref{lem:Ufinitary}; the functor $V$ is monadic by Lemma~\ref{lem:Vmonadic} and finitary by Lemma~\ref{lem:Vfinitary}; since $\nColl$ is a slice of a presheaf category, it is also a presheaf category, so is locally finitely presentable (see \cite[Proposition~1.1.7 and Appendix~G]{Lei04} and \cite[Example~5.2.2(b)]{Bor94}). Specifically, given a small category $\mathbb{A}$ and a presheaf $X \colon \mathbb{A}^{\op} \rightarrow \Set$, we have
      \[
        [\mathbb{A}^{\op}, \Set]/X \simeq [(\mathbb{A}/X)^{\op}, \Set],
      \]
    where $\mathbb{A}/X$ is the \emph{category of elements of $X$}, which has:
      \begin{itemize}
        \item objects: pairs $(A, x)$, where $A \in \mathbb{A}$ and $x \in XA$;
        \item morphisms: a morphism $f \colon (A, x) \rightarrow (A', x')$ in $\mathbb{A}/X$ consists of a map
          \[
            f \colon A \rightarrow A'
          \]
        in $\mathbb{A}$ such that $Xf(x') = x$.
      \end{itemize}
    In the case of $\nColl$ where $\mathbb{A} = \nGSet$ and $X = T1$, the category of elements $(\mathbb{A}/X)^{\op}$ has one object for each globular pasting diagram, and morphisms generated by source and target maps.

    Thus, the forgetful functor $\OCS \rightarrow \nColl$ has a left adjoint, so $\OCS$ has an initial object.
  \end{proof}

  \section{Leinster weak $n$-categories}  \label{sect:Leinster}

  Various authors \cite{Ber02, Lei02, Lei00, Cis07, Gar10, vdBG11, Che11} have considered variants of Batanin's definition.  Many of these variants take the approach of relaxing the choice of operad, by defining a weak $n$-category to be an algebra for any operad that can be equipped with a contraction and system of compositions.  In this section, we recall a variant of the definition that takes a different approach, due to Leinster~\cite{Lei98};  we refer to the resulting notion of weak $n$-category as a ``Leinster weak $n$-category''.  The key distinction between Leinster's variant and Batanin's original definition is that, rather than using a contraction and system of compositions, Leinster ensures the existence of both composition operations and contraction operations using a single piece of extra structure, called an ``unbiased contraction'' (note that Leinster simply uses the term ``contraction'' for this concept, and uses the term ``coherence'' for Batanin's contractions).  An unbiased contraction on an operad lifts all cells from $T1$, not just identity cells as in a contraction.  As well as giving the usual constraint cells, an unbiased contraction gives a composition operation for each non-identity cell in $T1$.  Thus for any globular pasting diagram there is an operation, specified by the unbiased contraction, which we think of as telling us how to compose a pasting diagram of that shape ``all at once''.  Operads equipped with unbiased contractions form a category, and this category has an initial object; a Leinster weak $n$-category is defined to be an algebra for this initial operad.

   We recall briefly the definition of Leinster weak $n$-categories.  The majority of this section concerns the relationship between the unbiased contractions of Leinster and the contractions and systems of compositions of Batanin.  First we recall a result of Leinster~\cite[Examples 10.1.2 and 10.1.4]{Lei04} that any operad with an unbiased contraction can be equipped with a contraction and system of compositions in a canonical way.  We then prove a conjecture of Leinster~\cite[10.1]{Lei04} which states that any operad with a contraction and system of compositions can be equipped with an unbiased contraction.  The proof consists of picking a binary bracketing for each pasting diagram in $T1$, then composing these bracketings with contraction cells to obtain unbiased contraction cells with the correct sources and targets.  Since we have to make arbitrary choices of bracketings during this process, there is no canonical way doing this.

   We begin by recalling the definition of unbiased contraction \cite{Lei98}.

  \begin{defn}
    An \emph{unbiased contraction} $\gamma$ on an $n$-globular collection
      \[
        \xymatrix{ K \ar[r]^k & T1 }
      \]
    consists of, for all $1 \leq m \leq n$, and for each $\pi \in (T1)_m$, a function
      \[
        \gamma_{\pi} \colon C_K(\pi) \rightarrow K(\pi)
      \]
    such that, for all $(a, b) \in C_K(\pi)$,
      \[
        s\gamma_{\pi}(a, b) = a, \; t\gamma_{\pi}(a, b) = b.
      \]
    We also require that, for $\alpha$, $\beta \in K_n$, if
      \[
        s(\alpha) = s(\beta), \; t(\alpha) = t(\beta), \; k(\alpha) = k(\beta),
      \]
    then $\alpha = \beta$.
  \end{defn}

  The key difference between the unbiased contractions of Leinster and the (biased) contractions of Batanin, defined in Definition~\ref{defn:Bcontr}, is that unbiased contractions lift all cells from $T1$, not just the identities.  Thus in an operad $\xymatrix{ K \ar[r]^k & T1 }$ equipped with an unbiased contraction, there is a contraction cell for each cell of $T1$, giving us a composition operation in $K$ of each arity.  This gives unbiased composition in $K$ (rather than just the binary composition given by a system of compositions); thus, when using operads with unbiased contraction, we have no need for a system of compositions.

  \begin{defn}  \label{defn:OUC}
    Define $\OUC$ to be the category with
      \begin{itemize}
        \item objects: operads $\xymatrix{ K \ar[r]^k & T1 }$ equipped with an unbiased contraction $\gamma$;
        \item morphisms: for operads $\xymatrix{ K \ar[r]^k & T1 }$, $\xymatrix{ K' \ar[r]^{k'} & T1 }$, respectively equipped with unbiased contractions $\gamma$, $\gamma'$, a morphisms $u \colon K \rightarrow K'$ consists of a map of the underlying operads such that, for all $1 \leq m \leq n$, $\pi \in (T1)_m$, $(a, b) \in C_K(\pi)$,
              \[
                u_m(\gamma_{\pi}(a, b)) = \gamma'_{\pi}(u_{m - 1}(a), u_{m - 1}(b)).
              \]
      \end{itemize}
  \end{defn}

  We often refer to an operad with an unbiased contraction simply as a \emph{Leinster operad}.

  \begin{lemma}
    The category $\OUC$ has an initial object, denoted $\xymatrix{ L \ar[r]^l & T1 }$.
  \end{lemma}

  This lemma was originally proved by Leinster in his thesis \cite{Lei00}; an explicit construction of $\xymatrix{ L \ar[r]^l & T1 }$ is given by Cheng in \cite{Che10}.

  \begin{defn}
    A Leinster weak $n$-category is an algebra for the $n$-globular operad $\xymatrix{ L \ar[r]^l & T1 }$.  The category of Leinster weak $n$-categories is $L\Alg$.
  \end{defn}

  We now discuss the relationship between Leinster operads and Batanin operads.  We first recall a theorem of Leinster~\cite[Examples 10.1.2 and 10.1.4]{Lei04}, which states that every Leinster operad can be equipped with a contraction and system of compositions (thus giving it the structure of a Batanin operad) in a canonical way.  We then prove a conjecture of Leinster~\cite[Section~10.1]{Lei04}, which states that any Batanin operad can be equipped with an unbiased contraction (giving it the structure of a Leinster operad), though not in a canonical way;  the proof of this result is new.

  Since the algebras for an operad are not affected by a choice of system of compositions, contraction, or unbiased contraction, one consequence of these theorems is that any result that holds for algebras for a Batanin operad also holds for algebras for a Leinster operad (and vice versa).  We use this fact in Section~\ref{sect:globopcoh} to prove several coherence theorems that are valid for both Batanin and Leinster weak $n$-categories, whilst working with whichever notion is more technically convenient in the case of each proof.

  The following the theorem is due to Leinster~\cite[Examples 10.1.2 and 10.1.4]{Lei04}.

  \begin{thm}  \label{thm:OUCtoOCS}
    Let $K$ be an $n$-globular operad with unbiased contraction $\gamma$.  Then $K$ can be equipped with a contraction and a system of compositions in a canonical way.
  \end{thm}

  The converse of this theorem is a conjecture of Leinster~\cite[Section~10.1]{Lei04}; we now prove it for the first time.

  \begin{thm}  \label{thm:OCStoOUC}
    Let $K$ be an $n$-globular operad with contraction $\gamma$ and system of compositions $\sigma$.  Then $K$ can be equipped with an unbiased contraction.
  \end{thm}

  Our approach to prove this is as follows: first, we define a map $\hat{k} \colon T1 \rightarrow K$, which uses the contraction on $k$ to lift identity cells in $T1$, and picks a binary bracketing for each non-identity cell.  This bracketing is constructed using the system of compositions on $K$; the choice of bracketing is arbitrary.  To extend this to an unbiased contraction on $k$ we need to specify, for all $1 \leq m \leq n$, and for each $\pi \in T1_m$ and $a$, $b \in C_K(\pi)$, an unbiased contraction cell
    \[
      \gamma_{\pi}(a, b) \colon a \longrightarrow b.
    \]
  To obtain this unbiased contraction cell we start with the cell $\hat{k}(\pi)$; since $\hat{k}$ is a section to $k$ this cell maps to $\pi$ under $k$, but in general it does not have the desired source and target.  In order to obtain a cell with source $a$ and target $b$ we compose $\hat{k}(\pi)$ with contraction cells, first composing $\hat{k}(\pi)$ with contraction $1$-cells to obtain a cell with the desired source and target $0$-cells, then composing the resulting cell with contraction $2$-cells to obtain a cell with the desired source and target $1$-cells, and so on; this composition is performed using the system of compositions on $K$.  The resulting cell has the desired source and target and, since contraction cells map to identities under $k$, and $\hat{k}$ is a section to $k$, this cell maps to $\pi$ under $k$.  Note that the section $\hat{k}$ is not just used to prove that $k$ is surjective at each dimension; it is also used in the construction of an unbiased contraction on a Batanin operad in the proof of Theorem~\ref{thm:OCStoOUC}.

  \begin{lemma}  \label{lem:khat}
    Let $K$ be an $n$-globular operad with contraction $\gamma$ and system of compositions $\sigma$.  Then $k$ has a section $\hat{k} \colon T1 \rightarrow K$ in $\nGSet$, so for all $0 \leq m  \leq n$, $k_m \colon K_m \rightarrow T1_m$ is surjective.
  \end{lemma}

  \begin{proof}
    Our approach is first to define $\hat{k}$, then show it is a section to $k$ and therefore each $k_m$ is surjective. To define $\hat{k} \colon T1 \rightarrow K$, we use a description of $T1$ due to Leinster~\cite[Section~8.1]{Lei04}.  For a set $X$, write $X^*$ for the underlying set of the free monoid on $X$ (so $X^*$ is the set of all finite strings of elements of $X$, including the empty string, which we write as $\emptyset$).  Define $T1$ inductively as follows:
      \begin{itemize}
        \item $T1_0 = 1$;
        \item for $1 \leq m \leq n$, $T1_m = T1^*_{m - 1}$.
      \end{itemize}
    The source and target maps are defined as follows:
      \begin{itemize}
        \item for $m = 1$, $s = t = \, ! \colon T1_m \rightarrow T1_0$;
        % m = 1 -- weird; spacing around pling is also weird/wrong
        \item for $m > 1$, $s = t \colon T1_m \rightarrow T1_{m - 1}$ is defined by, for $(\pi_1, \pi_2, \dotsc, \pi_i) \in T1_m$,
              \[
                s(\pi_1, \pi_2, \dotsc, \pi_i) = (s(\pi_1), s(\pi_2), \dotsc, s(\pi_i)).
              \]
      \end{itemize}

    This description of $T1$ is technically convenient, but it hides what is going on conceptually.  The element $(\pi_1, \pi_2, \dotsc, \pi_i)$ of $T1_m$ should not be visualised as a string of $(m - 1)$-cells; instead, we increase the dimension of each cell in each $\pi_i$ by $1$, then compose $\pi_1$, $\pi_2$, $\dotsc$, $\pi_i$ along their boundary $0$-cells.  So the element
      \[
        \underbrace{(\bullet, \bullet, \dotsc, \bullet)}_i
      \]
    of $T1_1$ should be thought of as
      \[
        \underbrace{
        \xy
          % POINTS
          (0, 0)*+{\bullet}="1";
          (10, 0)*+{\bullet}="2";
          (20, 0)*+{\bullet}="3";
          (30, 0)*+{\bullet}="4";
          (25, 0)*+{\dotsc};
          (40, 0)*+{\bullet}="5";
          (42, -1)*+{,};
          % ARROWS
          {\ar "1" ; "2"};
          {\ar "2" ; "3"};
          {\ar "4" ; "5"};
        \endxy
        }_{i \text{ } 1 \text{-cells}}
      \]
    the element
      \[
        (\emptyset, \bullet \longrightarrow \bullet \longrightarrow \bullet, \bullet \longrightarrow \bullet)
      \]
    of $T1_2$ should be thought of as
      \[
        \xy
          % POINTS
          (-16, 0)*+{\bullet}="0";
          (0, 0)*+{\bullet}="1";
          (16, 0)*+{\bullet}="2";
          (32, 0)*+{\bullet}="3";
          (34, -1)*+{,};
          (8, 3.5)*+{\Downarrow};
          (8, -3.5)*+{\Downarrow};
          (24, 0)*+{\Downarrow};
          % ARROWS
          {\ar "0" ; "1"};
          {\ar@/^1.75pc/ "1" ; "2"};
          {\ar "1" ; "2"};
          {\ar@/_1.75pc/ "1" ; "2"};
          {\ar@/^1.25pc/ "2" ; "3"};
          {\ar@/_1.25pc/ "2" ; "3"};
        \endxy
      \]
    and so on.

    We now define $\hat{k} \colon T1 \rightarrow K$ by defining its components $\hat{k}_m$, for $0 \leq m \leq n$, inductively over $m$.  We use the following notational abbreviations:
      \begin{itemize}
        \item for each $m$ we write $\eta_m$ for the $m$-cell $\sigma_m(\beta^m_m) = \eta^K(1)$ of $K$;
        \item for $m \geq 1$ we write $\id_m$ for the identity $m$-cell on $\eta_0$.  Recall that identity cells in $K$ are defined via the contraction on $k$, so $\id_m$ is defined inductively over $m$ as follows:
              \begin{itemize}
                \item when $m = 1$, $\id_m := \gamma_{\emptyset}(1, 1)$;
                \item when $m > 1$, $\id_m := \gamma_{\emptyset}(\id_{m - 1}, \id_{m - 1})$.
              \end{itemize}
      \end{itemize}
    We also denote binary composition of $m$-cells along $p$-cells, defined using the system of compositions on $K$, by $\comp^m_p$, the same notation used in the definition of magma, Definition~\ref{defn:magma}.

    When $m = 0$, define
      \[
        \hat{k}_0(1) = \eta_0.
      \]

    When $1 \leq m \leq n$ the construction becomes notationally complicated, so we first describe it by example in the cases $m = 1$, $2$.

    When $m = 1$, by the construction of $T1$ above, an element of $T1_m$ is a string
      \[
        \underbrace{(\bullet, \bullet, \dotsc, \bullet)}_i
      \]
    for some natural number $i$.  When $i = 0$, define
      \[
        \hat{k}_1(\emptyset) = \id_1.
      \]
    When $i \geq 1$, there are three steps to the construction of $\hat{k}_1$.  First, we apply $\hat{k}_0$ to all elements in the string, which gives
      \[
        \underbrace{(\eta_0, \eta_0, \dotsc, \eta_0)}_i,
      \]
    a string of $0$-cells in $K$.  Now, we add $1$ to the dimension of each cell in the string by replacing each instance of $\eta_0$ with $\eta_1$, which gives
      \[
        \underbrace{(\eta_1, \eta_1, \dotsc, \eta_1)}_i,
      \]
    a string of $1$-cells in $K$.  Finally, we compose these $1$-cells along boundary $0$-cells, using the system of compositions on $K$, with the bracketing on the left.  Thus, for example, in the case of $i = 4$, we obtain
      \[
        \hat{k}_1(\bullet, \bullet, \bullet, \bullet) := \left( \left( \eta_1 \comp^1_0 \eta_1 \right) \comp^1_0 \eta_1 \right) \comp^1_0 \eta_1.
      \]

    When $m = 2$, an element of $T1_m$ is a string of elements of $T1_1$
      \[
        \mathbf{\pi} = (\pi_1, \pi_2, \dotsc, \pi_i),
      \]
    for some natural number $i$.  When $i = 0$, define
      \[
        \hat{k}_2(\emptyset) = \id_2.
      \]
    For the case $i \geq 1$, we explain with reference to the example
      \[
        \xy
          % POINTS
          (-16, 0)*+{\bullet}="0";
          (0, 0)*+{\bullet}="1";
          (16, 0)*+{\bullet}="2";
          (32, 0)*+{\bullet}="3";
          (34, -1)*+{.};
          (8, 3.5)*+{\Downarrow};
          (8, -3.5)*+{\Downarrow};
          (24, 0)*+{\Downarrow};
          % ARROWS
          {\ar "0" ; "1"};
          {\ar@/^1.75pc/ "1" ; "2"};
          {\ar "1" ; "2"};
          {\ar@/_1.75pc/ "1" ; "2"};
          {\ar@/^1.25pc/ "2" ; "3"};
          {\ar@/_1.25pc/ "2" ; "3"};
        \endxy
      \]
    Recall that, as a string of elements of $T1_1$, this is written as
      \[
        (\pi_1, \pi_2, \pi_3) = (\emptyset, \bullet \longrightarrow \bullet \longrightarrow \bullet, \bullet \longrightarrow \bullet).
      \]
    As in the case $m = 1$, there are three steps to the construction of $\hat{k}_2(\pi_1, \pi_2, \pi_3)$.  First, we apply $\hat{k}_1$ to all elements in the string, which gives
      \[
        \left( \hat{k}_1(\pi_1), \hat{k}_1(\pi_2), \hat{k}_1(\pi_3) \right) = \left( \id_1, \eta_1 \comp^1_0 \eta_1, \eta_1 \right) .
      \]
    In general each $\hat{k}_1(\pi_j)$ is either $\id_1$ or a composite of $\eta_1$'s.  The next step is to add $1$ to the dimension of each $\hat{k}_1(\pi_j)$ by replacing
      \begin{itemize}
        \item every instance of $\id_1$ with $\id_2$;
        \item every instance of $\eta_1$ with $\eta_2$;
        \item every instance of $\comp^1_0$ with $\comp^2_1$.
      \end{itemize}
    The cell we obtain from $\hat{k}_1(\pi_j)$ is denoted $\hat{k}^+_1(\pi_j)$.  Thus our example becomes
      \[
        \left( \hat{k}^+_1(\pi_1), \hat{k}^+_1(\pi_2), \hat{k}^+_1(\pi_3) \right) = \left( \id_2, \eta_2 \comp^2_1 \eta_2, \eta_2 \right) .
      \]
    Finally, we compose these cells along boundary $0$-cells, using the system of compositions on $K$, with the bracketing on the left.  In our example, this gives
      \[
        \hat{k}_2(\pi) := \left( \eta_2 \comp^2_1 \left( \eta_2 \comp^2_1 \eta_2 \right) \right) \comp^2_0 \id_2.
      \]

    We now describe the construction in general for $1 \leq m \leq n$.  Suppose that we have defined $\hat{k}_{m - 1}$ in such a way that, for all $\pi \in T1_{m - 1}$, $\hat{k}_{m - 1}(\pi)$ consists of a composite of copies of $\eta_{m - 1}$ and $\id_{m - 1}$, composed via operations of the form $\comp^{m - 1}_p$ for some $0 \leq p < m - 1$.

    Let $(\pi_1, \pi_2, \dotsc, \pi_i)$ be an element of $T1_m$.  When $i = 0$, we define
      \[
        \hat{k}_m(\pi_1, \pi_2, \dotsc, \pi_i) = \hat{k}_m(\emptyset) = \id_m.
      \]
    When $i \geq 1$ we define $\hat{k}_m(\pi_1, \pi_2, \dotsc, \pi_i)$ in three steps, as described above.  First, we apply $\hat{k}_{m - 1}$ to each $\pi_j$ to obtain
      \[
        \left( \hat{k}_{m - 1}(\pi_1), \hat{k}_{m - 1}(\pi_2), \dotsc, \hat{k}_{m - 1}(\pi_i) \right) .
      \]
    Next, we obtain from each $\hat{k}_{m - 1}(\pi_j)$ a cell $\hat{k}^+_{m - 1}(\pi_j) \in K_m$ by replacing
      \begin{itemize}
        \item every instance of $\id_{m - 1}$ with $\id_m$;
        \item every instance of $\eta_{m - 1}$ with $\eta_m$;
        \item every instance of $\comp^{m - 1}_p$, for all $0 \leq p < m - 1$, with $\comp^m_{p + 1}$.
      \end{itemize}
    This gives
      \[
        \left( \hat{k}^+_{m - 1}(\pi_1), \hat{k}^+_{m - 1}(\pi_2), \dotsc, \hat{k}^+_{m - 1}(\pi_i) \right) .
      \]
    Finally, we compose these cells along boundary $0$-cells, using the system of compositions on $K$, with the bracketing on the left.  This gives
      \begin{align*}
        & \hat{k}_m(\pi_1, \pi_2, \dotsc, \pi_i) :=  \\
        & \Big( \dotsc \Big( \hat{k}^+_{m - 1}(\pi_i) \comp^m_0 \hat{k}^+_{m - 1}(\pi_{i - 1}) \Big) \comp^m_0 \dotsb \comp^m_0 \hat{k}^+_{m - 1}(\pi_2) \Big) \comp^m_0 \hat{k}^+_{m - 1}(\pi_1).
      \end{align*}

    This defines a map of $n$-globular sets $\hat{k} \colon T1 \longrightarrow K$.

    We now show that $\hat{k}$ is a section to $k$.  At dimension $0$, $k_0 \hat{k}_0 = \id_{T1_0}$ since $T1_0$ is terminal, so $\hat{k}_0$ is a section to $k_0$.  Suppose we have shown that, for $1 \leq m \leq n$, $k_{m - 1} \hat{k}_{m - 1} = \id_{T1_{m - 1}}$.  For $\pi \in T1_{m - 1}$,
      \[
        k_m\hat{k}^+_{m - 1}(\pi) = (\pi),
      \]
    so for $(\pi_1, \pi_2, \dotsc, \pi_i) \in T1_m$, we have
      \[
        k_m\hat{k}_m(\pi_1, \pi_2, \dotsc, \pi_i) = (\pi_1, \pi_2, \dotsc, \pi_i),
      \]
    as required.  When $i = 0$,
      \[
        k_m\hat{k}_m(\emptyset) = \emptyset.
      \]
    Hence $\hat{k}$ is a section to $k$.
  \end{proof}

  We now use the map $\hat{k}$ to define an unbiased contraction on $k \colon K \rightarrow T1$.

  \begin{proof}[Proof of Theorem \ref{thm:OCStoOUC}]
    We define an unbiased contraction $\delta$ on the operad $K$; that is, for all $1 \leq m \leq n$, and for each $\pi \in T1_m$, a function
      \[
        \delta_{\pi} \colon C_K(\pi) \rightarrow K(\pi)
      \]
    such that, for all $(a, b) \in C_K(\pi)$,
      \[
        s \delta_{\pi}(a, b), \; t \delta_{\pi}(a, b) = b.
      \]

    To make the construction easier to follow, we first present the cases $m = 1$ and $m = 2$ separately, before giving the construction for general $m$.  Throughout the construction, we use the map $\hat{k} \colon T1 \rightarrow K$ defined in the proof of Lemma~\ref{lem:khat}, which we showed to be a section to $k \colon K \rightarrow T1$.

    Let $m = 1$, let $\pi \in T1_m = T1_1$, and let $(a, b) \in C_K(\pi)$.  If $\pi = \id_{\alpha}$ for some $\alpha \in T1_0$ we already have a corresponding contraction cell from the contraction $\gamma$ on $k$, so we define
      \[
        \delta_{\pi}(a, b) := \gamma_{\id_{\alpha}}(a, b).
      \]
    Now suppose that $\pi \neq \id_{\alpha}$ for any $\alpha \in T1_0$.  We seek a $1$-cell
      \[
        \delta_{\pi}(a, b) \colon a \longrightarrow b
      \]
    in $K$ such that $k_1\delta_{\pi}(a, b) = \pi$.  We have a $1$-cell $\hat{k}_1(\pi)$ in $K$, and since $\hat{k}$ is a section to $k$, we have
      \[
        k_1 \hat{k}_1(\pi) = \pi.
      \]
    However, $\hat{k}_1(\pi)$ does not necessarily have the required source and target.  In order to obtain a cell with the desired source and target, we first observe that
      \[
        k_1s\hat{k}_1(\pi) = sk_1\hat{k}_1(\pi) = s(\pi)
      \]
    and
      \[
        k_1t\hat{k}_1(\pi) = tk_1\hat{k}_1(\pi) = t(\pi).
      \]
    Thus, from the contraction $\gamma$, we have contraction $1$-cells
      \[
        \gamma_{\id_{k_0(a)}}(a, s\hat{k}_1(\pi)) \colon a \longrightarrow s\hat{k}_1(\pi)
      \]
    and
      \[
        \gamma_{\id_{k_0(b)}}(t\hat{k}_1(\pi), b) \colon t\hat{k}_1(\pi) \longrightarrow b
      \]
    in $K$.  Thus in $K$ we have composable $1$-cells
      \[
        \xy
          % POINTS
          (0, 0)*+{a}="a";
          (16, 0)*+{\bullet}="s";
          (32, 0)*+{\bullet}="t";
          (48, 0)*+{b,}="b";
          % ARROWS
          {\ar@{-->} "a" ; "s"};
          {\ar^{\hat{k}_1(\pi)} "s" ; "t"};
          {\ar@{-->} "t" ; "b"};
        \endxy
      \]
    where the dashed arrows denote the contraction cells.  We define the contraction cell $\delta_{\pi}(a, b)$ to be given by a composite of these cells; as in the definition of $\hat{k}$, we bracket this composite on the left, so
      \[
        \delta_{\pi}(a, b) := \Big( \gamma_{\id_{k(b)}}(t\hat{k}(\pi), b) \comp^1_0 \hat{k}(\pi) \Big) \comp^1_0 \gamma_{\id_{k(a)}}(a, s\hat{k}(\pi)).
      \]
    Since $k$ maps the contraction cells to identities and $\hat{k}_1(\pi)$ to $\pi$, and since in $K$ the arity of a composite is the composite of the arities, we have
      \[
        k\delta_{\pi}(a, b) = \pi,
      \]
    as required.  This defines the unbiased contraction $\delta$ on $k \colon K \rightarrow T1$ at dimension $1$.

    Before defining $\delta$ for $m = 2$ or for general $m$, we establish some notation.  For repeated application of source and target maps in $K$, we write
      \[
        s^p := \underbrace{s \comp s \comp \dotsb \comp s}_{p \text{ times}}, \; t^p := \underbrace{t \comp t \comp \dotsb \comp t}_{p \text{ times}},
      \]
    so for $1 \leq p < m \leq n$, and for an $m$-cell $\alpha$ of $K$, $s^p(\alpha)$ is the source $(m - p)$-cell of $\alpha$, and $t^p(\alpha)$ is the target $(m - p)$-cell of $\alpha$.  For all $m < l \leq n$, we write $\id^l \alpha$ for the identity $l$-cell on $\alpha$; so, for example,
      \[
        \id^{m + 1} \alpha = \id_{\alpha}, \; \id^{m + 2} \alpha = \id_{\id_{\alpha}},
      \]
    and so on.

    Now let $m = 2$, let $\pi \in T1_m = T1_2$, and let $(a, b) \in C_K(\pi)$.  As in the case $m = 1$, if $\pi = \id_{\alpha}$ for some $\alpha \in T1_1$, define
      \[
        \delta_{\pi}(a, b) := \gamma_{\id_{\alpha}}(a, b)
      \]
    for all $(a, b) \in C_K(\pi)$.

    Now suppose that $\pi \neq \id_{\alpha}$ for any $\alpha \in T1_1$.  We seek a $2$-cell
      \[
        \delta_{\pi}(a, b) \colon a \Longrightarrow b
      \]
    in $K$ such that $k_2\delta_{\pi}(a, b) = \pi$.  We have a $2$-cell $\hat{k}_2(\pi)$ in $K$, and since $\hat{k}$ is a section to $k$, we have
      \[
        k_2\hat{k}_2(\pi) = \pi.
      \]
    However, $\hat{k}_2(\pi)$ does not necessarily have the required source and target cells at any dimension.  We construct $\delta_{\pi}(a, b)$ from $\hat{k}_2(\pi)$ in two stages: first we compose with contraction $1$-cells to obtain a $2$-cell with the required source and target $0$-cells, then we compose this with contraction $2$-cells to obtain a $2$-cell with the required source and target $1$-cells.  To obtain a cell with the required source and target $0$-cells, observe that, since $T1_0$ is terminal,
      \[
        ks(a) = ks^2\hat{k}(\pi)
      \]
    and
      \[
        kt(b) = kt^2\hat{k}(\pi).
      \]
    Thus, from the contraction $\gamma$, we have contraction $1$-cells
      \[
         \gamma_{\id_1}(s(a), s^2\hat{k}(\pi)) \colon s(a) \longrightarrow s^2\hat{k}(\pi)
      \]
    and
      \[
        \gamma_{\id_1}(t^2\hat{k}(\pi), t(b)) \colon t^2\hat{k}(\pi) \longrightarrow t(b)
      \]
    in $K$.  Thus we have the following composable diagram of cells in $K$:
      \[
        \xy
          % POINTS
          (0, 0)*+{s(a)}="a";
          (16, 0)*+{\bullet}="s";
          (36, 0)*+{\bullet}="t";
          (52, 0)*+{t(b),}="b";
          % ARROWS
          {\ar@{-->} "a" ; "s"};
          {\ar@/^1.5pc/ "s" ; "t"};
          {\ar@/_1.5pc/ "s" ; "t"};
          {\ar@{-->} "t" ; "b"};
          {\ar@{=>}^{\hat{k}_2(\pi)} (26, 4) ; (26, -4)};
        \endxy
      \]
    where the dashed arrows denote identity $2$-cells on the contraction cells mentioned above.  We compose this diagram to obtain a $2$-cell in $K$ with the required source and target $0$-cells, which we denote $\delta^0_{\pi}(a, b)$.  Formally, this is defined by
      \[
        \delta^0_{\pi}(a, b) := \Big( \id^2\gamma_{\id_1}(t^2\hat{k}(\pi), t(b)) \comp^2_0 \hat{k}(\pi) \Big) \comp^2_0 \id^2\gamma_{\id_1}(s(a), s^2\hat{k}(\pi)).
      \]
    As before, we bracket this composite on the left, though this choice is arbitrary.

    We now repeat this process at dimension $2$ to obtain a cell with the required source and target $1$-cells.  We have
      \[
        s(a) = s(b) = s^2\delta^0_{\pi}(a, b)
      \]
    and
      \[
        t(a) = t(b) = t^2\delta^0_{\pi}(a, b),
      \]
    so we have contraction $2$-cells
      \[
        \gamma_{\id_{k(a)}}(a, s\delta^0_{\pi}(a, b)) \colon a \Longrightarrow s\delta^0_{\pi}(a, b)
      \]
    and
      \[
        \gamma_{\id_{k(b)}}(t\delta^0_{\pi}(a, b), b) \colon t\delta^0_{\pi}(a, b) \Longrightarrow b
      \]
    in $K$.  Thus we have the following composable diagram of cells in $K$:
      \[
        \xy
          % POINTS
          (0, 0)*+{s(a)}="a";
          (16, 0)*+{\bullet}="s";
          (36, 0)*+{\bullet}="t";
          (52, 0)*+{t(b),}="b";
          % ARROWS
          {\ar@/^4pc/^a "a" ; "b"};
          {\ar@{==>} (26, 14) ; (26, 8)};
          {\ar@{-->} "a" ; "s"};
          {\ar@/^1.5pc/ "s" ; "t"};
          {\ar@/_1.5pc/ "s" ; "t"};
          {\ar@{-->} "t" ; "b"};
          {\ar@{=>}^{\hat{k}_2(\pi)} (26, 4) ; (26, -4)};
          {\ar@{==>} (26, -8) ; (26, -14)};
          {\ar@/_4pc/_b "a" ; "b"};
        \endxy
      \]
    where the dashed arrows denote contraction cells.  We compose this diagram to obtain the unbiased contraction cell $\delta_{\pi}(a, b)$ in $K$.  Formally, this is defined by
      \[
        \delta_{\pi}(a, b) := \Big( \gamma_{\id_{k(b)}}(t\delta^0_{\pi}(a, b), b) \comp^2_1 \delta^0_{\pi}(a, b) \Big) \comp^2_1 \gamma_{\id_{k(a)}}(a, s\delta^0_{\pi}(a, b)).
      \]
    By construction, we see that $s\delta_{\pi}(a, b) = a$, $t\delta_{\pi}(a, b) = b$.  As before, since $k$ maps the contraction cells to identities and $\hat{k}_2(\pi)$ to $\pi$, and since in $K$ the arity of a composite is the composite of the arities, we have
      \[
        k\delta_{\pi}(a, b) = \pi,
      \]
    as required.  This defines the unbiased contraction $\delta$ on $k \colon K \rightarrow T1$ at dimension $2$.

    We now give the definition of $\delta$ for higher dimensions.  Our approach is the same as that for dimensions $1$ and $2$; we build our contraction cells in stages, first constructing a cell with the desired source and target $0$-cells, then constructing from that a cell with the desired source and target $1$-cells, and so on.

    Let $3 \leq m \leq n$, let $\pi \in T1_m$, and let $(a, b) \in C_K(\pi)$.  If $\pi = \id_{\alpha}$ for some $\alpha \in T1_{m - 1}$, we define
      \[
        \delta_{\pi}(a, b) := \gamma_{\id_{\alpha}}(a, b).
      \]

    Now suppose that $\pi \neq \id_{\alpha}$ for any $\alpha \in T1_{m - 1}$.  We seek an $m$-cell
      \[
        \delta_{\pi}(a, b) \colon a \longrightarrow b
      \]
    in $K$ such that $k_m\delta_{\pi}(a, b) = \pi$.  As before, we have an $m$-cell $\hat{k}_m(\pi)$ in $K$, and since $\hat{k}$ is a section to $k$, we have
      \[
        k_m\hat{k}_m(\pi) = \pi.
      \]
    However, $\hat{k}_m(\pi)$ does not necessarily have the required source and target cells at any dimension.  We obtain a cell with the required source and target by defining, for each $0 \leq j \leq m - 1$, an $m$-cell $\delta^j_{\pi}(a, b)$ which has the required source and target $j$-cells, and maps to $\pi$ under $k$.  We define this by induction over $j$.  Note that, since this construction is very notation heavy, we henceforth omit subscripts indicating the dimensions of components of maps of $n$-globular sets, so we write $k$ for $k_m$, $\hat{k}$ for $\hat{k}_m$, etc.

    Let $j = 0$.  Since $T1_0$ is the terminal set, we have
      \[
        ks^{m - 1}(a) = ks^m\hat{k}(\pi)
      \]
    and
      \[
        kt^{m - 1}(b) = kt_m\hat{k}(\pi)
      \]
    in $K$, so we have contraction $1$-cells
      \[
        \gamma_{\id_1}(s^{m - 1}(a), s^m\hat{k}(\pi))
      \]
    and
      \[
        \gamma_{\id_1}(t^m\hat{k}(\pi), t_{m - 1}(b))
      \]
    in $K$.  We obtain $\delta^0_{\pi}(a, b)$ by composing $\hat{k}(\pi)$ with the $m$-cell identities on these contraction cells, so we define
      \[
        \delta^0_{\pi}(a, b) := \Big( \id^m\gamma_{\id_1}(t^m\hat{k}(\pi), t_{m - 1}(b)) \comp^m_0 \hat{k}(\pi) \Big) \comp^m_0 \id^m\gamma_{\id_1}(s^{m - 1}(a), s^m\hat{k}(\pi)).
      \]
    By construction, we have
      \[
        s^{m - 1}(a) = s^{m - 1}(b) = s^m\delta^0_{\pi}(a, b)
      \]
    and
      \[
        t^{m - 1}(a) = t^{m - 1}(b) = t^m\delta^0_{\pi}(a, b),
      \]
    so this has the required source and target $0$-cells.  Since $k$ sends contraction cells to identities, and since $\hat{k}$ is a section to $k$, we have
      \[
        k\delta^0_{\pi}(a, b) = \pi.
      \]

    Now let $0 \leq j < m - 1$, and suppose we have defined $\delta^j_{\pi}(a, b)$ such that
      \[
        s^{m - j - 1}(a) = s^{m - j - 1}(b) = s^{m - j}\delta^j_{\pi}(a, b),
      \]
      \[
        t^{m - j - 1}(a) = t^{m - j - 1}(b) = t^{m - j}\delta^j_{\pi}(a, b),
      \]
    so $\delta^j_{\pi}(a, b)$ has the required source and target $j$-cells, and
      \[
        k\delta^j_{\pi}(a, b) = \pi.
      \]
    Applying $k$ to the source and target conditions above, we have
      \[
        ks^{m - j - 2}(a) = ks^{m - j - 1}\delta^j_{\pi}(a, b)
      \]
    and
      \[
        kt^{m - j - 2}(b) = kt^{m - j - 1}\delta^j_{\pi}(a, b).
      \]
    Thus we have contraction cells
      \[
        \gamma_{\id_{ks^{m - j - 2}(a)}}(s^{m - j - 2}(a), s^{m - j - 1}\delta^j_{\pi}(a, b)),
      \]
    and
      \[
        \gamma_{\id_{ks^{m - j - 2}(b)}}(t^{m - j - 1}\delta^j_{\pi}(a, b), t^{m - j - 2}(b)).
      \]
    in $K$.  We obtain $\delta^{j + 1}_{\pi}(a, b)$ by composing $\delta^j_{\pi}(a, b)$ with the $m$-cell identities on these contraction cells (or with the contraction cells themselves in the case $j + 1 = m$), so we define
      \begin{align*}
        \delta^{j + 1}_{\pi}(a, b) := \Big( & \id^m\gamma_{\id_{ks^{m - j - 2}(b)}}(t^{m - j - 1}\delta^j_{\pi}(a, b), t^{m - j - 2}(b)) \comp^m_{j + 1} \delta^j_{\pi}(a, b) \Big)  \\
        & \comp^m_{j + 1} \id^m\gamma_{\id_{ks^{m - j - 2}(a)}}(s^{m - j - 2}(a), s^{m - j - 1}\delta^j_{\pi}(a, b)).
      \end{align*}
    By construction, we see that
      \[
        s^{m - j - 1}\delta^{j + 1}_{\pi}(a, b) = s^{m - j - 2}(a)
      \]
    and
      \[
        t^{m - j - 1}\delta^{j + 1}_{\pi}(a, b) = t^{m - j - 2}(b),
      \]
    so $\delta^{j + 1}_{\pi}(a, b)$ has the required source and target $(j + 1)$-cells.  Since
      \[
        k\delta^{j}_{\pi}(a, b) = \pi,
      \]
    and $k$ maps contraction cells to identities, we have
      \[
        k\delta^{j + 1}_{\pi}(a, b) = \pi.
      \]

    This defines an $m$-cell $\delta^{j}_{\pi}(a, b)$ in $K$, for each $0 \leq j \leq m - 1$, with the required source and target $j$-cells, and such that
      \[
        k\delta^{j}_{\pi}(a, b) = \pi.
      \]
    In particular, we have
      \[
        \delta^{m - 1}_{\pi}(a, b) \colon a \longrightarrow b.
      \]
    Thus we define
      \[
        \delta_{\pi}(a, b) := \delta^{m - 1}_{\pi}(a, b).
      \]
    This defines an unbiased contraction $\delta$ on the operad $K$, as required.
  \end{proof}

  Thus any operad with a contraction and system of compositions can be equipped with an unbiased contraction.  In the proof above we had to make several arbitrary choices.  Most of these involved picking a binary bracketing for a composite; we also chose to define the unbiased contraction to be the same as the original contraction on all cells for which this makes sense, which we did not have to do.  There is no canonical choice in any of these cases, and thus no canonical way of equipping an operad in $\OCS$ with an unbiased contraction.

  Note that various authors use variants of Batanin's definition in which a choice of globular operad is not specified, and instead a weak $n$-category is defined either to be an algebra for any operad that \emph{can be} equipped with a contraction and system of compositions, or an algebra for any operad that \emph{can be} equipped with an unbiased contraction (\cite[Definitions~B2 and L2]{Lei02}, \cite{Ber02, Gar10, vdBG11, Che11}).  By Theorems~\ref{thm:OUCtoOCS} and \ref{thm:OCStoOUC}, these two ``less algebraic'' variants of Batanin's definition are equivalent, since any operad that can be equipped with a contraction and system of compositions can also be equipped with an unbiased contraction, and vice versa.

  \section{Coherence for algebras for $n$-globular operads}  \label{sect:globopcoh}

  In this section we prove three new coherence theorems for algebras for any Batanin operad or Leinster operad $K$.  Roughly speaking, our coherence theorems say the following:
    \begin{itemize}
      \item every free $K$-algebra is equivalent to a free strict $n$-category;
      \item every diagram of constraint $n$-cells commutes in a free $K$-algebra;
      \item in any $K$-algebra there is a certain class of diagrams of constraint $n$-cells that always commute; these should be thought of as the diagrams of shapes that can arise in a free algebra.
    \end{itemize}
  In the first two of these theorems freeness is crucial; these theorems do not hold in general for non-free $K$-algebras, so this does not mean that every weak $n$-category is equivalent to a strict one, which we know should not be true for $n \geq 3$ in a fully weak theory.  All of these theorems have analogues in the case of tricategories, which appear in Gurski's thesis~\cite{Gur06} and book~\cite{Gur13} on coherence for tricategories; these are noted throughout the section.  Note that there is no theorem corresponding to the coherence theorem for tricategories that states ``every tricategory is triequivalent to a $\mathbf{Gray}$-category'' \cite[Theorem 8.1]{GPS95}, since we have no analogue of $\mathbf{Gray}$-categories in this case.  There are also no coherence theorems for maps of $K$-algebras, since there is no established notion of weak map of $K$-algebras.

  These coherence theorems hold for Batanin weak $n$-categories and Leinster weak $n$-categories; in Section~\ref{sect:Penonop} we prove that there is a Batanin operad whose algebras are Penon weak $n$-categories, so the theorems in this section also hold for Penon weak $n$-categories.  Note that, by Theorems~\ref{thm:OUCtoOCS} and \ref{thm:OCStoOUC}, we need only prove each coherence theorem either in the case of algebras for a Batanin operad or algebras for a Leinster operad; thus in each case we use whichever of these is more technically convenient for the purposes of the proof.  Throughout this section we write $K$ to denote either a Batanin operad or Leinster operad (with the exception of Definition~\ref{defn:TAlgintoKAlg} and Proposition~\ref{prop:TAlgintoKAlg}, in which a little more generality is possible).

  Our first coherence theorem corresponds to the coherence theorem for tricategories stating that the free tricategory on a $\Cat$-enriched $2$-graph $X$ is triequivalent to the free strict $3$-category on $X$ \cite[Theorem~10.4]{Gur13}. Since the theorem involves comparing $K$-algebras with strict $n$-categories, before stating the theorem we first define, for any $n$-globular operad $K$, a functor $T\Alg \rightarrow K\Alg$; in fact, we do this for a $T$-operad $K$ for any suitable choice of monad $T$.  This functor is induced by the natural transformation $k \colon K \Rightarrow T$.  We then prove that, under certain circumstances (and in particular, when $K$ is an $n$-globular operad with unbiased contraction), this functor is full, faithful, and injective on objects, so we can consider $T\Alg$ to be a full subcategory of $K\Alg$.  This tells us that, for any definition of weak $n$-categories as algebras for an $n$-globular operad, every strict $n$-category is a weak $n$-category.  The fact that the inclusion functor is full comes from the fact that, since $K\Alg$ is the category of algebras for a monad, we only have strict maps of $K$-algebras.

  \begin{defn} \label{defn:TAlgintoKAlg}
    Let $T$ be a cartesian monad on a cartesian category $\mathcal{C}$, which has an initial object $1$, and let $K$ be a $T$-operad.  Then there is a functor $- \comp k \colon T\Alg \rightarrow K\Alg$ defined by
      \[
        \xy
          % POINTS
            (0, 0)*+{\longrightarrow};
            (0, -14)*+{\longmapsto};
            (-18, 0)*+{- \comp k \colon T\Alg};
            (12, 0)*+{K\Alg};
            (-13, -6)*+{X}="Al";
            (-29, -6)*+{TX}="TAl";
            (-13, -22)*+{Y}="Bl";
            (-29, -22)*+{TY}="TBl";
            (12, -6)*+{KX}="KAr";
            (28, -6)*+{TX}="TAr";
            (44, -6)*+{X}="Ar";
            (12, -22)*+{KY}="KBr";
            (28, -22)*+{TY}="TBr";
            (44, -22)*+{Y}="Br";
          % ARROWS
            {\ar_{Tu} "TAl" ; "TBl"};
            {\ar^{u} "Al" ; "Bl"};
            {\ar^{\phi} "TAl" ; "Al"};
            {\ar_{\psi} "TBl" ; "Bl"};
            {\ar_{Ku} "KAr" ; "KBr"};
            {\ar^{k_X} "KAr" ; "TAr"};
            {\ar_{k_Y} "KBr" ; "TBr"};
            {\ar^{u} "Ar" ; "Br"};
            {\ar^{\phi} "TAr" ; "Ar"};
            {\ar_{\psi} "TBr" ; "Br"};
        \endxy
      \]
  \end{defn}

  \begin{prop} \label{prop:TAlgintoKAlg}
    Let $T$ be a cartesian monad on a cartesian category $\mathcal{C}$, which has an initial object $1$, and let $K$ be a $T$-operad such that, for any object $X$ in $\mathcal{C}$, the component $k_X \colon KX \rightarrow TX$ of the natural transformation $k \colon K \Rightarrow T$ is an epimorphism.  Then the functor $- \comp k \colon T\Alg \rightarrow K\Alg$ is full, faithful, and injective on objects; hence we can consider $T\Alg$ to be a full subcategory of $K\Alg$.
  \end{prop}

  \begin{proof}
    First, faithfulness is immediate since when we apply $- \comp k$ to a map of $T$-algebras it retains the same underlying map of $n$-globular sets.

    For fullness, suppose we have $T$-algebras $\xymatrix{ TX \ar[r]^-{\phi} & X }$, $\xymatrix{ TY \ar[r]^-{\psi} & Y }$, and a map $u$ between their images in $K\Alg$.  By naturality of $k$,
      \[
        \xy
          % POINTS
          (-8, 0)*+{KX}="KA";
          (8, 0)*+{KY}="KB";
          (-8, -16)*+{TX}="TA";
          (8, -16)*+{TY}="TB";
          % ARROWS
          {\ar^-{Ku} "KA" ; "KB"};
          {\ar_{k_X} "KA" ; "TA"};
          {\ar^{k_Y} "KB" ; "TB"};
          {\ar_-{Tu} "TA" ; "TB"};
        \endxy
     \]
    commutes, so
      \[
        \xy
          % POINTS
          (-8, 0)*+{KX}="KA";
          (8, 0)*+{TX}="TAr";
          (-8, -16)*+{TX}="TA";
          (8, -16)*+{TY}="TB";
          (-8, -32)*+{X}="A";
          (8, -32)*+{Y}="B";
          % ARROWS
          {\ar^{k_X} "KA" ; "TAr"};
          {\ar_{k_X} "KA" ; "TA"};
          {\ar^{Tu} "TAr" ; "TB"};
          {\ar_{\phi} "TA" ; "A"};
          {\ar^{\psi} "TB" ; "B"};
          {\ar_{u} "A" ; "B"};
        \endxy
     \]
    commutes.  Since $k_X$ is an epimorphism, the diagram above gives us that
      \[
        \xy
          % POINTS
          (-8, 0)*+{TX}="TA";
          (8, 0)*+{TY}="TB";
          (-8, -16)*+{X}="A";
          (8, -16)*+{Y}="B";
          % ARROWS
          {\ar^{Tu} "TA" ; "TB"};
          {\ar_{\phi} "TA" ; "A"};
          {\ar^{\psi} "TB" ; "B"};
          {\ar_{u} "A" ; "B"};
        \endxy
     \]
    commutes, so $u$ is a map of $T$-algebras.  Hence $- \comp k$ is full.

    Finally, suppose we have $T$-algebras $\xymatrix{ TX \ar[r]^-{\phi} & X }$, $\xymatrix{ TX \ar[r]^-{\psi} & X }$, with
      \[
        - \comp k \big( \xymatrix { TX \ar[r]^-{\phi} & X } \big) = - \comp k \big( \xymatrix { TX \ar[r]^-{\psi} & X } \big).
      \]
    Then
      \[
        \xy
          % POINTS
          (-8, 0)*+{KX}="KA";
          (8, 0)*+{TX}="TAr";
          (-8, -16)*+{TX}="TA";
          (8, -16)*+{X}="A";
          % ARROWS
          {\ar^{k_X} "KA" ; "TAr"};
          {\ar_{k_X} "KA" ; "TA"};
          {\ar^{\psi} "TAr" ; "A"};
          {\ar_{\phi} "TA" ; "A"};
        \endxy
     \]
    commutes.  Since $k_X$ is an epimorphism, this gives us that $\phi = \psi$, so $- \comp k$ is injective on objects.
  \end{proof}

  In the case in which $K$ is a Batanin operad or Leinster operad, each component $k_X$ is surjective on all dimensions of cell, so we have the following corollary.

  \begin{corol}
    Let $K$ be a Batanin operad or Leinster operad.  Then the functor $- \comp k \colon T\Alg \rightarrow K\Alg$ is full, faithful, and injective on objects.
  \end{corol}

  For the remainder of this section, when we say ``strict $n$-category'', we mean it in the sense of a $K$-algebra in the image of the functor $- \comp k \colon T\Alg \rightarrow K\Alg$.

  Before we state our first coherence theorem, we must also define what it means for two $K$-algebras to be equivalent.

  \begin{defn}  \label{defn:globopalgequiv}
    Let $K$ be an $n$-globular operad, and let
      \[
        \xymatrix{ KX \ar[r]^-{\theta} & X }, \xymatrix{ KY \ar[r]^-{\phi} & Y }
      \]
    be $K$-algebras.  We say that the algebras $\xymatrix{ KX \ar[r]^-{\theta} & X }$ and $\xymatrix{ KY \ar[r]^-{\phi} & Y }$ are \emph{equivalent} if there exists a map of $K$-algebras $u \colon X \rightarrow Y$ or $u \colon Y \rightarrow X$ such that $u$ is surjective on $0$-cells, full on $m$-cells for all $1 \leq m \leq n$, and faithful on $n$-cells.  The map $u$ is referred to as an \emph{equivalence of $K$-algebras}.
  \end{defn}

  Observe that, since maps of $K$-algebras preserve the $K$-algebra structure strictly, this definition of equivalence is much more strict (and thus much less general) than it ``ought'' to be.  This is also why we require that the map $u$ can go in either direction; having a map $X \rightarrow Y$ satisfying the conditions does not imply the existence of a map $Y \rightarrow X$ satisfying the conditions.  We will use this definition of equivalence only in the next theorem, and, in spite of its lack of generality, it is sufficient for our purposes.  If we required a more general definition of equivalence of $K$-algebras, there are various approaches we could take.  One option would be to replace the map $u$ with a weak map of $K$-algebras; a definition of weak maps of $K$-algebras is given by Garner in \cite{Gar10}, and is valid for any $n$-globular operad $K$.  Another option is to replace the map $u$ with a span of maps of $K$-algebras, similar to the approach used by Smyth and Woolf to define an equivalence of Whitney $n$-categories \cite{SW11}.  However, pursuing definitions of equivalence given by either of these approaches is beyond the scope of this thesis.

  In this definition of equivalence we asked for surjectivity on $0$-cells, rather than essential surjectivity.  This is another way in which our definition of equivalence is less general than it ``ought'' to be, but once again, asking for surjectivity is enough for our purposes.  This approach of using surjectivity instead of essential surjectivity to simplify the definition of equivalence has previously been taken by Simpson~\cite{Sim97}; we discuss this in more detail in the definition of Tamsamani--Simpson weak $n$-category in Section~\ref{subsect:tamsimn}.

  \begin{thm}  \label{thm:Kalgcohone}
    Let $K$ be an $n$-globular operad with unbiased contraction $\gamma$, and let $X$ be an $n$-globular set.  Then the free $K$-algebra on $X$ is equivalent to the free strict $n$-category on $X$.
  \end{thm}

  \begin{proof}
    As a $K$-algebra, the free strict $n$-category on $X$ is
      \[
        \xy
          % POINTS
          (-18, 0)*+{KTX}="KTX";
          (0, 0)*+{T^2 X}="TTX";
          (18, 0)*+{TX.}="TX";
          % ARROWS
          {\ar^{k_{TX}} "KTX" ; "TTX"};
          {\ar^{\mu^T_X} "TTX" ; "TX"};
        \endxy
      \]
      We first show that $k_X$ is a map of $K$-algebras, and then show that it is an equivalence of $K$-algebras.

      The diagram
        \[
          \xy
          % POINTS
          (-18, 0)*+{K^2 X}="KKA";
          (18, 0)*+{KTX}="KTA";
          (0, -16)*+{TKX}="TKA";
          (18, -16)*+{T^2 X}="TTA";
          (-18, -32)*+{KX}="KA";
          (18, -32)*+{TX}="TA";
          % ARROWS
          {\ar^{k_{KX}} "KKA" ; "TKA"};
          {\ar^{Kk_X} "KKA" ; "KTA"};
          {\ar^{Tk_X} "TKA" ; "TTA"};
          {\ar_{\mu^K_X} "KKA" ; "KA"};
          {\ar^{k_{TX}} "KTA" ; "TTA"};
          {\ar^{\mu^T_X} "TTA" ; "TA"};
          {\ar_{k_X} "KA" ; "TA"};
          \endxy
        \]
      commutes; the top square is a naturality square for $k$, and the bottom part is axiom for the monad map $k$.  Thus $k_X$ is a map of $K$-algebras, as required.

      We now show that $k_X$ is surjective on $0$-cells.  By definition of the unit $\eta^K \colon 1 \Rightarrow K$, the diagram
        \[
          \xy
            % POINTS
            (0, 0)*+{X}="X";
            (16, 0)*+{KX}="KX";
            (16, -16)*+{TX}="TX";
            % ARROWS
            {\ar^-{\eta^K_X} "X" ; "KX"};
            {\ar_{\eta^T_X} "X" ; "TX"};
            {\ar^{k_X} "KX" ; "TX"};
          \endxy
        \]
      commutes.  We have $TX_0 = X_0$, and $(\eta^T_X)_0 = \id_X$, so at dimension $0$ the diagram above becomes
        \[
          \xy
            % POINTS
            (0, 0)*+{X_0}="X";
            (16, 0)*+{KX_0}="KX";
            (16, -16)*+{X_0.}="TX";
            % ARROWS
            {\ar^-{(\eta^K_X)_0} "X" ; "KX"};
            {\ar_{\id_{X_0}} "X" ; "TX"};
            {\ar^{(k_X)_0} "KX" ; "TX"};
          \endxy
        \]
      Hence $(k_X)_0$ is surjective, i.e. $k_X$ is surjective on $0$-cells.

      We now show that $k_X$ is full on $m$-cells for all $1 \leq m \leq n$.  Let $(\alpha, p)$, $(\beta, q) \in KX_{m - 1}$ be parallel $(m - 1)$-cells, and let $\pi \colon k_X(f) \rightarrow k_X(g)$ be an $m$-cell in $TX$.  Then we have an $m$-cell
        \[
          (\pi, \gamma_{T!(\pi)}(p, q)) \colon (\alpha, p) \longrightarrow (\beta, q)
        \]
      in $KX$ with $k_X(\pi, \gamma_{T!(\pi)}(p, q)) = \pi$.  Hence $k_X$ is full at dimension $m$.

      Finally, we show that $k_X$ is faithful at dimension $n$.  Let $(\alpha, p)$, $(\beta, q)$ be $n$-cells in $KA$, such that
        \[
          s(\alpha, p) = s(\beta, q), \; t(\alpha, p) = t(\beta, q), \; k_X(\alpha, p) = k_X(\beta, q).
        \]
      The first two equations above give us that $s(p) = s(q)$ and $t(p) = t(q)$, and the third equation gives
        \[
          \alpha = k_X(\alpha, p) = k_X(\beta, q) = \beta.
        \]
      Now, since $(\alpha, p)$, $(\beta, q) \in KX_n$, and since $\alpha = \beta$, we have
        \[
          k(p) = T!(\alpha) = T!(\beta) = k(q),
        \]
      and since $k$ has an unbiased contraction $\gamma$, it is faithful at dimension $n$, and we get that $p = q$.  Hence $(\alpha, p) = (\beta, q)$, so $k_X$ is faithful at dimension $n$.

      Hence $k_X$ is an equivalence of $K$-algebras, so the free $K$-algebra on $X$ is equivalent to the free strict $n$-category on $X$.
  \end{proof}

  The remaining coherence theorems require only a contraction on the operad $K$, not a system of compositions or an unbiased contraction.  These theorems concern which diagrams of constraint cells commute in a $K$-algebra, so in order to state them, we must first define what we mean by a ``diagram'' in a $K$-algebra, and what it means for a diagram to commute.

  \begin{defn}
    Let $K$ be an $n$-globular operad, let $\xymatrix{ KX \ar[r]^-{\theta} & X }$ be a $K$-algebra, and let $1 \leq m \leq n$.  A \emph{diagram of $m$-cells} in $\xymatrix{ KX \ar[r]^-{\theta_X} & X }$ consists of an unordered pair of $m$-cells $(\alpha, p)$, $(\beta, q) \in KX_m$ such that $\theta(\alpha, p)$ and $\theta(\beta, q)$ are parallel, i.e.
      \[
        s\theta(\alpha, p) = s\theta(\beta, q), \; t\theta(\alpha, p) = t\theta(\beta, q).
      \]
    We write such a diagram as $((\alpha, p), (\beta, q))$.

    We say that the diagram $((\alpha, p), (\beta, q))$ \emph{commutes} if
      \[
        \theta(\alpha, p) = \theta(\beta, q).
      \]
  \end{defn}

  Our second coherence theorem states that in a free $K$-algebra every diagram of constraint $n$-cells commutes.  This corresponds to the coherence theorem for tricategories due to Gurski which states that, in the free tricategory on a $\Cat$-enriched $2$-graph whose set of $3$-cells is empty, every diagram of $3$-cells commutes (\cite[Corollary~10.6]{Gur13}, originally \cite[Theorem~10.2.2]{Gur06}).  Since the constraint $3$-cells in a free tricategory do not depend on the generating $3$-cells, this implies that in a free tricategory all diagrams of constraint $3$-cells commute.  Our theorem is analogous to this last result, and our approach is the same as that of Gurski: first, we prove a lemma which states that, in the free $K$-algebra on an $n$-globular set whose set of $n$-cells is empty, all diagrams of $n$-cells commute;  note that in a free $K$-algebra of this type, all $n$-cells are constraint cells.  We then use this lemma, combined with the fact that the constraint $n$-cells in a free $K$-algebra depend only on dimension $n - 1$, to prove the theorem.

  \begin{lemma}  \label{thm:Kalgcohtwo}
    Let $K$ be an $n$-globular operad with contraction $\gamma$, and let $X$ be an $n$-globular set with $X_n = \emptyset$.  Then in the free $K$-algebra on $X$, every diagram of $n$-cells commutes.
  \end{lemma}

  \begin{proof}
    Let $((\alpha, p), (\beta, q))$ be a diagram of $n$-cells in $\xymatrix{ K^2 X \ar[r]^-{\mu^X_A} & KX }$.  Since $X_n = \emptyset$, the only $n$-cells in $TX$ are identities, so we have $TX_n \iso TX_{n - 1}$, and the source and target maps $s$, $t \colon TX_n \rightarrow TX_{n - 1}$ are isomorphisms with $s(\pi) = t(\pi)$ for all $\pi \in TX_n$.  Since $(\alpha, p)$, $(\beta, q)$ are parallel and $k_X$ is a map of $n$-globular sets, so preserves sources and targets, $k_X(\alpha, p)$ and $k_X(\beta, q)$ are parallel $n$-cells in $TX$ so must be equal.  As shown in the proof of Theorem~\ref{thm:Kalgcohone}, $k_X$ is faithful at dimension $n$, hence $(\alpha, p) = (\beta, q)$.
  \end{proof}

  Before we use Lemma~\ref{thm:Kalgcohtwo} to prove our second coherence theorem, we must first give a formal definition of constraint cells in a $K$-algebra.  Constraint cells are cells that arise from the contraction on $k \colon K \rightarrow T1$; these include identities, and mediating cells between different composites of the same pasting diagram.  Note that constraint $m$-cells for $m < n$ depend on the choice of contraction on $k$, even though the algebras for $K$ do not; constraint $n$-cells do not depend on the choice of contraction on $k$, since faithfulness of $k$ at dimension $n$ ensures that there is only ever one valid choice at this dimension.

  \begin{defn}  \label{defn:constraintcell}
    Let $K$ be an $n$-globular operad with unbiased contraction $\gamma$, and let $\xymatrix{ KX \ar[r]^{\theta} & X }$ be a $K$-algebra.  There is a contraction $\delta$ on $k_X \colon KX \rightarrow TX$ given by, for each $1 \leq m \leq n$, $\pi \in TX_{m - 1}$, the function
      \[
        \xy
          % POINTS
            (0, 0)*+{\longrightarrow};
            (0, -6)*+{\longmapsto};
            (-18, 0)*+{\delta_{\id_{\pi}} \colon C_{KX}(\id_{\pi})};
            (18, 0)*+{KX(\id_{\pi})};
            (-18, -6)*+{((\pi, p), (\pi, q))};
            (18, -6)*+{(\pi, \gamma_{T!(\id_{\pi})}(p, q)).};
        \endxy
      \]
    A \emph{constraint $m$-cell} in $\xymatrix{ KX \ar[r]^{\theta} & X }$ is an $m$-cell of $X$ in the image of the map
      \[
        \xy
          % POINTS
          (-36, 0)*+{C_{KX}(\id_{\pi})}="CKA";
          (-12, 0)*+{KX(\id_{\pi})}="KAid";
          (9, 0)*+{(KX)_m}="KAm";
          (27, 0)*+{X_m,}="Am";
          % ARROWS
          {\ar^{\delta_{\id_{\pi}}} "CKA" ; "KAid"};
          {\ar@{^{(}->} "KAid" ; "KAm"};
          {\ar^{\theta_m} "KAm" ; "Am"};
        \endxy
      \]
    for some $\pi \in TX_{m - 1}$.
  \end{defn}

  \begin{corol}  \label{lem:diaginfree}
    Let $K$ be an $n$-globular operad with contraction $\gamma$, and let $X$ be an $n$-globular set.  In the free $K$-algebra on $X$, $\xymatrix{ K^2 X \ar[r]^{\mu^K_X} & KX }$, every diagram of constraint $n$-cells commutes.
  \end{corol}

  \begin{proof}
    Write $X'$ for the $n$-globular set defined by
      \[
        X'_m =
            \left\{
              \begin{array}{ll}
              X_m & \text{if } m < n, \\
              \emptyset & \text{if } m = n,
              \end{array}
            \right.
      \]
    with source and target maps the same as those in $X$ for dimensions $m < n$; write $u \colon X' \rightarrow X$ for the map which is the identity on all dimensions $m < n$.  For all $\pi \in TKX_{n - 1} = TKX'_{n - 1}$ we have $C_{K^2 X'}(\id_{\pi}) = C_{K^2 X}(\id_{\pi})$, and the diagram
      \[
        \xy
          % POINTS
          (-36, 0)*+{C_{K^2 X'}(\id_{\pi})}="CKA'";
          (-12, 0)*+{K^2 X'(\id_{\pi})}="KA'id";
          (9, 0)*+{(K^2 X')_n}="KA'n";
          (27, 0)*+{KX'_m,}="A'n";
          (-36, -16)*+{C_{K^2 X}(\id_{\pi})}="CKA";
          (-12, -16)*+{K^2 X(\id_{\pi})}="KAid";
          (9, -16)*+{K^2 X_n}="KAn";
          (27, -16)*+{KX_m,}="An";
          % ARROWS
          {\ar^{{\delta'}_{\id_{\pi}}} "CKA'" ; "KA'id"};
          {\ar@{^{(}->} "KA'id" ; "KA'n"};
          {\ar^{\mu^K_{X'}} "KA'n" ; "A'n"};
          {\ar^{\delta_{\id_{\pi}}} "CKA" ; "KAid"};
          {\ar@{^{(}->} "KAid" ; "KAn"};
          {\ar^{\mu^K_X} "KAn" ; "An"};
          {\ar@{=} "CKA'" ; "CKA"};
          {\ar_{K^2u_n} "KA'id" ; "KAid"};
          {\ar_{K^2u_n} "KA'n" ; "KAn"};
          {\ar^{Ku_n} "A'n" ; "An"};
        \endxy
      \]
    commutes.

    Let $((\alpha, p), (\beta, q))$ be a diagram of $n$-cells in $K^2 X$ such that $\alpha$, $\beta \in TKX_n$ are composites of constraint $n$-cells of $KX$.  Since constraint $n$-cells are determined by $(n - 1)$-cells, and $TKX_{n - 1} = TKX'_{n - 1}$, we have $(\alpha, p)$, $(\beta, q) \in K^2 X'_n$, with $\mu^K_{X'}(\alpha, p)$ and $\mu^K_{X'}(\beta, q)$ parallel.  Thus, by Theorem~\ref{thm:Kalgcohtwo}, $(\alpha, p) = (\beta, q)$.
  \end{proof}

  The final coherence theorem describes a class of diagrams of constraint $n$-cells which commute in any $K$-algebra.  These diagrams should be thought of as those that are ``free-shaped'', i.e. they are diagrams of constraint cells that could arise in a free $K$-algebra.  This rules out diagrams in which the sources and targets of the constraint cells involve non-constraint cells with constraint cells in their boundaries, and non-composite cells with composites in their boundaries.  This is the analogue of a coherence theorem for tricategories due to Gurski, which describes a similar class of diagrams of constraint $3$-cells in the context of tricategories \cite[Corollary 10.2.5]{Gur06}.  We call such a diagram $F_K$-admissible, where $F_K$ is the left adjoint to the forgetful functor
    \[
      U_K \colon K\Alg \longrightarrow \nGSet,
    \]
  which sends a $K$-algebra to its underlying $n$-globular set; this terminology is taken from the theorem of Gurski mentioned above.

  \begin{defn}
    Let $K$ be an $n$-globular operad with contraction $\gamma$, and let $\xymatrix{ KX \ar[r]^{\theta} & X }$ be a $K$-algebra.  A diagram $((\alpha, p), (\beta, q))$ of constraint $n$-cells in $X$ is said to be \emph{$F_K$-admissible} if there exists a sub-$n$-globular set $E$ of $X$, with $E_n = \emptyset$ and inclusion map $i \colon E \hookrightarrow X$, and a diagram $((\alpha', p'), (\beta', q'))$ of constraint $n$-cells in $F_K E$ such that $((\alpha, p), (\beta, q))$ is the image of $((\alpha', p'), (\beta', q'))$ under the map
      \[
        \xy
          % POINTS
          (0, 0)*+{K^2 E}="0,0";
          (16, 0)*+{KX}="1,0";
          (0, -16)*+{KE}="0,1";
          (16, -16)*+{X,}="1,1";
          % ARROWS
          {\ar^-{K\overline{i}} "0,0" ; "1,0"};
          {\ar_{\mu^K_E} "0,0" ; "0,1"};
          {\ar^{\theta} "1,0" ; "1,1"};
          {\ar_-{\overline{i}} "0,1" ; "1,1"};
        \endxy
      \]
    where $\overline{i}$ is the transpose under the adjunction $F_K \ladj U_K$ of $i$.
  \end{defn}

  The following is now an immediate corollary of Lemma~\ref{thm:Kalgcohtwo}, and the fact that $\overline{i}$ is a map of $K$-algebras.

  \begin{corol}
    Let $K$ be an $n$-globular operad with contraction $\gamma$, and let $\xymatrix{ KX \ar[r]^{\theta} & X }$ be a $K$-algebra.  Then every $F_K$-admissible diagram of constraint $n$-cells in $\xymatrix{ KX \ar[r]^{\theta} & X }$ commutes.
  \end{corol}

  \chapter{Comparisons between algebraic definitions of weak $n$-category}  \label{chap:Penonop}

  In this chapter we investigate the connections between the algebraic definitions of weak $n$-category given in the previous two chapters.  We begin by recalling that there is a Batanin operad whose algebras are Penon weak $n$-categories, which was originally proved by Batanin~\cite{Bat02}; we give a new, more direct proof of this using our construction of Penon's left adjoint from Section~\ref{sect:ladj}.  This tells us that the coherence theorems from Section~\ref{sect:globopcoh} hold for Penon weak $n$-categories.  It also allows for comparison of Penon's definition with other operadic definitions of weak $n$-category, though we make no such comparison here.

  We then take the first steps towards a comparison between Batanin weak $n$-categories and Leinster weak $n$-categories, using the correspondence between Batanin operads and Leinster operads from Section~\ref{sect:Leinster}.  It has long been believed that these definitions are in some sense equivalent \cite[end of Section~4.5]{Lei00}, but formalising this statement is difficult since it is not clear what ``equivalent'' should mean in this context, so no such comparison has previously been made.  We derive comparison functors between the categories of Batanin weak $n$-categories and Leinster weak $n$-categories using the universal properties of the operads $B$ and $L$.  These functors should be higher-dimensional equivalences of some kind;  there is currently no way of stating this formally, so we give a preliminary approximation of what this might mean.

  \section{The operad for Penon weak $n$-categories}  \label{sect:Penonop}

  In~\cite{Bat02}, Batanin proved that there is an $n$-globular operad whose algebras are Penon weak $n$-categories, and that this operad can be equipped with a contraction and system of compositions.  In this section we give a new, alternative proof of this fact using the construction of Penon's left adjoint from Section~\ref{sect:ladj}.  Although it is not a new result, our proof is more direct than that of Batanin, offering an alternative point of view in a way that elucidates the structure of the operad, and makes clear the fact that the contraction and system of compositions arise naturally from the contraction and magma structure in the original definition of the monad $P$.

  Note that we do not use this fact, or the proof, elsewhere in the thesis.  The result implies that the coherence theorems from Section~\ref{sect:globopcoh} apply to Penon weak $n$-categories.  In order to prove it, we use an alternative statement of the definition of $n$-globular operad (see \cite{Web04}), which describes an $n$-globular operad as a cartesian map of monads.

  \begin{prop}  \label{defn:Weberops}
    An \emph{$n$-globular operad} consists of a monad $K$ on $\nGSet$, and a cartesian map of monads $k \colon K \Rightarrow T$  (by which we mean a cartesian natural transformation $k \colon K \Rightarrow T$ respecting the monad structure).  Given operads $k \colon K \Rightarrow T$, $k' \colon K' \Rightarrow T$, a \emph{map of operads} $f \colon K \Rightarrow K'$ is a map of monads such that the diagram
      \[
        \xy
          % POINTS
          (-10, 14)*+{K}="K";
          (10, 14)*+{K'}="K'";
          (0, 0)*+{T1}="T";
          % ARROWS
          {\ar^{f} "K" ; "K'"};
          {\ar_{k} "K" ; "T"};
          {\ar^{k'} "K'" ; "T"}
        \endxy
      \]
    commutes.  The category of algebras for an operad $k \colon K \Rightarrow T$ is the category $K\Alg$ of algebras for the monad $K$.
  \end{prop}

  It is a straightforward and enlightening exercise to prove that this definition is equivalent to Definition~\ref{defn:globop}, and also that the monad $K$ is necessarily cartesian.  We leave this to the reader.

  To prove that there is an operad whose algebras are Penon weak $n$-categories using Proposition~\ref{defn:Weberops} we must prove three facts:  that there is a natural transformation
    \[
      p \colon P \Longrightarrow T,
    \]
  that this natural transformation is cartesian, and that it is a map of monads.  Note that we know that the source of this natural transformation must be $P$ to ensure that the algebras for the resulting operad are indeed $P$-algebras.

  \begin{prop}  \label{prop:pnatural}
    Recall from Definition~\ref{defn:Penonncat} that $P \colon \nGSet \rightarrow \nGSet$ is the monad induced by the adjunction
        \[
          \xy
            % POINTS
            (0, 0)*+{\nGSet}="nG";
            (20, 0)*+{\Qn .}="Qn";
            % ARROWS
            {\ar@<1ex>^-{F}_-*!/u1pt/{\labelstyle \bot} "nG" ; "Qn"};
            {\ar@<1ex>^-{U} "Qn" ; "nG"};
          \endxy
       \]
    There is a natural transformation $p \colon P \Rightarrow T$ whose component $p_X$ at $X~\in~\nGSet$ is given by the map part of
      \[
        F(X) = (\xymatrix{ PX \ar[r]^{p_X} & TX }),
      \]
    an object of $\Qn$.
  \end{prop}

  \begin{proof}
    Recall that there is a forgetful functor
      \[
        U_T \colon \nCat \longrightarrow \nGSet
      \]
    that sends a strict $n$-category to its underlying $n$-globular set, and that the category $\Rn$ can be considered as the comma category
      \[
        \nGSet \downarrow U_T.
      \]
    Write
      \[
        \pi_1 \colon \nGSet \downarrow U_T \rightarrow \nGSet
      \]
    and
      \[
        \pi_2 \colon \nGSet \downarrow U_T \rightarrow \nCat
      \]
    for the projection maps, and consider the following diagram:
      \[
        \xy
          % POINTS
          (0, 0)*+{\nGSet}="0,0";
          (0, -10)*+{\Qn}="0,1";
          (0, -20)*+{\nGSet \downarrow G}="0,2";
          (-16, -36)*+{\nGSet}="-1,3";
          (16, -36)*+{\nCat.}="1,3";
          % ARROWS
          {\ar^F "0,0" ; "0,1"};
          %{\ar@/_1.5pc/ "0,0" ; "-1,3"};
          %{\ar@/^1.5pc/ "0,0" ; "1,3"};
          {\ar^W "0,1" ; "0,2"};
          {\ar@/_0.5pc/_-{\pi_1} "0,2" ; "-1,3"};
          {\ar@/^0.5pc/^-{\pi_2} "0,2" ; "1,3"};
          {\ar^G "1,3" ; "-1,3"};
          {\ar@{=>} (-4, -27) ; (4, -31)};
        \endxy
      \]
    Then the universal property of $\nGSet \downarrow U_T$ as a $2$-limit induces a unique natural transformation $p \colon P \Rightarrow T$ such that
      \[
        \xy
          % POINTS
          (0, 0)*+{\nGSet}="0,0";
          (0, -10)*+{\Qn}="0,1";
          (0, -20)*+{\nGSet \downarrow U_T}="0,2";
          (-16, -36)*+{\nGSet}="-1,3";
          (16, -36)*+{\nCat}="1,3";
          (-58, 0)*+{\nGSet}="-3,0";
          (-42, -36)*+{\nCat}="-2,3";
          (-74, -36)*+{\nGSet}="-4,3";
          (-29, -20)*+{=};
          % ARROWS
          {\ar^F "0,0" ; "0,1"};
          {\ar@/_1.5pc/_P "0,0" ; "-1,3"};
          {\ar@/^1.5pc/^{F_T} "0,0" ; "1,3"};
          {\ar^W "0,1" ; "0,2"};
          {\ar@/_0.5pc/_-{\pi_1} "0,2" ; "-1,3"};
          {\ar@/^0.5pc/^-{\pi_2} "0,2" ; "1,3"};
          {\ar^G "1,3" ; "-1,3"};
          {\ar@{=>} (-4, -27) ; (4, -31)};
          % ARROWS for left-hand section
          {\ar@/_1.5pc/_P "-3,0" ; "-4,3"};
          {\ar@/^1.5pc/^{F_T} "-3,0" ; "-2,3"};
          {\ar^G "-2,3" ; "-4,3"};
          {\ar@{=>}^{p} (-62, -18) ; (-54, -22)};
        \endxy
      \]
    where $F_T$ is the free strict $n$-category functor.
  \end{proof}

  \begin{prop}  \label{prop:pcartes}
    The natural transformation $p \colon P \Rightarrow T$ is cartesian.
  \end{prop}

  To prove this, we must show that each naturality square for $p$ is a pullback square.  To do so, we use the construction of the adjunction
        \[
          \xy
            % POINTS
            (0, 0)*+{\nGSet}="nG";
            (20, 0)*+{\Qn .}="Qn";
            % ARROWS
            {\ar@<1ex>^-{F}_-*!/u1pt/{\labelstyle \bot} "nG" ; "Qn"};
            {\ar@<1ex>^-{U} "Qn" ; "nG"};
          \endxy
       \]
  from Section~\ref{sect:Penondefn}.  Recall that this adjunction can be decomposed as
      \[
        \xy
          % POINTS
          (0, 0)*+{\nGSet}="nG";
          (20, 0)*+{\Rn}="Rn";
          (36, 0)*+{\Qn.}="Qn";
          % ARROWS
          {\ar@<1ex>^-{H}_-*!/u1pt/{\labelstyle \bot} "nG" ; "Rn"};
          {\ar@<1ex>^-{V} "Rn" ; "nG"};
          {\ar@<1ex>^-{J}_-*!/u1pt/{\labelstyle \bot} "Rn" ; "Qn"};
          {\ar@<1ex>^-{W} "Qn" ; "Rn"};
        \endxy
     \]
  Given a map $f \colon X \rightarrow Y$ in $\nGSet$, the corresponding naturality square is obtained by applying the functor $J \colon \Rn \rightarrow \Qn$ to the map
      \[
        \xy
          % POINTS
          (0, 0)*+{X}="A";
          (16, 0)*+{Y}="B";
          (0, -16)*+{TX}="TA";
          (16, -16)*+{TY.}="TB";
          % ARROWS
          {\ar^-{f} "A" ; "B"};
          {\ar_{\eta^T_X} "A" ; "TA"};
          {\ar_-{Tf} "TA" ; "TB"};
          {\ar^{\eta^T_Y} "B" ; "TB"};
        \endxy
      \]
  in $\Rn$, which is a pullback square in $\nGSet$, since the free strict $n$-category monad $T$ is cartesian \cite[4.1.18 and F.2.2]{Lei04}.  Thus we prove that $p$ is cartesian by proving that the functor $J$ sends maps that are pullback squares to maps that are pullback squares (in fact, we do so only for a certain class of such maps).  Recall that the adjunction $J \ladj W$ can be decomposed as the following chain of adjunctions:
    \[
      \xymatrix{
        \Rn = \mathcal{R}_{0, 0} \ar@<1ex>[r]^-{C_{0, 1}}_-*!/u1pt/{\labelstyle \bot} & \mathcal{R}_{0, 1} \ar@<1ex>[l]^-{D_{0, 1}} \ar@<1ex>[r]^-{M_{1, 1}}_-*!/u1pt/{\labelstyle \bot} & \dotsc \ar@<1ex>[l]^-{N_{1, 1}} \ar@<1ex>[r]^-{C_{n - 1, n}}_-*!/u1pt/{\labelstyle \bot} & \mathcal{R}_{n - 1, n} \ar@<1ex>[l]^-{D_{n - 1, n}} \ar@<1ex>[r]^-{M_{n, n}}_-*!/u1pt/{\labelstyle \bot}  & \mathcal{R}_{n, n} \ar@<1ex>[l]^-{N_{n, n}} \ar@<1ex>[r]^-{C_{n, n + 1}}_-*!/u1pt/{\labelstyle \bot} & \mathcal{R}_{n, n + 1} = \Qn \ar@<1ex>[l]^-{D_{n, n + 1}},
               }
    \]
  where the functor $C_{m, m + 1}$ freely adds the contraction structure at dimension $m + 1$, and the functor $M_{m, m}$ freely adds the magma structure at dimension $m$.  We now prove three lemmas to show that each of these functors sends maps that are pullback squares to maps that are pullback squares, thus showing that their composite $J$ does so as well.  Note that there are three lemmas since the functor $C_{n, n + 1}$ must be treated separately from the functors $C_{m, m + 1}$ for $0 \leq m \leq n - 1$.

  Note that we only consider maps whose the strict $n$-category part is a map in the image of $T$ between free strict $n$-categories; this is as general as we need it to be to prove Proposition~\ref{prop:pcartes}, and it allows us to use the fact that $T$ is cartesian in the proofs of the lemmas.

  \begin{lemma}  \label{lem:Cpbstopbs}
    Let $0 \leq m \leq n - 1$ and suppose we have a morphism
      \[
        \xy
          % POINTS
          (0, 0)*+{X}="X";
          (16, 0)*+{Y}="Y";
          (0, -16)*+{TA}="TA";
          (16, -16)*+{TB}="TB";
          % ARROWS
          {\ar^-{u} "X" ; "Y"};
          {\ar_{x} "X" ; "TA"};
          {\ar_-{Tf} "TA" ; "TB"};
          {\ar^{y} "Y" ; "TB"};
          % PULLBACK STUFF
          (6,-1)*{}; (6,-5)*{} **\dir{-};
          (2,-5)*{}; (6,-5)*{} **\dir{-};
        \endxy
      \]
    in $\mathcal{R}_{m, m}$ that is a pullback square in $\nGSet$.  Then its image under the functor
      \[
        C_{m, m + 1} \colon \mathcal{R}_{m, m} \longrightarrow \mathcal{R}_{m, m + 1}
      \]
    is also a pullback square in $\nGSet$.
  \end{lemma}

  \begin{proof}
    The idea of the proof is as follows: the functor $C_{m, m + 1}$ freely adds contraction $(m + 1)$-cells to $X$ and $Y$.  These contraction cells are obtained by taking pullbacks in $\Set$, and then added to the sets of $(m + 1)$-cells $X_{m + 1}$ and $Y_{m + 1}$ by taking coproducts in $\Set$.  The action of $C_{m, m + 1}$ on the map itself is then induced by the universal properties of these pullbacks and coproducts.  Thus the image of this map under the functor $C_{m, m + 1}$ is a coproduct of pullback squares (with some adjustments at the bottom to ensure that the strict $n$-category parts $TX$ and $TY$ remain unchanged).  Since pullbacks commute with coproducts in $\Set$ \cite[IX.2 Exercise 3]{Mac98}, this coproduct of pullback squares is itself a pullback square.

    Recall from Definition~\ref{defn:contr} that we have
      \[
        \xy
          % POINTS
          (0, 0)*+{X^c_{m + 1}}="Xc";
          (36, 0)*+{X_m}="Xr";
          (0, -16)*+{X_m}="Xb";
          (36, -16)*+{X_{m - 1} \times X_{m - 1} \times TA_m}="XXTA";
          % ARROWS
          {\ar "Xc" ; "Xr"};
          {\ar "Xc" ; "Xb"};
          {\ar_-{(s, t, x_m)} "Xb" ; "XXTA"};
          {\ar^{(s, t, x_m)} "Xr" ; "XXTA"};
          % PULLBACK STUFF
          (6,-1)*{}; (6,-5)*{} **\dir{-};
          (2,-5)*{}; (6,-5)*{} **\dir{-};
        \endxy
      \]
    For $k \neq m + 1$, we have $C_{m, m}(u, Tf)_k = (u, Tf)_k$, and since pullbacks in $\nGSet$ are computed pointwise, we only need to check that $C_{m, m}(u, Tf)_{m + 1}$ is a pullback square, i.e. that
      \[
        \xy
          % POINTS
          (0, 0)*+{X_{m + 1} \amalg X^c_{m + 1}}="X";
          (44, 0)*+{Y_{m + 1} \amalg X^c_{m + 1}}="Y";
          (0, -16)*+{TA_{m + 1}}="TA";
          (44, -16)*+{TB_{m + 1}}="TB";
          % ARROWS
          {\ar^-{u_{m + 1} \amalg u^c_{m + 1}} "X" ; "Y"};
          {\ar_{x_{m + 1} \amalg x^c_{m + 1}} "X" ; "TA"};
          {\ar_-{Tf_{m + 1}} "TA" ; "TB"};
          {\ar^{y_{m + 1} \amalg y^c_{m + 1}} "Y" ; "TB"};
        \endxy
      \]
    is a pullback square.  Since coproducts commute with pullbacks in $\Set$ \cite[IX.2, exercise 3]{Mac98}, this is true if the squares
      \[
        \xy
          % POINTS
          (0, 0)*+{X_{m + 1}}="X";
          (24, 0)*+{Y_{m + 1}}="Y";
          (0, -16)*+{TA_{m + 1}}="TA";
          (24, -16)*+{TB_{m + 1}}="TB";
          (48, 0)*+{X^c_{m + 1}}="Xc";
          (72, 0)*+{Y^c_{m + 1}}="Yc";
          (48, -16)*+{TA_{m + 1}}="TAc";
          (72, -16)*+{TB_{m + 1}}="TBc";
          % ARROWS
          {\ar^-{u_{m + 1}} "X" ; "Y"};
          {\ar_{x_{m + 1}} "X" ; "TA"};
          {\ar_-{Tf_{m + 1}} "TA" ; "TB"};
          {\ar^{y_{m + 1}} "Y" ; "TB"};
          {\ar^-{u^c_{m + 1}} "Xc" ; "Yc"};
          {\ar_{x^c_{m + 1}} "Xc" ; "TAc"};
          {\ar_-{Tf_{m + 1}} "TAc" ; "TBc"};
          {\ar^{y^c_{m + 1}} "Yc" ; "TBc"};
        \endxy
      \]
    are both pullback squares.  The left-hand square is a pullback square by hypothesis.  For the right-hand square, suppose we have a cone
      \[
        \xy
          % POINTS
          (0, 0)*+{V}="V";
          (24, 0)*+{Y^c_{m + 1}}="Yc";
          (0, -16)*+{TA_{m + 1}}="TAc";
          (24, -16)*+{TB_{m + 1}}="TBc";
          % ARROWS
          {\ar^-{v_1} "V" ; "Yc"};
          {\ar_{v_2} "V" ; "TAc"};
          {\ar_-{Tf_{m + 1}} "TAc" ; "TBc"};
          {\ar^{y^c_{m + 1}} "Yc" ; "TBc"};
        \endxy
      \]
    in $\Set$.  Recall that we have source and target maps $s$, $t \colon Y^c_{m + 1} \rightarrow Y_m$ given by the projections from the pullback defining $Y^c_{m + 1}$.  Composing with these, and source and target maps for $TA$ and $TB$, induces maps
      \[
        \xy
          % POINTS
          (-10, 10)*+{V}="V";
          (0, 0)*+{X_m}="X";
          (20, 0)*+{Y_m}="Y";
          (0, -16)*+{TA_m}="TA";
          (20, -16)*+{TB_m,}="TB";
          (44, 10)*+{V}="Vr";
          (54, 0)*+{X_m}="Xr";
          (74, 0)*+{Y_m}="Yr";
          (54, -16)*+{TA_m}="TAr";
          (74, -16)*+{TB_m.}="TBr";
          % ARROWS
          {\ar^-{u_m} "X" ; "Y"};
          {\ar_{x_m} "X" ; "TA"};
          {\ar_-{Tf_m} "TA" ; "TB"};
          {\ar^{y_m} "Y" ; "TB"};
          {\ar@/^1pc/^-{sv_1} "V" ; "Y"};
          {\ar@/_1pc/_{sv_2} "V" ; "TA"};
          {\ar@{-->}^{!\sigma} "V" ; "X"};
          {\ar^-{u_m} "Xr" ; "Yr"};
          {\ar_{x_m} "Xr" ; "TAr"};
          {\ar_-{Tf_m} "TAr" ; "TBr"};
          {\ar^{y_m} "Yr" ; "TBr"};
          {\ar@/^1pc/^-{tv_1} "Vr" ; "Yr"};
          {\ar@/_1pc/_{tv_2} "Vr" ; "TAr"};
          {\ar@{-->}^{!\tau} "Vr" ; "Xr"};
          % PULLBACK STUFF
          (6,-1)*{}; (6,-5)*{} **\dir{-};
          (2,-5)*{}; (6,-5)*{} **\dir{-};
          (60,-1)*{}; (60,-5)*{} **\dir{-};
          (56,-5)*{}; (60,-5)*{} **\dir{-};
        \endxy
      \]
    The maps $\sigma$ and $\tau$ give us a cone over the pullback square defining $X^c_{m + 1}$;  commutativity of this cone comes from the globularity conditions, and the fact that every cell in the image of $v_2$ is an identity, so has the same source and target.  Thus the universal property of $X^c_{m + 1}$ induces a unique map such that the diagram
      \[
        \xy
          % POINTS
          (-10, 10)*+{V}="V";
          (0, 0)*+{X^c_{m + 1}}="Xc";
          (38, 0)*+{X_m}="Xr";
          (0, -16)*+{X_m}="Xb";
          (38, -16)*+{X_{m - 1} \times X_{m - 1} \times TA_m.}="bl";
          % ARROWS
          {\ar@/^1pc/^-{\sigma} "V" ; "Xr"};
          {\ar@/_1pc/_{\tau} "V" ; "Xb"};
          {\ar@{-->}^v "V" ; "Xc"};
          {\ar "Xc" ; "Xr"};
          {\ar "Xc" ; "Xb"};
          {\ar_-{(s, t, x_m)} "Xb" ; "bl"};
          {\ar^{(s, t, x_m)} "Xr" ; "bl"};
          % PULLBACK STUFF
          (6,-1)*{}; (6,-5)*{} **\dir{-};
          (2,-5)*{}; (6,-5)*{} **\dir{-};
        \endxy
      \]
    commutes.

    We now check that $v$ makes the diagram
      \[
        \xy
          % POINTS
          (-10, 10)*+{V}="V";
          (0, 0)*+{X^c_{m + 1}}="Xc";
          (24, 0)*+{Y^c_{m + 1}}="Yc";
          (0, -16)*+{TA_{m + 1}}="TAc";
          (24, -16)*+{TB_{m + 1}}="TBc";
          % ARROWS
          {\ar@/^1pc/^-{v_1} "V" ; "Yc"};
          {\ar@/_1pc/_{v_2} "V" ; "TAc"};
          {\ar^v "V" ; "Xc"};
          {\ar^{u^c_{m + 1}} "Xc" ; "Yc"};
          {\ar^{x^c_{m + 1}} "Xc" ; "TAc"};
          {\ar_-{Tf_{m + 1}} "TAc" ; "TBc"};
          {\ar^{y^c_{m + 1}} "Yc" ; "TBc"};
        \endxy
      \]
    commute.  To show that the top triangle commutes, observe that the map $v_1 = u^c_{m + 1} \comp v$ makes the following diagram commute:
      \[
        \xy
          % POINTS
          (-40, 20)*+{V}="V";
          (-20, 10)*+{X^c_{m + 1}}="Xc";
          (18, 10)*+{X_m}="Xr";
          (-20, -6)*+{X_m}="Xb";
          (0, 0)*+{Y^c_{m + 1}}="Yc";
          (38, 0)*+{Y_m}="Yr";
          (0, -16)*+{Y_m}="Yb";
          (38, -16)*+{Y_{m - 1} \times Y_{m - 1} \times TB_m.}="bl";
          % ARROWS
          {\ar^v "V" ; "Xc"};
          {\ar@/^1pc/^{\sigma} "V" ; "Xr"};
          {\ar@/_1pc/_{\tau} "V" ; "Xb"};
          {\ar^{u^c_{m + 1}} "Xc" ; "Yc"};
          {\ar^s "Xc" ; "Xr"};
          {\ar_t "Xc" ; "Xb"};
          {\ar^{u_m} "Xr" ; "Yr"};
          {\ar_{u_m} "Xb" ; "Yb"};
          {\ar^s "Yc" ; "Yr"};
          {\ar_t "Yc" ; "Yb"};
          {\ar_-{(s, t, y_m)} "Yb" ; "bl"};
          {\ar^{(s, t, y_m)} "Yr" ; "bl"};
          % PULLBACK STUFF
          (6,-1)*{}; (6,-5)*{} **\dir{-};
          (2,-5)*{}; (6,-5)*{} **\dir{-};
        \endxy
      \]
    Since $u_m \sigma = sv_1$ and $u_m \tau = tv_1$, by the universal property of $Y^c_{m + 1}$, we have $u^c_{m + 1} \comp v = v_1$.

    To show that the left-hand triangle commutes, write $i \colon TA_m \rightarrow TA_{m + 1}$ for the map that sends an $m$-cell to its identity $(m + 1)$-cell, and consider that we can factorise $x^c_{m + 1} \comp v$ as
      \[
        \xy
          % POINTS
          (-20, 0)*+{V}="V";
          (0, 0)*+{X^c_{m + 1}}="Xc";
          (20, 0)*+{X_m}="X";
          (40, 0)*+{TA_m}="TA";
          (60, 0)*+{TA_{m + 1}.}="TAr";
          % ARROWS
          {\ar^v "V" ; "Xc"};
          {\ar^{s} "Xc" ; "X"};
          {\ar@/^1.5pc/^{x^c_{m + 1}} "Xc" ; "TAr"};
          {\ar^{x_m} "X" ; "TA"};
          {\ar^i "TA" ; "TAr"};
          {\ar@/_1.5pc/_{\sigma} "V" ; "X"};
          {\ar@/_3pc/_{sv_2} "V" ; "TA"};
        \endxy
      \]
    Thus we have $x^c_{m + 1} \comp v = isv_2 = v_2$, since all cells in the image of $v_2$ are identities.

    Finally, uniqueness of $v$ comes from the universal property of $X^c_{m + 1}$.  Hence
      \[
        \xy
          % POINTS
          (0, 0)*+{X^c_{m + 1}}="Xc";
          (24, 0)*+{Y^c_{m + 1}}="Yc";
          (0, -16)*+{TA_{m + 1}}="TAc";
          (24, -16)*+{TB_{m + 1}}="TBc";
          % ARROWS
          {\ar^-{u^c_{m + 1}} "Xc" ; "Yc"};
          {\ar_{x^c_{m + 1}} "Xc" ; "TAc"};
          {\ar_-{Tf_{m + 1}} "TAc" ; "TBc"};
          {\ar^{y^c_{m + 1}} "Yc" ; "TBc"};
          % PULLBACK STUFF
          (6,-1)*{}; (6,-5)*{} **\dir{-};
          (2,-5)*{}; (6,-5)*{} **\dir{-};
        \endxy
      \]
    is a pullback square, so $C_{m, m + 1}(u, Tf)$ is a pullback square.
  \end{proof}

  We must treat the case $m = n$ separately.

  \begin{lemma}  \label{lem:Cnpbstopbs}
    Suppose we have a morphism
      \[
        \xy
          % POINTS
          (0, 0)*+{X}="X";
          (16, 0)*+{Y}="Y";
          (0, -16)*+{TA}="TA";
          (16, -16)*+{TB}="TB";
          % ARROWS
          {\ar^-{u} "X" ; "Y"};
          {\ar_{x} "X" ; "TA"};
          {\ar_-{Tf} "TA" ; "TB"};
          {\ar^{y} "Y" ; "TB"};
          % PULLBACK STUFF
          (6,-1)*{}; (6,-5)*{} **\dir{-};
          (2,-5)*{}; (6,-5)*{} **\dir{-};
        \endxy
      \]
    in $\mathcal{R}_{n, n}$ that is a pullback square in $\nGSet$.  Then its image under the functor
      \[
        C_{n, n + 1} \colon \mathcal{R}_{n, n} \longrightarrow \mathcal{R}_{n, n + 1} = \Qn
      \]
    is also a pullback square in $\nGSet$.
  \end{lemma}

  \begin{proof}
   Recall from Definition~\ref{defn:freenplusonecontr} that we have
      \[
        \xy
          % POINTS
          (0, 0)*+{X^c_{n + 1}}="Xc";
          (38, 0)*+{X_n}="Xr";
          (0, -16)*+{X_n}="Xb";
          (38, -16)*+{X_{n - 1} \times X_{n - 1} \times TA_n,}="bl";
          % ARROWS
          {\ar^{\pi_1} "Xc" ; "Xr"};
          {\ar_{\pi_2} "Xc" ; "Xb"};
          {\ar_-{(s, t, x_n)} "Xb" ; "bl"};
          {\ar^{(s, t, x_n)} "Xr" ; "bl"};
          % PULLBACK STUFF
          (6,-1)*{}; (6,-5)*{} **\dir{-};
          (2,-5)*{}; (6,-5)*{} **\dir{-};
        \endxy
      \]
    and that $\tilde{X}_n$ is defined to be the coequaliser of the diagram
      \[
        \xy
          % POINTS
          (0, 0)*+{X^c_{n + 1}}="Xc";
          (20, 0)*+{X_n}="X";
          % ARROWS
          {\ar@<1ex>^-{\pi_1} "Xc" ; "X"};
          {\ar@<-1ex>_-{\pi_2} "Xc" ; "X"};
        \endxy
      \]
    in $\Set$.  We write $q \colon X_n \rightarrow \tilde{X}_n$ for the coprojection.  The set $\tilde{Y}_n$ is defined similarly, and we write $r \colon Y_n \rightarrow \tilde{Y}_n$ for the coprojection.  For all $0 \leq m < n$ we have
      \[
        C_{n, n + 1}(u, Tf)_m = (u, Tf)_m,
      \]
    and for $m = n$, we have that $C_{n, n + 1}(u, Tf)_n$ is given by
      \[
        \xy
          % POINTS
          (0, 0)*+{\tilde{X}_n}="X";
          (16, 0)*+{\tilde{Y}_n}="Y";
          (0, -16)*+{TA_n}="TA";
          (16, -16)*+{TB_n,}="TB";
          % ARROWS
          {\ar^-{\tilde{u}_n} "X" ; "Y"};
          {\ar_{\tilde{x}_n} "X" ; "TA"};
          {\ar_-{Tf_n} "TA" ; "TB"};
          {\ar^{\tilde{y}_n} "Y" ; "TB"};
        \endxy
      \]
    so we only need to check that this is a pullback square in $\Set$.

    Write $w$ for the unique map making the diagram
      \[
        \xy
          % POINTS
          (-10, 10)*+{X_n}="X";
          (0, 0)*+{\bullet}="pb";
          (6, 10)*+{Y_n}="Yt";
          (16, 0)*+{\tilde{Y}_n}="Y";
          (0, -16)*+{TA_n}="TA";
          (16, -16)*+{TB_n}="TB";
          % ARROWS
          {\ar^{u_n} "X" ; "Yt"};
          {\ar^{r} "Yt" ; "Y"};
          {\ar "pb" ; "Y"};
          {\ar "pb" ; "TA"};
          {\ar_-{Tf_n} "TA" ; "TB"};
          {\ar^{\tilde{y}_n} "Y" ; "TB"};
          {\ar@/_1pc/_{x_n} "X" ; "TA"};
          {\ar@{-->}^w "X" ; "pb"};
          % PULLBACK STUFF
          (6,-1)*{}; (6,-5)*{} **\dir{-};
          (2,-5)*{}; (6,-5)*{} **\dir{-};
        \endxy
      \]
    commute.  We will show that, for $a$, $b \in X_n$, $w(a) = w(b)$ if and only if $(a, b) \in X^c_{n + 1}$, and also that $w$ is surjective; and thus $C_{n, n + 1}(u, Tf)_n$ is a pullback square and $w = q$.

    Let $(a, b) \in X^c_{n + 1}$, so $x_n(a) = x_n(b)$, $s(a) = s(b)$, $t(a) = t(b)$.  We have $(u_n(a), u_n(b)) \in Y^c_{n + 1}$, so $ru_n(a) = ru_n(b)$.  Thus
      \[
        w(a) = (x_n(a), ru_n(a)) = (x_n(b), ru_n(b)) = w(b).
      \]

    Now let $a$, $b \in X_n$ with $w(a) = w(b)$, so $x_n(a) = x_n(b)$, $ru_n(a) = ru_n(b)$.  The source map $s \colon X_n \rightarrow X_{n - 1}$ is the unique map making the diagram
      \[
        \xy
          % POINTS
          (-16, 10)*+{X_n}="Xt";
          (4, 10)*+{Y_n}="Yt";
          (-16, -6)*+{TA_n}="TAt";
          (0, 0)*+{X_{n - 1}}="X";
          (20, 0)*+{Y_{n - 1}}="Y";
          (0, -16)*+{TA_{n - 1}}="TA";
          (20, -16)*+{TB_{n - 1}}="TB";
          % ARROWS
          {\ar^-{u_{n - 1}} "X" ; "Y"};
          {\ar^{x_{n - 1}} "X" ; "TA"};
          {\ar_-{Tf_{n - 1}} "TA" ; "TB"};
          {\ar^{y_{n - 1}} "Y" ; "TB"};
          {\ar^{u_n} "Xt" ; "Yt"};
          {\ar^s "Xt" ; "X"};
          {\ar_{x_n} "Xt" ; "TAt"};
          {\ar_s "TAt" ; "TA"};
          {\ar^s "Yt" ; "Y"};
          % PULLBACK STUFF
          (6,-1)*{}; (6,-5)*{} **\dir{-};
          (2,-5)*{}; (6,-5)*{} **\dir{-};
        \endxy
      \]
    commute.  Thus, since $su_n(a) = su_n(b)$ and $sx_n(a) = sx_n(b)$, we have $s(a) = s(b)$.  Similarly, $t(a) = t(b)$. Hence $(a, b) \in X^c_{n + 1}$.

    Now let $\pi \in TA_n$, $c \in \tilde{Y}_n$, with $Tf_n(\pi) = \tilde{y}_n(c)$.  We wish to show that there is some $a \in X_n$ with $w(a) = (\pi, c)$, and thus that $w$ is surjective.  Since $\tilde{y}_n$ is surjective, there exists $c' \in Y_n$ with $r(c') = c$.  Since $X_n$ is given by the pullback
      \[
        \xy
          % POINTS
          (0, 0)*+{X_n}="X";
          (16, 0)*+{Y_n}="Y";
          (0, -16)*+{TA_n}="TA";
          (16, -16)*+{TB_n}="TB";
          % ARROWS
          {\ar^-{u_n} "X" ; "Y"};
          {\ar_{x_n} "X" ; "TA"};
          {\ar_-{Tf_n} "TA" ; "TB"};
          {\ar^{y_n} "Y" ; "TB"};
          % PULLBACK STUFF
          (6,-1)*{}; (6,-5)*{} **\dir{-};
          (2,-5)*{}; (6,-5)*{} **\dir{-};
        \endxy
      \]
    and $yr(c') = Tf_n(\pi)$, we have $a \in X_n$ with $x_n(a) = \pi$, $u_n(a) = c'$.  Thus $w(a) = (\pi, c)$, so $w$ is surjective.  Hence
      \[
        \xy
          % POINTS
          (0, 0)*+{\tilde{X}_n}="X";
          (16, 0)*+{\tilde{Y}_n}="Y";
          (0, -16)*+{TA_n}="TA";
          (16, -16)*+{TB_n}="TB";
          % ARROWS
          {\ar^-{\tilde{u}_n} "X" ; "Y"};
          {\ar_{\tilde{x}_n} "X" ; "TA"};
          {\ar_-{Tf_n} "TA" ; "TB"};
          {\ar^{\tilde{y}_n} "Y" ; "TB"};
          % PULLBACK STUFF
          (6,-1)*{}; (6,-5)*{} **\dir{-};
          (2,-5)*{}; (6,-5)*{} **\dir{-};
        \endxy
      \]
    is a pullback square.
  \end{proof}

  Thus we have shown that the functors adding the free contraction cells send maps that are pullback square to maps that are pullback squares.  We now do the same for the functors adding the free magma structure.

  \begin{lemma}  \label{lem:Mpbstopbs}
    Let $0 < m \leq n$ and suppose we have a morphism
      \[
        \xy
          % POINTS
          (0, 0)*+{X}="X";
          (16, 0)*+{Y}="Y";
          (0, -16)*+{TA}="TA";
          (16, -16)*+{TB}="TB";
          % ARROWS
          {\ar^-{u} "X" ; "Y"};
          {\ar_{x} "X" ; "TA"};
          {\ar_-{Tf} "TA" ; "TB"};
          {\ar^{y} "Y" ; "TB"};
          % PULLBACK STUFF
          (6,-1)*{}; (6,-5)*{} **\dir{-};
          (2,-5)*{}; (6,-5)*{} **\dir{-};
        \endxy
      \]
    in $\mathcal{R}_{m - 1, m}$ that is a pullback square in $\nGSet$.  Then its image under the functor
      \[
        M_{m, m} \colon \mathcal{R}_{m - 1, m} \longrightarrow \mathcal{R}_{m, m}
      \]
    is also a pullback square in $\nGSet$.
  \end{lemma}

  \begin{proof}
    The idea of this proof is similar to that of the proof of Lemma~\ref{lem:Cpbstopbs}, but is slightly more complicated since the construction of $M_{m, m}$ uses filtered colimits as well as coproducts.  The functor $M_{m, m}$ freely adds binary composites of $m$-cells to $X$ and $Y$.  These composites are added through a process of taking pullbacks, coproducts, and filtered colimits in $\Set$.  The action of $M_{m, m}$ on the map itself is then induced by the universal properties of these pullbacks, coproducts, and filtered colimits.  Thus the image of this map under the functor $M_{m, m}$ is a filtered colimit of coproducts of pullback squares (with some adjustments at the bottom to ensure that the strict $n$-category parts $TX$ and $TY$ remain unchanged).  Since pullbacks commute with both coproducts and filtered colimits in $\Set$ \cite[IX.2, Exercise 3 and Theorem 1]{Mac98}, this filtered colimit of coproducts of pullback squares is itself a pullback square.

    Recall the notation from Definition~\ref{defn:Mj}: we write
      \[
        M_{m, m}(\xymatrix{ X \ar[r]^x & TA }) = \xymatrix{ \hat{X} \ar[r]^{\hat{x}} & TA },
      \]
      \[
        M_{m, m}(\xymatrix{ Y \ar[r]^y & TB }) = \xymatrix{ \hat{Y} \ar[r]^{\hat{y}} & TB }.
      \]
    Since $M_{m, m}$ changes only dimension $m$, and since pullbacks in $\nGSet$ are computed pointwise, we just need to check that
      \[
        \xy
          % POINTS
          (0, 0)*+{\hat{X}_m}="X";
          (16, 0)*+{\hat{Y}_m}="Y";
          (0, -16)*+{TA_m}="TA";
          (16, -16)*+{TB_m}="TB";
          % ARROWS
          {\ar^-{\hat{u}_m} "X" ; "Y"};
          {\ar_{\hat{x}_m} "X" ; "TA"};
          {\ar_-{Tf_m} "TA" ; "TB"};
          {\ar^{\hat{y}_m} "Y" ; "TB"};
        \endxy
      \]
    is a pullback square in $\Set$.  Recall that $\hat{X}_m$ and $\hat{Y}_m$ are defined as filtered colimits in $\Set$, with
      \[
        \hat{X}_m := \colim_{j \geq 1} X^{(j)}_{m}, \; \hat{Y}_m := \colim_{j \geq 1} Y^{(j)}_{m}.
      \]
    Since pullbacks commute with filtered colimits in $\Set$, we can prove that the above diagram is a pullback square by proving that, for each $j \geq 1$, the diagram
      \[
        \xy
          % POINTS
          (0, 0)*+{X^{(j)}_m}="X";
          (16, 0)*+{Y^{(j)}_m}="Y";
          (0, -16)*+{TA_m}="TA";
          (16, -16)*+{TB_m}="TB";
          % ARROWS
          {\ar^-{u^{(j)}_m} "X" ; "Y"};
          {\ar_{x^{(j)}_m} "X" ; "TA"};
          {\ar_-{Tf_m} "TA" ; "TB"};
          {\ar^{y^{(j)}_m} "Y" ; "TB"};
        \endxy
      \]
    is a pullback square in $\Set$.  We do this by induction.  When $j = 1$, we have $X^{(j)}_m = X_m$, $Y^{(j)}_m = Y_m$, and the square above becomes is a pullback square by hypothesis.

    Now suppose that $j > 1$, and we have shown that
      \[
        \xy
          % POINTS
          (0, 0)*+{X^{(j - 1)}_m}="X";
          (24, 0)*+{Y^{(j- 1)}_m}="Y";
          (0, -16)*+{TA_m}="TA";
          (24, -16)*+{TB_m}="TB";
          % ARROWS
          {\ar^-{u^{(j - 1)}_m} "X" ; "Y"};
          {\ar_{x^{(j - 1)}_m} "X" ; "TA"};
          {\ar_-{Tf_m} "TA" ; "TB"};
          {\ar^{y^{(j - 1)}_m} "Y" ; "TB"};
          % PULLBACK STUFF
          (6,-1)*{}; (6,-5)*{} **\dir{-};
          (2,-5)*{}; (6,-5)*{} **\dir{-};
        \endxy
      \]
    is a pullback square; we will show that
      \[
        \xy
          % POINTS
          (0, 0)*+{X^{(j)}_m}="X";
          (16, 0)*+{Y^{(j)}_m}="Y";
          (0, -16)*+{TA_m}="TA";
          (16, -16)*+{TB_m}="TB";
          % ARROWS
          {\ar^-{u^{(j)}_m} "X" ; "Y"};
          {\ar_{x^{(j)}_m} "X" ; "TA"};
          {\ar_-{Tf_m} "TA" ; "TB"};
          {\ar^{y^{(j)}_m} "Y" ; "TB"};
        \endxy
      \]
    is a pullback square.  Recall that $X^{(j)}_m$ is defined by
      \[
        X^{(j)}_m := X_m \amalg \coprod_{0 \leq p < m} X^{(j - 1)}_m \times_{X_p} X^{(j - 1)}_m,
      \]
    and similarly for $Y^{(j)}_m$.  Since pullbacks commute with coproducts in $\Set$, the above diagram is a pullback square if, for all $0 \leq p < m$, the diagram
    % Need labels on maps at the sides.
    % Might be nicer with j, not j - 1
      \[
        \xy
          % POINTS
          (0, 0)*+{X^{(j - 1)}_m \times_{X_p} X^{(j - 1)}_m}="X";
          (50, 0)*+{Y^{(j - 1)}_m \times_{Y_p} Y^{(j - 1)}_m}="Y";
          (0, -16)*+{TA_m}="TA";
          (50, -16)*+{TB_m}="TB";
          % ARROWS
          {\ar^-{(u^{(j - 1)}_m, u^{(j - 1)}_m)} "X" ; "Y"};
          {\ar "X" ; "TA"};
          {\ar_-{Tf_m} "TA" ; "TB"};
          {\ar "Y" ; "TB"};
        \endxy
      \]
    is a pullback square.  We can write this as
      \[
        \xy
          % POINTS
          (0, 0)*+{X^{(j - 1)}_m \times_{X_p} X^{(j - 1)}_m}="X";
          (50, 0)*+{Y^{(j - 1)}_m \times_{Y_p} Y^{(j - 1)}_m}="Y";
          (0, -16)*+{TA_m \times_{TA_p} TA_m}="TATA";
          (50, -16)*+{TB_m \times_{TB_p} TB_m}="TBTB";
          (0, -32)*+{TA_m}="TA";
          (50, -32)*+{TB_m.}="TB";
          % ARROWS
          {\ar^-{(u^{(j - 1)}_m, u^{(j - 1)}_m)} "X" ; "Y"};
          {\ar_{(x^{(j, 1)}_m, x^{(j, 1)}_m)} "X" ; "TATA"};
          {\ar_-{(Tf_m, Tf_m)} "TATA" ; "TBTB"};
          {\ar^{(y^{(j, 1)}_m, y^{(j, 1)}_m)} "Y" ; "TBTB"};
          {\ar_{\comp^m_p} "TATA" ; "TA"};
          {\ar^{\comp^m_p} "TBTB" ; "TB"};
          {\ar_-{Tf_m} "TA" ; "TB"};
        \endxy
      \]
    The top square is a pullback of pullback squares, and hence is itself a pullback square.  The fact that the bottom square is a pullback square is left as a straightforward exercise to the reader; it is an application of the fact that $T$ is a cartesian monad \cite[Example 4.1.18 and Theorem F.2.2]{Lei04}, so the naturality squares for its multiplication $\mu^T$ are pullbacks squares, and the fact that $T^2 A$ and $T^2 B$ can be constructed via a series a pullbacks in $\nGSet$ (see \cite[F.1]{Lei04} and \cite{Che11b}, which give constructions of $T$ using this method).

    Thus the diagram
      \[
        \xy
          % POINTS
          (0, 0)*+{\hat{X}_m}="X";
          (16, 0)*+{\hat{Y}_m}="Y";
          (0, -16)*+{TA_m}="TA";
          (16, -16)*+{TB_m}="TB";
          % ARROWS
          {\ar^-{\hat{u}_m} "X" ; "Y"};
          {\ar_{\hat{x}_m} "X" ; "TA"};
          {\ar_-{Tf_m} "TA" ; "TB"};
          {\ar^{\hat{y}_m} "Y" ; "TB"};
          % PULLBACK STUFF
          (6,-1)*{}; (6,-5)*{} **\dir{-};
          (2,-5)*{}; (6,-5)*{} **\dir{-};
        \endxy
      \]
    is a pullback square in $\nGSet$.  Hence $M_{m, m}$ sends maps that are pullback squares to maps that are pullback squares, as required.
  \end{proof}

  We now combine these results to prove that $p \colon P \Rightarrow T$ is cartesian.

  \begin{proof}[Proof of Proposition~\ref{prop:pcartes}]
    Combining the above results, and using the fact that $J \colon \Rn \rightarrow \Qn$ is defined as the composite
      \[
        J = C_{n, n + 1} \comp M_{n, n} \comp C_{n - 1, n} \comp \dotsb \comp M_{1, 1} \comp C_{0, 1},
      \]
    we see that, given a map $(u, Tf)$ in $\Rn$ such that
      \[
        \xy
          % POINTS
          (0, 0)*+{X}="X";
          (16, 0)*+{Y}="Y";
          (0, -16)*+{TA}="TA";
          (16, -16)*+{TB}="TB";
          % ARROWS
          {\ar^-{u} "X" ; "Y"};
          {\ar_{x} "X" ; "TA"};
          {\ar_-{Tf} "TA" ; "TB"};
          {\ar^{y} "Y" ; "TB"};
          % PULLBACK STUFF
          (6,-1)*{}; (6,-5)*{} **\dir{-};
          (2,-5)*{}; (6,-5)*{} **\dir{-};
        \endxy
      \]
    is a pullback square in $\nGSet$, the map $J(u, Tf)$ in $\Qn$ is also a pullback square in $\nGSet$.  Take $(u, Tf)$ to be
      \[
        \xy
          % POINTS
          (0, 0)*+{A}="A";
          (16, 0)*+{B}="B";
          (0, -16)*+{TA}="TA";
          (16, -16)*+{TB}="TB";
          % ARROWS
          {\ar^-{f} "A" ; "B"};
          {\ar_{\eta^T_A} "A" ; "TA"};
          {\ar_-{Tf} "TA" ; "TB"};
          {\ar^{\eta^T_B} "B" ; "TB"};
          % PULLBACK STUFF
          (6,-1)*{}; (6,-5)*{} **\dir{-};
          (2,-5)*{}; (6,-5)*{} **\dir{-};
        \endxy
      \]
    for any $f \colon A \rightarrow B$ in $\nGSet$, which is a pullback square since $T$ is cartesian.  Applying $J$ gives us that
      \[
        \xy
          % POINTS
          (0, 0)*+{PA}="PA";
          (16, 0)*+{PB}="PB";
          (0, -16)*+{TA}="TA";
          (16, -16)*+{TB}="TB";
          % ARROWS
          {\ar^-{Pf} "PA" ; "PB"};
          {\ar_{p_A} "PA" ; "TA"};
          {\ar_-{Tf} "TA" ; "TB"};
          {\ar^{p_B} "PB" ; "TB"};
          % PULLBACK STUFF
          (6,-1)*{}; (6,-5)*{} **\dir{-};
          (2,-5)*{}; (6,-5)*{} **\dir{-};
        \endxy
     \]
    is a pullback square in $\nGSet$.  Thus $p \colon P \Rightarrow T$ is a cartesian natural transformation.
  \end{proof}

  Thus the natural transformation $p \colon P \Rightarrow T$ satisfies one of the conditions in Proposition~\ref{defn:Weberops}; to prove that it is an operad, we now only need to prove the following:

  \begin{prop}  \label{prop:pmonadmap}
    The natural transformation $p \colon P \Rightarrow T$ is a map of monads.
  \end{prop}

  \begin{proof}
    We need to check that $p$ satisfies the monad map axioms.  To do so, recall that $P$ is the monad induced by the adjunction
      \[
        \xy
          % POINTS
          (0, 0)*+{\nGSet}="nG";
          (20, 0)*+{\Qn}="Qn";
          % ARROWS
          {\ar@<1ex>^-{F}_-*!/u1pt/{\labelstyle \bot} "nG" ; "Qn"};
          {\ar@<1ex>^-{U} "Qn" ; "nG"};
        \endxy
     \]
    defined in Section~\ref{sect:ladj}, and that this adjunction can be decomposed as
      \[
        \xy
          % POINTS
          (0, 0)*+{\nGSet}="nG";
          (20, 0)*+{\Rn}="Rn";
          (36, 0)*+{\Qn .}="Qn";
          % ARROWS
          {\ar@<1ex>^-{H}_-*!/u1pt/{\labelstyle \bot} "nG" ; "Rn"};
          {\ar@<1ex>^-{V} "Rn" ; "nG"};
          {\ar@<1ex>^-{J}_-*!/u1pt/{\labelstyle \bot} "Rn" ; "Qn"};
          {\ar@<1ex>^-{W} "Qn" ; "Rn"};
        \endxy
     \]
    Write $\alpha$, $\beta$ for the unit and counit of $H \ladj V$, and write $\kappa$, $\zeta$ for the unit and counit of $J \ladj W$.  Then the unit $\eta = \eta^P$ of the adjunction $F \ladj U$ is given by the composite
      \[
        \xy
          % POINTS
          (0, 0)*+{1}="1";
          (12, 0)*+{VH}="VH";
          (36, 0)*+{VWJH = UF}="VWJH";
          % ARROWS
          {\ar^-{\alpha} "1" ; "VH"};
          {\ar^-{V\kappa H} "VH" ; "VWJH"};
        \endxy
     \]
    and the counit $\epsilon$ of $F \ladj U$ is given by the composite
      \[
        \xy
          % POINTS
          (0, 0)*+{FU = JHVW}="JHVW";
          (24, 0)*+{JW}="JW";
          (36, 0)*+{1.}="1";
          % ARROWS
          {\ar^-{J \beta W} "JHVW" ; "JW"};
          {\ar^-{\zeta} "JW" ; "1"};
        \endxy
     \]

    To show that $p$ satisfies the axioms for a monad map we consider the unit $\eta^P$ and counit $\epsilon$ for the adjunction $F \ladj U$.  Write $\alpha$ for the unit of the adjunction $H \ladj V$ and $\kappa$ for the unit of the adjunction $J \ladj W$.  By Proposition~\ref{prop:HladjV}, $\alpha = \id$, so $\eta^P = V \kappa H$.  For all $X \in \nGSet$, $\kappa_{HX}$ is the map
      \[
        \xy
          % POINTS
          (0, 0)*+{X}="1";
          (16, 0)*+{PX}="P1";
          (0, -16)*+{TX}="T1";
          (16, -16)*+{TX}="T1r";
          % ARROWS
          {\ar^-{\eta^P_X} "1" ; "P1"};
          {\ar_{\eta^T_X} "1" ; "T1"};
          {\ar_-{\id_{TX}} "T1" ; "T1r"};
          {\ar^{p_X} "P1" ; "T1r"};
        \endxy
     \]
    in $\Rn$.  Commutativity of this diagram shows that $p$ satisfies the first axiom for a monad map.

    For all $X \in \nGSet$, $\epsilon_{FX}$ is the map
      \[
        \xy
          % POINTS
          (0, 0)*+{P^2 X}="PP";
          (32, 0)*+{PX}="P1";
          (0, -16)*+{TPX}="TP1";
          (16, -16)*+{T^2 X}="TT";
          (32, -16)*+{TX}="T1";
          % ARROWS
          {\ar^-{\mu^P_X} "PP" ; "P1"};
          {\ar_-{p_{PX}} "PP" ; "TP1"};
          {\ar^-{p_{X}} "P1" ; "T1"};
          {\ar_-{Tp_X} "TP1" ; "TT"};
          {\ar_-{\mu^T_X} "TT" ; "T1"};
        \endxy
      \]
    in $\Qn$.  Commutativity of this diagram shows $p$ satisfies the second axiom for a monad map.

    Thus $p \colon P \Rightarrow T$ is a monad map.
  \end{proof}

  Combining Propositions~\ref{prop:pcartes} and \ref{prop:pmonadmap} gives us the following theorem:

  \begin{thm}
    There is an operad whose algebras are Penon weak $n$-categories, given by the cartesian map of monads $p \colon P \Rightarrow T$.
  \end{thm}

  \begin{proof}
    The natural transformation $p \colon P \Rightarrow T$ is cartesian by Proposition~\ref{prop:pcartes}, and is a monad map by Proposition~\ref{prop:pmonadmap}.  Thus it is an operad, and its category of algebras is $P\Alg$, the category of Penon weak $n$-categories.
  \end{proof}

  In keeping with our notation for the operads for Batanin weak $n$-categories and Leinster weak $n$-categories, we henceforth abuse notation write $P := P1$ and $p := p_1$, so the underlying collection of this operad is denoted
      \[
        \xymatrix{
          P \ar[d]^-{p} \\
          T1.
                 }
      \]

  \begin{prop}
    The operad $P$ for Penon weak $n$-categories can be equipped with a contraction and system of compositions which arise naturally from the contraction on $p_1 \colon P1 \rightarrow T1$ and the magma structure on $P1$ respectively.
  \end{prop}

  \begin{proof}
    The presence of the contraction is immediate, since
      \[
        \xymatrix{
          P \ar[d]^-{p} \\
          T1
                 }
      \]
    is an object of $\Qn$, so is equipped with a contraction as constructed in Section~\ref{sect:ladj}.  Similarly, $P$, is equipped with a magma structure;  we use this to define a system of compositions
      \[
        \xy
          % POINTS
          (-10, 14)*+{S}="S";
          (10, 14)*+{P}="P";
          (0, 0)*+{T1}="T";
          % ARROWS
          {\ar^{\sigma} "S" ; "P"};
          {\ar_{s} "S" ; "T"};
          {\ar^{p} "P" ; "T"}
        \endxy
      \]
    as follows: for all $0 \leq m \leq n$,
      \begin{itemize}
        \item $\sigma_m(\beta^m_m) := (\eta^P_1)_m(1) = 1$;
        \item for $0 \leq l \leq m$, $\sigma_m(\beta^m_l) := 1 \comp^m_l 1$.
      \end{itemize}
    From the definition of the magma structure on $P$ given in Definition~\ref{defn:Mj}, this satisfies the source and target conditions for a map of $n$-globular sets, and the commutativity conditions required to be a map of collections.  By definition of $\sigma_m(\beta^m_m)$,
      \[
        \xy
        % POINTS
        (-15, 0)*+{1}="1";
        (0, 0)*+{S}="S";
        (15, 0)*+{P}="P";
        % ARROWS
        {\ar^{\epsilon^S} "1" ; "S"};
        {\ar^{\sigma} "S" ; "P"};
        {\ar@/_1.5pc/_{\epsilon^P = \eta^P_1} "1" ; "P"}
        \endxy
      \]
    commutes.  Thus, $\sigma$ is a system of compositions on $P$.
  \end{proof}

  Thus we now have $P$ as an operad with a contraction and a system of compositions.  Consequently, the coherence theorems from Section~\ref{sect:globopcoh} are valid for Penon weak $n$-categories.

\section{Towards a comparison between $B\Alg$ and $L\Alg$}  \label{sect:compareBL}

  In this section we give some steps towards a comparison between Batanin weak $n$-categories and Leinster weak $n$-categories; everything in this section is new.  The fact that both Batanin weak $n$-categories and Leinster weak $n$-categories are defined as algebras for $n$-globular operads means we can make some statements about the relationship between the two definitions by comparing the operads $B$ and $L$.  Some of these statements are preliminary, but we hope that they will pave the way for a more comprehensive comparison in the future.

  We use the correspondence between Batanin operads and Leinster operads (Theorems~\ref{thm:OUCtoOCS} and \ref{thm:OCStoOUC}), along with the universal properties of the operads $B$ and $L$, to derive comparison functors
    \[
      u_* \colon L\Alg \longrightarrow B\Alg \text{ and } v_* \colon B\Alg \longrightarrow L\Alg.
    \]

  We then give an explicit construction of a left adjoint to $u_*$.  We can think of $u_*$ as a forgetful functor that forgets the unbiased composition structure on an $L$-algebra, and remembers only its binary-biased composition structure.  Although the existence of the left adjoint to $u_*$ can be proved by abstract means, our construction illustrates the fact that the left adjoint freely adds unbiased composites to a $B$-algebra, while leaving the original $B$-algebra structure unchanged.  The construction is also applicable in a more general context; the exact level of generality is noted at the beginning of the subsection.

  The functors $u_*$ and $v_*$ are not equivalences of categories; they should be higher-dimensional equivalences of some kind, but we do not have a formal way of saying this, so instead we approximate this statement.  To do so, we consider what happens when we start with an $L$-algebra, apply $u_*$ to obtain a $B$-algebra, then apply $v_*$ to that to obtain an $L$-algebra;  in particular, we take some steps investigating the relationship between the resulting $L$-algebra and the original $L$-algebra.  We expect these $L$-algebras to be in some sense equivalent, but it is not clear how to make this precise, due to the lack of a well-established notion of weak map of $L$-algebras.  The underlying $n$-globular sets of these $L$-algebras are the same; they differ only on their algebra actions.  We argue that these algebra actions differ only ``up to a constraint cell''; we make this statement precise, defining a new notion of weak map of $L$-algebras in the process.

  \subsection{Comparison functors between $B\Alg$ and $L\Alg$}  \label{subsect:defineuandv}

  Recall from Definition~\ref{defn:OCS} that we write $\OCS$ for the category of Batanin operads, and from Definition~\ref{defn:OUC} that we write $\OUC$ for the category of Leinster operads.  By Theorem~\ref{thm:OUCtoOCS} we have a canonical functor
    \[
      \OUC \longrightarrow \OCS
    \]
  which is the identity on the underlying operads.  Applying this functor to $L$ equips it with a contraction and a system of compositions.  Thus, since $B$ is initial in $\OCS$, there is a unique map
      \[
        \xy
          % POINTS
          (-10, 14)*+{B}="B";
          (10, 14)*+{L}="L";
          (0, 0)*+{T1}="T";
          % ARROWS
          {\ar^{u} "B" ; "L"};
          {\ar_{b} "B" ; "T"};
          {\ar^{l} "L" ; "T"}
        \endxy
      \]
  in $\OCS$.

  By Theorem~\ref{thm:OCStoOUC} we can equip the operad $B$ with an unbiased contraction to obtain an object of $\OUC$.  However, unlike the process of equipping $L$ with a contraction and system of compositions, there is no canonical way of doing this; the unbiased contraction on $B$ depends on a choice of section to $b$, as described in Lemma~\ref{lem:khat}.  Suppose we have chosen a section to $b$ and thus equipped $B$ with an unbiased contraction.  Since $L$ is initial in $\OUC$, there is a unique map
      \[
        \xy
          % POINTS
          (-10, 14)*+{L}="L";
          (10, 14)*+{B}="B";
          (0, 0)*+{T1}="T";
          % ARROWS
          {\ar^{v} "L" ; "B"};
          {\ar^{b} "B" ; "T"};
          {\ar_{l} "L" ; "T"}
        \endxy
      \]
  in $\OUC$.

  Recall from Proposition~\ref{defn:Weberops} that every map of operads gives rise to a corresponding map of the induced monads.  Thus the maps $u$ and $v$ induce functors between the categories of algebras $B\Alg$ and $L\Alg$; we write
    \[
      u_* \colon L\Alg \longrightarrow B\Alg
    \]
  for the functor induced by $u$, and
    \[
      v_* \colon B\Alg \longrightarrow L\Alg
    \]
  for the functor induced by $v$.

  \begin{prop}  \label{prop:uvid}
    The functor $v_*$ is a retraction of the functor $u_*$, i.e.
      \[
        u_*v_* = \id_{B\Alg}.
      \]
  \end{prop}

  \begin{proof}
    Recall that we have a functor $\OUC \rightarrow \OCS$.  Applying this functor to $v$ gives that $v$ is a map in $\OCS$, so the composite
      \[
        \xy
          % POINTS
          (-16, 14)*+{B}="Bl";
          (0, 14)*+{L}="L";
          (16, 14)*+{B}="Br";
          (0, 0)*+{T1}="T";
          % ARROWS
          {\ar^{u} "Bl" ; "L"};
          {\ar^{v} "L" ; "Br"};
          {\ar_-{b} "Bl" ; "T"};
          {\ar^-{l} "L" ; "T"};
          {\ar^-{b} "Br" ; "T"};
        \endxy
      \]
    is the identity $\id_B$, since $B$ is initial in $\OCS$.  Thus $u_*v_*$ is the functor induced by $vu = \id_B$, so $u_*v_* = \id_{B\Alg}$, as required.
  \end{proof}

  We now consider the composite
    \[
      \xy
        % POINTS
        (0, 0)*+{L\Alg}="0";
        (20, 0)*+{B\Alg}="1";
        (40, 0)*+{L\Alg.}="2";
        % ARROWS
        {\ar^{u_*} "0" ; "1"};
        {\ar^{v_*} "1" ; "2"};
      \endxy
    \]
  Note that $v_* u_*$ does not change the underlying $n$-globular set of an $L$-algebra, it only changes the algebra structure.  We describe a small example which illustrates the way in which the new algebra structure differs from the original one.  In Section~\ref{subsect:v*u*} we investigate the relationship between an $L$-algebra and its image under the functor $v_* u_*$ more fully, using this example to motivate a definition of weak map of $L$-algebras that we use in the general case.  Let $n \geq 2$ and let $A$ denote the $n$-globular set consisting of three composable $1$-cells:
    \[
      \xy
        % POINTS
        (0, 0)*+{\bullet}="0";
        (16, 0)*+{\bullet}="1";
        (32, 0)*+{\bullet}="2";
        (48, 0)*+{\bullet}="3";
        % ARROWS
        {\ar^{f} "0" ; "1"};
        {\ar^{g} "1" ; "2"};
        {\ar^{h} "2" ; "3"};
      \endxy
    \]
  We consider the free $L$-algebra on $A$, i.e.
    \[
      \xy
        % POINTS
        (0, 0)*+{L^2 A}="LLA";
        (0, -16)*+{LA.}="LA";
        % ARROWS
        {\ar^{\mu^L_A} "LLA" ; "LA"};
      \endxy
    \]
  This has:
    \begin{itemize}
      \item $0$-cells: the same as those of $A$;
      \item $1$-cells:
        \begin{itemize}
          \item generating cells $f$, $g$, $h$,
          \item binary composites $g \comp f$, $h \comp g$, $h \comp (g \comp f)$, $(h \comp g) \comp f$,
          \item a ternary composite $h \comp g \comp f$,
          \item identities and composites involving identities;
        \end{itemize}
      \item $2$-cells: for every pair of parallel $1$-cells $a$, $b \in LA_1$, a constraint cell which we write as
        \[
          [a, b] \colon a \Longrightarrow b.
        \]
        In particular, this includes constraint cells mediating between different composites of the same cells, e.g.
          \[
            [h \comp g \comp f , (h \comp g) \comp f],
          \]
          \[
            [(h \comp g) \comp f, h \comp (g \comp f)],
          \]
        etc.  We also have freely generated composites of these;
      \item $m$-cells for $m \geq 3$: constraint cells, and composites of constraint cells.
    \end{itemize}

    Applying $u_*v_*$ to this gives the $L$-algebra
    \[
      \xy
        % POINTS
        (0, 0)*+{L^2 A}="LLAt";
        (0, -10)*+{BLA}="BLA";
        (0, -20)*+{L^2 A}="LLA";
        (0, -30)*+{LA,}="LA";
        % ARROWS
        {\ar^{v_{LA}} "LLAt" ; "BLA"};
        {\ar^{u_{LA}} "BLA" ; "LLA"};
        {\ar^{\mu^L_A} "LLA" ; "LA"};
      \endxy
    \]
  which has the same underlying $n$-globular set as the free $L$-algebra on $A$, but has a different composition structure.  Write $\newcomp$ for the new composition operation on $1$-cells, which is defined as follows:
    \begin{itemize}
      \item binary composition remains the same, so we have
        \[
          g \newcomp f = g \comp f, \; h \newcomp g = h \comp g, \; h \newcomp (g \newcomp f) = h \comp (g \comp f),
        \]
      etc.;
      \item ternary composition is given by bracketing on the left, i.e.
        \[
          h \newcomp g \newcomp f = (h \comp g) \comp f.
        \]
    \end{itemize}

  Consider the diagram
    \[
      \xy
        % POINTS
        (0, 0)*+{L^2 A}="LLAt";
        (20, 0)*+{L^2 A}="LLAr";
        (0, -10)*+{BLA}="BLA";
        (0, -20)*+{L^2 A}="LLA";
        (0, -30)*+{LA}="LA";
        (20, -30)*+{LA}="LAr";
        % ARROWS
        {\ar^-{L\id_{LA}} "LLAt" ; "LLAr"};
        {\ar_{v_{LA}} "LLAt" ; "BLA"};
        {\ar^{\mu^L_A} "LLAr" ; "LAr"};
        {\ar_{u_{LA}} "BLA" ; "LLA"};
        {\ar_{\mu^L_A} "LLA" ; "LA"};
        {\ar_-{\id_{LA}} "LA" ; "LAr"};
      \endxy
    \]
  in $\nGSet$.  If this diagram commuted it would be a map of $L$-algebras, and since its underlying map of $n$-globular sets is an identity, this would show that the free $L$-algebra on $A$ is isomorphic to its image under $v_*u_*$.  In fact this diagram does not commute; although it does commute on the underlying $B$-algebra structure, i.e. it commutes on generating cells, binary composites, and constraint cells mediating between these, it does not commute on cells that only exist in the $L$-algebra structure, such as ternary composites.  Consider the freely generated ternary composite
    \[
      h \comp g \comp f \in L^2 A.
    \]
  We have
    \begin{itemize}
      \item $\theta \comp L\id_{LA} (h \comp g \comp f) = \theta(h \comp g \comp f) = h \comp g \comp f$;
      \item $\id_{LA} \comp \theta \comp u_{LA} \comp v_{LA}(h \comp g \comp f) = h \newcomp g \newcomp f = (h \comp g) \comp f$,
    \end{itemize}
  and
    \[
      h \comp g \comp f \neq (h \comp g) \comp f,
    \]
  so the diagram does not commute.  However, there is a constraint cell mediating between these two $1$-cells:
    \[
      [h \comp g \comp f, (h \comp g) \comp f] \colon h \comp g \comp f \Rightarrow (h \comp g) \comp f.
    \]
  Similarly, for any other cell in $L^2 A$ that is not part of the underlying $B$-algebra structure (such as non-binary composites of $1$-cells involving identities, and non-binary composites at higher dimensions) we also have a constraint cell mediating between its images under the maps $\theta \comp L\id_{LA}$ and $\id_{LA} \comp \theta \comp u_{LA} \comp v_{LA}$.  Thus, we can think of the diagram as ``commuting up to a constraint cell''.  By the definition of constraint cells as those induced by the contraction $L$, combined with the fact that all diagrams of constraint $n$-cells commute in a free $L$-algebra (Theorem~\ref{lem:diaginfree}), these constraint cells are equivalences in the $L$-algebra, and any diagram of them commutes up to a constraint cell at the dimension above, with strict commutativity for diagrams of constraint $n$-cells.  Thus these constraint cells are ``well-behaved enough'' to act as the mediating cells in a weak map; any commutativity conditions we would need to check are automatically satisfied by coherence for $L$-algebras.

  \subsection{Left adjoint to $u_*$}  \label{subsect:ladjtou}

  We now construct a functor
    \[
      F \colon B\Alg \longrightarrow L\Alg,
    \]
  and prove that this is left adjoint to the functor $u_*$.  Recall that $u_*$ is the functor induced by the unique map of operads with contractions and systems of compositions
    \[
      u \colon B \longrightarrow L
    \]
  induced by the universal property of $B$, the initial object in $\OCS$.  We can think of $u_*$ as a forgetful functor that sends an $L$-algebra to its underlying $B$-algebra by forgetting its unbiased composition structure, and remembering only the binary composition structure and the necessary constraint cells.  The left adjoint $F$ takes a $B$-algebra and freely adds an unbiased composition structure, along with all the required constraint cells to make an $L$-algebra, but retains the original binary composition structure (note that new binary composites are not added freely).

  It is a result of Blackwell--Kelly--Power \cite[Theorem 5.12]{BKP} that any functor induced by a map of monads has a left adjoint (their result is for $2$-monads, but can be applied to monads by considering them as a special case of $2$-monads).  Consequently, one may ask why the adjunction
    \[
          \xy
            % POINTS
            (0, 0)*+{B\Alg}="BAlg";
            (20, 0)*+{L\Alg}="LAlg";
            % ARROWS
            {\ar@<1ex>_-*!/u1pt/{\labelstyle \bot}^-F "BAlg" ; "LAlg"};
            {\ar@<1ex>^-{u_*} "LAlg" ; "BAlg"};
          \endxy
    \]
  should be considered significant, and, in particular, why it is more significant than the adjunction in which $v_*$ is the right adjoint.  There are two reasons for this.  First, $u \colon B \rightarrow L$ is canonical in the sense that it is the only such map of monads that preserves the contraction and system of compositions on $B$.  In contrast, $v \colon L \rightarrow B$ is not canonical; there is no canonical way of equipping $B$ with an unbiased contraction (Theorem~\ref{thm:OCStoOUC}), so $v$ depends on the choices we made when doing so.  Second, this adjunction formalises the idea that the key difference between $B$-algebras and $L$-algebras is that $B$-algebras have binary-biased composition whereas $L$-algebras have unbiased composition, and describes how to obtain an $L$-algebra from a $B$-algebra by adding unbiased composites, as well as the necessary constraint cells, freely.

  The construction of the left adjoint described in this section is valid in greater generality than just this case;  we can replace $\nGSet$ with any cocomplete category, $L$ with any finitary monad, $B$ with any other monad on the same category, and $u \colon B \rightarrow L$ with any map of monads.  We first explain the construction with reference to the specific case of a left adjoint to $u_*$, then state the construction in more generality.

  Note that the left adjoint we construct is not induced by a map of monads;  a functor $B\Alg \rightarrow L\Alg$
  induced by a map of monads $L \Rightarrow B$ would leave the underlying $n$-globular set of a $B$-algebra unchanged, but the left adjoint to $u_*$ freely adds unbiased composites (and various contraction cells) to obtain an $L$-algebra structure, rather than using cells already present in the original $B$-algebra.

  Let
      \[
        \xy
          % POINTS
          (0, 0)*+{BX}="BX";
          (0, -16)*+{X,}="X";
          % ARROWS
          {\ar^-{\theta} "BX" ; "X"};
        \endxy
      \]
  be a $B$-algebra; we will now construct an $n$-globular set $\bar{X}$, which will be the underlying $n$-globular set of the $L$-algebra obtained by applying $F$ to the $B$-algebra above.  First, we apply $L$ to $X$, which freely adds an $L$-algebra structure, while ignoring the existing $B$-algebra structure.  This free $L$-algebra structure has a free $B$-algebra structure inside it, which is picked out by the map
    \[
      u_X \colon BX \longrightarrow LX.
    \]
  We identify this free $B$-algebra structure with the original $B$-algebra structure on $X$ by taking the following pushout:
      \[
        \xy
          % POINTS
          (0, 0)*+{BX}="BX";
          (16, 0)*+{LX}="LX";
          (0, -16)*+{X}="X";
          (16, -16)*+{X^{(1)}.}="X1";
          % ARROWS
          {\ar^{u_A} "BX" ; "LX"};
          {\ar_-{\theta} "BX" ; "X"};
          {\ar^{\phi^{(1)}} "LX" ; "X1"};
          {\ar_-{x^{(1)}} "X" ; "X1"};
          % PUSHOUT STUFF
          (10,-14)*{}; (10,-10)*{} **\dir{-};
          (14,-10)*{}; (10,-10)*{} **\dir{-};
        \endxy
      \]
  Taking this pushout identifies any cell in the free $B$-algebra structure inside $LX$ with the corresponding cell in the original $B$-algebra on $X$.  So, for example, any free binary composite in $LX$ is identified with the binary composite of the same cells, as evaluated by $\theta$, in $X$.

  However, this is not the end of the construction for two reasons:  first, in principle the act of identifying cells causes more cells to share common boundaries, thus making more cells composable; second, taking this pushout does nothing to cells in $LX$ that involve both the $B$-algebra and non-$B$-algebra structure.  Such cells include non-binary composites of binary composites; for example, suppose we have a string of four composable $1$-cells
    \[
      \xy
        % POINTS
        (0, 0)*+{\bullet}="0";
        (16, 0)*+{\bullet}="1";
        (32, 0)*+{\bullet}="2";
        (48, 0)*+{\bullet}="3";
        (64, 0)*+{\bullet}="4";
        % ARROWS
        {\ar^a "0" ; "1"};
        {\ar^b "1" ; "2"};
        {\ar^c "2" ; "3"};
        {\ar^d "3" ; "4"};
      \endxy
    \]
  in $X$.  In $X^{(1)}$, we have distinct cells
    \[
      d \comp c \comp (b \comp a) \neq d \comp c \comp \theta(b \comp a),
    \]
  but in the $L$-algebra we are constructing we want these cells to be equal.

  To rectify these problems we apply $L$ to $X^{(1)}$, thus freely adding composites of the newly composable cells, then identify the free $L$-algebra structure on $LX^{(1)}$ with the partial $L$-algebra structure on $X^{(1)}$ given by $\phi^{(1)} \colon LX \rightarrow X^{(1)}$ by taking the pushout
      \[
        \xy
          % POINTS
          (0, 0)*+{LX}="LX";
          (0, -16)*+{X^{(1)}}="X1";
          (0, 16)*+{L^2X}="LLX";
          (20, 16)*+{LX^{(1)}}="LX1";
          (20, -16)*+{X^{(2)}.}="X2";
          % ARROWS
          {\ar_{\phi^{(1)}} "LX" ; "X1"};
          {\ar_{\mu^L_X} "LLX" ; "LX"};
          {\ar^-{L\phi^{(1)}} "LLX" ; "LX1"};
          {\ar^{\phi^{(2)}} "LX1" ; "X2"};
          {\ar_-{x^{(2)}} "X1" ; "X2"};
          % PUSHOUT STUFF
          (14,-14)*{}; (14,-10)*{} **\dir{-};
          (18,-10)*{}; (14,-10)*{} **\dir{-};
        \endxy
      \]
  Once again, the act of identifying cells causes more cells to become composable.  Also, although in $X^{(1)}$ we now have the desired equalities between non-binary composites involving binary composites such as
    \[
      d \comp c \comp (b \comp a) = d \comp c \comp \theta(b \comp a),
    \]
  this is not true for composites whose binary parts appear at greater ``depths'', such as non-binary composites of non-binary composites of binary composites.  We thus must repeat the procedure above indefinitely to obtain the following sequence of pushouts in $\nGSet$:
      \[
        \xy
          % POINTS
          (0, 0)*+{BX}="BX";
          (20, 0)*+{LX}="LX";
          (0, -16)*+{X = X^{(0)}}="X";
          (20, -16)*+{X^{(1)}}="X1";
          (20, 16)*+{L^2X}="LLX";
          (40, 16)*+{LX^{(1)}}="LX1";
          (40, -16)*+{X^{(2)}}="X2";
          (40, 32)*+{L^2X^{(1)}}="LLX1";
          (60, 32)*+{LX^{(2)}}="LX2";
          (60, -16)*+{X^{(3)}}="X3";
          (70, 32)*+{\dotsb};
          (70, -16)*+{\dotsb};
          % ARROWS
          {\ar^{u_A} "BX" ; "LX"};
          {\ar_-{\phi^{(0)} = \theta} "BX" ; "X"};
          {\ar^{\phi^{(1)}} "LX" ; "X1"};
          {\ar_-{x^{(1)}} "X" ; "X1"};
          {\ar_{\mu^L_X} "LLX" ; "LX"};
          {\ar^-{L\phi^{(1)}} "LLX" ; "LX1"};
          {\ar^{\phi^{(2)}} "LX1" ; "X2"};
          {\ar_-{x^{(2)}} "X1" ; "X2"};
          {\ar_{\mu^L_{X^{(1)}}} "LLX1" ; "LX1"};
          {\ar^-{L\phi^{(2)}} "LLX1" ; "LX2"};
          {\ar^{\phi^{(3)}} "LX2" ; "X3"};
          {\ar_-{x^{(3)}} "X2" ; "X3"};
          % PUSHOUT STUFF
          (14,-14)*{}; (14,-10)*{} **\dir{-};
          (18,-10)*{}; (14,-10)*{} **\dir{-};
          (34,-14)*{}; (34,-10)*{} **\dir{-};
          (38,-10)*{}; (34,-10)*{} **\dir{-};
          (54,-14)*{}; (54,-10)*{} **\dir{-};
          (58,-10)*{}; (54,-10)*{} **\dir{-};
        \endxy
      \]
  The bottom row of this diagram is a sequence of $n$-globular sets
      \[
        \{ X^{(i)} \}_{i \geq 0}.
      \]
  We define $\bar{X}$ to be given by
      \[
        \bar{X} := \colim_{i \geq 0} X^{(i)},
      \]

  We now describe the construction in general.  Throughout the rest of this section, let $\mathcal{C}$ denote a cocomplete category, let $R$ and $S$ be monads on $\mathcal{C}$ with $S$ finitary (i.e. the functor part of $S$ preserves filtered colimits), and let $p \colon R \rightarrow S$ be a map of monads.  The map $p$ induces a functor
    \[
      p_* \colon S\Alg \longrightarrow R\Alg,
    \]
  and we will construct a left adjoint $F$ to $p_*$.  Let
      \[
        \xy
          % POINTS
          (0, 0)*+{RX}="RX";
          (0, -16)*+{X,}="X";
          % ARROWS
          {\ar^-{\theta} "RX" ; "X"};
        \endxy
      \]
  be an $R$-algebra.  We define a sequence
      \[
        \{ X^{(i)} \}_{i \geq 0}.
      \]
  of objects in $\mathcal{C}$ by the following sequence of pushouts in $\mathcal{C}$:
      \[
        \xy
          % POINTS
          (0, 0)*+{RX}="RX";
          (20, 0)*+{SX}="SX";
          (0, -16)*+{X = X^{(0)}}="X";
          (20, -16)*+{X^{(1)}}="X1";
          (20, 16)*+{S^2X}="SSX";
          (40, 16)*+{SX^{(1)}}="SX1";
          (40, -16)*+{X^{(2)}}="X2";
          (40, 32)*+{S^2X^{(1)}}="SSX1";
          (60, 32)*+{SX^{(2)}}="SX2";
          (60, -16)*+{X^{(3)}}="X3";
          (70, 32)*+{\dotsb};
          (70, -16)*+{\dotsb};
          % ARROWS
          {\ar^{u_A} "RX" ; "SX"};
          {\ar_-{\phi^{(0)} = \theta} "RX" ; "X"};
          {\ar^{\phi^{(1)}} "SX" ; "X1"};
          {\ar_-{x^{(1)}} "X" ; "X1"};
          {\ar_{\mu^S_X} "SSX" ; "SX"};
          {\ar^-{S\phi^{(1)}} "SSX" ; "SX1"};
          {\ar^{\phi^{(2)}} "SX1" ; "X2"};
          {\ar_-{x^{(2)}} "X1" ; "X2"};
          {\ar_{\mu^S_{X^{(1)}}} "SSX1" ; "SX1"};
          {\ar^-{S\phi^{(2)}} "SSX1" ; "SX2"};
          {\ar^{\phi^{(3)}} "SX2" ; "X3"};
          {\ar_-{x^{(3)}} "X2" ; "X3"};
          % PUSHOUT STUFF
          (14,-14)*{}; (14,-10)*{} **\dir{-};
          (18,-10)*{}; (14,-10)*{} **\dir{-};
          (34,-14)*{}; (34,-10)*{} **\dir{-};
          (38,-10)*{}; (34,-10)*{} **\dir{-};
          (54,-14)*{}; (54,-10)*{} **\dir{-};
          (58,-10)*{}; (54,-10)*{} **\dir{-};
        \endxy
      \]
  We then define an object $\bar{X}$ of $\mathcal{C}$ by
      \[
        \bar{X} := \colim_{i \geq 0} X^{(i)}.
      \]
  This will be the underlying object of $\mathcal{C}$ of the $S$-algebra obtained by applying the functor $F$ to the $R$-algebra $\theta \colon RX \rightarrow X$.

  We now equip $\bar{X}$ with an $S$-algebra action
    \[
      \phi \colon S\bar{X} \longrightarrow \bar{X}.
    \]
  Since $S$ is finitary, we can write $S\bar{X}$ as
      \[
        S\bar{X} = \colim_{i \geq 0} SX^{(i)}.
      \]
  We wish to use the universal property of this colimit to define the $S$-algebra action $\phi$.  To do so, we now describe the cocone that induces $\phi$, and prove that it commutes.

  \begin{lemma}
    There is a cocone under the diagram
      \[
        \{ SX^{(i)} \}_{i \geq 0}
      \]
    with vertex $\bar{X}$, given by
      \[
        \xy
          % POINTS
          (0, 0)*+{SX^{(0)}}="SX0";
          (20, 0)*+{SX^{(1)}}="SX1";
          (40, 0)*+{SX^{(2)}}="SX2";
          (60, 0)*+{\dotsb}="dots";
          (0, -16)*+{X^{(1)}}="X1";
          (20, -16)*+{X^{(2)}}="X2";
          (40, -16)*+{X^{(3)}}="X3";
          (30, -32)*+{\bar{X}}="barX";
          % ARROWS
          {\ar^-{Sx^{(1)}} "SX0" ; "SX1"};
          {\ar_{\phi^{(1)}} "SX0" ; "X1"};
          {\ar^-{Sx^{(2)}} "SX1" ; "SX2"};
          {\ar_{\phi^{(2)}} "SX1" ; "X2"};
          {\ar^-{Sx^{(3)}} "SX2" ; "dots"};
          {\ar_{\phi^{(3)}} "SX2" ; "X3"};
          {\ar@/_1pc/_-{c^{(1)}} "X1" ; "barX"};
          {\ar_-{c^{(2)}} "X2" ; "barX"};
          {\ar^-{c^{(3)}} "X3" ; "barX"};
        \endxy
      \]
  \end{lemma}

  \begin{proof}
    We must show that, for each $i \geq 0$, the diagram
      \[
        \xy
          % POINTS
          (20, 0)*+{SX^{(i)}}="SX1";
          (44, 0)*+{SX^{(i + 1)}}="SX2";
          (20, -16)*+{X^{(i + 1)}}="X2";
          (44, -16)*+{X^{(i + 2)}}="X3";
          (32, -32)*+{\bar{X}}="barX";
          % ARROWS
          {\ar^-{Sx^{(i + 1)}} "SX1" ; "SX2"};
          {\ar_{\phi^{(i + 1)}} "SX1" ; "X2"};
          {\ar^{\phi^{(i + 2)}} "SX2" ; "X3"};
          {\ar_-{c^{(i + 1)}} "X2" ; "barX"};
          {\ar^-{c^{(i + 2)}} "X3" ; "barX"};
        \endxy
      \]
    commutes.  We can write this diagram as
      \[
        \xy
          % POINTS
          (0, 0)*+{SX^{(i)}}="SXit";
          (0, -16)*+{S^2X^{(i)}}="SSXi";
          (24, -16)*+{SX^{(i + 1)}}="SXi1";
          (0, -32)*+{SX^{(i)}}="SXi";
          (0, -48)*+{X^{(i + 1)}}="Xi1";
          (24, -48)*+{X^{(i + 2)}}="Xi2";
          (24, -64)*+{\bar{X}.}="barX";
          % ARROWS
          {\ar_{S\eta^S_{X^{(i)}}} "SXit" ; "SSXi"};
          {\ar^{Sx^{(i + 1)}} "SXit" ; "SXi1"};
          {\ar@/_4pc/_{\id_{SX^{(i)}}} "SXit" ; "SXi"};
          {\ar_{S\phi^{(i + 1)}} "SSXi" ; "SXi1"};
          {\ar_{\mu^S_{X^{(i)}}} "SSXi" ; "SXi"};
          {\ar^{\phi^{(i + 2)}} "SXi1" ; "Xi2"};
          {\ar_{\phi^{(i + 1)}} "SXi" ; "Xi1"};
          {\ar_{x^{(i + 2)}} "Xi1" ; "Xi2"};
          {\ar_{c^{(i + 1)}} "Xi1" ; "barX"};
          {\ar^{c^{(i + 2)}} "Xi2" ; "barX"};
        \endxy
      \]
    The rectangle commutes since it is the pushout square defining $X^{(i + 2)}$, the top-left triangle commutes by the unit axiom for the monad $S$, and the bottom triangle commutes by definition of $\bar{X}$; thus we need only check that the top-right triangle commutes.  We do so by showing that, for all $i \geq 0$, the diagram
     \[
       \xy
         % POINTS
         (0, 0)*+{X^{(i)}}="Xi";
         (0, -16)*+{SX^{(i)}}="SXi";
         (20, -16)*+{X^{(i + 1)}}="Xi1";
         % ARROWS
         {\ar_{\eta^S_{X^{(i)}}} "Xi" ; "SXi"};
         {\ar^{x^{(i + 1)}} "Xi" ; "Xi1"};
         {\ar_-{\phi^{(i + 1)}} "SXi" ; "Xi1"};
       \endxy
     \]
    commutes, then applying $S$ to this diagram.

    When $i = 0$, the diagram above can be written as
     \[
       \xy
         % POINTS
         (0, 16)*+{X}="Xt";
         (0, 0)*+{RX}="RX";
         (20, 0)*+{SX}="SX";
         (0, -16)*+{X}="X";
         (20, -16)*+{X^{(1)}.}="X1";
         % ARROWS
         {\ar_{\eta^R_X} "Xt" ; "RX"};
         {\ar@/_2pc/_{\id_X} "Xt" ; "X"};
         {\ar^{\eta^S_X} "Xt" ; "SX"};
         {\ar^{p_X} "RX" ; "SX"};
         {\ar_{\theta} "RX" ; "X"};
         {\ar^{\phi^{(1)}} "SX" ; "X1"};
         {\ar_{x^{(1)}} "X" ; "X1"};
       \endxy
     \]
    The square commutes since it is the pushout square defining $X^{(1)}$, the left-hand triangle commutes by the unit axiom for $\theta$, and the top-right triangle commutes by the unit axiom for the monad map $p$.  Thus this diagram commutes.

    Now let $i \geq 1$.  The diagram
      \[
        \xy
          % POINTS
          (0, 0)*+{SX^{(i - 1)}}="SXim1t";
          (20, 0)*+{X^{(i)}}="Xit";
          (0, -16)*+{S^2X^{(i - 1)}}="SSXim1";
          (20, -16)*+{SX^{(i)}}="SXi";
          (0, -32)*+{SX^{(i - 1)}}="SXim1";
          (0, -48)*+{X^{(i)}}="Xi";
          (20, -48)*+{X^{(i + 1)}}="Xi1";
          % ARROWS
          {\ar_{\eta^S_{SX^{(i)}}} "SXim1t" ; "SSXim1"};
          {\ar^-{\phi^{(i)}} "SXim1t" ; "Xit"};
          {\ar^{\eta^S_{X^{(i)}}} "Xit" ; "SXi"};
          {\ar@/_4pc/_{\id_{SX^{(i - 1)}}} "SXim1t" ; "SXim1"};
          {\ar_{S\phi^{(i)}} "SSXim1" ; "SXi"};
          {\ar_{\mu^S_{X^{(i - 1)}}} "SSXim1" ; "SXim1"};
          {\ar^{\phi^{(i + 1)}} "SXi" ; "Xi1"};
          {\ar_{\phi^{(i)}} "SXim1" ; "Xi"};
          {\ar_-{x^{(i + 1)}} "Xi" ; "Xi1"};
        \endxy
      \]
   commutes; the bottom rectangle commutes since it is the pushout square defining $X^{(i + 1)}$, the top square commutes since it is a naturality square for $\eta^S$, and the top-left triangle commutes by the unit axiom for the monad $S$. We wish to cancel the $\phi^{(i)}$'s in the diagram above, in order to obtain the desired triangle.  We can do this if $\phi^{(i)}$ is an epimorphism; we now show that this is true by induction over $i$.

   To show that this is true when $i = 1$, observe that $\eta^R_X$ is a section to $\theta$, so $\theta$ is epic; since the pushout of an epimorphism is also an epimorphism \cite[Proposition 2.5.3]{Bor94a}, we have that $\phi^{(1)}$ is epic.

   Now suppose $i > 1$ and that we have shown that $\phi^{(i - 1)}$ is epic.  By the unit axiom for the monad $S$, $\eta^S_{X^{(i - 2)}}$ is a section to $\mu^S_{X^{(i - 2)}}$, so $\mu^S_{X^{(i - 2)}}$ is epic.  Hence the composite
     \[
       \phi^{i - 1} \comp \mu^S_{X^{(i - 2)}}
     \]
   is epic; since the pushout of an epimorphism is also an epimorphism, we have that $\phi^{(i)}$ is epic.

   Hence, for each $i \geq 0$, the diagram
     \[
       \xy
         % POINTS
         (0, 0)*+{X^{(i)}}="Xi";
         (0, -16)*+{SX^{(i)}}="SXi";
         (20, -16)*+{X^{(i + 1)}}="Xi1";
         % ARROWS
         {\ar_{\eta^S_{X^{(i)}}} "Xi" ; "SXi"};
         {\ar^{x^{(i + 1)}} "Xi" ; "Xi1"};
         {\ar_-{\phi^{(i + 1)}} "SXi" ; "Xi1"};
       \endxy
     \]
   commutes, and thus the diagram
      \[
        \xy
          % POINTS
          (20, 0)*+{SX^{(i)}}="SX1";
          (44, 0)*+{SX^{(i + 1)}}="SX2";
          (20, -16)*+{X^{(i + 1)}}="X2";
          (44, -16)*+{X^{(i + 2)}}="X3";
          (32, -32)*+{\bar{X}}="barX";
          % ARROWS
          {\ar^-{Sx^{(i + 1)}} "SX1" ; "SX2"};
          {\ar_{\phi^{(i + 1)}} "SX1" ; "X2"};
          {\ar^{\phi^{(i + 2)}} "SX2" ; "X3"};
          {\ar_-{c^{(i + 1)}} "X2" ; "barX"};
          {\ar^-{c^{(i + 2)}} "X3" ; "barX"};
        \endxy
      \]
   commutes, as required.
  \end{proof}

   We now define $\phi \colon S\bar{X} \rightarrow \bar{X}$ to be the unique map induced by the universal property of $S\bar{X}$ such that, for all $i \geq 0$, the diagram
     \[
       \xy
         % POINTS
         (0, 0)*+{SX^{(i)}}="SXi";
         (20, 0)*+{S\bar{X}}="SbarX";
         (0, -16)*+{X^{(i + 1)}}="Xi1";
         (20, -16)*+{\bar{X}}="barX";
         % ARROWS
         {\ar^{Sc^{(i)}} "SXi" ; "SbarX"};
         {\ar_{\phi^{(i + 1)}} "SXi" ; "Xi1"};
         {\ar@{-->}^{\phi} "SbarX" ; "barX"};
         {\ar_{c^{(i + 1)}} "Xi1" ; "barX"};
       \endxy
     \]
   commutes.  To check that
      \[
        \xy
          % POINTS
          (0, 0)*+{S\bar{X}}="SbarX";
          (0, -16)*+{\bar{X}}="barX";
          % ARROWS
          {\ar^-{\phi} "SbarX" ; "barX"};
        \endxy
      \]
   is an $S$-algebra we must show that it satisfies the $S$-algebra axioms.

  \begin{lemma}
    The map $\phi \colon S\bar{X} \rightarrow \bar{X}$ satisfies the $S$-algebra axioms.
  \end{lemma}

  \begin{proof}
   For the unit axiom, we must check that the diagram
      \[
        \xy
          % POINTS
          (-16, 0)*+{\bar{X}}="barXl";
          (0, 0)*+{S\bar{X}}="SbarX";
          (0, -16)*+{\bar{X}}="barX";
          % ARROWS
          {\ar^{\eta^S_{\bar{X}}} "barXl" ; "SbarX"};
          {\ar_{\id_{\bar{X}}} "barXl" ; "barX"};
          {\ar^-{\phi} "SbarX" ; "barX"};
        \endxy
      \]
   commutes.  Since $\bar{X}$ is defined as a colimit, we check this by comparing the cocones corresponding to the maps on either side of the diagram.  The cocone corresponding to $\phi \comp \eta^S_{\bar{X}}$ is given by, for each $i \geq 0$, the composite
     \[
       \xy
         % POINTS
         (0, 0)*+{X^{(i)}}="Xi";
         (20, 0)*+{SX^{(i)}}="SXi";
         (42, 0)*+{X^{(i + 1)}}="Xi1";
         (60, 0)*+{\bar{X}.}="barX";
         % ARROWS
         {\ar^-{\eta^S_{X^{(i)}}} "Xi" ; "SXi"};
         {\ar^-{\phi^{(i + 1)}} "SXi" ; "Xi1"};
         {\ar^-{c^{(i + 1)}} "Xi1" ; "barX"};
       \endxy
     \]
   The cocone corresponding to $\id_{\bar{X}}$ is the universal cocone given by the coprojections $c^{(i)}$.  The diagram
     \[
       \xy
         % POINTS
         (0, 0)*+{X^{(i)}}="Xi";
         (20, 0)*+{SX^{(i)}}="SXi";
         (20, -16)*+{X^{(i + 1)}}="Xi1";
         (20, -32)*+{\bar{X}.}="barX";
         % ARROWS
         {\ar^-{\eta^S_{X^{(i)}}} "Xi" ; "SXi"};
         {\ar^{x^{(i)}} "Xi" ; "Xi1"};
         {\ar_{c^{(i)}} "Xi" ; "barX"};
         {\ar^-{\phi^{(i + 1)}} "SXi" ; "Xi1"};
         {\ar^-{c^{(i + 1)}} "Xi1" ; "barX"};
       \endxy
     \]
   commutes, so these cocones are equal.  Thus the unit axiom is satisfied.

   For the associativity axiom, we must check that the diagram
     \[
       \xy
         % POINTS
         (0, 0)*+{S^2\bar{X}}="SSX";
         (16, 0)*+{S\bar{X}}="SXr";
         (0, -16)*+{S\bar{X}}="SXl";
         (16, -16)*+{\bar{X}}="X";
         % ARROWS
         {\ar^{S\phi} "SSX" ; "SXr"};
         {\ar_{\mu^S_{\bar{X}}} "SSX" ; "SXl"};
         {\ar^{\phi} "SXr" ; "X"};
         {\ar_{\phi} "SXl" ; "X"};
       \endxy
     \]
   commutes.  Since $S$ is finitary, we have
     \[
       S^2 X = \colim_{i \geq 0} S^2 X^{(i)}.
     \]
   Thus we can check that the diagram commutes by comparing the cocones corresponding to the maps on either side of the diagram.  The cocone corresponding to $\phi \comp S \phi$ is given by, for each $i \geq 0$, the composite
     \[
       \xy
         % POINTS
         (0, 0)*+{S^2 X^{(i)}}="SSXi";
         (22, 0)*+{S X^{(i + 1)}}="SXi1";
         (44, 0)*+{X^{(i + 2)}}="Xi2";
         (64, 0)*+{\bar{X}}="barX";
         % ARROWS
         {\ar^-{S\phi^{(i + 1)}} "SSXi" ; "SXi1"};
         {\ar^-{\phi^{(i + 2)}} "SXi1" ; "Xi2"};
         {\ar^-{c^{(i + 2)}} "Xi2" ; "barX"};
       \endxy
     \]
   The cocone corresponding to $\phi \comp \mu^S_{\bar{X}}$ is given by, for each $i \geq 0$, the composite
     \[
       \xy
         % POINTS
         (0, 0)*+{S^2 X^{(i)}}="SSXi";
         (22, 0)*+{S X^{(i)}}="SXi";
         (44, 0)*+{X^{(i + 1)}}="Xi1";
         (64, 0)*+{X^{(i + 2)}}="Xi2";
         (84, 0)*+{\bar{X}.}="barX";
         % ARROWS
         {\ar^-{\mu^S_{X^{(i)}}} "SSXi" ; "SXi"};
         {\ar^-{\phi^{(i + 1)}} "SXi" ; "Xi1"};
         {\ar^-{x^{(i + 2)}} "Xi1" ; "Xi2"};
         {\ar^-{c^{(i + 2)}} "Xi2" ; "barX"};
       \endxy
     \]
   From the definition of $\bar{X}$, for all $i \geq 0$, the diagram
      \[
        \xy
          % POINTS
          (0, 0)*+{S^2X^{(i)}}="SSXim1";
          (20, 0)*+{SX^{(i + 1)}}="SXi";
          (0, -16)*+{SX^{(i)}}="SXim1";
          (0, -32)*+{X^{(i + 1)}}="Xi";
          (20, -32)*+{X^{(i + 2)}}="Xi1";
          (20, -48)*+{\bar{X}}="barX";
          % ARROWS
          {\ar^-{S\phi^{(i + 1)}} "SSXim1" ; "SXi"};
          {\ar_{\mu^S_{X^{(i)}}} "SSXim1" ; "SXim1"};
          {\ar^{\phi^{(i + 2)}} "SXi" ; "Xi1"};
          {\ar_{\phi^{(i + 1)}} "SXim1" ; "Xi"};
          {\ar_-{x^{(i + 2)}} "Xi" ; "Xi1"};
          {\ar_{c^{(i + 1)}} "Xi" ; "barX"};
          {\ar^{c^{(i + 2)}} "Xi1" ; "barX"};
        \endxy
      \]
   commutes, so these cocones are equal.  Thus the associativity axiom is satisfied.

   Hence
      \[
        \xy
          % POINTS
          (0, 0)*+{S\bar{X}}="SbarX";
          (0, -16)*+{\bar{X}}="barX";
          % ARROWS
          {\ar^-{\phi} "SbarX" ; "barX"};
        \endxy
      \]
   is an $S$-algebra.
  \end{proof}

  This gives us the action of the left adjoint to $p_*$ on objects.  To prove that this gives a left adjoint, we use the following result of Mac Lane~\cite[Theorem~IV.1.2]{Mac98}, which allows us to avoid describing the action of the left adjoint on morphisms.

  \begin{lemma} \label{lem:adjuniarrow}
    Given a functor $U \colon \mathcal{D} \longrightarrow \mathcal{C}$, an adjunction
        \[
          \xy
            % POINTS
            (0, 0)*+{\mathcal{C}}="C";
            (16, 0)*+{\mathcal{D}}="D";
            % ARROWS
            {\ar@<1ex>^-{F}_-*!/u1pt/{\labelstyle \bot} "C" ; "D"};
            {\ar@<1ex>^-{U} "D" ; "C"};
          \endxy
       \]
    is completely determined by, for all objects $x$ in $\mathcal{C}$, an object $F_0(x)$ in $\mathcal{D}$ and a universal arrow $\eta_x \colon x \rightarrow UF_0(x)$ from $x$ to $U$.
  \end{lemma}

  As is suggested by the notation, here the assignment $F_0(x)$ gives the action of the left adjoint $F$ on objects, and the maps $\eta_x$ are the components of the unit of the adjunction.

  \begin{prop}
    There is an adjunction $F \ladj p_*$.
  \end{prop}

  \begin{proof}
    As discussed above, we prove this using Lemma~\ref{lem:adjuniarrow}, thus allowing us to avoid constructing the action of $F$ on morphisms.  Let
      \[
        \xy
          % POINTS
          (0, 0)*+{RX}="RX";
          (0, -16)*+{X}="X";
          % ARROWS
          {\ar^-{\theta} "RX" ; "X"};
        \endxy
      \]
    be an $R$-algebra.  By the construction described earlier we have a corresponding $S$-algebra
      \[
        \xy
          % POINTS
          (0, 0)*+{S\bar{X}}="SbarX";
          (0, -16)*+{\bar{X}.}="barX";
          % ARROWS
          {\ar^-{\phi} "SbarX" ; "barX"};
        \endxy
      \]
    To show that this gives the action on objects of the left adjoint $F \colon R\Alg \rightarrow S\Alg$ to $p_*$, we require a map of $R$-algebras
      \[
        \eta_X \colon X \longrightarrow \bar{X}
      \]
    which is a universal arrow from $X$ to $p_*$.  This is given by the coprojection map $c^{(0)} \colon X \rightarrow \bar{X}$.  This is indeed a map of $R$-algebras, since the diagram
      \[
        \xy
          % POINTS
          (0, 0)*+{RX}="RX";
          (32, 0)*+{R\bar{X}}="RbarX";
          (16, -16)*+{SX}="SX";
          (32, -16)*+{S\bar{X}}="SbarX";
          (0, -32)*+{X}="X";
          (16, -32)*+{X^{(1)}}="X1";
          (32, -32)*+{\bar{X}}="barX";
          % ARROWS
          {\ar^{Rc^{(0)}} "RX" ; "RbarX"};
          {\ar^{p_X} "RX" ; "SX"};
          {\ar_{\theta} "RX" ; "X"};
          {\ar^{p_{\bar{X}}} "RbarX" ; "SbarX"};
          {\ar^{Sc^{(0)}} "SX" ; "SbarX"};
          {\ar_{\phi^{(1)}} "SX" ; "X1"};
          {\ar^{\phi} "SbarX" ; "barX"};
          {\ar_{x^{(1)}} "X" ; "X1"};
          {\ar_{c^{(1)}} "X1" ; "barX"};
          {\ar@/_2pc/_{c^{(0)}} "X" ; "barX"};
        \endxy
      \]
    commutes.  We now show universality.  Suppose we have an $S$-algebra
      \[
        \xy
          % POINTS
          (0, 0)*+{SY}="SY";
          (0, -16)*+{Y}="Y";
          % ARROWS
          {\ar^-{\psi} "SY" ; "Y"};
        \endxy
      \]
    and a map of $R$-algebras
      \[
        \xy
          % POINTS
          (0, 0)*+{RX}="RX";
          (16, 0)*+{RY}="RY";
          (16, -10)*+{SY}="SY";
          (0, -20)*+{X}="X";
          (16, -20)*+{Y.}="Y";
          % ARROWS
          {\ar^{Rf} "RX" ; "RY"};
          {\ar_-{\theta} "RX" ; "X"};
          {\ar^-{p_Y} "RY" ; "SY"};
          {\ar^-{\psi} "SY" ; "Y"};
          {\ar_{f} "X" ; "Y"};
        \endxy
      \]
    We seek a unique map of $S$-algebras
      \[
        \xy
          % POINTS
          (0, 0)*+{S\bar{X}}="SX";
          (16, 0)*+{SY}="SY";
          (0, -16)*+{\bar{X}}="X";
          (16, -16)*+{Y}="Y";
          % ARROWS
          {\ar^{S\bar{f}} "SX" ; "SY"};
          {\ar_-{\phi} "SX" ; "X"};
          {\ar^-{\psi} "SY" ; "Y"};
          {\ar_{\bar{f}} "X" ; "Y"};
        \endxy
      \]
    such that the diagram
      \[
        \xy
          % POINTS
          (0, 0)*+{X}="X";
          (16, 0)*+{\bar{X}}="barX";
          (16, -16)*+{Y}="Y";
          % ARROWS
          {\ar^-{c^{(0)}} "X" ; "barX"};
          {\ar_-{f} "X" ; "Y"};
          {\ar^-{\bar{f}} "barX" ; "Y"};
        \endxy
      \]
    commutes.  We define $\bar{f}$ by defining a cocone
    \[
      \bar{f}^{(i)} \colon X^{(i)} \longrightarrow \bar{Y}
    \]
  by induction over $i$.

  When $i = 0$, $X^{(i)} = X$, and we define $\bar{f}^{(i)} = \bar{f}^{(0)}$ to be given by $f \colon X \rightarrow Y$.

  When $i = 1$, $\bar{f}^{(i)} = \bar{f}^{(1)}$ is the unique map such that the diagram
          \[
            \xy
              % POINTS
              (0, 0)*+{RX}="RX";
              (16, 0)*+{SX}="SXr";
              (0, -16)*+{X}="X";
              (16, -16)*+{X^{(1)}}="X1";
              (32, -10)*+{SY}="SY";
              (32, -26)*+{Y}="Y";
              % ARROWS
              {\ar^{p_X} "RX" ; "SXr"};
              {\ar^{\phi^{(1)}} "SXr" ; "X1"};
              {\ar^{Sf} "SXr" ; "SY"};
              {\ar^{\psi} "SY" ; "Y"};
              {\ar_{\theta} "RX" ; "X"};
              {\ar_{x^{(1)}} "X" ; "X1"};
              {\ar@/_1.5pc/_{f} "X" ; "Y"};
              {\ar@{-->}^{\epsilon^{(1)}} "X1" ; "Y"};
              % PUSHOUT STUFF
              (10,-14)*{}; (10,-10)*{} **\dir{-};
              (14,-10)*{}; (10,-10)*{} **\dir{-};
            \endxy
          \]
  commutes.  To check that this is well-defined we must check that the outside of this diagram commutes; this is true, since
          \[
            \xy
              % POINTS
              (0, 0)*+{RX}="RX";
              (16, 0)*+{SX}="SXr";
              (16, -10)*+{RY}="RY";
              (0, -16)*+{X}="X";
              (32, -10)*+{SY}="SY";
              (32, -26)*+{Y}="Y";
              % ARROWS
              {\ar^{p_X} "RX" ; "SXr"};
              {\ar_-{Rf} "RX" ; "RY"};
              {\ar_{p_Y} "RY" ; "SY"};
              {\ar^{Sf} "SXr" ; "SY"};
              {\ar^{\psi} "SY" ; "Y"};
              {\ar_{\theta} "RX" ; "X"};
              {\ar@/_1.5pc/_{f} "X" ; "Y"};
            \endxy
          \]
  commutes.

  Now let $i > 1$ and suppose that we have defined $\bar{f}^{(i - 1)} \colon Y^{(i - 1)} \rightarrow \bar{Y}$.  We define $\bar{f}^{(i)}$ to be the unique map induced by the universal property of the pushout $X^{(i)}$ such that the diagram
          \[
            \xy
              % POINTS
              (0, 0)*+{S^2X^{(i - 2)}}="SSX";
              (24, 0)*+{SX^{(i - 1)}}="SXr";
              (0, -16)*+{SX^{(i - 2)}}="SXl";
              (40, -16)*+{SY}="SY";
              (0, -32)*+{X^{(i - 1)}}="X";
              (24, -32)*+{X^{(i)}}="X1";
              (40, -48)*+{Y}="Y";
              % ARROWS
              {\ar^{S\phi^{(i - 1)}} "SSX" ; "SXr"};
              {\ar_{\mu^S_{X^{(i - 2)}}} "SSX" ; "SXl"};
              {\ar^{\phi^{(i)}} "SXr" ; "X1"};
              {\ar^{S\bar{f}^{(i - 1)}} "SXr" ; "SY"};
              {\ar_{\phi^{(i - 1)}} "SXl" ; "X"};
              {\ar^{\psi} "SY" ; "Y"};
              {\ar_{x^{(i)}} "X" ; "X1"};
              {\ar@/_1.5pc/_{\bar{f}^{(i - 1)}} "X" ; "Y"};
              {\ar@{-->}^-{\bar{f}^{(i)}} "X1" ; "Y"};
              % PUSHOUT STUFF
              (18,-30)*{}; (18,-26)*{} **\dir{-};
              (22,-26)*{}; (18,-26)*{} **\dir{-};
            \endxy
          \]
  commutes.  Note that the fact that the bottom triangle in this diagram commutes gives us commutativity of the cocone.  To check that this is well-defined we must check that the outside of this diagram commutes; this is true, since
      \[
        \xy
          % POINTS
          (0, 0)*+{S^2 X^{(i - 2)}}="SSXim2";
          (24, 0)*+{SX^{(i - 1)}}="SXim1";
          (44, 0)*+{SY}="SXr";
          (24, -12)*+{S^2 Y}="SSX";
          (0, -20)*+{SX^{(i - 2)}}="SXim2";
          (24, -28)*+{SY}="SXm";
          (0, -40)*+{X^{(i - 1)}}="Xim1";
          (44, -40)*+{Y}="X";
          % ARROWS
          {\ar^-{S\phi^{(i - 1)}} "SSXim2" ; "SXim1"};
          {\ar_{S^2 \bar{f}^{(i - 2)}} "SSXim2" ; "SSX"};
          {\ar_{\mu^S_{X^{(i - 2)}}} "SSXim2" ; "SXim2"};
          {\ar^-{S\bar{f}^{(i - 1)}} "SXim1" ; "SXr"};
          {\ar^{\psi} "SXr" ; "X"};
          {\ar_{S\psi} "SSX" ; "SXr"};
          {\ar^{\mu^S_{X}} "SSX" ; "SXm"};
          {\ar_{S\bar{f}^{(i - 2)}} "SXim2" ; "SXm"};
          {\ar_{\phi^{(i - 1)}} "SXim2" ; "Xim1"};
          {\ar_{\psi} "SXm" ; "X"};
          {\ar_{\bar{f}^{(i - 1)}} "Xim1" ; "X"};
        \endxy
      \]
  commutes.

  We then define $\bar{f}$ to be the unique map such that, for each $i \geq 0$, the diagram
      \[
        \xy
          % POINTS
          (0, 0)*+{X^{(i)}}="X";
          (16, 0)*+{\bar{X}}="barX";
          (16, -16)*+{Y}="Y";
          % ARROWS
          {\ar^-{c^{(i)}} "X" ; "barX"};
          {\ar_-{\bar{f}^{(i)}} "X" ; "Y"};
          {\ar^-{\bar{f}} "barX" ; "Y"};
        \endxy
      \]
  commutes.  When $i = 0$ this gives us the required commutativity condition for $c^{(0)}$ to be a universal arrow, and uniqueness of $\bar{f}$ comes from the universal property of $\bar{X}$.  All that remains is to check that $\bar{f}$ is a map of $S$-algebras, i.e. that the diagram
      \[
        \xy
          % POINTS
          (0, 0)*+{S\bar{X}}="SX";
          (16, 0)*+{SY}="SY";
          (0, -16)*+{\bar{X}}="X";
          (16, -16)*+{Y}="Y";
          % ARROWS
          {\ar^{S\bar{f}} "SX" ; "SY"};
          {\ar_-{\phi} "SX" ; "X"};
          {\ar^-{\psi} "SY" ; "Y"};
          {\ar_{\bar{f}} "X" ; "Y"};
        \endxy
      \]
  commutes.  Since $S$ is finitary, we can write $S\bar{X}$ as
      \[
        S\bar{X} = \colim_{i \geq 0} SX^{(i)}.
      \]
    Thus we can check that the square above commutes by comparing the cocones corresponding to the maps $\psi \comp S\bar{f}$ and $\bar{f} \comp \phi$.  The cocone corresponding to $\psi \comp S\bar{f}$ has components given, for each $i \geq 1$, by the composite
      \[
        \xy
          % POINTS
          (0, 0)*+{SX^{(i)}}="SXi";
          (20, 0)*+{SY}="SbarY";
          (36, 0)*+{Y.}="barY";
          % ARROWS
          {\ar^-{\bar{f}^{(i)}} "SXi" ; "SbarY"};
          {\ar^-{\psi} "SbarY" ; "barY"};
        \endxy
      \]
    The cocone corresponding to $\bar{f} \comp \phi$ has components given, for each $i \geq 1$, by the composite
      \[
        \xy
          % POINTS
          (0, 0)*+{SX^{(i)}}="SXi";
          (20, 0)*+{X^{(i + 1)}}="Xi1";
          (36, 0)*+{Y.}="barY";
          % ARROWS
          {\ar^-{\phi^{(i + 1)}} "SXi" ; "Xi1"};
          {\ar^-{\bar{f}^{(i + 1)}} "Xi1" ; "barY"};
        \endxy
      \]
    From the definition of $\bar{f}^{(i + 1)}$ we see that the diagram
      \[
        \xy
          % POINTS
          (0, 0)*+{SX^{(i)}}="SXi";
          (20, 0)*+{SY}="SbarY";
          (0, -16)*+{X^{(i + 1)}}="Xi1";
          (20, -16)*+{Y}="barY";
          % ARROWS
          {\ar^-{\bar{f}^{(i)}} "SXi" ; "SbarY"};
          {\ar_-{\phi^{(i + 1)}} "SXi" ; "Xi1"};
          {\ar^-{\psi} "SbarY" ; "barY"};
          {\ar_-{\bar{f}^{(i + 1)}} "Xi1" ; "barY"};
        \endxy
      \]
    commutes for all $i \geq 0$; hence the cocones described above are equal, so $\bar{f}$ is a map of $S$-algebras.

    Hence we have an adjunction $F \ladj p_*$, as required.
  \end{proof}

  Finally, to show that this construction does indeed give an adjunction
    \[
          \xy
            % POINTS
            (0, 0)*+{B\Alg}="BAlg";
            (20, 0)*+{L\Alg}="LAlg";
            % ARROWS
            {\ar@<1ex>_-*!/u1pt/{\labelstyle \bot}^-F "BAlg" ; "LAlg"};
            {\ar@<1ex>^-{u_*} "LAlg" ; "BAlg"};
          \endxy
    \]
  in the case $\mathcal{C} = \nGSet$, $R = B$, $S = L$, $p = u$, we must show that $L$ is finitary.  In fact, this is true of any monad induced by an $n$-globular operad.

  \begin{lemma}  \label{lem:opsfiltcolims}
    Let $K$ be an $n$-globular operad.  Then the monad induced by $K$ is finitary, i.e. its underlying endofunctor preserves filtered colimits.
  \end{lemma}

  \begin{proof}
    It is a result of Leinster that the free strict $n$-category monad $T$ is finitary \cite[Theorem F.2.2]{Lei04}; the proof that the monad $K$ is finitary is an application of this and of the fact that filtered colimits commute with pullbacks in $\Set$ (Lemma~\ref{lem:maclane}).

    Let $\mathbb{I}$ be a small, filtered category and let
      \[
        D \colon \mathbb{I} \longrightarrow \nGSet
      \]
    be a diagram in $\nGSet$.  Then for each $i \in \mathbb{I}$, $KD(i)$ is given by the pullback
      \[
        \xy
          % POINTS
          (0, 0)*+{KD(i)}="KX";
          (16, 0)*+{K}="K";
          (0, -16)*+{TD(i)}="TX";
          (16, -16)*+{T1,}="T";
          % ARROWS
          {\ar^-{K!} "KX" ; "K"};
          {\ar_{k_{D(i)}} "KX" ; "TX"};
          {\ar_-{T!} "TX" ; "T"};
          {\ar^{k} "K" ; "T"};
          % PULLBACK STUFF
          (6,-1)*{}; (6,-5)*{} **\dir{-};
          (2,-5)*{}; (6,-5)*{} **\dir{-};
        \endxy
     \]
    in $\nGSet$.  Write
      \[
        X := \colim_{i \in \mathbb{I}} D(i).
      \]
    Then $KX$ is given by the pullback
      \[
        \xy
          % POINTS
          (0, 0)*+{KX}="KX";
          (16, 0)*+{K}="K";
          (0, -16)*+{TX}="TX";
          (16, -16)*+{T1,}="T";
          % ARROWS
          {\ar^-{K!} "KX" ; "K"};
          {\ar_{k_X} "KX" ; "TX"};
          {\ar_-{T!} "TX" ; "T"};
          {\ar^{k} "K" ; "T"};
          % PULLBACK STUFF
          (6,-1)*{}; (6,-5)*{} **\dir{-};
          (2,-5)*{}; (6,-5)*{} **\dir{-};
        \endxy
     \]
    in $\nGSet$.  Since $T$ is finitary, we have
      \[
        TX \iso \colim_{i \in \mathbb{I}} TD(i).
      \]
    Since filtered colimits commute with pullbacks in $\Set$, and since limits and colimits are computed pointwise in $\nGSet$, we have that filtered colimits commutes with pullbacks in $\nGSet$, so
      \[
        KX = \colim_{i \in \mathbb{I}} KD(i).
      \]
    Hence $K$ preserves filtered colimits, i.e. $K$ is finitary.
  \end{proof}

  Hence there is an adjunction $F \ladj u_*$.

  \subsection{The relationship between $u_*$ and $v_*$}  \label{subsect:v*u*}

  Recall that, in Proposition~\ref{prop:uvid}, we showed that $u_* v_* = \id_{B\Alg}$.  We then gave a small example of an $L$-algebra
    \[
      \xy
        % POINTS
        (0, 0)*+{L^2 A}="LLA";
        (0, -16)*+{LA.}="LA";
        % ARROWS
        {\ar^{\mu^L_A} "LLA" ; "LA"};
      \endxy
    \]
  and described its image under the composite
    \[
      \xy
        % POINTS
        (0, 0)*+{L\Alg}="0";
        (20, 0)*+{B\Alg}="1";
        (40, 0)*+{L\Alg.}="2";
        % ARROWS
        {\ar^{u_*} "0" ; "1"};
        {\ar^{v_*} "1" ; "2"};
      \endxy
    \]
  Specifically, we argued that the diagram
    \[
      \xy
        % POINTS
        (0, 0)*+{L^2 A}="LLAt";
        (20, 0)*+{L^2 A}="LLAr";
        (0, -10)*+{BLA}="BLA";
        (0, -20)*+{L^2 A}="LLA";
        (0, -30)*+{LA}="LA";
        (20, -30)*+{LA}="LAr";
        % ARROWS
        {\ar^-{L\id_{LA}} "LLAt" ; "LLAr"};
        {\ar_{v_{LA}} "LLAt" ; "BLA"};
        {\ar^{\mu^L_A} "LLAr" ; "LAr"};
        {\ar_{u_{LA}} "BLA" ; "LLA"};
        {\ar_{\mu^L_A} "LLA" ; "LA"};
        {\ar_-{\id_{LA}} "LA" ; "LAr"};
      \endxy
    \]
  in $\nGSet$ ``commutes up to a contraction cell''.  We now extend these ideas to a definition of weak map of $L$-algebras that uses constraint cells for mediating cells, in order to formalise this idea and thus investigate the relationship between a general $L$-algebra and its image under the functor $v_* u_*$ more fully.  The idea is that, by using constraint cells, any axioms we would require will automatically be satisfied, so we do not have to state any axioms in the definition.  This approach is beneficial, since it is straightforward to specify the data required for a weak map of $L$-algebras (i.e. to specify where we require mediating cells), but difficult to state the axioms that this data must satisfy.

  Note that the definition of weak map that this approach gives is not optimal, for several reasons.  First, the fact that the mediating cells must be constraint cells means that this definition lacks generality, since in a fully general definition of weak map we would be able to use any choice of cells that interacted with one another in a suitably coherent way.  Second, the composite of two weak maps is not necessarily a weak map, since the mediating cells in the composite are composites of constraint cells, and these are not necessarily constraint cells.  Finally, in a non-free $L$-algebra not all diagrams of constraint $n$-cells commute, and not all diagrams of constraint cells commute up to a higher constraint cell.

  \begin{defn} \label{defn:weakmaps}
    Let $K$ be an $n$-globular operad with a contraction and system of compositions and let
        \[
          \xy
            % POINTS
            (0, 0)*+{KX}="KX";
            (0, -16)*+{X,}="X";
            (16, 0)*+{KY}="KY";
            (16, -16)*+{Y,}="Y";
            % ARROWS
            {\ar_-{\theta} "KX" ; "X"};
            {\ar^-{\phi} "KY" ; "Y"};
          \endxy
        \]
    be $K$-algebras.  A \emph{weak map of $K$-algebras} consists of a (not necessarily commuting) square
        \[
          \xy
            % POINTS
            (0, 0)*+{KX}="KX";
            (0, -16)*+{X}="X";
            (16, 0)*+{KY}="KY";
            (16, -16)*+{Y}="Y";
            % ARROWS
            {\ar^-{Kf} "KX" ; "KY"};
            {\ar_-{\theta} "KX" ; "X"};
            {\ar^-{\phi} "KY" ; "Y"};
            {\ar_-f "X" ; "Y"};
          \endxy
        \]
    in $\nGSet$, equipped with the following constraint cells:
      \begin{itemize}
        \item for all $0$-cells $x$ in $KX$, a constraint $1$-cell
          \[
            f_x \colon \phi \comp Kf(x) \longrightarrow f \comp \theta(x)
          \]
            in $Y$;
        \item for all $1$-cells $a \colon x \rightarrow y$ in $KX$, a constraint $2$-cell
          \[
            \xy
              % POINTS
              (0, 0)*+{\phi \comp Kf(x)}="0,0";
              (28, 0)*+{\phi \comp Kf(y)}="1,0";
              (0, -16)*+{f \comp \theta(x)}="0,1";
              (28, -16)*+{f \comp \theta(y)}="1,1";
              % ARROWS
              {\ar^-{\phi \comp Kf(a)} "0,0" ; "1,0"};
              {\ar_{f_x} "0,0" ; "0,1"};
              {\ar^{f_y} "1,0" ; "1,1"};
              {\ar_-{f \comp \theta(a)} "0,1" ; "1,1"};
              {\ar@{=>}_{f_a} (18, -4) ; (10, -12)};
            \endxy
          \]
        in $Y$;
        \item for all $2$-cells
          \[
            \xy
              % POINTS
              (0, 0)*+{x}="0,0";
              (16, 0)*+{y}="1,0";
              % ARROWS
              {\ar@/^1.5pc/ "0,0" ; "1,0"};
              {\ar@/_1.5pc/ "0,0" ; "1,0"};
              {\ar@{=>}^{\alpha} (8, 3) ; (8, -3)};
            \endxy
          \]
        in $KX$, a constraint $3$-cell
          \[
            \xy
              % POINTS
              (0, 0)*+{\bullet}="0,0";
              (24, 0)*+{\bullet}="1,0";
              (0, -16)*+{\bullet}="0,1";
              (24, -16)*+{\bullet}="1,1";
              (36,-8)*+{\Rrightarrow};
              (48, 0)*+{\bullet}="2,0";
              (72, 0)*+{\bullet}="3,0";
              (48, -16)*+{\bullet}="2,1";
              (72, -16)*+{\bullet}="3,1";
              % ARROWS -- FIRST CYLINDER
              {\ar@/^1.5pc/ "0,0" ; "1,0"};
              {\ar@/_1.5pc/ "0,0" ; "1,0"};
              {\ar@{=>}^{\phi \comp Kf(\alpha)} (8, 3) ; (8, -3)};
              {\ar_{f_x} "0,0" ; "0,1"};
              {\ar^{f_y} "1,0" ; "1,1"};
              {\ar@/_1.5pc/ "0,1" ; "1,1"};
              {\ar@{=>}^{f_b} (15, -10) ; (9, -15)};
              % ARROWS -- SECOND CYLINDER
              {\ar@/^1.5pc/ "2,0" ; "3,0"};
              {\ar_{f_x} "2,0" ; "2,1"};
              {\ar^{f_y} "3,0" ; "3,1"};
              {\ar@{=>}^{f_a} (63, -1) ; (57, -6)};
              {\ar@/^1.5pc/ "2,1" ; "3,1"};
              {\ar@/_1.5pc/ "2,1" ; "3,1"};
              {\ar@{=>}^{f \comp \theta(\alpha)} (60, -13) ; (60, -19)};
              % 3-CELL LABEL
              {\ar@{}_{f_{\alpha}} (24, -10) ; (48, -10)};
            \endxy
          \]
        in $Y$.  We abuse notation slightly and write this as
          \[
            f_{\alpha} \colon f_{b} \comp (\phi \comp Kf(\alpha)) \Rrightarrow (f \comp \theta(\alpha)) \comp f_a,
          \]
        omitting the $1$-cells $f_x$ and $f_y$; this makes little difference here, but at higher dimensions it allows us to avoid unwieldy notation;
        \item for $3 \leq m \leq n - 1$, and for all $m$-cells $\alpha$ in $KX$, a constraint cell
          \[
            f_{\alpha} \colon f_{t(\alpha)} \comp (\phi \comp Kf(\alpha)) \rightarrow (f \comp \theta(\alpha)) \comp f_{s(\alpha)}.
          \]
        As described above, we omit lower-dimensional constraint cells from the source and target to avoid unwieldy notation.  When $m = 3$, $f_{\alpha}$ is a constraint $4$-cell with source
          \[
            \xy
              % SOURCE CYLINDER
              % POINTS
              (0, 0)*+{\bullet}="0,0";
              (24, 0)*+{\bullet}="1,0";
              (12, -2)*+{\Rrightarrow};
              (0, -16)*+{\bullet}="0,1";
              (24, -16)*+{\bullet}="1,1";
              (36, -8)*+{\Rrightarrow};
              (48, 0)*+{\bullet}="2,0";
              (72, 0)*+{\bullet}="3,0";
              (48, -16)*+{\bullet}="2,1";
              (72, -16)*+{\bullet}="3,1";
              % INVISIBLE POINTS FOR 2-CELLS
              (9, 2)*+{}="1s";
              (9, -5)*+{}="1t";
              (15, 2)*+{}="2s";
              (15, -5)*+{}="2t";
              % ARROWS -- FIRST CYLINDER HALF
              {\ar@{}^{\phi \comp Kf(\alpha)} (0, 1) ; (24, 1)};
              {\ar@/^1.5pc/ "0,0" ; "1,0"};
              {\ar@/_1.5pc/ "0,0" ; "1,0"};
              {\ar@{=>}@/_.4pc/ "1s" ; (9, -4.85)};
              {\ar@{=}@/_.4pc/ "1s" ; "1t"};
              {\ar@{=>}@/^.4pc/ "2s" ; (15, -4.85)};
              {\ar@{=}@/^.4pc/ "2s" ; "2t"};
              {\ar "0,0" ; "0,1"};
              {\ar "1,0" ; "1,1"};
              {\ar@/_1.5pc/ "0,1" ; "1,1"};
              {\ar@{=>}^{f_{tt(\alpha)}} (15, -10) ; (9, -15)};
              % ARROWS -- SECOND CYLINDER HALF
              {\ar@/^1.5pc/ "2,0" ; "3,0"};
              {\ar "2,0" ; "2,1"};
              {\ar "3,0" ; "3,1"};
              {\ar@{=>}^{f_{ss(\alpha)}} (63, -1) ; (57, -6)};
              {\ar@/^1.5pc/ "2,1" ; "3,1"};
              {\ar@/_1.5pc/ "2,1" ; "3,1"};
              {\ar@{=>}^{f \comp \theta(t(\alpha))} (56, -13) ; (56, -19)};
              % 3-CELL LABEL
              {\ar@{}_{f_{t(\alpha)}} (24, -10) ; (48, -10)};
            \endxy
          \]
        and target
          \[
            \xy
              % TARGET CYLINDER
              % POINTS
              (0, 0)*+{\bullet}="0,0";
              (24, 0)*+{\bullet}="1,0";
              (0, -16)*+{\bullet}="0,1";
              (24, -16)*+{\bullet}="1,1";
              (36, -8)*+{\Rrightarrow};
              (48, 0)*+{\bullet}="2,0";
              (72, 0)*+{\bullet}="3,0";
              (48, -16)*+{\bullet}="2,1";
              (72, -16)*+{\bullet;}="3,1";
              (60, -14)*+{\Rrightarrow};
              % INVISIBLE POINTS FOR 2-CELLS
              (57, -11)*+{}="1s";
              (57, -18)*+{}="1t";
              (63, -11)*+{}="2s";
              (63, -18)*+{}="2t";
              % ARROWS -- FIRST CYLINDER HALF
              {\ar@/^1.5pc/ "0,0" ; "1,0"};
              {\ar@/_1.5pc/ "0,0" ; "1,0"};
              {\ar@{=>}^{\phi \comp Kf(s(\alpha))} (6, 3) ; (6, -3)};
              {\ar "0,0" ; "0,1"};
              {\ar "1,0" ; "1,1"};
              {\ar@/_1.5pc/ "0,1" ; "1,1"};
              {\ar@{=>}^{f_{tt(\alpha)}} (15, -10) ; (9, -15)};
              % ARROWS -- SECOND CYLINDER HALF
              {\ar@/^1.5pc/ "2,0" ; "3,0"};
              {\ar "2,0" ; "2,1"};
              {\ar "3,0" ; "3,1"};
              {\ar@{=>}^{f_{ss(\alpha)}} (63, -1) ; (57, -6)};
              {\ar@/^1.5pc/ "2,1" ; "3,1"};
              {\ar@/_1.5pc/ "2,1" ; "3,1"};
              {\ar@{=>}@/_.4pc/ "1s" ; (57, -17.85)};
              {\ar@{=}@/_.4pc/ "1s" ; "1t"};
              {\ar@{=>}@/^.4pc/ "2s" ; (63, -17.85)};
              {\ar@{=}@/^.4pc/ "2s" ; "2t"};
              {\ar@{}_{f \comp \theta(\alpha)} (48, -17.5) ; (72, -17.5)};
              % 3-CELL LABEL
              {\ar@{}_{f_{s(\alpha)}} (24, -10) ; (48, -10)};
            \endxy
          \]
        \item for all $n$-cells $\alpha$ in $KX$, an equality (which we can think of as a ``constraint $(n + 1)$-cell'')
          \[
            f_{t(\alpha)} \comp (\phi \comp Kf(\alpha)) = (f \comp \theta(\alpha)) \comp f_{s(\alpha)}.
          \]
      \end{itemize}
  \end{defn}

  Note that, given two weak maps of $K$-algebras
        \[
          \xy
            % POINTS
            (0, 0)*+{KX}="KX";
            (0, -16)*+{X}="X";
            (16, 0)*+{KY}="KY";
            (16, -16)*+{Y}="Y";
            (32, 0)*+{KZ}="KZ";
            (32, -16)*+{Z,}="Z";
            % ARROWS
            {\ar^{Kf} "KX" ; "KY"};
            {\ar_-{\theta} "KX" ; "X"};
            {\ar^{Kg} "KY" ; "KZ"};
            {\ar^-{\phi} "KY" ; "Y"};
            {\ar^-{\psi} "KZ" ; "Z"};
            {\ar_f "X" ; "Y"};
            {\ar_g "Y" ; "Z"};
          \endxy
        \]
  although their underlying maps of $n$-globular sets are composable, this composite is not necessarily a weak map of $K$-algebras, since in a weak map we require the mediating cells to be constraint cells, whereas in the composite $gf$ we only have composites of constraint cells.  Thus there is no category of $K$-algebras with morphisms given by weak maps.  There are two ways we could get around this: we could either modify our definition of weak map to allow us to use composites of constraint cells as mediating cells, or we could take the closure under composition of the class of weak maps; either approach is beyond the scope of this thesis.

  Recall that in Section~\ref{sect:globopcoh} we gave a definition of equivalence of $K$-algebras which used strict maps (Definition~\ref{defn:globopalgequiv}).  We now modify this definition by using weak maps instead of strict maps, to obtain a notion of weak equivalence of $K$-algebras.

  \begin{defn}  \label{defn:weakequiv}
    Let $K$ be an $n$-globular operad with a contraction and system of compositions.  We say that two $K$-algebras
        \[
          \xy
            % POINTS
            (0, 0)*+{KX}="KX";
            (0, -16)*+{X,}="X";
            (16, 0)*+{KY}="KY";
            (16, -16)*+{Y}="Y";
            % ARROWS
            {\ar_-{\theta} "KX" ; "X"};
            {\ar^-{\phi} "KY" ; "Y"};
          \endxy
        \]
    are \emph{weakly equivalent} if there exists a weak map
        \[
          \xy
            % POINTS
            (0, 0)*+{KX}="0,0";
            (0, -16)*+{X}="0,1";
            (16, 0)*+{KY}="1,0";
            (16, -16)*+{Y}="1,1";
            (26, 0)*+{\text{or}};
            (36, 0)*+{KY}="2,0";
            (36, -16)*+{Y}="2,1";
            (52, 0)*+{KX}="3,0";
            (52, -16)*+{X}="3,1";
            % ARROWS
            {\ar^-{Kf} "0,0" ; "1,0"};
            {\ar_-{\theta} "0,0" ; "0,1"};
            {\ar^-{\phi} "1,0" ; "1,1"};
            {\ar_-f "0,1" ; "1,1"};
            {\ar^-{Kf} "2,0" ; "3,0"};
            {\ar_-{\phi} "2,0" ; "2,1"};
            {\ar^-{\theta} "3,0" ; "3,1"};
            {\ar_-f "2,1" ; "3,1"};
          \endxy
        \]
    such that $f$ is surjective on $0$-cells, full on $m$-cells for all $1 \leq m \leq n$, and faithful on $n$-cells.  The map $f$ is referred to as an \emph{weak equivalence of $K$-algebras}.
  \end{defn}

  As in Definition~\ref{defn:globopalgequiv}, in this definition we require that the weak equivalence can go in either direction, since having a weak equivalence in one direction does not guarantee the existence of a weak equivalence in the opposite direction.  This is caused by the fact that in our definition of weak map only allows for mediating cells that are constraint cells, rather than allowing any suitably coherent choice of cells.

  We now consider the composite
    \[
      \xy
        % POINTS
        (0, 0)*+{L\Alg}="0";
        (20, 0)*+{B\Alg}="1";
        (40, 0)*+{L\Alg,}="2";
        % ARROWS
        {\ar^{u_*} "0" ; "1"};
        {\ar^{v_*} "1" ; "2"};
      \endxy
    \]
  and show that any $L$-algebra is weakly equivalent to its image under this functor.

  \begin{prop}
    Let
      \[
        \xymatrix{ LX \ar[r]^-{\theta} & X }
      \]
    be an $L$-algebra.  Then the diagram
    \[
      \xy
        % POINTS
        (0, 0)*+{LX}="LLAt";
        (20, 0)*+{LX}="LLAr";
        (0, -10)*+{BX}="BLA";
        (0, -20)*+{LX}="LLA";
        (0, -30)*+{X}="LA";
        (20, -30)*+{X}="LAr";
        % ARROWS
        {\ar^-{L\id_{X}} "LLAt" ; "LLAr"};
        {\ar_{v_{X}} "LLAt" ; "BLA"};
        {\ar^{\theta} "LLAr" ; "LAr"};
        {\ar_{u_{X}} "BLA" ; "LLA"};
        {\ar_{\theta} "LLA" ; "LA"};
        {\ar_-{\id_{X}} "LA" ; "LAr"};
      \endxy
    \]
    can be equipped with the structure of a weak map of $L$-algebras, and this weak map is a weak equivalence.
  \end{prop}

  \begin{proof}
    We need only show that this diagram can be equipped with the structure of a weak map of $L$-algebras; if so, it will automatically be a weak equivalence since its underlying map of $n$-globular sets is the identity.  In this proof we write $f := \id_X$ to avoid double subscripts and misleading notation for the mediating constraint cells.  Note that although $L$ is defined using an unbiased contraction, in this proof we only use its contraction and system of compositions, as described in Theorem~\ref{thm:OUCtoOCS}.

    Recall from Definition~\ref{defn:indmonad} that $LX$ is defined by the pullback
      \[
        \xy
          % POINTS
          (0, 0)*+{LX}="LX";
          (16, 0)*+{L}="L";
          (0, -16)*+{TX}="TX";
          (16, -16)*+{T1.}="T";
          % ARROWS
          {\ar^-{L!} "LX" ; "L"};
          {\ar_{l_X} "LX" ; "TX"};
          {\ar_-{T!} "TX" ; "T"};
          {\ar^{l} "L" ; "T"};
          % PULLBACK STUFF
          (6,-1)*{}; (6,-5)*{} **\dir{-};
          (2,-5)*{}; (6,-5)*{} **\dir{-};
        \endxy
     \]
    For $1 \leq m \leq n$ we write the elements of $LX_m$ in the form $(\alpha, \chi)$, where $\alpha \in TX$, $\chi \in L$, and $T!(\alpha) = l(\chi)$ (note that we do not do this when $m = 0$ since $L_0$ has only one element).  Our approach is to find the required constraint cells using the fact that
      \[
        \xy
          % POINTS
          (-16, 14)*+{L}="Ll";
          (0, 14)*+{B}="B";
          (16, 14)*+{L}="Lr";
          (0, 0)*+{T1}="T";
          % ARROWS
          {\ar^-{u} "Ll" ; "B"};
          {\ar^-{v} "B" ; "Lr"};
          {\ar_{l} "Ll" ; "T"};
          {\ar^{b} "B" ; "T"};
          {\ar^{l} "Lr" ; "T"};
        \endxy
      \]
    commutes; thus when we apply $u_X v_X$ to a cell in $LX$ we end up ``a contraction cell away'' from where we started.

    Since the lower-dimensional mediating cells appear in the sources and targets of the mediating cells at the dimensions above, we must define our choice of mediating cells by induction over dimension.  As in Definition~\ref{defn:weakmaps}, for dimensions greater than $1$ we abuse notation slightly by omitting lower-dimensional constraint cells from sources and targets.

    At dimension $0$ the diagram commutes, so for all $x \in LX_0$ we define
      \[
        f_x := \id_{\theta(x)} \colon \theta(x) \longrightarrow \theta u_Xv_X(x).
      \]

    At dimension $1$, let
      \[
        (a, p) \colon x \longrightarrow y
      \]
    be a $1$-cell in $LX$, where $a \in TX_1$ and $p \in L_1$.  We seek a constraint $2$-cell
          \[
            \xy
              % POINTS
              (0, 0)*+{\phi \comp Kf(x)}="0,0";
              (28, 0)*+{\phi \comp Kf(y)}="1,0";
              (0, -16)*+{f \comp \theta(x)}="0,1";
              (28, -16)*+{f \comp \theta(y)}="1,1";
              % ARROWS
              {\ar^-{\phi \comp Kf(a, p)} "0,0" ; "1,0"};
              {\ar_{f_x = \id_{\theta(x)}} "0,0" ; "0,1"};
              {\ar^{f_y = \id_{\theta(y)}} "1,0" ; "1,1"};
              {\ar_-{f \comp \theta(a, p)} "0,1" ; "1,1"};
              {\ar@{=>}_{f_{(a, p)}} (18, -4) ; (10, -12)};
            \endxy
          \]
    in $LX$.  We have
      \[
        u_X v_X(a, p) = (a, uv(p)),
      \]
    and we can write the source and target of the required $2$-cell as
      \[
        \id_{\theta(y)} \comp \theta(a, p) = \theta(a, \id \comp p)
      \]
    and
      \[
        \theta(a, uv(p)) \comp \id_{\theta(x)} = \theta(a, uv(p) \comp \id)
      \]
    respectively, where $\id$ denotes the identity on the unique $0$-cell of $L$.  Now, since $(uv)_0 = \id_{LX}$, we have
      \[
        s(\id \comp p) = s(uv(p) \comp \id),
      \]
      \[
        t(\id \comp p) = t(uv(p) \comp \id),
      \]
      \[
        l(\id \comp p) = l(uv(p) \comp \id).
      \]
    Hence there is a contraction $2$-cell $\gamma(\id \comp p, uv(p) \comp \id)$ in $L$.  We denote this by $\kappa_p$, and define $f_{(a, p)}$ to be the constraint cell
      \[
        f_{(a, p)} := \theta \big( \id_a, \kappa_p \big) \colon \theta(a, \id \comp p) \Rightarrow \theta(a, uv(p) \comp \id)
      \]
    in $X$.

    Now let $2 \leq m \leq n$, and suppose that for all $j < m$ and for all $j$-cells $(a, p)$ in $LX$ we have defined a constraint cell
      \[
        f_{(a, p)} = \theta(\id_a, \kappa_p) \colon \theta(a, \kappa_{t(p)} \comp p) \Rightarrow \theta(a, uv(p) \comp \kappa_{s(p)})
      \]
    in $X$.  Let $(\alpha, \chi)$ be an $m$-cell in $LX$.  We have $u_X v_X(\alpha, \chi) = (\alpha, uv(\chi))$, so we seek a constraint $(m + 1)$-cell
      \[
        f_{(\alpha, \chi)} \colon f_{t(\alpha, \chi)} \comp \theta(\alpha, \chi) \rightarrow \theta(\alpha, uv(\chi)) \comp f_{s(\alpha, \chi)}.
      \]
    We can write the source of this as
      \[
        f_{t(\alpha, \chi)} \comp \theta(\alpha, \chi) = \theta(\alpha, \kappa_{t(\chi)} \comp \chi),
      \]
    and the target as
      \[
        \theta(\alpha, uv(\chi)) \comp f_{s(\alpha, \chi)} = \theta(\alpha, uv(\chi) \comp \kappa_{s(\chi)}).
      \]
    The cells $\kappa_{t(\chi)} \comp \chi$ and $uv(\chi) \comp \kappa_{s(\chi)}$ are parallel;  since the diagram
      \[
        \xy
          % POINTS
          (-16, 14)*+{L}="Ll";
          (0, 14)*+{B}="B";
          (16, 14)*+{L}="Lr";
          (0, 0)*+{T1}="T";
          % ARROWS
          {\ar^-{u} "Ll" ; "B"};
          {\ar^-{v} "B" ; "Lr"};
          {\ar_{l} "Ll" ; "T"};
          {\ar^{b} "B" ; "T"};
          {\ar^{l} "Lr" ; "T"};
        \endxy
      \]
    commutes, and since $l$ maps contraction cells to identities in $T1$, we have
      \[
        l(\kappa_{t(\chi)} \comp \chi) = l(\chi) = luv(\chi) = l(uv(\chi) \comp \kappa_{s(\chi)}).
      \]
    Hence there is a contraction $(m + 1)$-cell
      \[
        \gamma(\kappa_{t(\chi)} \comp \chi, uv(\chi) \comp \kappa_{s(\chi)})
      \]
    in $L$ (an equality if $m = n$).  We denote this by $\kappa_{\chi}$, and define $f_{(\alpha, \chi)}$ to be the constraint cell
      \[
        f_{(\alpha, \chi)} = \theta(\id_{\alpha}, \kappa_{\chi}) \colon \theta(\alpha, \kappa_{t(\chi)} \comp \chi) \rightarrow \theta(\alpha, uv(\chi) \comp \kappa_{s(\chi)}).
      \]

    This equips the diagram
    \[
      \xy
        % POINTS
        (0, 0)*+{LX}="LLAt";
        (20, 0)*+{LX}="LLAr";
        (0, -10)*+{BX}="BLA";
        (0, -20)*+{LX}="LLA";
        (0, -30)*+{X}="LA";
        (20, -30)*+{X}="LAr";
        % ARROWS
        {\ar^-{L\id_{X}} "LLAt" ; "LLAr"};
        {\ar_{v_{X}} "LLAt" ; "BLA"};
        {\ar^{\theta} "LLAr" ; "LAr"};
        {\ar_{u_{X}} "BLA" ; "LLA"};
        {\ar_{\theta} "LLA" ; "LA"};
        {\ar_-{\id_{X}} "LA" ; "LAr"};
      \endxy
    \]
    with the structure of a weak map of $L$-algebras; thus, since its underlying map of $n$-globular sets is $\id_{X}$, it is a weak equivalence of $L$-algebras.
  \end{proof}

  We now justify that, in the example above, the use of constraint cells as mediating cells allows us to avoid having to check any axioms.  All of the constraint cells we used in this example were first formed in $LX$, with the correct source and target; we then applied $\theta$ to obtain a constraint cell in $X$.  Thus any diagram we would want to commute, as one of the axioms for a weak map, is the image under $\theta$ of a diagram of constraint cells in the free $L$-algebra
    \[
      \xy
        % POINTS
        (0, 0)*+{L^2 X}="0";
        (16, 0)*+{LX;}="1";
        % ARROW
        {\ar^-{\mu^L_X} "0" ; "1"};
      \endxy
    \]
  thus any such diagram commutes (up to a constraint cell of the dimension above in the case of diagrams of cells of dimension less than $n$), by coherence for $L$-algebras (see Corollary~\ref{lem:diaginfree}).  Note that this will not be true for a general weak map of $L$-algebras, since in a non-free $L$-algebra not all diagrams of constraint cells commute.

  All of this highlights many of the difficulties involved in defining and working with weak maps when using an algebraic definition of weak $n$-category.  The natural notion of map in the algebraic setting is that of strict map; to define a notion of weak map we need to specify a large amount of extra structure and axioms.  Our approach to defining weak maps, using constraint cells to give maps that are ``automatically coherent'' without the need to check any axioms, is comparable to the approach taken in non-algebraic definitions of weak $n$-category.  In the non-algebraic setting it is meaningless to say that a map is strict, since we have no specified composites for maps to preserve.  The natural notion of map is more like a weak map (or a normalised map if the definition has a notion of degeneracies), and consists simply of a map of the underlying data.  This is sufficient since the roles of the cells are encoded in their shapes, which are recorded in the underlying data;  this is in contrast to the algebraic case in which, once we have applied an algebra action, all cells are globular so we are unable to tell what role they play in the algebra.  Maps in the non-algebraic case are more comparable to weak maps , and it is meaningless to ask for a map to be strict since we do not have specified composites.  This concludes our analysis of algebraic definitions, and leads us on to Part~\ref{Part2}, in which take the first steps towards a comparison between the algebraic definition of Penon and the non-algebraic definition of Tamsamani--Simpson.

\part{A multisimplicial nerve construction for Penon weak $n$-categories}  \label{Part2}

\chapter{Tamsamani--Simpson weak $n$-categories}  \label{chap:TamSim}  \label{chap:nerves}

  We now move on to the study of non-algebraic definitions of weak $n$-category.  This part focusses on the construction of a nerve functor for Penon weak $n$-categories, which allows us to compare them with a non-algebraic definition of weak $n$-category, specifically Simpson's variant of Tamsamani's definition.  This definition, which originates in \cite{Tam99}, is a higher-dimensional generalisation of the nerve of a category; Tamsamani's idea was, instead of using simplicial sets, presheaves on $\Delta$, to use $n$-simplicial sets, presheaves on $\Delta^n$.  He defined a weak $n$-category to be an $n$-simplicial set satisfying a generalisation of the nerve condition, called the Segal condition; this name is inherited from an analogous condition that arises in the study of Segal categories \cite{Seg74, DKS89}.  The definition we use is a variant of Tamsamani's definition, given by Simpson \cite{Sim97}.  Simpson refined and slightly simplified Tamsamani's definition, giving a more direct approach at the cost of a little generality.

  This definition is appropriate to use for the purposes of making a comparison between an algebraic and non-algebraic definition, since in a Tamsamani--Simpson weak $n$-category the cells are separated by dimension, as they are in a globular algebraic definition such as that of Penon.  This is often not the case with non-algebraic definitions of weak $n$-category.

  In this chapter we recall the nerve construction for categories, and the definition of Tamsamani--Simpson weak $n$-category.  None of the material in this chapter is new.

\section{Nerves of categories} \label{subsect:nercat}

  It is well-known that every small category has associated to it a simplicial set, known as its \emph{nerve} \cite{Seg68}.  The nerve is constructed by expanding upon the underlying graph of a category in a way that captures all the information about composition, identities, and associativity, information which is not present in the underlying graph.  The nerve construction gives a \emph{nerve functor} \hbox{$\nerve \colon \Cat \rightarrow \SSet$}, which sends a category to its nerve.  This functor is full and faithful \cite[Section 3.3]{Lei04}, so functors between categories correspond precisely to morphisms between their nerves.

  Not every simplicial set arises as the nerve of a category, but those that do arise in this way can be characterised using a number of different (equivalent) conditions.  Such a condition is called a \emph{nerve condition}.  Nerve conditions check whether or not we can define composition from a simplicial set in a way that is well-defined, associative, and unital.  Along with fullness and faithfulness of the nerve functor, they allow us to give alternative statements of the definitions of category and functor purely in terms of simplicial sets.  There are various nerve conditions; the nerve condition we use is that used by Tamsamani~\cite{Tam99}, since this is the condition that is generalised in the definition of Tamsamani--Simpson weak $n$-category.  This condition is a \emph{Segal condition}, and originates in \cite{Seg68}.

  We begin by recalling the definition of simplicial sets.  There are several different ways of defining simplicial sets, and the morphisms between them.  For our purposes, we will think of the category of simplicial sets as the category of presheaves on the \emph{simplicial category} $\Delta$.

  \begin{defn}
    We write $\Delta$ for the category with
      \begin{itemize}
        \item objects: for each $k = 0$, $1$, $2$, $\dotsc$, the totally ordered set $[k] = \{ 0, 1, \dotsc, k \}$;
        \item morphisms: order preserving functions.
      \end{itemize}
    The \emph{category of simplicial sets} $\SSet$ is defined to be the presheaf category
      \[
        \SSet := [\Delta^{\op}, \Set].
      \]
  \end{defn}

  We now establish some notation and terminology that we will use when talking about simplicial sets, all of which is standard.  For a simplicial set $A \colon \Delta^{\op} \rightarrow \Set$, we usually write $A_k$ to denote the set $A[k]$, where $[k] \in \Delta$.  We refer to the elements of $A_k$ as the ``(simplicial) $k$-cells of $A$''.  We refer to maps $A[p]$, where $[p] \colon [j] \rightarrow [k]$ in $\Delta$ is injective, as ``face maps'', since such maps give us the $j$-cell faces of the $k$-cells of $A$.  For all $k > 0$, $0 \leq i \leq k$, we have a map $d_i \colon [k - 1] \rightarrow [k]$ in $\Delta$ given by
    \[
      d_{i}(j) =
            \left\{ \begin{array}{ll}
                j & \text{if } j < i, \\
                j + 1 & \text{if } j \geq i,
            \end{array} \right.
    \]
  and the face maps $A(d_i)$ generate all face maps in $A$.  Similarly, we refer to maps $A[q]$, where $[q] \colon [j] \rightarrow [k]$ is surjective, as ``degeneracy maps'' since such maps give us degenerate $j$-cells; in the nerve of a category, these degenerate cells will be those made up at least partially from identities.

  An ordered set can be considered as a category in which the elements of the set are the objects of the category and, for elements $a$, $b$, there is one morphism $a \rightarrow b$ if $a \leq b$, and no such morphism otherwise.  Order-preserving maps then correspond precisely to functors.  Thus we can consider $\Delta$ to be a full subcategory of $\Cat$; we write
    \[
      I \colon \Delta \longrightarrow \Cat
    \]
  for the inclusion functor.  For each $[k] \in \Delta$, $I[k]$ is the category consisting of a string of $k$ composable morphisms.  In the nerve of a category $\mathcal{C}$, the set of $k$-cells is given by $\Cat(I[k], \mathcal{C})$, the set of composable strings of length $k$ of morphisms in $\mathcal{C}$.

  We now use this inclusion functor to define the nerve functor for categories.

  \begin{defn}
    The \emph{nerve functor} $\nerve \colon \Cat \rightarrow \SSet$ is defined to be:
      \[
        \xymatrix@C=0pt@R=2pt{
          \nerve \colon \Cat & \longrightarrow & [\Delta^{\op}, \Set]  \\
          \mathcal{C} \ar[dddd]_{F} && \Cat(I(-), \mathcal{C}) \ar[dddd]^{F \comp -}  \\
          \\
          & \longmapsto & \\
          \\
          \mathcal{D} && \Cat(I(-), \mathcal{D}),
                 }
      \]
    where $\Cat(I(-), \mathcal{C})$, referred to as the \emph{nerve} of the category $\mathcal{C}$, is the simplicial set defined by
      \[
        \xymatrix@C=0pt@R=2pt{
          \Cat(I(-), \mathcal{C}) \colon \Delta^{\op} & \longrightarrow & \Set  \\
          [k] \ar[dddd]_{p} && \Cat(I[k], \mathcal{C}) \ar[dddd]^{- \comp p}  \\
          \\
          & \longmapsto & \\
          \\
          [j] && \Cat(I[j], \mathcal{C}).
                 }
      \]
  \end{defn}

   In the nerve of a category $\mathcal{C}$, the set of $0$-cells $\Cat(I[0], \mathcal{C})$ is the set of objects in $\mathcal{C}$, and the set of $1$-cells $\Cat(I[1], \mathcal{C})$ is the set of morphisms in $\mathcal{C}$, with the face maps giving the sources and targets.  The set $\Cat(I[2], \mathcal{C})$ of composable pairs of morphisms tells us about composition in $\mathcal{C}$, and the sets of $k$-cells for $k > 2$ tell us about associativity.  Since identities in the nerve are given by the degeneracy maps, when we take the nerve of a category we retain all the information in that category.

  The following result is well-known (see, for example, \cite[Section 3.3]{Lei04}).

  \begin{prop}  \label{prop:nfandf}
    The nerve functor $\nerve \colon \Cat \rightarrow \SSet$ is full and faithful.
  \end{prop}

  Thus functors between categories correspond precisely to maps of simplicial sets between their nerves.

  We now recall the nerve condition that will be generalised in the definition of Tamsamani--Simpson weak $n$-category in the next section.  The nerve condition we use is a Segal condition, rather than a horn-filling condition.

  In  the category $\Delta$, for each object $[k]$, $k \geq 1$, and for each $i$, $0 \leq i \leq k - 1$, we have a map
    \begin{align*}
      \iota_{i + 1} \colon [1] & \rightarrow [k] \\
      0 & \mapsto i \\
      1 & \mapsto i + 1. \\
    \end{align*}
  We also have maps $\sigma$, $\tau: [0] \rightarrow [1]$, with $\sigma(0) = 0$ and $\tau(0) = 1$.  Thus we can construct the following diagram in $\Delta$:
     \[
       \xymatrix@C=12pt@R=12pt{
       &&&&& [k] \\
       \\
       [1] \ar[uurrrrr]^(.3){\iota_1} && [1] \ar[uurrr]^(.25){\iota_2} && [1] \ar[uur]^(.25){\iota_3} && \dotsc && [1] \ar[uulll]_(.25){\iota_{k - 1}} && [1] \ar[uulllll]_(.3){\iota_{k}} \\
       & [0] \ar[ul]^{\tau} \ar[ur]_{\sigma} && [0] \ar[ul]^{\tau} \ar[ur]_{\sigma} &&&&&& [0] \ar[ul]^{\tau} \ar[ur]_{\sigma}
                }
     \]
  and this diagram commutes.

  Let $A \colon \Delta^{\op} \rightarrow \Set$ be a simplicial set.  Applying $A \colon \Delta^{\op} \rightarrow \Set$ to the diagram above gives the following diagram in $\Set$:
    \[
      \xymatrix@C=12pt@R=12pt{
        &&&&& A_k \ar[ddlllll]_(.7){i_1} \ar[ddlll]_(.75){i_2} \ar[ddl]_(.75){i_3} \ar[ddrrr]^(.75){i_{k - 1}} \ar[ddrrrrr]^(.7){i_k} \\
        \\
        A_1 \ar[dr]_{t} && A_1 \ar[dl]^{s} \ar[dr]_{t} && A_1 \ar[dl]^{s} && \dotsc && A_1 \ar[dr]_{t} && A_1 \ar[dl]^{s} \\
        & A_0 && A_0 &&&&&& A_0
               }
    \]
  where $i_j$ sends a $k$-cell to the $j$th $1$-cell in its spine, and $s$ and $t$ send a $1$-cell to its source and target $0$-cells respectively.  Since functors preserve commutative diagrams, this diagram commutes and is a cone with vertex $A_k$ over the diagram
    \[
      \overbrace{
      \xymatrix@C=12pt@R=12pt{
        A_1 \ar[dr]_{t} && A_1 \ar[dl]^{s} \ar[dr]_{t} && A_1 \ar[dl]^{s} && \dotsc && A_1 \ar[dr]_{t} && A_1 \ar[dl]^{s} \\
        & A_0 && A_0 &&&&&& A_0
               }
               }^k
    \]
  Since $\Set$ is complete, we can take the limit of this diagram, called a wide pullback, which we denote by $\underbrace{A_1 \times_{A_0} \dotsb \times_{A_0} A_1}_k$.  The $k$th Segal map
  \[
    S_k: A_k \rightarrow \underbrace{A_1 \times_{A_0} \dotsb \times_{A_0} A_1}_k
  \]
  is the unique map induced by the universal property of the wide pullback such that the diagram
    \[
      \xymatrix@C=14pt@R=14pt{
        &&& A_k \ar@/_2.5pc/[ddddlll]_(.5){i_1} \ar@/_2pc/[ddddl]_(.5){i_2} \ar@{-->}[dd]^{!S_k}  \ar@/^2pc/[ddddr]^(.5){i_{k - 1}} \ar@/^2.5pc/[ddddrrr]^(.5){i_k} \\
        \\
        &&& A_1 \times_{A_0} \dotsb \times_{A_0} A_1 \ar[ddlll] \ar[ddl] \ar[ddr] \ar[ddrrr] \\
        \\
        A_1 \ar[dr]_{t} && A_1 \ar[dl]^{s} & \dotsc & A_1 \ar[dr]_{t} && A_1 \ar[dl]^{s} \\
        & A_0 &&&& A_0
               }
    \]
  commutes.

  The set $A_k$ is the set of simplicial $k$-cells in $A$.  If $A$ arises as the nerve of a category, we should be able to think of these as composable strings of $k$ morphisms, together with their composites.  Now, an element of
    \[
      \underbrace{A_1 \times_{A_0} \dotsb \times_{A_0} A_1}_k
    \]
  is a $k$-tuple $(f_1, \dotsc, f_k)$ of $1$-cells in $A$, such that $t(f_i) = s(f_{i+1})$ for $0 \leq i < k$.  We can think of such a $k$-tuple as being like a string of $k$ ``composable'' $1$-cells in our simplicial set.

  \begin{defn} \label{defn:nervecond}
    A simplicial set $A: \Delta^{\op} \rightarrow \Set$, is said to satisfy the \emph{nerve condition} if, for all $k$, the $k$th Segal map $S_k$ is an isomorphism.
  \end{defn}

  The following result originates in \cite{Seg68}.

  \begin{prop} \label{prop:nervecond}
    Given a category $\mathcal{C}$, the nerve $\nerve(\mathcal{C})$ satisfies the nerve condition.  Given a simplicial set $A$, if $A$ satisfies the nerve condition then there exists a category $\mathcal{C}$ such that $\nerve(\mathcal{C}) \iso A$.
  \end{prop}

  Proposition~\ref{prop:nervecond} and Proposition~\ref{prop:nfandf} allow us to give the following alternative statement of the definition of the category of small categories.

  \begin{defn}  \label{defn:catasnerve}
    The category of small categories $\Cat$ is the full subcategory of $[\Delta^{\op}, \Set]$ whose objects are those simplicial sets that satisfy the nerve condition.
  \end{defn}

\section{Tamsamani--Simpson weak $n$-categories} \label{subsect:tamsimn}

  In this section we recall Simpson's variant of Tamsamani's definition of weak $n$-category \cite{Tam99, Sim97}.  We begin by generalising the definition of simplicial set to that of an \emph{$n$-simplicial set} (often known as a \emph{multisimplicial set} when not specifying the value of $n$).

  \begin{defn}
    The \emph{category of $n$-simplicial sets} $\nSSet$ is defined inductively as follows:
      \begin{itemize}
        \item $0\text{-}\SSet := \Set$;
        \item for $n \geq 1$, $n\text{-}\SSet := [\Delta^{\op}, (n - 1)\text{-}\SSet] \iso [(\Delta^n)^{\op}, \Set]$, by cartesian closedness of $\Cat$.
      \end{itemize}
    \end{defn}

  We could have defined $n$-simplicial sets to be presheaves on $\Delta^n$ directly, but the form of the definition stated above highlights the fact that $n$-simplicial sets can be obtained by a process of repeated internalisation, which is a well-established method of adding extra dimensions; thus this illustrates why $\Delta^n$ is a reasonable category on which to take presheaves in a definition of weak $n$-category.  Note that the inductive nature of this definition means that the definition of Tamsamani--Simpson weak $n$-category does not a priori allow for the case $n = \omega$.  We write an object of $\Delta^n$ as an $n$-tuple
    \[
      (k_1, k_2, \dotsc, k_n),
    \]
  where, for all $1 \leq i \leq n$, $k_i \in \mathbb{N}$.

  We now explain how we should think of the shapes of cells in an $n$-simplicial set for the purposes of the definition of Tamsamani--Simpson weak $n$-category.  In $\Delta$, the object $[k]$ can be thought of as a string of $k$ composable morphisms.  Similarly, in an $n$-simplicial set $A \colon (\Delta^n)^{\op} \rightarrow \Set$, the set $A(k_1, k_2, \dotsc, k_n)$ can be thought of as the set of pasting diagrams called ``cuboidal'' by Leinster~\cite{Lei04}.  A cuboidal pasting diagram $(k_1, k_1, \dotsc, k_n) \in \Delta^n$ can be thought of as a grid of $n$-cells which is $k_1$ $n$-cells long, $k_2$ $n$-cells high, \ldots, and $k_n$ $n$-cells wide; for example, the cuboidal pasting diagram $(2,3) \in \Delta^2$ is shown in the diagram below.
            \[
              \xymatrix@R=7pt@C=40pt{
                & {} \ar@{=>}[dd] && {} \ar@{=>}[dd] & \\
                \\
                & {} \ar@{=>}[dd] && {} \ar@{=>}[dd] & \\
                \bullet \ar@/^3.5pc/[rr] \ar@/^1.3pc/[rr] \ar@/_1.3pc/[rr] \ar@/_3.5pc/[rr] && \bullet \ar@/^3.5pc/[rr] \ar@/^1.3pc/[rr] \ar@/_1.3pc/[rr] \ar@/_3.5pc/[rr] && \bullet \\
                & {} \ar@{=>}[dd] && {} \ar@{=>}[dd] & \\
                \\
                & {} && {}
                       }
            \]

  For this to give a globular notion of weak $n$-category, we need to ensure that, if $g:x \rightarrow x'$ is a $k$-cell in a weak $n$-category, then the $(k - 1)$-cells $x$ and $x'$ have the same source and the same target.  To do so we require that, for any $j$, if $k_j = 0$, i.e. the pasting diagram is $0$ $j$-cells wide, then $j - 1$ should be the maximum dimension of cell in the diagram.  In order to deal with this issue we use Simpson's method, which is to use presheaves on a quotient of $\Delta^n$, denoted $\Theta^n$, rather than using presheaves on $\Delta^n$ itself.  Note that if we do not ensure that our cells are globular, we obtain a definition of weak $n$-tuple category (also known as a weak $n$-fold category).

  We define $\Theta^n$ as a coequaliser in $\Cat$.  The idea is to identify objects in $\Delta^n$ if they are to be thought of as the same cuboidal pasting diagram.  For example, in $\Theta^2$, given an object $(j, k)$, if $j = 0$ the pasting diagram has zero width, so the value of $k$ should make no difference since the pasting diagram must also have zero height.  Thus in $\Theta^2$ we identify all objects of the form $(0, k)$, so $\Theta^2$ looks like:
    \[
      \xy
        % POINTS
        (0, 0)*+{(0, 0)}="0,0";
        (0, 4)*+{\phantom{(0, 0)}}="ph0,0";
        (24, 0)*+{(1, 0)}="1,0";
        (48, 0)*+{(2, 0)}="2,0";
        (56, 0)*+{\dotsc};
        (24, 16)*+{(1, 1)}="1,1";
        (48, 16)*+{(2, 1)}="2,1";
        (56, 16)*+{\dotsc};
        (24, 32)*+{(1, 2)}="1,2";
        (48, 32)*+{(2, 2)}="2,2";
        (56, 32)*+{\dotsc};
        (24, 40)*+{\vdots};
        (48, 40)*+{\vdots};
        % ARROWS
        % HORIZONTAL
        {\ar@<1ex> "0,0" ; "1,0"};
        {\ar@<-1ex> "0,0" ; "1,0"};
        {\ar "1,0" ; "0,0"};
        {\ar@<1ex>@/^1pc/ "0,0" ; "1,1"};
        {\ar@<-1ex>@/^1pc/ "0,0" ; "1,1"};
        {\ar@/_1pc/ "1,1" ; "0,0"};
        {\ar@<1ex>@/^2pc/ "ph0,0" ; "1,2"};
        {\ar@<-1ex>@/^2pc/ "ph0,0" ; "1,2"};
        {\ar@/_2pc/ "1,2" ; "ph0,0"};
        {\ar@<2ex> "1,0" ; "2,0"};
        {\ar "1,0" ; "2,0"};
        {\ar@<-2ex> "1,0" ; "2,0"};
        {\ar@<1ex> "2,0" ; "1,0"};
        {\ar@<-1ex> "2,0" ; "1,0"};
        {\ar@<2ex> "1,1" ; "2,1"};
        {\ar "1,1" ; "2,1"};
        {\ar@<-2ex> "1,1" ; "2,1"};
        {\ar@<1ex> "2,1" ; "1,1"};
        {\ar@<-1ex> "2,1" ; "1,1"};
        {\ar@<2ex> "1,2" ; "2,2"};
        {\ar "1,2" ; "2,2"};
        {\ar@<-2ex> "1,2" ; "2,2"};
        {\ar@<1ex> "2,2" ; "1,2"};
        {\ar@<-1ex> "2,2" ; "1,2"};
        % VERTICAL
        {\ar@<1ex> "1,0" ; "1,1"};
        {\ar@<-1ex> "1,0" ; "1,1"};
        {\ar "1,1" ; "1,0"};
        {\ar@<1ex> "2,0" ; "2,1"};
        {\ar@<-1ex> "2,0" ; "2,1"};
        {\ar "2,1" ; "2,0"};
        {\ar@<2ex> "1,1" ; "1,2"};
        {\ar "1,1" ; "1,2"};
        {\ar@<-2ex> "1,1" ; "1,2"};
        {\ar@<1ex> "1,2" ; "1,1"};
        {\ar@<-1ex> "1,2" ; "1,1"};
        {\ar@<2ex> "2,1" ; "2,2"};
        {\ar "2,1" ; "2,2"};
        {\ar@<-2ex> "2,1" ; "2,2"};
        {\ar@<1ex> "2,2" ; "2,1"};
        {\ar@<-1ex> "2,2" ; "2,1"};
      \endxy
    \]
  Similarly, for higher values of $n$, objects of $\Delta^n$ are identified in $\Theta^n$ if they differ only after a $0$.

  \begin{defn}  \label{defn:Theta}
    We define a category $\Theta^n$ as a coequaliser in $\Cat$ as follows:  first, let $R$ be the subcategory of $\Delta^n \times \Delta^n$ with
      \begin{itemize}
        \item objects: for all objects $(k_1, k_2, \dotsc, k_n)$ of $\Delta^n$,
          \[
            ((k_1, k_2, \dotsc, k_n), (k_1, k_2, \dotsc, k_n))
          \]
        is in $R$; also, for a fixed $j$ with $1 \leq j < n$,
          \[
            ((k_1, k_2, \dotsc, k_j, \dotsc, k_n), (k'_1, k'_2, \dotsc, k'_j, \dotsc, k'_n))
          \]
        is in $R$ if $k_j = 0 = k'_j$ and $k_i = k'_i$ for all $i < j$;
        \item morphisms: let $l$, $m \leq n$, and let $(\vectk,0) = (k_1, \dotsc, k_l, 0, \dotsc, 0)$ and $(\vectk^\prime,0) = (k^\prime_1, \dotsc, k^\prime_m, 0, \dotsc, 0)$ be objects of $\Delta^n$.  Then the morphism
              \[
                (\phi, \psi) \colon ((\vectk, 0), (\vectk, 0)) \rightarrow ((\vectk^\prime, 0), (\vectk^\prime, 0)),
              \]
            where $\phi = (\phi_1, \dotsc, \phi_n)$, $\psi = (\psi_1, \dotsc, \psi_n)$, is in $R$ if
      \begin{itemize}
        \item for all $i \leq j$, $\phi_i = \psi_i$;
        \item $\phi_j: [k_j] \rightarrow [{k^\prime}_j]$ factors through $[0]$ in $\Delta$.
      \end{itemize}
      \end{itemize}

    Since $R$ is a subcategory of $\Delta^n \times \Delta^n$, it comes equipped with projection maps $\pi_1$, $\pi_2 \colon R \rightarrow \Delta^n$.  The category $\Theta^n$ is defined to be the coequaliser of the diagram
      \[
        \xy
          % POINTS
          (0, 0)*+{R}="R";
          (16, 0)*+{\Delta^n}="Deltan";
          % ARROWS
          {\ar@<1ex>^-{\pi_1} "R" ; "Deltan"};
          {\ar@<-1ex>_-{\pi_2} "R" ; "Deltan"};
        \endxy
      \]
    in $\Cat$.  A presheaf
      \[
        A \colon (\Theta^n)^{\op} \longrightarrow \Set
      \]
    is called an \emph{$n$-precategory}.
  \end{defn}

  Note that this is not the only way of ensuring that we have globular cells; in the original definition, Tamsamani takes presheaves on $\Delta^n$, then includes an extra condition to ensure that the cells are globular.  In their expositions of Simpson's definition, both Cheng and Lauda \cite{Guide} and Leinster \cite{Lei02} also take this approach.  Using Simpson's approach does make a difference, since it leads to a definition of a weak $n$-category as a presheaf satisfying the Segal condition, with no extra conditions; this allows us to work with a presheaf category, with all the usual desirable properties these have, such as completeness, cocompleteness, and the existence of the Yoneda embedding.

  We now discuss the Segal condition, Tamsamani's $n$-dimensional generalisation of the nerve condition for categories given in Definition~\ref{defn:nervecond}.  Like the nerve condition, the Segal condition is a condition on a family of morphisms of $n$-precategories, called the \emph{Segal maps}; these Segal maps are defined to be induced by wide pullbacks in a way analogous to the definition of the Segal maps in Section~\ref{subsect:nercat}.

  In the nerve condition we required that the Segal maps were isomorphisms, to ensure that well-defined, associative, unital composition could be extracted from the nerve.  In the Segal condition for weak $n$-categories, we wish to weaken this since we only want composition that is associative and unital up to coherent isomorphism.  If the Segal maps were maps of $n$-categories we would instead require them to be equivalences.  However, the Segal maps are merely maps of $n$-precategories, so we cannot use the same notion of equivalence.  A functor is an equivalence if it is full, faithful, and essentially surjective on objects; for a map of $n$-precategories, we can still define fullness and faithfulness in the same way, but we cannot define what it means for a map to be essentially surjective since we do not have a composition structure, and thus no notion of isomorphism between cells.

  It was Simpson's insight that, instead of asking for essential surjectivity, we can demand surjectivity on $0$-cells.  Simpson observed that the resulting notion, which we call \emph{contractibility}, is enough for the purposes of the Segal condition, although it is not enough to define equivalences in general.  Recall that we previously used this insight of Simpson when defining equivalences of algebras for an operad in Definition~\ref{defn:globopalgequiv}, since contractibility was all that we required in this context as well.

  Before defining contractibility, we establish some notation used in the definition.  Let $0 \leq p \leq n$, and write $\vectone_p$ for the equivalence class in $\Theta^n$ of the object
      \[
        (\underbrace{1, \dotsc, 1}_p, \underbrace{0, \dotsc, 0}_{n - p})
      \]
    of $\Delta^n$, which should be thought of as a single globular $p$-cell.

    Let $A \colon (\Theta^n)^{\op} \rightarrow \Set$ be an $n$-precategory.  In $\Delta$, we have maps $\sigma$, $\tau: [0] \rightarrow [1]$, with $\sigma(0) = 0$ and $\tau(0) = 1$.  We define the source and target maps (denoted $s$ and $t$ respectively) in $A$, for each $p$, as follows:
      \begin{align*}
        s = A(\underbrace{\id, \dotsc, \id}_{p - 1}, \sigma, \id, \dotsc, \id): A(\vectone_p) & \rightarrow A(\vectone_{p - 1}); \\
        t = A(\underbrace{\id, \dotsc, \id}_{p - 1}, \tau, \id, \dotsc, \id): A(\vectone_p) & \rightarrow A(\vectone_{p - 1}).
      \end{align*}
  Note that this defines the underlying $n$-globular set of the $n$-precategory $A$, with the set of $p$-cells for each $- \leq p \leq n$ given by $A(\vectone_p)$.

  We now give the definition of contractibility.

  \begin{defn} \label{defn:simcontr}
    Let $m \geq 1$, let $A$, $B \colon (\Theta^m)^{\op} \rightarrow \Set$ be $m$-precategories, and let $\alpha \colon A \rightarrow B$ be a map of $m$-precategories.  For each $0 \leq p \leq m - 1$, we write $A(\vectone_p) \times_{B(\vectone_p)} B(\vectone_{p + 1}) \times_{B(\vectone_p)} A(\vectone_p)$ for the limit of the diagram
      \[
        \xymatrix{
          && A(\vectone_p) \ar[d]^{\alpha_{\vectone_p}}  \\
          & B(\vectone_{p + 1}) \ar[r]_s \ar[d]^t & B(\vectone_p)  \\
          A(\vectone_p) \ar[r]_{\alpha_{\vectone_p}} & B(\vectone_p)
                 }
      \]
    in $\Set$.  We also have a cone over this diagram with vertex $A(\vectone_{p + 1})$, as shown in the diagram below:
      \[
        \xymatrix{
          A(\vectone_{p + 1}) \ar[rr]^s \ar[dd]_t \ar[dr]_{\alpha_{\vectone_{p + 1}}} && A(\vectone_p) \ar[d]^{\alpha_{\vectone_p}}  \\
          & B(\vectone_{p + 1}) \ar[r]_s \ar[d]^t & B(\vectone_p)  \\
          A(\vectone_p) \ar[r]_{\alpha_{\vectone_p}} & B(\vectone_p)
                 }
      \]
    The universal property of the limit induces a unique map
      \[
        \tilde{\alpha}_{\vectone_{p + 1}} \colon A(\vectone_{p + 1}) \rightarrow A(\vectone_p) \times_{B(\vectone_p)} B(\vectone_{p + 1}) \times_{B(\vectone_p)} A(\vectone_p)
      \]
    such that
      \[
        \xymatrix{
          A_{\vectone_{p + 1}} \ar@/^1.5pc/[drrr]^s \ar@/_4pc/[dddr]_t \ar@/_4pc/[ddrr]_{\alpha_{\vectone_{p + 1}}} \ar@{-->}[dr]^{\tilde{\alpha}_{\vectone_{p + 1}}} \\
          & A(\vectone_p) \times_{B(\vectone_p)} B(\vectone_{p + 1}) \times_{B(\vectone_p)} A(\vectone_p) \ar[rr] \ar[dd] \ar[dr] && A(\vectone_p) \ar[d]^{\alpha_{\vectone_p}}  \\
          && B(\vectone_{p + 1}) \ar[r]_s \ar[d]^t & B(\vectone_p)  \\
          & A(\vectone_p) \ar[r]_{\alpha_{\vectone_p}} & B(\vectone_p)
                 }
      \]
    commutes.

    The map $\alpha \colon A \rightarrow B$ is  said to be contractible if:
      \begin{itemize}
        \item the map $\alpha_{\vectone_0} \colon A(\vectone_0) \rightarrow B(\vectone_0)$ is surjective (this is surjectivity of $\alpha$ on objects);
        \item for each $0 \leq p \leq m - 1$, the map
          \[
            \tilde{\alpha}_{\vectone_{p + 1}} \colon A(\vectone_{p + 1}) \rightarrow A(\vectone_p) \times_{B(\vectone_p)} B(\vectone_{p + 1}) \times_{B(\vectone_p)} A(\vectone_p)
          \]
           is surjective (this gives fullness at dimension $(p + 1)$);
        \item for each $p = m - 1$, the map
          \[
            \tilde{\alpha}_{\vectone_{p + 1}} \colon A(\vectone_{p + 1}) \rightarrow A(\vectone_p) \times_{B(\vectone_p)} B(\vectone_{p + 1}) \times_{B(\vectone_p)} A(\vectone_p)
          \]
           is injective (this gives faithfulness at dimension $m$).
      \end{itemize}
  \end{defn}

  Note that the definition of contractibility above is only concerned with the effect of $A$ and $B$ on $\vectone_p$.  The set $A(\vectone_p)$ is the set of ``globular $p$-cells'', i.e. $p$-cells in $A$ that are one $1$-cell long, one $2$-cell high, etc.; there are no cells composed end-to-end (and similarly for $B$).

  We now give the construction of the Segal maps.  In the nerve condition for categories we considered composable strings of $k$ morphisms for every $k \in \mathbb{N}$; here we consider, for every $0 \leq m \leq n$, the composable strings of $k$ $m$-cells for every $k$ and every composite of $m$-cells.

  Let $A \colon (\Theta^n)^{\op} \rightarrow \Set$ be an $n$-precategory.  Then, for all $1 \leq m \leq n$, and all $\vectk = (k_1, \dotsc, k_{m - 1})$, we have a functor
    \begin{align*}
      A(\vectk, -, -): \Delta^{\op} & \rightarrow [(\Theta^{n - m})^{\op}, \Set] \\
      [k] & \mapsto A(\vectk, k, -),
    \end{align*}
  with the effect on morphisms given by composition.

  Consider, as in Subsection~\ref{subsect:nercat}, the following diagram in $\Delta$:
     \[
       \xymatrix@C=12pt@R=12pt{
       &&&&& [k] \\
       \\
       [1] \ar[uurrrrr]^(.3){\iota_1} && [1] \ar[uurrr]^(.25){\iota_2} && [1] \ar[uur]^(.25){\iota_3} && \dotsc && [1] \ar[uulll]_(.25){\iota_{k - 1}} && [1] \ar[uulllll]_(.3){\iota_{k}} \\
       & [0] \ar[ul]^{\tau} \ar[ur]_{\sigma} && [0] \ar[ul]^{\tau} \ar[ur]_{\sigma} &&&&&& [0] \ar[ul]^{\tau} \ar[ur]_{\sigma}
                }
     \]

  Applying the functor $A(\vectk, -, -)$ to this diagram gives us the following diagram in $[(\Theta^{n - m})^{\op}, \Set]$:
     \[
      \xymatrix@C=-6pt@R=12pt{
        &&& A(\vectk,k,-) \ar[ddlll]_(.7){i_1} \ar[ddl]_(.75){i_2} \ar[ddr]^(.75){i_{k - 1}} \ar[ddrrr]^(.7){i_k} \\
        \\
        A(\vectk,1,-) \ar[dr]_{t} && A(\vectk,1,-) \ar[dl]^{s} & \dotsc & A(\vectk,1,-) \ar[dr]_{t} && A(\vectk,1,-) \ar[dl]^{s} \\
        & A(\vectk,0,-) &&&& A(\vectk,0,-)
               }
     \]
  and this is a cone over the diagram:
    \[
      \overbrace{
      \xymatrix@C=-1pt@R=12pt{
        A(\vectk,1,-) \ar[dr]_{t} && A(\vectk,1,-) \ar[dl]^{s} & \dotsc & A(\vectk,1,-) \ar[dr]_{t} && A(\vectk,1,-) \ar[dl]^{s} \\
        & A(\vectk,0,-) &&&& A(\vectk,0,-)
               }
               }^k
    \]

  Since $\Set$ is complete, $[(\Theta^{n - m})^{\op}, \Set]$ is complete, so we can take the limit of this diagram, denoted
  \[
    A(\vectk, 1, -) \times_{A(\vectk, 0, -)} \dotsb \times_{A(\vectk, 0, -)} A(\vectk, 1, -),
  \]
  called a ``wide pullback''.  The universal property of this wide pullback induces a unique morphism such that the diagram
    \[
      \xy
        % POINTS
        (0, 8)*+{A(\vectk,k,-)}="Akk";
        (0, -16)*+{A(\vectk,1,-) \times_{A(\vectk,0,-)} \dotsb \times_{A(\vectk,0,-)} A(\vectk,1,-)}="wpb";
        (-54, -32)*+{A(\vectk,1,-)}="Ak11";
        (-36, -32)*+{A(\vectk,1,-)}="Ak12";
        (36, -32)*+{A(\vectk,1,-)}="Ak13";
        (54, -32)*+{A(\vectk,1,-)}="Ak14";
        (-45, -48)*+{A(\vectk,0,-)}="Ak01";
        (45, -48)*+{A(\vectk,0,-)}="Ak02";
        % ARROWS
        {\ar@{-->}^{S_{\vectk, k}} "Akk" ; "wpb"};
        {\ar@/_3.5pc/_{i_1} "Akk" ; "Ak11"};
        {\ar@/_3pc/_{i_2} "Akk" ; "Ak12"};
        {\ar@/^3pc/^{i_{k - 1}} "Akk" ; "Ak13"};
        {\ar@/^3.5pc/^{i_k} "Akk" ; "Ak14"};
        {\ar "wpb" ; "Ak11"};
        {\ar "wpb" ; "Ak12"};
        {\ar "wpb" ; "Ak13"};
        {\ar "wpb" ; "Ak14"};
        {\ar_t "Ak11" ; "Ak01"};
        {\ar^s "Ak12" ; "Ak01"};
        {\ar_t "Ak13" ; "Ak02"};
        {\ar^s "Ak14" ; "Ak02"};
      \endxy
    \]
  commutes.  The maps $S_{\vectk, k}$, for all $\vectk = (k_1, \dotsc, k_{m - 1})$ and all $k \in \mathbb{N}$, are called the \emph{Segal maps}.

  We now give Simpson's variant of Tamsamani's definition of weak $n$-category.

  \begin{defn} \label{defn:simncat}
    Let $n \in \mathbb{N}$. A Tamsamani--Simpson weak $n$-category is an $n$-precategory $A \colon (\Theta^n)^{\op} \rightarrow \Set$ such that, for all $1 \leq m \leq n$, $\vectk = (k_1, \dotsc, k_{m - 1}) \in \Delta^m$, and $[k] \in \Delta$, the Segal map
      \[
        S_{\vectk, k}: A(\vectk, k, -) \rightarrow A(\vectk, 1, -) \times_{A(\vectk, 0, -)} \dotsb \times_{A(\vectk, 0, -)} A(\vectk, 1, -)
      \]
    is contractible.
  \end{defn}

\chapter{Nerves of Penon weak $2$-categories} \label{chap:nerveconstr}

  This chapter concerns our nerve construction for Penon weak $2$-categories.  We begin by recalling a nerve construction for bicategories due to Leinster~\cite{Lei02} in which the nerve of a bicategory is a $2$-precategory, and we prove that this nerve is a Tamsamani--Simpson weak $2$-category (a result previously stated without proof by Leinster).  We then define our own nerve functor for Penon weak $2$-categories using Leinster's nerve construction for bicategories as a prototype. Finally, we prove that the nerve of a Penon weak $2$-category is a Tamsamani--Simpson weak $2$-category.

  Note that, throughout this chapter, given a $2$-precategory
    \[
      A \colon (\Theta^2)^{\op} \longrightarrow \Set,
    \]
  we refer to an element of the set $A(j, k)$, for $j$, $k \in \mathbb{N}$, as a ``$(j, k)$-cell''.

\section{Leinster's nerve construction for bicategories}  \label{sect:Leinsternerves}

  In this section we describe and expand upon a nerve construction for bicategories originally given by Leinster in \cite{Lei02}, which will serve as a prototype for our nerve construction for Penon weak $n$-categories in Section~\ref{sect:2nerveconstr}.  This nerve construction takes a bicategory and produces from it a $2$-precategory as its nerve.  Leinster defines this nerve construction only on objects; we extend the construction to a definition of a nerve functor
    \[
      \nerve \colon \Bicat \longrightarrow [(\Theta^2)^{\op}, \Set]
    \]
  by describing the action on morphisms.  Leinster states without proof that the nerve $\nerve \mathcal{B}$ satisfies the Segal condition, and is thus a Tamsamani--Simpson weak $2$-category; we prove this result for the first time.  Apart from the definition of the action of $\nerve$ on objects (Definition~\ref{defn:Leinsterbicatnerve}) and an unpacked version of the definition of Tamsamani--Simpson weak $n$-category in the case $n = 2$ (Definition~\ref{defn:2catv1}), everything in this section is new.

  Before stating the formal definition of this nerve functor, we discuss the shapes of the simplicial cells in the nerve.  The reason for giving this explanation is that the nerve construction is somewhat notation-heavy since it is defined using a completely elementary method; for a bicategory $\mathcal{B}$, the set of $(j, k)$-cells in the nerve of $\mathcal{B}$ is defined by simply specifying the cells of $\mathcal{B}$ which, together, make up a $(j, k)$-cell in the nerve of $\mathcal{B}$.  This explanation of shapes of cells also helps motivate the shapes of cells used in our nerve construction for Penon weak $n$-categories.

  Recall that, for all $k > 0$, $0 \leq i \leq k$, we have a map $d_i \colon [k - 1] \rightarrow [k]$ in $\Delta$ given by
    \[
      d_{i}(j) =
            \left\{ \begin{array}{ll}
                j & \text{if } j < i, \\
                j + 1 & \text{if } j \geq i.
            \end{array} \right.
    \]
  In the nerve of a category $\nerve \mathcal{C}$, a simplicial $k$-cell consists of a string of $k$ composable morphisms, and the face maps $\nerve \mathcal{C}(d_i)$ are defined either to omit a single cell at one end of this string, or to compose a single pair of cells within the string.  One would expect the definition of a $(k, 0)$-cell in the nerve of a bicategory to be similar; however, we cannot define these face maps in exactly the same way, since composition of $1$-cells in a bicategory is not associative.  We now explain why this causes problems.

  Suppose we define a $(k, 0)$-cell in the nerve of a bicategory to consist just of a string of $k$ composable morphisms, which we write as $(f_1, f_2, \dotsc, f_k)$, with the face maps defined using composition in the same way as in the nerve of a category.  In $\Delta^2$, the diagram
    \[
      \xy
        % POINTS
        (0, 0)*+{(3, 0)}="A";
        (20, 0)*+{(2, 0)}="B";
        (0, -16)*+{(2, 0)}="C";
        (20, -16)*+{(1, 0)}="D";
        % ARROWS
        {\ar_-{(d_1, 1)} "B" ; "A"};
        {\ar^{(d_2, 1)} "C" ; "A"};
        {\ar_{(d_1, 1)} "D" ; "B"};
        {\ar^-{(d_1, 1)} "D" ; "C"};
      \endxy
    \]
  commutes.  Write $\nerve \mathcal{B}$ for the nerve of $\mathcal{B}$; then, in order for $\nerve \mathcal{B}$ to be a bisimplicial set, the diagram
    \[
      \xy
        % POINTS
        (0, 0)*+{\nerve \mathcal{B}(3, 0)}="A";
        (30, 0)*+{\nerve \mathcal{B}(2, 0)}="B";
        (0, -20)*+{\nerve \mathcal{B}(2, 0)}="C";
        (30, -20)*+{\nerve \mathcal{B}(1, 0)}="D";
        % ARROWS
        {\ar^-{\nerve \mathcal{B}(d_1, 1)} "A" ; "B"};
        {\ar_{\nerve \mathcal{B}(d_2, 1)} "A" ; "C"};
        {\ar^{\nerve \mathcal{B}(d_1, 1)} "B" ; "D"};
        {\ar_-{\nerve \mathcal{B}(d_1, 1)} "C" ; "D"};
      \endxy
    \]
  must commute in $\Set$.  However, consider a $(3, 0)$-cell $(f, g, h) \in \nerve \mathcal{B}(3, 0)$.  Applying the maps along the top and right of the diagram above gives
    \[
      \xymatrix{
        (f, g, h) \ar@{|->}[rr]^-{\nerve \mathcal{B}(d_1, 1)} && (g \comp f, h) \ar@{|->}[rr]^-{\nerve \mathcal{B}(d_1, 1)} && (h \comp (g \comp f)),
               }
    \]
  whereas applying the maps along the left and bottom of the diagram gives
    \[
      \xymatrix{
        (f, g, h) \ar@{|->}[rr]^-{\nerve \mathcal{B}(d_2, 1)} && (f, h \comp g) \ar@{|->}[rr]^-{\nerve \mathcal{B}(d_1, 1)} && ((h \comp g) \comp f),
               }
    \]
  so the diagram does not commute.

  Thus instead we define a $(k, 0)$-cell in the nerve of a bicategory to consist not only of a string of $k$ composable $1$-cells, but of a whole $k$-simplex of $1$-cells and isomorphism $2$-cells; the data for each $(k, 0)$-cell includes all of its faces, not just those which make up the composable string of $1$-cells.  For example, a $(2, 0)$-cell looks like:
    \[
      \xy
        % POINTS
        (0, 10)*+{a_1}="a1";
        (-10, -10)*+{a_0}="a0";
        (10, -10)*+{a_2}="a2";
        (0,0)*{}="i0";
        (0, -5)*{}="i1";
        % ARROWS
        {\ar^{f_{01}} "a0" ; "a1"};
        {\ar^{f_{12}} "a1" ; "a2"};
        {\ar_{f_{02}} "a0" ; "a2"};
        {\ar@{=>}^{i_{012}}_{\sim} "i0" ; "i1"}
      \endxy
    \]
  This should be thought of as a pair of composable $1$-cells, together with another $1$-cell that would be a ``valid choice'' for their composite (but not necessarily their actual composite in the bicategory).

  Similarly, a $(3, 0)$-cell looks like
  \[
    \xy
      % POINTS
      (-50, 15)*+{a_1}="a1";
      (-30, 15)*+{a_2}="a2";
      (-60, -5)*+{a_0}="a0";
      (-20, -5)*+{a_3}="a3";
      (-5, 5)*+{=};
      (20, 15)*+{a_1}="a1r";
      (40, 15)*+{a_2}="a2r";
      (10, -5)*+{a_0}="a0r";
      (50, -5)*+{a_3}="a3r";
      % INVISIBLE POINTS
      (-48, 13)*+{}="i0s";
      (-43, 8)*+{}="i0t";
      (38, 13)*+{}="i1s";
      (33, 8)*+{}="i1t";
      (-32, 5)*+{}="i2s";
      (-34, -1)*+{}="i2t";
      (22, 5)*+{}="i3s";
      (24, -1)*+{}="i3t";
      % ARROWS
      {\ar^{f_{01}} "a0" ; "a1"};
      {\ar^{f_{12}} "a1" ; "a2"};
      {\ar^{f_{23}} "a2" ; "a3"};
      {\ar_{f_{03}} "a0" ; "a3"};
      {\ar_{f_{02}} "a0" ; "a2"};
      {\ar^{f_{01}} "a0r" ; "a1r"};
      {\ar^{f_{12}} "a1r" ; "a2r"};
      {\ar^{f_{23}} "a2r" ; "a3r"};
      {\ar_{f_{03}} "a0r" ; "a3r"};
      {\ar_{f_{13}} "a1r" ; "a3r"};
      % 2-CELLS:
      {\ar@{=>}_{\sim}^{\iota_{012}} "i0s" ; "i0t"};
      {\ar@{=>}^{\sim}_{\iota_{123}} "i1s" ; "i1t"};
      {\ar@{=>}^{\sim}_{\iota_{023}} "i2s" ; "i2t"};
      {\ar@{=>}_{\sim}^{\iota_{013}} "i3s" ; "i3t"};
    \endxy
  \]
  i.e. a commuting tetrahedron whose faces are isomorphism $2$-cells.

  The $(j, k)$-cells in the nerve, for $k > 0$, are ``simplicially weakened'' versions cuboidal pasting diagrams. We usually draw these as grids of $2$-cells; for example, we draw a $(3, 2)$-cell as:
    \[
      \xymatrix@R=0pt{
        & \big{\Downarrow} \alpha^1_{01} && \big{\Downarrow} \alpha^1_{12} && \big{\Downarrow} \alpha^1_{23}  \\
        a_0 \ar@/^2.75pc/[rr]^{f^0_{01}} \ar[rr]^(0.3){f^1_{01}} \ar@/_2.75pc/[rr]_{f^2_{01}} && a_1 \ar@/^2.75pc/[rr]^{f^0_{12}} \ar[rr]^(0.3){f^1_{12}} \ar@/_2.75pc/[rr]_{f^2_{12}} && a_2 \ar@/^2.75pc/[rr]^{f^0_{23}} \ar[rr]^(0.3){f^1_{23}} \ar@/_2.75pc/[rr]_{f^2_{23}} && a_3.  \\
        & \big{\Downarrow} \alpha^2_{01} && \big{\Downarrow} \alpha^2_{12} && \big{\Downarrow} \alpha^2_{23}  \\
               }
    \]
  However, such diagrams are misleading since they do not capture the whole simplicial shape of the cell.  In fact, each string of $k$ composable $1$-cells on the same ``level'' (i.e. with the same superscript) is a $(k, 0)$-cell, and all diagrams of $2$-cells within each $(j, k)$-cell commute.

  Note that the notation used in the diagrams above is the notation we use in the construction.  The subscripts and superscripts decorating each cell should be thought of as the coordinates of that cell, with the subscripts giving the horizontal coordinates, and superscripts giving the vertical coordinates.

  We state the definition of this nerve functor for bicategories in three parts.  In Definition~\ref{defn:Leinsterbicatnerve} we define, for a bicategory $\mathcal{B}$ and for each object $(j, k)$ in $\Theta^2$, a set $\NB(j, k)$, which is the set of $(j, k)$-cells in the nerve of $\mathcal{B}$;  this is the only part of the definition that Leinster states formally in \cite{Lei02}.  Then, in Definition~\ref{defn:Leinsterbicatnerveonmaps}, we extend this to a definition of a $2$-precategory
    \[
      \NB \colon (\Theta^2)^{\op} \longrightarrow \Set
    \]
  by defining the action of this presheaf on maps.  This defines the action of the nerve functor
    \[
      \nerve \colon \Bicat \longrightarrow [(\Theta^2)^{\op}, \Set].
    \]
  on objects; in Definition~\ref{defn:Leinsterbicatnervefunctor} we define the action of this functor on maps.

  Recall that an object of $\Theta^2$ is an equivalence class or objects of $\Delta^2$.  An object of $\Delta^2$ is in an equivalence class with more than one member if and only if it is of the form $(0, k)$.  Thus, for the purposes of the following definition we treat the equivalence class of $(0, k)$ as the object $(0, 0)$ of $\Delta^2$; all other equivalence classes are treated as their sole member.  Note that the exact choice of representative does not make a difference to the definition.

  Note that, ideally, we would give an abstract definition of the nerve of a bicategory by first defining a functor $i \colon \Theta^2 \rightarrow \Bicat$, then defining the nerve of a bicategory $\mathcal{B}$ to be given by $\Bicat(i(-), \mathcal{B})$, as we did when defining the nerve of a category.  However, this is not practical in the case of bicategories since the bicategories in the image of the functor $i$ are difficult to describe (in particular, they are not free, unlike in the case of the nerve of a category).  We believe that describing these bicategories would require extra machinery (for example, we believe it could be done using computads) and is thus beyond the scope of this thesis.  Note that this is one of the reasons for using Penon weak $n$-categories in the remainder of the thesis; in the case of Penon weak $n$-categories we are able to construct the nerve in this abstract way, by modifying the construction of a free Penon weak $n$-category, in a way that is not possible with bicategories.  We do this in Sections~\ref{sect:2nerveconstr} and \ref{sect:nnerveconstr}.

  In the original version of this definition in \cite{Lei02}, Leinster used $\Delta^2$ rather than $\Theta^2$, and then enforced a non-cubical condition, as described immediately after the definition of $\Theta^n$, Definition~\ref{defn:Theta}.  This is the only non-cosmetic difference between the following definition and Leinster's original definition; we also write the axioms as diagrams rather than equations.

  \begin{defn}  \label{defn:Leinsterbicatnerve}
    Let $\mathcal{B}$ be a bicategory.  We associate to $\mathcal{B}$ a $2$-precategory $\nerve \mathcal{B} \colon (\Theta^2)^{\op} \rightarrow \Set$, called the \emph{nerve of $\mathcal{B}$}, as follows:

    Given $(j,k) \in \Theta^2$, $\nerve \mathcal{B}(j,k)$ is the set which has as its elements all quadruples
      \[
        \left( (a_u)_{0 \leq u \leq j}, (f^z_{uv})_{\substack{0 \leq u < v \leq j \\ 0 \leq z \leq k}}, (\alpha^z_{uv})_{\substack{0 \leq u < v \leq j \\ 1 \leq z \leq k}}, (\iota^z_{uvw})_{\substack{0 \leq u < v < w \leq j \\ 0 \leq z \leq k}} \right)
      \]
    where
      \begin{itemize}
        \item each $a_u$ is an object of $\mathcal{B}$;
        \item each $f^z_{uv}: a_u \rightarrow a_v$ is a $1$-cell of $\mathcal{B}$;
        \item each $\alpha^z_{uv}: f^{z-1}_{uv} \rightarrow f^z_{uv}$ is a $2$-cell of $\mathcal{B}$;
        \item each $\iota^z_{uvw}: f^z_{vw} \comp f^z_{uv} \rightarrow f^z_{uw}$ is an isomorphism $2$-cell of $\mathcal{B}$, with inverse $(\iota^z_{uvw})^{-1}$;
      \end{itemize}
    and these cells satisfy the following axioms:
      \begin{itemize}
        \item for all $0 \leq u < v < w \leq j$, $1 \leq z \leq k$, the diagram
          \[
            \xy
              % POINTS
              (0, 0)*+{f^{z-1}_{vw} \comp f^{z-1}_{uv}}="0,0";
              (24, 0)*+{f^{z-1}_{uw}}="1,0";
              (0, -16)*+{f^{z}_{vw} \comp f^{z}_{uv}}="0,1";
              (24, -16)*+{f^{z}_{uw}}="1,1";
              % ARROWS
              {\ar^-{\iota^{z-1}_{uvw}} "0,0" ; "1,0"};
              {\ar_{\alpha^z_{vw} * \alpha^z_{uv}} "0,0" ; "0,1"};
              {\ar^{\alpha^z_{uw}} "1,0" ; "1,1"};
              {\ar_-{\iota^z_{uvw}} "0,1" ; "1,1"};
            \endxy
          \]
          commutes; alternatively, we can draw this axiom as
      \[
        \xy
          % POINTS
          (0, 0)*+{\bullet}="0,0";
          (32, 0)*+{\bullet}="2,0";
          (16, 16)*+{\bullet}="1,-1";
          (40, 0)*+{=};
          (48, 0)*+{\bullet}="3,0";
          (64, 0)*+{\bullet}="4,0";
          (80, 0)*+{\bullet}="5,0";
          % ARROWS
          {\ar@/^1.5pc/ "0,0" ; "2,0"};
          {\ar@/_1.5pc/ "0,0" ; "2,0"};
          {\ar@/^0.75pc/ "0,0" ; "1,-1"};
          {\ar@/^0.75pc/ "1,-1" ; "2,0"};
          {\ar@{=>}^{\iota^{z - 1}_{uvw}} (16, 14) ; (16, 8)};
          {\ar@{=>}^{\alpha^z_{uw}} (16, 3) ; (16, -3)};
          {\ar@/^1.5pc/ "3,0" ; "4,0"};
          {\ar@/_1.5pc/ "3,0" ; "4,0"};
          {\ar@/^1.5pc/ "4,0" ; "5,0"};
          {\ar@/_1.5pc/ "4,0" ; "5,0"};
          {\ar@/_3pc/ (48, -2) ; (80, -2)};
          {\ar@{=>}^{\alpha^z_{uv}} (56, 3) ; (56, -3)};
          {\ar@{=>}^{\alpha^z_{vw}} (72, 3) ; (72, -3)};
          {\ar@{=>}^{\iota^{z}_{uvw}} (64, -6) ; (64, -12)};
        \endxy
      \]
        \item for all $0 \leq u < v < w < x \leq j$, $0 \leq z \leq k$, the diagram
          \[
            \xy
              % POINTS
              (-2, 0)*+{(f^z_{wx} \comp f^z_{vw}) \comp f^z_{uv}}="0,0";
              (38, 0)*+{f^z_{wx} \comp (f^z_{vw} \comp f^z_{uv})}="1,0";
              (-10,-18)*+{f^z_{vx} \comp f^z_{uv}}="0,1";
              (46, -18)*+{f^z_{wx} \comp f^z_{uw}}="1,1";
              (18, -32)*+{f^z_{ux}}="0,2";
              % ARROWS
              {\ar^{s_{uvwx}} "0,0" ; "1,0"};
              {\ar_{\iota^z_{vwx} * 1_{f^z_{uv}}} "0,0" ; "0,1"};
              {\ar^{1_{f^z_{wx}} * \iota^z_{uvw}} "1,0" ; "1,1"};
              {\ar_{\iota^z_{uvx}} "0,1" ; "0,2"};
              {\ar^{\iota^z_{uwx}} "1,1" ; "0,2"};
            \endxy
          \]
          commutes, where
              \[
                s_{uvwx}:(f^z_{wx} \circ f^z_{vw}) \circ f^z_{uv} \rightarrow f^z_{wx} \circ (f^z_{vw} \circ f^z_{uv})
              \]
            is the component of the appropriate associativity isomorphism for $\mathcal{B}$; alternatively, we can draw this axiom as
  \[
    \xy
      % POINTS
      (-50, 15)*+{a_v}="a1";
      (-30, 15)*+{a_w}="a2";
      (-60, -5)*+{a_u}="a0";
      (-20, -5)*+{a_x}="a3";
      (-9, 5)*+{=};
      (12, 15)*+{a_v}="a1r";
      (32, 15)*+{a_w}="a2r";
      (2, -5)*+{a_u}="a0r";
      (42, -5)*+{a_x}="a3r";
      % INVISIBLE POINTS
      (-48, 13)*+{}="i0s";
      (-43, 8)*+{}="i0t";
      (30, 13)*+{}="i1s";
      (25, 8)*+{}="i1t";
      (-32, 5)*+{}="i2s";
      (-34, -1)*+{}="i2t";
      (14, 5)*+{}="i3s";
      (16, -1)*+{}="i3t";
      % ARROWS
      {\ar^{f^z_{uv}} "a0" ; "a1"};
      {\ar^{f^z_{vw}} "a1" ; "a2"};
      {\ar^{f^z_{wx}} "a2" ; "a3"};
      {\ar_{f^z_{ux}} "a0" ; "a3"};
      {\ar_(0.4){f^z_{uw}} "a0" ; "a2"};
      {\ar^{f^z_{uv}} "a0r" ; "a1r"};
      {\ar^{f^z_{vw}} "a1r" ; "a2r"};
      {\ar^{f^z_{wx}} "a2r" ; "a3r"};
      {\ar_{f^z_{ux}} "a0r" ; "a3r"};
      {\ar_(0.6){f^z_{vx}} "a1r" ; "a3r"};
      % 2-CELLS:
      {\ar@{=>}_{\sim}^{\iota_{uvw}} "i0s" ; "i0t"};
      {\ar@{=>}^{\sim}_{\iota_{vwx}} "i1s" ; "i1t"};
      {\ar@{=>}^{\sim}_{\iota_{uwx}} "i2s" ; "i2t"};
      {\ar@{=>}_{\sim}^{\iota_{uvx}} "i3s" ; "i3t"};
    \endxy
  \]
      \end{itemize}
  \end{defn}

  In \cite{Lei02} Leinster does not define the action of the presheaf $\NB$ on morphisms formally; he simply states that $\NB$ is ``defined on maps by a combination of inserting identities and forgetting data''.  We now explain this idea in more detail, then make it precise in the next definition.  Given a map $(p, q) \colon (l, m) \rightarrow (j, k)$ in $\Theta^2$, we define a map
      \[
        \nerve \mathcal{B}(p, q) \colon \nerve \mathcal{B}(j, k) \rightarrow \nerve \mathcal{B}(l, m).
      \]
  To understand what this map does, recall that an element of $\NB(j, k)$ consists of a collection of cells of $\mathcal{B}$ which form a $(j, k)$-cell, and that each of these cells has subscripts and (in some cases) superscripts which we think of as the coordinates of this cell within the $(j, k)$-cell.  Given an element of $\NB(j, k)$, its image under $\NB(p, q)$ is the element of $\NB(l, m)$ made up of those cells whose horizontal coordinates are in the image of $p$ and, where appropriate, whose vertical coordinate is in the image of $q$; any cells whose coordinates are not in the images of $p$ and $q$ are omitted, and cells with repeated coordinates are taken to be identities (or unitors in some cases).

  \begin{defn}  \label{defn:Leinsterbicatnerveonmaps}
    Let $\mathcal{B}$ be a bicategory, and write $l$ and $r$ for its left and right unitors respectively.  Let $(p, q) \colon (l, m) \rightarrow (j, k)$ be a map in $\Theta^2$.  We define a function of sets
      \[
        \nerve \mathcal{B}(p, q) \colon \nerve \mathcal{B}(j, k) \rightarrow \nerve \mathcal{B}(l, m)
      \]
    as follows:
      \begin{align*}
        \nerve \mathcal{B}(p, q) & : \left( (a_u)_{0 \leq u \leq j}, (f^z_{uv})_{\substack{0 \leq u < v \leq j \\ 0 \leq z \leq k}}, (\alpha^z_{uv})_{\substack{0 \leq u < v \leq j \\ 1 \leq z \leq k}}, (\iota^z_{uvw})_{\substack{0 \leq u < v < w \leq j \\ 0 \leq z \leq k}} \right) \\
        & \longmapsto \left( (b_u)_{0 \leq u \leq l}, (g^z_{uv})_{\substack{0 \leq u < v \leq l \\ 0 \leq z \leq m}}, (\beta^z_{uv})_{\substack{0 \leq u < v \leq l \\ 1 \leq z \leq m}}, (\kappa^z_{uvw})_{\substack{0 \leq u < v < w \leq l \\ 0 \leq z \leq m}} \right)
      \end{align*}
    where
      \begin{itemize}
        \item $b_u = a_{p(u)}$
        \item $g^z_{uv} =
            \left\{ \begin{array}{ll}
                f^{q(z)}_{p(u)p(v)} & \text{if } p(u) \neq p(v), \\
                \id_{a_{p(u)}} & \text{if } p(u) = p(v);
            \end{array} \right. $
        \item $\beta^z_{uv} =
            \left\{ \begin{array}{ll}
                \alpha^{q(z)}_{p(u)p(v)} & \text{if } p(u) \neq p(v), q(z - 1) \neq q(z), \\
                \id_{f^{q(z)}_{p(u)p(v)}} & \text{if } p(u) \neq p(v), q(z - 1) = q(z), \\
                \id_{\id_{a_{p(u)}}} & \text{if } p(u) = p(v);
            \end{array} \right. $
        \item $\kappa^z_{uvw} =
            \left\{ \begin{array}{ll}
                \iota^{q(z)}_{p(u)p(v)p(w)} & \text{if } p(u) \neq p(v) \neq p(w), \\
                l_{f^{q(z)}_{p(u)p(v)}} & \text{if } p(u) \neq p(v) = p(w), \\
                r_{f^{q(z)}_{p(u)p(v)}} & \text{if } p(u) = p(v) \neq p(w), \\
                \id_{\id_{a_{p(u)}}} & \text{if } p(u) = p(v) = p(w).
            \end{array} \right.$
      \end{itemize}
  \end{defn}

  This defines the action of the nerve functor on objects; we now give a new definition which extends this to a definition of a nerve functor
    \[
      \nerve \colon \Bicat \longrightarrow [(\Theta^2)^{\op}, \Set],
    \]
  by describing the action of this functor on morphisms.

  \begin{defn}  \label{defn:Leinsterbicatnervefunctor}
    Let $F \colon \mathcal{A} \rightarrow \mathcal{B}$ be a strict functor of bicategories.  We define a map of bisimplicial sets $\nerve F \colon \nerve \mathcal{A} \rightarrow \nerve \mathcal{B}$ to be the map whose component $\nerve F_{(j, k)} \colon \nerve \mathcal{A}(j, k) \rightarrow \nerve \mathcal{B}(j, k)$, for each $(j, k) \in \Delta^2$, is given by
      \begin{align*}
        \nerve F_{(j,k)} & \left( (a_u)_{0 \leq u \leq j}, (f^z_{uv})_{\substack{0 \leq u < v \leq j \\ 0 \leq z \leq k}}, (\alpha^z_{uv})_{\substack{0 \leq u < v \leq j \\ 1 \leq z \leq k}}, (\iota^z_{uvw})_{\substack{0 \leq u < v < w \leq j \\ 0 \leq z \leq k}} \right) \\
        = & \left( (F(a_u))_{0 \leq u \leq j}, (Ff^z_{uv})_{\substack{0 \leq u < v \leq j \\ 0 \leq z \leq k}}, (F\alpha^z_{uv})_{\substack{0 \leq u < v \leq j \\ 1 \leq z \leq k}}, (F\iota^z_{uvw})_{\substack{0 \leq u < v < w \leq j \\ 0 \leq z \leq k}} \right).
      \end{align*}

    The above defines a functor $\nerve \colon \Bicat \rightarrow [(\Theta^2)^{\op}, \Set]$, called the \emph{nerve functor}.
  \end{defn}

  In \cite{Lei02}, Leinster stated without proof that the nerve of a bicategory satisfies the Segal condition, and is thus a Tamsamani--Simpson weak $2$-category.  We will prove this for the first time; before doing so, we recall the definition of Tamsamani--Simpson weak $n$-category (Definition~\ref{defn:simncat}) in the case $n = 2$; the following is a slight unpacking of the definition, which treats Segal maps of the forms $S_k$ and $S_{j, k}$ separately.

  \begin{defn} \label{defn:2catv1}
    A \emph{Tamsamani--Simpson weak $2$-category} is a functor
      \[
        A \colon (\Theta^2)^{\op} \rightarrow \Set
      \]
    such that
      \begin{enumerate}[(i)]
        \item for each $k \geq 0$, the Segal map
          \[
            S_k: A(k,-) \longrightarrow \underbrace{A(1,-) \times_{A(0,1)} \dotsb \times_{A(0,1)} A(1,-)}_k
          \]
              is contractible, i.e. it is surjective on objects, and full and faithful on $1$-cells;
        \item for each $m$, $k \geq 0$, the Segal map
          \[
            S_{j,k}: A(j,k) \longrightarrow \underbrace{A(j,1) \times_{A(j,0)} \dotsb \times_{A(j,0)} A(j,1)}_k
          \]
              is a bijection.
      \end{enumerate}
  \end{defn}

  Thus to prove that the nerve of a bicategory is a Tamsamani--Simpson weak $2$-category, we break this statement down into four propositions: one stating that each of the Segal maps $S_{j, k}$ is a bijection, and the other three stating the three conditions required for contractibility of the Segal maps $S_k$.

  \begin{prop}  \label{prop:bicatSegalbij}
    Let $\mathcal{B}$ be a bicategory.  For all $j$, $k \geq 0$, the Segal map
          \[
            S_{j,k}: \NB(j,k) \longrightarrow \underbrace{\NB(j,1) \times_{\NB(j,0)} \dotsb \times_{\NB(j,0)} \NB(j,1)}_k
          \]
    is a bijection.
  \end{prop}

  \begin{proof}
    Let
      \[
        \left( (a_u)_{0 \leq u \leq j}, (f^z_{uv})_{\substack{0 \leq u < v \leq j \\ 0 \leq z \leq k}}, (\alpha^z_{uv})_{\substack{0 \leq u < v \leq j \\ 1 \leq z \leq k}}, (\iota^z_{uvw})_{\substack{0 \leq u < v < w \leq j \\ 0 \leq z \leq k}} \right)
      \]
    be an element of $\NB(j, k)$.  The function $S_{j, k}$ maps this to
      \begin{align*}
        \bigg( & \Big( (a_u)_{0 \leq u \leq j}, (f^z_{uv})_{\substack{0 \leq u < v \leq j \\ 0 \leq z \leq 1}}, (\alpha^1_{uv})_{\substack{0 \leq u < v \leq j}}, (\iota^z_{uvw})_{\substack{0 \leq u < v < w \leq j \\ 0 \leq z \leq 1}} \Big),  \\
        & \Big( (a_u)_{0 \leq u \leq j}, (f^z_{uv})_{\substack{0 \leq u < v \leq j \\ 1 \leq z \leq 2}}, (\alpha^2_{uv})_{\substack{0 \leq u < v \leq j}}, (\iota^z_{uvw})_{\substack{0 \leq u < v < w \leq j \\ 1 \leq z \leq 2}} \Big),  \\
        & \dotsc,  \\
        & \Big( (a_u)_{0 \leq u \leq j}, (f^z_{uv})_{\substack{0 \leq u < v \leq j \\ k - 1 \leq z \leq k}}, (\alpha^k_{uv})_{\substack{0 \leq u < v \leq j}}, (\iota^z_{uvw})_{\substack{0 \leq u < v < w \leq j \\ k - 1 \leq z \leq k}} \Big) \bigg).
      \end{align*}
    Every cell listed in the original element of $\NB(j, k)$ is listed in its image under $S_{j, k}$, so this function is injective.  Furthermore, any element of the wide pullback
      \[
        \underbrace{\NB(j,1) \times_{\NB(j,0)} \dotsb \times_{\NB(j,0)} \NB(j,1)}_k
      \]
    can be written in the form above.  Thus $S_{j, k}$ is surjective.

    Hence $S_{j, k}$ is a bijection.
  \end{proof}

  \begin{prop}  \label{prop:bicatSegalsurj}
    Let $\mathcal{B}$ be a bicategory.  For all $k \geq 0$, the Segal map
          \[
            S_k: \NB(k,-) \longrightarrow \underbrace{\NB(1,-) \times_{\NB(0,0)} \dotsb \times_{\NB(0,0)} \NB(1,-)}_k
          \]
    is surjective on objects.
  \end{prop}

  \begin{proof}
    Let
      \[
        \Big(
        \big( \xy (0, 0)*+{a_0}="0"; (14, 0)*+{a_1}="1"; {\ar^{f^0_{01}} "0" ; "1"}; \endxy \big),
        \big( \xy (0, 0)*+{a_1}="0"; (14, 0)*+{a_2}="1"; {\ar^{f^0_{12}} "0" ; "1"}; \endxy \big),
        \dotsc,
        \big( \xy (0, 0)*+{a_{k - 1}}="0"; (14, 0)*+{a_k}="1"; {\ar^{f^0_{k - 1, k}} "0" ; "1"}; \endxy \big)
        \Big)
      \]
    be an element of
      \[
        \underbrace{A(1,0) \times_{A(0,0)} \dotsb \times_{A(0,0)} A(1,0)}_k.
      \]
    This is a string of $k$ composable $1$-cells in $\mathcal{B}$.  We seek an element of $\NB(k, 0)$ that maps to this under $(S_{k})_0$.  We define an element
      \[
        \Big( (a_u)_{0 \leq u \leq j}, (f^0_{uv})_{\substack{0 \leq u < v \leq j}}, (\iota^0_{uvw})_{\substack{0 \leq u < v < w \leq j}} \Big)
      \]
    of $\NB(k, 0)$; to do so we must define $f^0_{uv}$ for every $v > u + 1$, and we must define the $\iota^0_{uvw}$ for all $0 \leq u < v < w \leq k$.  Our approach is to define the $f^0_{uv}$'s to be composites of the $f^0_{u, u + 1}$'s, then define the $\iota^0_{uvw}$'s to be composites of constraint cells in $\mathcal{B}$ that mediate between these composites.

    Let $0 \leq u < u + 1 < v \leq j$, and define $f^0_{uv}$ to be given by the composite
      \[
        f^0_{uv} := ( \dotsb (f^0_{v - 1, v} \comp f^0_{v - 2, v - 1}) \comp \dotsb ) \comp f^0_{u, u + 1}.
      \]
    Then, for all $0 \leq u < v < w \leq j$, there is a composite of constraint isomorphism $2$-cells
      \[
        \iota^0_{uvw}: f^0_{vw} \comp f^0_{uv} \rightarrow f^0_{uw}
      \]
    in $\mathcal{B}$, which is unique by coherence for bicategories \cite{GPS95, Lei98b}.

    This defines an element of $\NB(k, 0)$; by construction we see that this element maps to
      \[
        \Big(
        \big( \xy (0, 0)*+{a_0}="0"; (14, 0)*+{a_1}="1"; {\ar^{f^0_{01}} "0" ; "1"}; \endxy \big),
        \big( \xy (0, 0)*+{a_1}="0"; (14, 0)*+{a_2}="1"; {\ar^{f^0_{12}} "0" ; "1"}; \endxy \big),
        \dotsc,
        \big( \xy (0, 0)*+{a_{k - 1}}="0"; (14, 0)*+{a_k}="1"; {\ar^{f^0_{k - 1, k}} "0" ; "1"}; \endxy \big)
        \Big)
      \]
    under $(S_{k})_0$, as required.  Hence $S_k$ is surjective on objects.
  \end{proof}

  To show that the Segal maps are full and faithful on $1$-cells, we use the fact that there is some redundancy in the definition of $\NB(j, k)$.  Specifically, to specify an element of $\NB(j, k)$ we only need to specify $\alpha^z_{uv}$ for $v = u + 1$, rather than for all $u < v < j$ (note that we still have to specify every $a_u$, $f^z_{uv}$ and $\iota^z_{uvw}$).  Since this fact is used in the proofs of both fullness and faithfulness, we state and prove it as a separate lemma:

  \begin{lemma}  \label{lem:alphas}
    Let $\mathcal{B}$ be a bicategory, let $j$, $k \in \mathbb{N}$, and suppose we have the following data:
      \begin{itemize}
        \item for all $0 \leq u \leq j$, an object $a_u$ of $\mathcal{B}$;
        \item for all $0 \leq u < v \leq j$, $0 \leq z \leq k$, a $1$-cell $f^z_{uv}: a_u \rightarrow a_v$ in $\mathcal{B}$;
        \item for all $0 \leq u < j$, $1 \leq z \leq k$, a $2$-cell $\alpha^z_{u, u + 1}: f^{z-1}_{u, u + 1} \rightarrow f^z_{u, u + 1}$ in $\mathcal{B}$;
        \item for all $0 \leq u < v < w \leq j$, $0 \leq z \leq k$, an isomorphism $2$-cell $\iota^z_{uvw}: f^z_{vw} \comp f^z_{uv} \rightarrow f^z_{uw}$ in $\mathcal{B}$, with inverse $(\iota^z_{uvw})^{-1}$;
      \end{itemize}
    such that the isomorphism $2$-cells $\iota^z_{uvw}$ satisfy the pentagon axiom from the definition of $\NB$ on objects, Definition~\ref{defn:Leinsterbicatnerve}.  Then this specifies a unique element
      \[
        \left( (a_u)_{0 \leq u \leq j}, (f^z_{uv})_{\substack{0 \leq u < v \leq j \\ 0 \leq z \leq k}}, (\alpha^z_{uv})_{\substack{0 \leq u < v \leq j \\ 1 \leq z \leq k}}, (\iota^z_{uvw})_{\substack{0 \leq u < v < w \leq j \\ 0 \leq z \leq k}} \right)
      \]
    of $\NB(j, k)$.
  \end{lemma}

  \begin{proof}
    We need to show that, for all $0 \leq u < u + 1 < v \leq j$, $1 \leq z \leq k$, there is a unique choice of $2$-cell $\alpha^z_{uv}$ in $\mathcal{B}$ such that the axioms for an element of $\NB(j, k)$ are satisfied.  We do this by strong induction over $v$.

    First, let $v = u + 2$.  For all $1 \leq z \leq k$, write $w := u + 1$, and define $\alpha^z_{uv} = \alpha^z_{u, u + 2}$ to be given by the composite
      \[
        \xy
          % POINTS
          (0, 0)*+{\bullet}="3,0";
          (20, 0)*+{\bullet}="4,0";
          (40, 0)*+{\bullet}="5,0";
          % ARROWS
          {\ar@/^1.5pc/ "3,0" ; "4,0"};
          {\ar@/_1.5pc/ "3,0" ; "4,0"};
          {\ar@/^1.5pc/ "4,0" ; "5,0"};
          {\ar@/_1.5pc/ "4,0" ; "5,0"};
          {\ar@/_3.25pc/ (0, -2) ; (40, -2)};
          {\ar@/^3.25pc/ (0, 2) ; (40, 2)};
          {\ar@{=>}^{(\iota^{z - 1}_{uwv})^{-1}} (20, 12) ; (20, 6)};
          {\ar@{=>}^{\alpha^z_{uw}} (10, 3) ; (10, -3)};
          {\ar@{=>}^{\alpha^z_{wv}} (30, 3) ; (30, -3)};
          {\ar@{=>}^{\iota^{z}_{uwv}} (20, -6) ; (20, -12)};
        \endxy
      \]
    in $\mathcal{B}$.  By considering the composite $\alpha^z_{uv} \comp \iota^{z - 1}_{uwv}$, we see that $\alpha^z_{uv}$ satisfies the square axiom from the definition of $\NB(j, k)$, Definition~\ref{defn:Leinsterbicatnerve};  furthermore, it is the only $2$-cell of $\mathcal{B}$ satisfying these axioms, given that $\alpha^z_{uw}$, $\alpha^z_{wv}$, $\iota^{z - 1}_{uwv}$ and $\iota^{z}_{uwv}$ are fixed.

    Now let $m \geq 1$ and suppose we have defined $\alpha^z_{uv}$ for all $u + 1 \leq v \leq u + m$.  We define $\alpha^z_{uv}$ for $v = u + m + 1$ as follows:  let $w$ be a natural number with $u < w < v$, and define $\alpha^z_{uv}$ to be given by the composite
      \[
        \xy
          % POINTS
          (0, 0)*+{\bullet}="3,0";
          (20, 0)*+{\bullet}="4,0";
          (40, 0)*+{\bullet}="5,0";
          % ARROWS
          {\ar@/^1.5pc/ "3,0" ; "4,0"};
          {\ar@/_1.5pc/ "3,0" ; "4,0"};
          {\ar@/^1.5pc/ "4,0" ; "5,0"};
          {\ar@/_1.5pc/ "4,0" ; "5,0"};
          {\ar@/_3.25pc/ (0, -2) ; (40, -2)};
          {\ar@/^3.25pc/ (0, 2) ; (40, 2)};
          {\ar@{=>}^{(\iota^{z - 1}_{uwv})^{-1}} (20, 12) ; (20, 6)};
          {\ar@{=>}^{\alpha^z_{uw}} (10, 3) ; (10, -3)};
          {\ar@{=>}^{\alpha^z_{wv}} (30, 3) ; (30, -3)};
          {\ar@{=>}^{\iota^{z}_{uwv}} (20, -6) ; (20, -12)};
        \endxy
      \]
    Note that the pentagon axiom from the definition of $\NB(j, k)$ ensures that this is independent of our choice of $w$.  As before, by considering the composite $\alpha^z_{uv} \comp \iota^{z - 1}_{uwv}$, we see that $\alpha^z_{uv}$ satisfies the square axiom from the definition of $\NB(j, k)$, Definition~\ref{defn:Leinsterbicatnerve};  furthermore, it is the only $2$-cell of $\mathcal{B}$ satisfying these axioms, given that $\alpha^z_{uw}$, $\alpha^z_{wv}$, $\iota^{z - 1}_{uwv}$ and $\iota^{z}_{uwv}$ are fixed.

    This defines a unique element
      \[
        \left( (a_u)_{0 \leq u \leq j}, (f^z_{uv})_{\substack{0 \leq u < v \leq j \\ 0 \leq z \leq k}}, (\alpha^z_{uv})_{\substack{0 \leq u < v \leq j \\ 1 \leq z \leq k}}, (\iota^z_{uvw})_{\substack{0 \leq u < v < w \leq j \\ 0 \leq z \leq k}} \right)
      \]
    of $\NB(j, k)$, as required.
  \end{proof}

  This now allows us to prove the Segal maps are full and faithful on $1$-cells.

  \begin{prop}  \label{prop:bicatSegalfull}
    Let $\mathcal{B}$ be a bicategory.  For all $k \geq 0$, the Segal map
          \[
            S_k: \NB(k,-) \longrightarrow \underbrace{\NB(1,-) \times_{\NB(0,0)} \dotsb \times_{\NB(0,0)} \NB(1,-)}_k
          \]
    is full on $1$-cells.
  \end{prop}

  \begin{proof}
    Suppose we have two elements $f$, $g \in \NB(k, 0)$, which we denote
      \[
        f = \Big( (a_u)_{0 \leq u \leq k}, (f^0_{uv})_{\substack{0 \leq u < v \leq k}}, (\iota^0_{uvw})_{\substack{0 \leq u < v < w \leq k}} \Big)
      \]
    and
      \[
        g =\Big( (b_u)_{0 \leq u \leq k}, (g^0_{uv})_{\substack{0 \leq u < v \leq k}}, (\kappa^0_{uvw})_{\substack{0 \leq u < v < w \leq k}} \Big),
      \]
    and suppose we have an element $\alpha$ of
      \[
        \underbrace{\NB(1,1) \times_{\NB(0,0)} \dotsb \times_{\NB(0,0)} \NB(1,1)}_k,
      \]
    with $s(\alpha) = S_k(f)$ and $t(\alpha) = S_k(g)$.  Then, for all $0 \leq u \leq k$, $a_u = b_u$, and we can write $\alpha$ as
      \[
        \alpha = \left(
        \Bigg(
        \xy
          % POINTS
          (0, 0)*+{a_0}="0";
          (20, 0)*+{a_1}="1";
          % ARROWS
          {\ar@/^1.5pc/^{f^0_{01}} "0" ; "1"};
          {\ar@/_1.5pc/_{g^0_{01}} "0" ; "1"};
          {\ar@{=>}^{\alpha^1_{01}} (8, 3) ; (8, -3)};
        \endxy
        \Bigg),
        \Bigg(
        \xy
          % POINTS
          (0, 0)*+{a_1}="0";
          (20, 0)*+{a_2}="1";
          % ARROWS
          {\ar@/^1.5pc/^{f^0_{12}} "0" ; "1"};
          {\ar@/_1.5pc/_{g^0_{12}} "0" ; "1"};
          {\ar@{=>}^{\alpha^1_{12}} (8, 3) ; (8, -3)};
        \endxy
        \Bigg),
        \dotsc,
        \Bigg(
        \xy
          % POINTS
          (0, 0)*+{a_{k - 1}}="0";
          (20, 0)*+{a_k}="1";
          % ARROWS
          {\ar@/^1.5pc/^{f^0_{k - 1, k}} "0" ; "1"};
          {\ar@/_1.5pc/_{g^0_{k - 1, k}} "0" ; "1"};
          {\ar@{=>}^{\alpha^1_{k - 1, k}} (7, 3) ; (7, -3)};
        \endxy
        \Bigg)
        \right).
      \]
    By Lemma~\ref{lem:alphas}, $\alpha$, combined with the isomorphism $2$-cells $\iota^0_{uvw}$ and $\kappa^0_{uvw}$, defines a unique element
      \[
        \left( (a_u)_{0 \leq u \leq k}, (f^z_{uv})_{\substack{0 \leq u < v \leq k \\ 0 \leq z \leq 1}}, (\alpha^1_{uv})_{0 \leq u < v \leq k}, (\iota^z_{uvw})_{\substack{0 \leq u < v < w \leq k \\ 0 \leq z \leq 1}} \right)
      \]
    of $\NB(k, 1)$, where
      \begin{itemize}
        \item for all $0 \leq u < v \leq k$, $f^1_{uv} = g^0_{uv}$;
        \item for all $0 \leq u < v < w \leq k$, $\iota^1_{uvw} = \kappa^0_{uvw}$.
      \end{itemize}
    Denote this by $\hat{\alpha}$; then $s(\hat{\alpha}) = f$, $t(\hat{\alpha}) = g$, and $S_k(\hat{\alpha}) = \alpha$, so $S_k$ is full on $1$-cells.
  \end{proof}

  \begin{prop}  \label{prop:bicatSegalfaithful}
    Let $\mathcal{B}$ be a bicategory.  For all $k \geq 0$, the Segal map
          \[
            S_k: \NB(k,-) \longrightarrow \underbrace{\NB(1,-) \times_{\NB(0,0)} \dotsb \times_{\NB(0,0)} \NB(1,-)}_k
          \]
    is faithful on $1$-cells.
  \end{prop}

  \begin{proof}
    Suppose we have two parallel elements $\alpha$, $\beta \in \NB(k, 1)$ such that $(S_k)_1(\alpha) = (S_k)_1(\beta)$.  We wish to show that $\alpha = \beta$.  We can write $f$ and $g$ as
      \[
        \alpha = \left( (a_u)_{0 \leq u \leq k}, (f^z_{uv})_{\substack{0 \leq u < v \leq k \\ 0 \leq z \leq 1}}, (\alpha^1_{uv})_{0 \leq u < v \leq k}, (\iota^z_{uvw})_{\substack{0 \leq u < v < w \leq k \\ 0 \leq z \leq 1}} \right)
      \]
    and
      \[
        \beta = \left( (a_u)_{0 \leq u \leq k}, (f^z_{uv})_{\substack{0 \leq u < v \leq k \\ 0 \leq z \leq 1}}, (\beta^1_{uv})_{0 \leq u < v \leq k}, (\iota^z_{uvw})_{\substack{0 \leq u < v < w \leq k \\ 0 \leq z \leq 1}} \right).
      \]
    Note that the fact $\alpha$ and $\beta$ are parallel tells us that they can only differ on their $2$-cell parts.  We write $(S_k)_1(\alpha) = (S_k)_1(\beta)$ as
      \[
        \left(
        \Bigg(
        \xy
          % POINTS
          (0, 0)*+{a_0}="0";
          (20, 0)*+{a_1}="1";
          % ARROWS
          {\ar@/^1.5pc/^{f^0_{01}} "0" ; "1"};
          {\ar@/_1.5pc/_{g^0_{01}} "0" ; "1"};
          {\ar@{=>}^{\gamma^1_{01}} (8, 3) ; (8, -3)};
        \endxy
        \Bigg),
        \Bigg(
        \xy
          % POINTS
          (0, 0)*+{a_1}="0";
          (20, 0)*+{a_2}="1";
          % ARROWS
          {\ar@/^1.5pc/^{f^0_{12}} "0" ; "1"};
          {\ar@/_1.5pc/_{g^0_{12}} "0" ; "1"};
          {\ar@{=>}^{\gamma^1_{12}} (8, 3) ; (8, -3)};
        \endxy
        \Bigg),
        \dotsc,
        \Bigg(
        \xy
          % POINTS
          (0, 0)*+{a_{k - 1}}="0";
          (20, 0)*+{a_k}="1";
          % ARROWS
          {\ar@/^1.5pc/^{f^0_{k - 1, k}} "0" ; "1"};
          {\ar@/_1.5pc/_{g^0_{k - 1, k}} "0" ; "1"};
          {\ar@{=>}^{\gamma^1_{k - 1, k}} (7, 3) ; (7, -3)};
        \endxy
        \Bigg)
        \right),
      \]
    which is an element of
      \[
        \underbrace{\NB(1,1) \times_{\NB(0,0)} \dotsb \times_{\NB(0,0)} \NB(1,1)}_k.
      \]
    Furthermore, since $(S_k)_1(\alpha) = (S_k)_1(\beta)$, we have that, for all $0 \leq u < k$,
      \[
        \alpha^1_{u, u + 1} = \gamma^1_{u, u + 1} = \beta^1_{u, u + 1}.
      \]
    Thus, by Lemma~\ref{lem:alphas}, for all $0 \leq u < v \leq k$, we have
      \[
        \alpha^1_{uv} = \gamma^1_{uv} = \beta^1_{uv},
      \]
    so $\alpha = \beta$, as required.
  \end{proof}

  We now have everything we need to prove that the nerve of a bicategory satisfies the Segal condition.

  \begin{thm}  \label{thm:bicatnerveSegalcond}
    Let $\mathcal{B}$ be a bicategory.  Then the nerve of $\mathcal{B}$, $\nerve \mathcal{B}$, satisfies the Segal condition, and is thus a Tamsamani--Simpson weak $2$-category.
  \end{thm}

  \begin{proof}
    For all $j$, $k \geq 0$, the Segal map
          \[
            S_{j,k}: \NB(j,k) \longrightarrow \underbrace{\NB(j,1) \times_{\NB(j,0)} \dotsb \times_{\NB(j,0)} \NB(j,1)}_k
          \]
    is a bijection by Proposition~\ref{prop:bicatSegalbij}.

    For all $k \geq 0$, the Segal map
          \[
            S_k: \NB(k,-) \longrightarrow \underbrace{\NB(1,-) \times_{\NB(0,0)} \dotsb \times_{\NB(0,0)} \NB(1,-)}_k
          \]
    is surjective on $0$-cells by Proposition~\ref{prop:bicatSegalsurj}, full on $1$-cells by Proposition~\ref{prop:bicatSegalfull}, and faithful on $1$-cells by Proposition~\ref{prop:bicatSegalfaithful}.

    Thus $\NB$ satisfies the Segal condition, so it is a Tamsamani--Simpson weak $2$-category.
  \end{proof}

\section{The nerve construction for $n = 2$}  \label{sect:2nerveconstr}

  In this section we construct a nerve functor for Penon weak $2$-categories.  The construction for the case of general $n$ is given in the next chapter; we present the $2$-dimensional case separately since it is simpler, both conceptually and notationally, than the general case, but not too simple to exhibit all the key features of the $n$-dimensional construction.  We are also able to prove that nerves satisfy the Segal condition in the case $n = 2$; we do this in Section~\ref{sect:Segal}.  We use Leinster's nerve construction for bicategories as the prototype for our construction, and also use his notation.

  Recall that, when defining the nerve of a category, we defined a functor $I \colon \Delta \hookrightarrow \Cat$, and then defined the nerve $\nerve \mathcal{C}$ of a category $\mathcal{C}$ to be given by $\nerve \mathcal{C} = \Cat(I(-), \mathcal{C})$.  In analogy with this, to define our nerve functor for Penon weak $2$-categories, we first define a functor
    \[
      I_2: \Theta^2 \longrightarrow P\Alg.
    \]
  This functor should give us, for each object of $\Theta^2$, the corresponding cuboidal $2$-pasting diagram, expressed as a freely generated Penon weak $2$-category.  However, we have to be very careful about what we mean by ``freely generated'' in this context.  Each cuboidal $2$-pasting diagram has associated to it a $2$-globular set whose cells are those which we draw in the pasting diagram.  We could simply define $I_2$ to give us the free $P$-algebra on these $2$-globular sets.  Let $(j, k) \in \Theta^2$ and write $F_{P}(j, k)$ for the free $P$-algebra on the corresponding $2$-globular set.  We would then have, for a Penon weak $2$-category $\mathcal{A}$, the nerve defined by
    \[
      \nerve \mathcal{A}(j, k) = P\Alg(F_{P}(j, k), \mathcal{A}).
    \]
  Consider the object $(2, 0)$ of $\Theta^2$; writing $f$ and $g$ for the generating $1$-cells, the free $P$-algebra on the corresponding $2$-globular set looks like
    \[
      \xymatrix@C=1.25em@R=1.25em{
        & \bullet \ar[ddr]^g  \\
        \\
        \bullet \ar[uur]^f \ar[rr]_{g \comp f} && \bullet,
               }
    \]
  (omitting identities and any composites involving identities).  Thus, for $\mathcal{A} \in P\Alg$, the set $P\Alg(F_P(2, 0), \mathcal{A})$ is the set of all composable pairs of $1$-cells in $\mathcal{A}$.  However, we want an element of $\nerve \mathcal{A}(2, 0)$ to consist of a composable pair of $1$-cells together with a choice of alternative composite, so we want $I_2(2, 0)$ to look like
    \[
      \xymatrix@C=1.25em@R=1.25em{
        & \bullet \ar[ddr]^g  \\
        \\
        \bullet \ar[uur]^f \ar[rr]^{g \comp f}_{\rotatebox[origin=c]{270}{ $\iso$}} \ar@/_2pc/[rr]_h && \bullet,
               }
    \]
  (once again omitting identities, etc.), where $h$ is the choice of alternative composite.  Note that these alternative composites are also required to allow us to define the face maps in our nerve;  we cannot define the face maps using composition, as in the nerve of a category, because composition of $1$-cells is not strictly associative in a Penon weak $2$-category.  We can think of this as weakening the maps in $\nerve \mathcal{A}(2, 0)$ on composites, but keeping them strict on identities.  Thus, we may think we want to use a notion of normalised maps of Penon weak $n$-categories;  that is, maps which preserve identities strictly but preserve composition only up to coherent isomorphism (note that there is no established definition of normalised maps of $P$-algebras, but for the purposes of this thought experiment this is not important).  We would thus define
    \[
      \nerve \mathcal{A}(j, k) := P\Alg_{\mathrm{norm}}(F_{P}(j, k), \mathcal{A}),
    \]
  where $P\Alg_{\mathrm{norm}}$ is the category of $P$-algebras and normalised maps.  In fact, normalised maps turn out to be too weak, as we will now demonstrate.  Consider the pasting diagram $(2, 2)$ shown below:
    \[
      \xy
        % POINTS
        (0, 0)*+{a_0}="0";
        (20, 0)*+{a_1}="1";
        (40, 0)*+{a_2}="2";
        % ARROWS
        {\ar@/^2pc/^{f_1} "0" ; "1"};
        {\ar_(0.3){g_1} "0" ; "1"};
        {\ar@/_2pc/_{h_1} "0" ; "1"};
        {\ar@/^2pc/^{f_2} "1" ; "2"};
        {\ar_(0.3){g_2} "1" ; "2"};
        {\ar@/_2pc/_{h_2} "1" ; "2"};
        {\ar@{=>} (10, 7) ; (10, 2)};
        {\ar@{=>} (30, 7) ; (30, 2)};
        {\ar@{=>} (10, -2) ; (10, -7)};
        {\ar@{=>} (30, -2) ; (30, -7)};
      \endxy
    \]
  If we use normalised maps, we will add an extra $1$-cell isomorphic to each of the binary composites of $f$'s, $g$'s and $h$'s.  However, owing to the simplicial nature of Tamsamani--Simpson weak $n$-categories, we only wish to add such extra $1$-cells in place of $f_2 \comp f_1$, $g_2 \comp g_1$, and $h_2 \comp h_1$.  This is because we should have a $2$-simplex of $1$-cells at each ``level'' of the pasting diagram (here we have three such levels, one containing $f_1$ and $f_2$, one containing $g_1$ and $g_2$, and one containing $h_1$ and $h_2$), but there should be no extra interaction between the levels.  Recall from Definition~\ref{defn:2catv1} that the Segal map $S_{2, 2}$ divides pasting diagrams of this shape along the $1$-cells $g_1$ and $g_2$, and the Segal condition requires this map to be an isomorphism; if we add extra cells isomorphic to $h_2 \comp f_1$ and $h_1 \comp f_2$ to the diagram above, these cells are forgotten by $S_{2, 2}$ so it is not an isomorphism.

  We therefore want a method of weakening $P$-algebras that is biased towards specific choices of simplicial shapes. Such a method cannot be defined for a general $P$-algebra, since in general we have no notion of ``level'' like we do in a $2$-pasting diagram.  Thus, we define this weakening by explicitly stating which extra cells we are going to add.  We do so by modifying the construction of the free Penon weak $2$-category on a $2$-globular set, using the construction of Penon's left adjoint from Section~\ref{sect:ladj}.

  Recall from Section~\ref{sect:ladj} that the adjunction inducing $P$ can be decomposed as
      \[
        \xy
          % POINTS
          (0, 0)*+{\nGSet}="nG";
          (20, 0)*+{\Rn}="Rn";
          (36, 0)*+{\Qn .}="Qn";
          % ARROWS
          {\ar@<1ex>^-{H}_-*!/u1pt/{\labelstyle \bot} "nG" ; "Rn"};
          {\ar@<1ex>^-{V} "Rn" ; "nG"};
          {\ar@<1ex>^-{J}_-*!/u1pt/{\labelstyle \bot} "Rn" ; "Qn"};
          {\ar@<1ex>^-{W} "Qn" ; "Rn"};
        \endxy
     \]
  Thus we can write the free $P$-algebra functor as the composite
      \[
        \xy
          % POINTS
          (0, 0)*+{2\mathbf{GSet}}="nG";
          (20, 0)*+{\Rn}="Rn";
          (36, 0)*+{\Qn}="Qn";
          (56, 0)*+{P\Alg,}="PAlg";
          % ARROWS
          {\ar^-{H} "nG" ; "Rn"};
          {\ar^-{J} "Rn" ; "Qn"};
          {\ar^-{K} "Qn" ; "PAlg"};
        \endxy
     \]
  where $K$ is the Eilenberg--Moore comparison functor.  Thus, instead of starting in $2\mathbf{GSet}$, we can start with an object of $\Rn$ and apply $KJ$ to obtain a $P$-algebra that is ``partially free'' in the sense that the constraint cells and composites are still added freely (by the functor $J$), but the contraction is now taken over a different map, rather than a component of $\eta^T$.  This allows us to add the isomorphism $2$-cells we want using the contraction, thus avoiding the need to specify these cells individually.

  Before defining the process in general we first describe a small example; specifically, we construct the $P$-algebra $I_2(2, 1)$.  Write $X(2, 1)$ for the $2$-globular set illustrated below:
    \[
      \xy
        % POINTS
        (0, 0)*+{a_0}="0";
        (20, 0)*+{a_1}="1";
        (40, 0)*+{a_2}="2";
        % ARROWS
        {\ar@/^1.5pc/^{f^0_{01}} "0" ; "1"};
        {\ar@/_1.5pc/_{f^1_{01}} "0" ; "1"};
        {\ar@/^1.5pc/^{f^0_{12}} "1" ; "2"};
        {\ar@/_1.5pc/_{f^1_{12}} "1" ; "2"};
        {\ar@{=>}^{\alpha^1_{01}} (10, 4) ; (10, -4)}
        {\ar@{=>}^{\alpha^1_{12}} (30, 4) ; (30, -4)}
      \endxy
    \]
  This is the associated $2$-globular set of the pasting diagram, a concept introduced by Batanin~\cite[Proof of Proposition 4.2]{Bat98}.  As explained earlier, we want $I_2(2, 1)$ to be a ``simplicially weakened'' version of the free $P$-algebra on this $2$-globular set, and to do so we construct an object of $\Rn$, then generate the ``partially free'' $P$-algebra on it. We take the strict $2$-category part of this object of $\Rn$ to be the free strict $2$-category on $X(2, 1)$.  To obtain the $2$-globular set part of this object of $\Rn$ we add extra cells to $X(2, 1)$ in the places where we want to weaken the diagram.  Specifically, we add $1$-cells
        \[
          \xy
          % POINTS
          (0, 0)*+{a_0}="0";
          (16, 0)*+{a_2}="1";
          (32, 0)*+{a_0}="2";
          (48, 0)*+{a_2.}="3";
          (24, 0)*+{\text{and}};
          % ARROWS
          {\ar^{f^0_{02}} "0" ; "1"};
          {\ar^{f^1_{02}} "2" ; "3"};
          \endxy
        \]
  Based on Leinster's nerve construction for bicategories, we might also expect that we need to add a $2$-cell
    \[
      \xy
        % POINTS
        (0, 0)*+{a_0}="0";
        (20, 0)*+{a_2,}="1";
        % ARROWS
        {\ar@/^1.5pc/^{f^0_{02}} "0" ; "1"};
        {\ar@/_1.5pc/_{f^1_{02}} "0" ; "1"};
        {\ar@{=>}^{\alpha^1_{02}} (10, 4) ; (10, -4)}
      \endxy
    \]
  but this will be added automatically as a composite of other $2$-cells, as we shall see later.  We write $R(2, 1)$ for the resulting $2$-globular set; it can be drawn as:
      \[
        \xy
          % POINTS
          (0, 0)*+{a_0}="3,0";
          (16, 0)*+{a_1}="4,0";
          (32, 0)*+{a_2}="5,0";
          % ARROWS
          {\ar@/^3pc/^{f^0_{02}} (0, 2) ; (32, 2)};
          {\ar@/^1.5pc/^{f^0_{01}} "3,0" ; "4,0"};
          {\ar@/_1.5pc/_(0.6){f^1_{01}} "3,0" ; "4,0"};
          {\ar@/^1.5pc/^{f^0_{12}} "4,0" ; "5,0"};
          {\ar@/_1.5pc/_(0.4){f^1_{12}} "4,0" ; "5,0"};
          {\ar@/_3pc/_{f^1_{02}} (0, -2) ; (32, -2)};
          {\ar@{=>}^{\alpha^1_{01}} (8, 3) ; (8, -3)};
          {\ar@{=>}^{\alpha^1_{12}} (24, 3) ; (24, -3)};
        \endxy
      \]
  To get an object of $\Rn$, we define a map
    \[
      \theta_{(2, 1)} \colon R(2, 1) \longrightarrow TX(2, 1)
    \]
  as follows: $\theta_{(2, 1)}$ leaves cells in $R(2, 1)$ that are also in $X(2, 1)$ unchanged; on the extra cells, we have
    \begin{itemize}
      \item $\theta_{(2, 1)}(f^0_{02}) = f^0_{12} \comp f^0_{01}$;
      \item $\theta_{(2, 1)}(f^1_{02}) = f^1_{12} \comp f^1_{01}$.
    \end{itemize}

  We now explain what happens when we apply the functor
    \[
      J \colon \Rn \longrightarrow \Qn
    \]
  to
    \[
      \xy
        (0, 0)*+{R(2, 1)}="0";
        (24, 0)*+{TX(2, 1),}="1";
        {\ar^-{\theta_{(2, 1)}} "0" ; "1"};
      \endxy
    \]
  using the interleaving construction from Section~\ref{sect:ladj}.  First we add contraction $1$-cells; since $R(2, 1)$ and $TX(2, 1)$ have the same $0$-cells, this just adds identities.  We then generate composites of $1$-cells freely; this adds $f^0_{12} \comp f^0_{01}$, $f^1_{12} \comp f^1_{01}$, $f^1_{12} \comp f^0_{01}$ and $f^0_{12} \comp f^1_{01}$, as well as composites involving identities.  Next we add contraction $2$-cells; this is where the ``simplicial weakening'' manifests itself.  Observe that, after having generated $1$-cell composites, we have pairs of $1$-cells:
    \begin{itemize}
      \item $f^0_{02}$ and $f^0_{12} \comp f^0_{01}$, which are parallel and are mapped to the same cell in $TX(2, 1)$;
      \item $f^1_{02}$ and $f^1_{12} \comp f^1_{01}$, which are parallel and are mapped to the same cell in $TX(2, 1)$.
    \end{itemize}
  Thus, as well as the usual identities, associators, and unitors, generating contraction $2$-cells freely adds the following cells:
    \[
      \xy
        % POINTS
        (10, 0)*+{a_1}="1,1";
        (0, -16)*+{a_0}="1,0";
        (20, -16)*+{a_2}="1,2";
        (50, 0)*+{a_1}="2,1";
        (40, -16)*+{a_0}="2,0";
        (60, -16)*+{a_2}="2,2";
        % ARROWS
        {\ar^{f^0_{01}} "1,0" ; "1,1"};
        {\ar^{f^0_{12}} "1,1" ; "1,2"};
        {\ar_{f^0_{02}} "1,0" ; "1,2"};
        {\ar@{=>} (10, -6.5) ; (10, -12.5)};
        {\ar^{f^0_{01}} "2,0" ; "2,1"};
        {\ar^{f^0_{12}} "2,1" ; "2,2"};
        {\ar_{f^0_{02}} "2,0" ; "2,2"};
        {\ar@{=>} (50, -12.5) ; (50, -6.5)};
      \endxy
    \]
    \[
      \xy
        % POINTS
        (10, 0)*+{a_1}="1,1";
        (0, -16)*+{a_0}="1,0";
        (20, -16)*+{a_2}="1,2";
        (50, 0)*+{a_1}="2,1";
        (40, -16)*+{a_0}="2,0";
        (60, -16)*+{a_2}="2,2";
        % ARROWS
        {\ar^{f^1_{01}} "1,0" ; "1,1"};
        {\ar^{f^1_{12}} "1,1" ; "1,2"};
        {\ar_{f^1_{02}} "1,0" ; "1,2"};
        {\ar@{=>} (10, -6.5) ; (10, -12.5)};
        {\ar^{f^1_{01}} "2,0" ; "2,1"};
        {\ar^{f^1_{12}} "2,1" ; "2,2"};
        {\ar_{f^1_{02}} "2,0" ; "2,2"};
        {\ar@{=>} (50, -12.5) ; (50, -6.5)};
      \endxy
    \]
  We generate composites of $2$-cells, then ``add contraction $3$-cells'', which forces all diagrams of $2$-cells to commute.  In particular, this forces the pairs of triangular cells shown above to be inverses of one another (and thus isomorphisms), and also gives us a $2$-cell
      \[
        \xy
          % POINTS
          (0, 0)*+{a_0}="0,0";
          (32, 0)*+{a_2}="2,0";
          (40, 0)*+{=};
          (48, 0)*+{a_0}="3,0";
          (64, 0)*+{a_1}="4,0";
          (80, 0)*+{a_2}="5,0";
          % ARROWS
          {\ar@/^1.5pc/^(0.3){f^0_{02}} "0,0" ; "2,0"};
          {\ar@/_1.5pc/_{f^1_{02}} "0,0" ; "2,0"};
          {\ar@{=>}^{\alpha^1_{02}} (16, 3) ; (16, -3)};
          {\ar@/^3pc/^{f^0_{02}} (48, 2) ; (80, 2)};
          {\ar@/^1.5pc/^{f^0_{01}} "3,0" ; "4,0"};
          {\ar@/_1.5pc/_(0.6){f^1_{01}} "3,0" ; "4,0"};
          {\ar@/^1.5pc/^{f^0_{12}} "4,0" ; "5,0"};
          {\ar@/_1.5pc/_(0.4){f^1_{12}} "4,0" ; "5,0"};
          {\ar@/_3pc/_{f^1_{02}} (48, -2) ; (80, -2)};
          {\ar@{=>} (64, 12) ; (64, 6)};
          {\ar@{=>}^{\alpha^1_{01}} (56, 3) ; (56, -3)};
          {\ar@{=>}^{\alpha^1_{12}} (72, 3) ; (72, -3)};
          {\ar@{=>} (64, -6) ; (64, -12)};
        \endxy
      \]
  Observe that this corresponds to the first axiom from Leinster's nerve construction (see Definition~\ref{defn:Leinsterbicatnerve}); adding ``contraction $3$-cells'' also ensures that the second axiom holds when we perform this construction for longer cuboidal pasting diagrams.

  This whole process gives an object of $\Qn$, denoted
    \[
      \xy
        (0, 0)*+{Q(j, k)}="0";
        (24, 0)*+{TX(j, k).}="1";
        {\ar^-{\phi_{(j, k)}} "0" ; "1"};
      \endxy
    \]
  We obtain the $P$-algebra $I_2(2, 1)$ by applying the Eilenberg--Moore comparison functor; the resulting $P$-algebra has as its underlying magma the magma part of the object of $\Qn$ above.

  Note that the triangular cells added by the free contraction are considered contraction cells in the object of $\Qn$, but when we apply the Eilenberg--Moore comparison functor they are not contraction cells from the point of view of the $P$-algebra action.  They retain their commutativity properties, however, so given any other $P$-algebra $\mathcal{A}$, a map of $P$-algebras
    \[
      I_2(2, 1) \longrightarrow \mathcal{A}
    \]
  can map these cells to any suitably coherent choice of cells in $\mathcal{A}$; their images need not be contraction cells.

  We now describe this construction for a general object of $\Theta^2$.  As above, we use Leinster's notation from his nerve construction for bicategories (Section~\ref{sect:Leinsternerves}).  Recall that the subscripts and superscripts adorning each cell should be thought of as being the ``coordinates'' of that cell within the pasting diagram; the subscripts are the horizontal coordinates, and the superscripts are the vertical coordinates.

  Note that an object of $\Theta^2$ is an equivalence class of objects of $\Delta^2$.  An object of $\Delta^2$ is in an equivalence class with more than one member if and only if it has a $0$ in the first position.  Thus, for the purposes of the following definition we represent the equivalence class of $(0, k)$ for all $k \in \mathbb{N}$ by the object $(0, 0)$ of $\Delta^2$; all other equivalence classes are represented by their sole member.

  Let $(j, k)$ be an object of $\Theta^2$; we first define the $2$-globular set $X(j, k)$, the associated $2$-globular set of the cuboidal pasting diagram $(j, k)$, as follows:
    \begin{itemize}
      \item $X(j, k)_0 = \{ a_u \gt u \in \mathbb{N}, 0 \leq u \leq j \}$;
      \item $X(j, k)_1 = \{ f^z_{u, u + 1} \gt u, z \in \mathbb{N}, 0 \leq u < j, 0 \leq z \leq k \}$;
      \item $X(j, k)_2 = \{ \alpha^z_{u, u + 1} \gt u, z \in \mathbb{N}, 0 \leq u < j, 1 \leq z \leq k \}$,
    \end{itemize}
  with source and target maps given by
    \[
      s(f^z_{u, u + 1}) = a_u, \; t(f^z_{u, u + 1}) = a_{u + 1},
    \]
    \[
      s(\alpha^z_{u, u + 1}) = f^{z - 1}_{u, u + 1}, \; t(\alpha^z_{u, u + 1}) = f^z_{u, u + 1}.
    \]
  We then add extra $1$- and $2$-cells to this to obtain a $2$-globular set $R(j, k)$, defined as follows:
    \begin{itemize}
      \item $R(j, k)_0 = \{ a_u \gt u \in \mathbb{N}, 0 \leq u \leq j \}$;
      \item $R(j, k)_1 = \{ f^z_{uv} \gt u, v, z \in \mathbb{N}, 0 \leq u < v \leq j, 0 \leq z \leq k \}$;
      \item $R(j, k)_2 = \{ \alpha^z_{u, u + 1} \gt u, z \in \mathbb{N}, 0 \leq u < j, 1 \leq z \leq k \}$,
    \end{itemize}
  with source and target maps given by
    \[
      s(f^z_{uv}) = a_u, \; t(f^z_{uv}) = a_v,
    \]
    \[
      s(\alpha^z_{u, u + 1}) = f^{z - 1}_{u, u + 1}, \; t(\alpha^z_{u, u + 1}) = f^z_{u, u + 1}.
    \]
  It is important to note that, in spite of the notation, this does not define functors $X$ and $R$ into $2\text{-}\GSet$.  This is because, at this stage of the construction, there is no way to define the effect on maps in $\Theta^2$, since we cannot map cells to identities as we do not have these in the $2$-globular sets.

  We now construct, for each $(j, k) \in \Theta^2$, an object
    \[
      \xy
        (0, 0)*+{R(j, k)}="0";
        (24, 0)*+{TX(j, k)}="1";
        {\ar^-{\theta_{(j, k)}} "0" ; "1"};
      \endxy
    \]
  of $\Rtwo$.  We define the map $\theta_{(j, k)}$ as follows:
    \begin{itemize}
      \item on $0$-cells, $\theta_{(j, k)0}(a_u) = a_u$;
      \item on $1$-cells, $\theta_{(j, k)1}(f^z_{uv}) = f^z_{v - 1, v} \comp f^z_{v - 2, v - 1} \comp \dotsb \comp f^z_{u, u + 1}$;
      \item on $2$-cells, $\theta_{(j, k)2}(\alpha^z_{u, u + 1}) = \alpha^z_{u, u + 1}$.
    \end{itemize}
  This map coincides with $\eta^T_{X(j, k)}$, the unit for the monad $T$, for all cells in $X(j, k)$; the extra cells in $R(j, k)$ can be thought of as weakenings of the composites at each level of the cuboidal pasting diagram, and $\theta_{(j, k)}$ maps each of these cells to the corresponding freely generated strict composite in $TX(j, k)$.

  We now apply the functor $J \colon \Rtwo \rightarrow \Qtwo$ to the object of $\Rtwo$ described above; this adds to $R(j, k)$ all the required composites and contraction cells.  As demonstrated in the example above, this includes contraction cells in both directions between each of the extra $1$-cells (those in $R(j, k)_1$ but not in $X(j, k)_1$) and the corresponding freely generated composites at the same level of the pasting diagram (i.e. of cells with the same $z$-coordinate).  The tameness condition in the contraction ensures that these contraction $2$-cells are isomorphisms.  The extra $1$-cells will give us the choices of alternative composites in the nerve, and the contraction cells ensure that these are coherently isomorphic to the composites we originally had in the Penon weak $2$-category whose nerve we are taking.  We denote the resulting object of $\Qtwo$ by
    \[
      \xy
        (0, 0)*+{Q(j, k)}="0";
        (24, 0)*+{TX(j, k).}="1";
        {\ar^-{\phi_{(j, k)}} "0" ; "1"};
      \endxy
    \]

  We now extend this to a definition of a functor $E_2 \colon \Theta^2 \rightarrow \Qtwo$, with the action on objects as described above.  To describe the action on a morphism in $\Theta^2$, we first define a morphism in $\Rtwo$, and then take its transpose under the adjunction
      \[
        \xy
          % POINTS
          (0, 0)*+{\Rtwo}="Rn";
          (16, 0)*+{\Qtwo}="Qn";
          % ARROWS
          {\ar@<1ex>^-{J}_-*!/u1pt/{\labelstyle \bot} "Rn" ; "Qn"};
          {\ar@<1ex>^-{W} "Qn" ; "Rn"};
        \endxy
     \]
  to obtain a morphism in $\Qtwo$.

  Let $(p, q) \colon (j, k) \rightarrow (l, m)$ be a morphism in $\Theta^2$.  We define the strict $2$-category part of the morphism of $\Rtwo$ first.  Define a map of $2$-globular sets $x(p, q) \colon X(j, k) \rightarrow TX(l, m)$ as follows:
    \begin{itemize}
      \item for $a_u \in X(j, k)_0$, $x(p, q)_0(a_u) = a_{p(u)}$;
      \item for $f^z_{u, u + 1} \in X(j, k)_1$, $x(p, q)_1(f^z_{u, u + 1}) =$
        \[
            \left\{
              \begin{array}{ll}
              f^{q(z)}_{p(u + 1) - 1, p(u + 1)} \comp \dotsb \comp f^{q(z)}_{p(u), p(u) + 1} & \text{if } p(u) < p(u + 1), \\
              1_{a_{p(u)}} & \text{if } p(u) = p(u + 1);
              \end{array}
            \right.
        \]
      \item for $\alpha^z_{u, u + 1} \in X(j, k)_2$, $x(p, q)_2(\alpha^z_{u, u + 1}) =$
        \[
            \left\{
              \begin{array}{ll}
              \alpha^{q(z)}_{p(u + 1) - 1, p(u + 1)} * \dotsb * \alpha^{q(z)}_{p(u), p(u) + 1} & \text{if } p(u) < p(v), q(z - 1) < q(z), \\
              1_{TX(p, q)_1(f^z_{u, u + 1})} & \text{if } q(z - 1) = q(z).
              \end{array}
            \right.
        \]
    \end{itemize}
  To obtain a map $TX(j, k) \rightarrow TX(l, m)$ we apply $T$ and compose this with the multiplication for $T$, giving
    \[
      \xy
        % POINTS
        (0, 0)*+{TX(j, k)}="0";
        (28, 0)*+{T^2X(l, m)}="1";
        (56, 0)*+{TX(l, m)}="2";
        % ARROWS
        {\ar^-{Tx(p, q)} "0" ; "1"};
        {\ar^-{\mu^T_{X(l, m)}} "1" ; "2"};
      \endxy
    \]
  We now define a map
    \[
      \xy
        % POINTS
        (0, 0)*+{R(j, k)}="0,0";
        (56, 0)*+{Q(l, m)}="2,0";
        (0, -20)*+{TX(j, k)}="0,1";
        (28, -20)*+{T^2X(l, m)}="1,1";
        (56, -20)*+{TX(l, m),}="2,1";
        % ARROWS
        {\ar^{r(p, q)} "0,0" ; "2,0"};
        {\ar_{\theta_{(j, k)}} "0,0" ; "0,1"};
        {\ar^{\phi_{(l, m)}} "2,0" ; "2,1"};
        {\ar_-{Tx(p, q)} "0,1" ; "1,1"};
        {\ar_-{\mu^T_{X(l, m)}} "1,1" ; "2,1"};
      \endxy
    \]
  where the map $r(p, q)$ is defined as follows:
    \begin{itemize}
      \item for $a_u \in R(j, k)_0$, $R(p, q)_0(a_u) = a_{p(u)}$;
      \item for $f^z_{uv} \in R(j, k)_1$,
        \[
          R(p, q)_1(f^z_{uv}) =
            \left\{
              \begin{array}{rl}
              f^{q(z)}_{p(u) p(v)} & \text{if } p(u) < p(v), \\
              1_{a_{p(u)}} & \text{if } p(u) = p(v);
              \end{array}
            \right.
        \]
      \item for $\alpha^z_{uv} \in R(j, k)_2$,
        \[
          R(p, q)_2(\alpha^z_{uv}) =
            \left\{
              \begin{array}{rl}
              \alpha^{q(z)}_{p(u) p(v)} & \text{if } p(u) < p(v), q(z - 1) < q(z), \\
              1_{f^{q(z)}_{p(u) p(v)}} & \text{if } p(u) < p(v), q(z - 1) = q(z), \\
              1_{1_{a_{p(u)}}} & \text{if } p(u) = p(v).
              \end{array}
            \right.
        \]
    \end{itemize}

  Finally, we take the transpose of this map under the adjunction
      \[
        \xy
          % POINTS
          (0, 0)*+{\Rtwo}="Rn";
          (16, 0)*+{\Qtwo.}="Qn";
          % ARROWS
          {\ar@<1ex>^-{J}_-*!/u1pt/{\labelstyle \bot} "Rn" ; "Qn"};
          {\ar@<1ex>^-{W} "Qn" ; "Rn"};
        \endxy
     \]
  We write $\epsilon \colon JW \Rightarrow 1$ for the counit of this adjunction, and $\epsilon_{\phi_{(l, m)}}$ for the component corresponding to
    \[
      \xy
        (0, 0)*+{Q(l, m)}="0";
        (24, 0)*+{TX(l, m).}="1";
        {\ar^-{\phi_{(l, m)}} "0" ; "1"};
      \endxy
    \]
  Then the transpose is given by the composite
    \[
      \epsilon_{\phi_{(l, m)}} \comp J\big( r(p, q), \mu^T_{X(l, m)} \comp Tx(p, q) \big).
    \]
  This allows us to define the functors $E_2 \colon \Theta^2 \rightarrow \Qtwo$ and $I_2 \colon \Theta^2 \rightarrow P\Alg$.

  \begin{defn}
    Define a functor $E_2 \colon \Theta^2 \rightarrow \Qtwo$ as follows:
      \begin{itemize}
        \item given an object $(j, k) \in \Theta^2$, $E_2(j, k)$ is defined to be the object
      \[
      \xy
        (0, 0)*+{Q(j, k)}="0";
        (24, 0)*+{TX(j, k).}="1";
        {\ar^-{\phi_{(j, k)}} "0" ; "1"};
      \endxy
      \]
    of $\Qtwo$;
        \item given a morphism $(p, q): (j, k) \rightarrow (l, m)$ in $\Theta^2$, $E_2(p, q)$ is defined to be the map
    \[
      \epsilon_{\phi_{(l, m)}} \comp J\big( r(p, q), \mu^T_{X(l, m)} \comp Tx(p, q) \big).
    \]
      \end{itemize}
    Write $K \colon \Qtwo \rightarrow P\Alg$ for the Eilenberg--Moore comparison functor for the adjunction
      \[
        \xy
          % POINTS
          (0, 0)*+{\nGSet}="Rn";
          (24, 0)*+{\Qtwo.}="Qn";
          % ARROWS
          {\ar@<1ex>^-{F}_-*!/u1pt/{\labelstyle \bot} "Rn" ; "Qn"};
          {\ar@<1ex>^-{U} "Qn" ; "Rn"};
        \endxy
     \]
    We define a functor $I_2 := K \comp E_2 : \Theta^2 \rightarrow P\Alg$.
  \end{defn}

  We can now define the nerve functor for Penon weak $2$-categories.

  \begin{defn}
    The \emph{nerve functor} $\nerve$ for Penon weak $2$-categories is defined by
      \[
        \xymatrix@C=0pt@R=2pt{
          \nerve \colon P\Alg & \longrightarrow & [(\Theta^2)^{\op}, \Set]  \\
          \mathcal{A} \ar[dddd]_f && P\Alg(I_2(-), \mathcal{A}) \ar[dddd]^{f \comp -}  \\
          \\
          & \longmapsto & \\
          \\
          \mathcal{B} && P\Alg(I_2(-), \mathcal{B}).
                 }
      \]
    For a $P$-algebra $\mathcal{A}$, the presheaf $\nerve \mathcal{A} = P\Alg(I_2(-), \mathcal{A})$ is called the \emph{nerve of $\mathcal{A}$}.
  \end{defn}

\section{The Segal condition}  \label{sect:Segal}
  In this section we prove that the nerve of a Penon weak $2$-category satisfies the Segal condition, and is therefore a Tamsamani--Simpson weak $2$-category.  Recall from Definition~\ref{defn:2catv1} that $\nerve\mathcal{A}$ satisfies the Segal condition if
      \begin{enumerate}[(i)]
        \item for all $j \geq 0$, the Segal map
          \[
            S_j: \nerve\mathcal{A}(j,-) \longrightarrow \underbrace{\nerve\mathcal{A}(1,-) \times_{\nerve\mathcal{A}(0,1)} \dotsb \times_{\nerve\mathcal{A}(0,1)} \nerve\mathcal{A}(1,-)}_j
          \]
              is contractible, i.e. surjective on objects, full and faithful on $1$-cells;
        \item for all $j$, $k \geq 0$, the Segal map
          \[
            S_{j,k}: \nerve\mathcal{A}(j,k) \longrightarrow \underbrace{\nerve\mathcal{A}(j,1) \times_{\nerve\mathcal{A}(j,0)} \dotsb \times_{\nerve\mathcal{A}(j,0)} \nerve\mathcal{A}(j,1)}_k
          \]
              is a bijection.
      \end{enumerate}
  Our approach is to use the way in which nerve functor is defined to rewrite the Segal maps in terms of composition with certain maps of $P$-algebras;  this then allows us to express most parts of the Segal condition (everything except surjectivity on objects) as statements describing certain $P$-algebras in the image of $I_2$ as colimits of diagrams in the image of $I_2$.

  Before doing this, we establish some notation for certain free $P$-algebras in the image of $I_2$ that can be expressed as colimits of others; these $P$-algebras arise in the reformulation of the Segal condition described above.  Observe that the free $P$-algebra functor $F_{P}$ can be factorised as
    \[
      \xymatrix{
        \twoGSet \ar[rr]^{F_{P}} \ar[dr]_{F} && P\Alg  \\
        & \Qtwo \ar[ur]_{K}
               }
    \]
  Thus, we see from the construction of $I_2$ that, for $j$, $k \in \mathbb{N}$, if $R(j, k) = X(j, k)$, then $I_2(j, k) = F_{P}X(j, k)$.  Since $R(j, k)$ and $X(j, k)$ differ only on $1$-cells, this happens precisely when $R(j, k)_1 = X(j, k)_1$, which is true when $j = 0$ and $j = 1$:
    \begin{itemize}
      \item for $j = 0$, $R(j, k)_1 = \emptyset = X(j, k)_1$;
      \item for $j = 1$, $R(j, k)_1 = \{ f^z_{01} \gt 0 \leq z \leq k \} = X(j, k)_1$.
    \end{itemize}
  Thus $I_2(0, 0) = F_{P}X(0, 0)$, and $I_2(1, k) = F_{P}X(1, k)$ for all $k \in \mathbb{N}$.  For $j \geq 2$, we have $f^0_{02} \in R(j, k)$, but $f^0_{02} \not\in X(j, k)$, so this does not hold for $j \geq 2$.

  Recall that, for all $k > 0$, $0 \leq i \leq k$, we have a map $d_i \colon [k - 1] \rightarrow [k]$ in $\Delta$ given by
    \[
      d_{i}(j) =
            \left\{ \begin{array}{ll}
                j & \text{if } j < i, \\
                j + 1 & \text{if } j \geq i,
            \end{array} \right.
    \]
  and consider the following diagram in $P\Alg$:
    \[
      \underbrace{
        \xymatrix@C=12pt{
        & I_2(0, 0) \ar[dl]_{I_2(d_0, 1)} \ar[dr]^{I_2(d_1, 1)} &&&& I_2(0, 0) \ar[dl]_{I_2(d_0, 1)} \ar[dr]^{I_2(d_1, 1)}  \\
        I_2(1, 0) && I_2(1, 0) & \dotsc & I_2(1, 0) && I_2(1, 0).
               }
                  }_{j \text{ copies of } I_2(1, 0)}
    \]
  Write $I_2(1, 0)^{\pushj}$ for the colimit of this diagram in $P\Alg$.  By the observations above, this diagram is the image under $F_{P}$ of the diagram
    \[
      \underbrace{
        \xymatrix@C=12pt{
        & X(0, 0) \ar[dl]_{a_1} \ar[dr]^{a_0} &&&& X(0, 0) \ar[dl]_{a_1} \ar[dr]^{a_0}  \\
        X(1, 0) && X(1, 0) & \dotsc & X(1, 0) && X(1, 0)
               }
                  }_{j \text{ copies of } X(1, 0)}
    \]
  in $\twoGSet$, where $a_0 \colon X(0, 0) \rightarrow X(1, 0)$ maps the single $0$-cell of $X(0, 0)$ to $a_0$, and similarly for $a_1$.  The colimit in $\twoGSet$ of this diagram is $X(j, 0)$, and thus
    \[
      I_2(1, 0)^{\pushj} = F_{P}X(j, 0),
    \]
  the free $P$-algebra on a composable string of $j$ $1$-cells.

  Similarly, we write $I_2(1, 1)^{\pushj}$ for the colimit in $P\Alg$ of the diagram
    \[
      \underbrace{
        \xymatrix@C=12pt{
        & I_2(0, 1) \ar[dl]_{I_2(d_0, 1)} \ar[dr]^{I_2(d_1, 1)} &&&& I_2(0, 1) \ar[dl]_{I_2(d_0, 1)} \ar[dr]^{I_2(d_1, 1)}  \\
        I_2(1, 1) && I_2(1, 1) & \dotsc & I_2(1, 1) && I_2(1, 1),
               }
                  }_{j \text{ copies of } I_2(1, 1)}
    \]
  which is the image under $F_{P}$ of the diagram
    \[
      \underbrace{
        \xymatrix@C=12pt{
        & X(0, 1) \ar[dl]_{a_1} \ar[dr]^{a_0} &&&& X(0, 1) \ar[dl]_{a_1} \ar[dr]^{a_0}  \\
        X(1, 1) && X(1, 1) & \dotsc & X(1, 1) && X(1, 1)
               }
                  }_{j \text{ copies of } X(1, 1)}
    \]
  in $\twoGSet$.  The colimit in $\twoGSet$ of this diagram is $X(j, 1)$, and thus
    \[
      I_2(1, 1)^{\pushj} = F_{P}X(j, 1),
    \]
  the free $P$-algebra on a string of $j$ $2$-cells composable along boundary $0$-cells.

  We now rewrite the Segal maps of the form $S_j$ in terms of composition with certain maps of $P$-algebras.

  \begin{lemma}  \label{lem:sk}
    Let $\mathcal{A}$ be a Penon weak $2$-category.  For all $j > 0$, we have
      \[
        \underbrace{\nerve\mathcal{A}(1, -) \times_{\nerve\mathcal{A}(0, -)} \dotsb \times_{\nerve\mathcal{A}(0, -)} \nerve\mathcal{A}(1, -)}_j \iso P\Alg(I_2(1, -)^{\pushj}, \mathcal{A})
      \]
    and the Segal map $S_j$ is given by
      \[
        S_j = \dotblank \comp d^{\pushj} \colon P\Alg(I_2(j, -), \mathcal{A}) \longrightarrow P\Alg(I_2(1, -)^{\pushj}, \mathcal{A}),
      \]
    where $d^{\pushj} \colon I_2(1, -)^{\pushj} \rightarrow I_2(j, -)$ is a map in $[\Delta, P\Alg]$, defined in the proof.
  \end{lemma}

  \begin{proof}
    We have the following functors:
      \[
        \xymatrix@C=0pt@R=2pt{
          \nerve^2\mathcal{A}(\dotblank, -) \colon \Delta^{\op} & \longrightarrow & [\Delta^{\op}, \Set]  \\
          k \ar[dddd]_{\alpha} && P\Alg(I_2(k, -), \mathcal{A}) \ar[dddd]^{\dotblank \comp I_2(\alpha, -)}  \\
          \\
          & \longmapsto & \\
          \\
          j && P\Alg(I_2(j, -), \mathcal{A}),
                 }
      \]
      \[
        \xymatrix@C=0pt@R=2pt{
          I_2(\dotblank, -) \colon \Delta & \longrightarrow & [\Delta, P\Alg]  \\
          j \ar[dddd]_{\alpha} && I_2(j, -) \ar[dddd]^{I_2(\alpha, -)}  \\
          \\
          & \longmapsto & \\
          \\
          k && I_2(k, -),
                 }
      \]
  and
      \[
        \xymatrix@C=0pt@R=2pt{
          P\Alg(-, \mathcal{A}) \colon [\Delta, P\Alg]^{\op} & \longrightarrow & [\Delta^{\op}, \Set]  \\
          X \ar[dddd]_{\delta} && P\Alg(X(-), \mathcal{A}) \ar[dddd]^{- \comp \delta}  \\
          \\
          & \longmapsto & \\
          \\
          Y && P\Alg(Y(-), \mathcal{A}).
                 }
      \]

  We can factorise $\nerve\mathcal{A}(\dotblank, -)$ as follows:
    \[
      \xymatrix{
        \Delta^{\op} \ar[rrrr]^{\nerve\mathcal{A}(\dotblank, -)} \ar[drr]_{I_2(\dotblank, -)} &&&& [\Delta^{\op}, \Set]  \\
        && [\Delta, P\Alg]^{\op} \ar[urr]_{\, \, \, \, \, P\Alg(-, \mathcal{A})}
               }
    \]

  For each, $[j] \in \Delta$, we consider the actions of the functors $\nerve\mathcal{A}(\dotblank, -)$ and $I_2(\dotblank, -)$ on the diagram
     \[
       \xymatrix@C=12pt@R=12pt{
       &&&&& [j] \\
       \\
       [1] \ar[uurrrrr]^(.3){\iota_1} && [1] \ar[uurrr]^(.25){\iota_2} && [1] \ar[uur]^(.25){\iota_3} && \dotsc && [1] \ar[uulll]_(.25){\iota_{j - 1}} && [1] \ar[uulllll]_(.3){\iota_{j}} \\
       & [0] \ar[ul]^{\tau} \ar[ur]_{\sigma} && [0] \ar[ul]^{\tau} \ar[ur]_{\sigma} &&&&&& [0] \ar[ul]^{\tau} \ar[ur]_{\sigma}
                }
     \]
  in $\Delta$.

  Applying $\nerve\mathcal{A}(\dotblank, -)$ to this diagram gives
    \[
      \xy
        % POINTS
        (0, 0)*+{P\Alg(I_2(j, -), \mathcal{A})}="0,0";
        (-45, -16)*+{P\Alg(I_2(1, -), \mathcal{A})}="-2,1";
        (-24, -24)*+{P\Alg(I_2(1, -), \mathcal{A})}="-1,1";
        (24, -24)*+{P\Alg(I_2(1, -), \mathcal{A})}="1,1";
        (45, -16)*+{P\Alg(I_2(1, -), \mathcal{A})}="2,1";
        (-40,-40)*+{P\Alg(I_2(0, -), \mathcal{A})}="-1,2";
        (40,-40)*+{P\Alg(I_2(0, -), \mathcal{A})}="1,2";
        (0, -24)*+{\dotsc};
        % ARROWS
        {\ar "0,0" ; "-2,1"};
        {\ar "0,0" ; "-1,1"};
        {\ar "0,0" ; "1,1"};
        {\ar "0,0" ; "2,1"};
        {\ar_t "-2,1" ; "-1,2"};
        {\ar^s "-1,1" ; "-1,2"};
        {\ar_t "1,1" ; "1,2"};
        {\ar^s "2,1" ; "1,2"};
      \endxy
    \]
  which is a cone over the diagram
    \[
      \xy
        % POINTS
        (-45, 0)*+{P\Alg(I_2(1, -), \mathcal{A})}="-2,1";
        (-24, -8)*+{P\Alg(I_2(1, -), \mathcal{A})}="-1,1";
        (24, -8)*+{P\Alg(I_2(1, -), \mathcal{A})}="1,1";
        (45, 0)*+{P\Alg(I_2(1, -), \mathcal{A})}="2,1";
        (-40,-24)*+{P\Alg(I_2(0, -), \mathcal{A})}="-1,2";
        (40,-24)*+{P\Alg(I_2(0, -), \mathcal{A})}="1,2";
        (0, -8)*+{\dotsc};
        % ARROWS
        {\ar_t "-2,1" ; "-1,2"};
        {\ar^s "-1,1" ; "-1,2"};
        {\ar_t "1,1" ; "1,2"};
        {\ar^s "2,1" ; "1,2"};
      \endxy
    \]

  Applying $I_2(\dotblank, -)^{\op}$ to the original diagram gives
     \[
      \xymatrix@C=-6pt@R=12pt{
        &&& I_2(j, -) \ar[ddlll] \ar[ddl] \ar[ddr] \ar[ddrrr] \\
        \\
        I_2(1, -) \ar[dr]_{t} && I_2(1, -) \ar[dl]^{s} & \dotsc & I_2(1, -) \ar[dr]_{t} && I_2(1, -) \ar[dl]^{s} \\
        & I_2(0, -) &&&& I_2(0, -)
               }
     \]
  in $[\Delta, P\Alg]^{\op}$, which is a cone over the diagram
     \[
      \overbrace{\xymatrix@C=0pt@R=12pt{
        I_2(1, -) \ar[dr]_{t} && I_2(1, -) \ar[dl]^{s} & \dotsc & I_2(1, -) \ar[dr]_{t} && I_2(1, -) \ar[dl]^{s} \\
        & I_2(0, -) &&&& I_2(0, -)
               }}^j
     \]
  The limit of this diagram is $I_2(1, -)^{\pushj}$, and this limit induces a unique map $d^{\pushj}$ such that the diagram
     \[
      \xymatrix@C=-6pt@R=12pt{
        &&& I_2(j, -) \ar@/_2.7pc/[ddddlll] \ar@/_1.5pc/[ddddl] \ar@/^1.5pc/[ddddr] \ar@/^2.7pc/[ddddrrr] \ar@{-->}[dd]^{!d^{\pushj}} \\
        \\
        &&& I_2(1, -)^{\pushj} \ar[ddlll] \ar[ddl] \ar[ddr] \ar[ddrrr] \\
        \\
        I_2(1, -) \ar[dr]_{t} && I_2(1, -) \ar[dl]^{s} & \dotsc & I_2(1, -) \ar[dr]_{t} && I_2(1, -) \ar[dl]^{s} \\
        & I_2(0, -) &&&& I_2(0, -)
               }
     \]
  Applying $P\Alg(-, \mathcal{A})$ to this diagram, we get:
    \[
      \xy
        % POINTS
        (0, 20)*+{P\Alg(I_2(j, -), \mathcal{A})}="0,-1";
        (0, 0)*+{P\Alg(I_2(1, -)^{\pushj}, \mathcal{A})}="0,0";
        (-45, -16)*+{P\Alg(I_2(1, -), \mathcal{A})}="-2,1";
        (-24, -24)*+{P\Alg(I_2(1, -), \mathcal{A})}="-1,1";
        (24, -24)*+{P\Alg(I_2(1, -), \mathcal{A})}="1,1";
        (45, -16)*+{P\Alg(I_2(1, -), \mathcal{A})}="2,1";
        (-40,-40)*+{P\Alg(I_2(0, -), \mathcal{A})}="-1,2";
        (40,-40)*+{P\Alg(I_2(0, -), \mathcal{A})}="1,2";
        (0, -24)*+{\dotsc};
        % ARROWS
        {\ar@{-->}^{- \comp d^{\pushj}} "0,-1" ; "0,0"};
        {\ar@/_3pc/ "0,-1" ; "-2,1"};
        {\ar@/_2pc/ "0,-1" ; "-1,1"};
        {\ar@/^2pc/ "0,-1" ; "1,1"};
        {\ar@/^3pc/ "0,-1" ; "2,1"};
        {\ar "0,0" ; "-2,1"};
        {\ar "0,0" ; "-1,1"};
        {\ar "0,0" ; "1,1"};
        {\ar "0,0" ; "2,1"};
        {\ar_t "-2,1" ; "-1,2"};
        {\ar^s "-1,1" ; "-1,2"};
        {\ar_t "1,1" ; "1,2"};
        {\ar^s "2,1" ; "1,2"};
      \endxy
    \]
  Since $P\Alg(-, \mathcal{A})$ is representable, it preserves limits \cite[V.6 Theorem 3]{Mac98}, so we have that
    \begin{align*}
      & \underbrace{P\Alg(I_2(1, -), \mathcal{A}) \times_{P\Alg(I_2(0, -), \mathcal{A})} \dotsb \times_{P\Alg(I_2(0, -), \mathcal{A})} P\Alg(I_2(1, -), \mathcal{A})}_k  \\
      & \iso P\Alg(I_2(1, -)^{\pushj}, \mathcal{A})
    \end{align*}
  and the Segal map $S_j$ is given by composition with $d^{\pushj}$, as required.
  \end{proof}

  Similarly, we now rewrite the Segal maps of the form $S_{j, k}$ in terms of composition with certain maps of $P$-algebras.

  \begin{lemma}  \label{lem:sjk}
    Let $\mathcal{A}$ be a Penon weak $2$-category.  For all $j$, $k > 0$, we have
      \[
        \underbrace{\nerve\mathcal{A}(j, 1) \times_{\nerve\mathcal{A}(j, 0)} \dotsb \times_{\nerve\mathcal{A}(j, 0)} \nerve\mathcal{A}(j, 1)}_k \iso P\Alg(I_2(j, 1)^{\pushk}, \mathcal{A})
      \]
    and the Segal map $S_{j, k}$ is given by
      \[
        S_{j, k} = \dotblank \comp d^{\pushk} \colon P\Alg(I_2(j, k), \mathcal{A}) \longrightarrow P\Alg(I_2(j, 1)^{\pushk}, \mathcal{A}),
      \]
    where $d^{\pushk} \colon I_2(j, 1)^{\pushk} \rightarrow I_2(j, k)$ is a map of $P$-algebras, defined in the proof.
  \end{lemma}

  \begin{proof}
  We take a similar approach to that used in the proof of Lemma~\ref{lem:sk}.  For each $j > 0$, we have the following functors:
      \[
        \xymatrix@C=0pt@R=2pt{
          \nerve^2\mathcal{A}(j, \dotblank) \colon \Delta^{\op} & \longrightarrow & [\Delta^{\op}, \Set]  \\
          l \ar[dddd]_{\alpha} && P\Alg(I_2(j, l), \mathcal{A}) \ar[dddd]^{\dotblank \comp I_2(1_j, \alpha)}  \\
          \\
          & \longmapsto & \\
          \\
          k && P\Alg(I_2(j, k), \mathcal{A}),
                 }
      \]
  and
      \[
        \xymatrix@C=0pt@R=2pt{
          I_2(j, -) \colon \Delta & \longrightarrow & [\Delta, P\Alg]  \\
          k \ar[dddd]_{\alpha} && I_2(j, k) \ar[dddd]^{I_2(1_j, \alpha)}  \\
          \\
          & \longmapsto & \\
          \\
          l && I_2(j, l),
                 }
      \]
  and we can factorise $\nerve\mathcal{A}(j, \dotblank)$ as follows:
    \[
      \xymatrix{
        \Delta^{\op} \ar[rrrr]^{\nerve\mathcal{A}(j, \dotblank)} \ar[drr]_{I_2(j, \dotblank)} &&&& [\Delta^{\op}, \Set]  \\
        && [\Delta, P\Alg]^{\op} \ar[urr]_{\, \, \, \, \, P\Alg(-, \mathcal{A})}
               }
    \]

  For each, $[k] \in \Delta$, we consider the effects of the functors $\nerve\mathcal{A}(j, \dotblank)$ and $I_2(j, \dotblank)$ on the diagram
     \[
       \xymatrix@C=12pt@R=12pt{
       &&&&& [k] \\
       \\
       [1] \ar[uurrrrr]^(.3){\iota_1} && [1] \ar[uurrr]^(.25){\iota_2} && [1] \ar[uur]^(.25){\iota_3} && \dotsc && [1] \ar[uulll]_(.25){\iota_{k - 1}} && [1] \ar[uulllll]_(.3){\iota_{k}} \\
       & [0] \ar[ul]^{\tau} \ar[ur]_{\sigma} && [0] \ar[ul]^{\tau} \ar[ur]_{\sigma} &&&&&& [0] \ar[ul]^{\tau} \ar[ur]_{\sigma}
                }
     \]
  in $\Delta$.  By exactly the same argument as the case of $S_j$, we have a unique map $d^{\pushk}$ such that
     \[
      \xymatrix@C=-6pt@R=12pt{
        &&& I_2(j, k) \ar@/_2.7pc/[ddddlll] \ar@/_1.5pc/[ddddl] \ar@/^1.5pc/[ddddr] \ar@/^2.7pc/[ddddrrr] \ar@{-->}[dd]^{!d^{\pushk}} \\
        \\
        &&& I_2(j, 1)^{\pushk} \ar[ddlll] \ar[ddl] \ar[ddr] \ar[ddrrr] \\
        \\
        I_2(j, 1) \ar[dr]_{t} && I_2(j, 1) \ar[dl]^{s} & \dotsc & I_2(j, 1) \ar[dr]_{t} && I_2(j, 1) \ar[dl]^{s} \\
        & I_2(j, 0) &&&& I_2(j, 0)
               }
     \]
  and applying the functor $P\Alg(-, \mathcal{A})$ gives us the diagram
    \[
      \xy
        % POINTS
        (0, 20)*+{P\Alg(I_2(j, k), \mathcal{A})}="0,-1";
        (0, 0)*+{P\Alg(I_2(j, 1)^{\pushk}, \mathcal{A})}="0,0";
        (-45, -16)*+{P\Alg(I_2(j, 1), \mathcal{A})}="-2,1";
        (-24, -24)*+{P\Alg(I_2(j, 1), \mathcal{A})}="-1,1";
        (24, -24)*+{P\Alg(I_2(j, 1), \mathcal{A})}="1,1";
        (45, -16)*+{P\Alg(I_2(j, 1), \mathcal{A})}="2,1";
        (-40,-40)*+{P\Alg(I_2(j, 0), \mathcal{A})}="-1,2";
        (40,-40)*+{P\Alg(I_2(j, 0), \mathcal{A})}="1,2";
        (0, -24)*+{\dotsc};
        % ARROWS
        {\ar@{-->}^{- \comp d^{\pushk}} "0,-1" ; "0,0"};
        {\ar@/_3pc/ "0,-1" ; "-2,1"};
        {\ar@/_2pc/ "0,-1" ; "-1,1"};
        {\ar@/^2pc/ "0,-1" ; "1,1"};
        {\ar@/^3pc/ "0,-1" ; "2,1"};
        {\ar "0,0" ; "-2,1"};
        {\ar "0,0" ; "-1,1"};
        {\ar "0,0" ; "1,1"};
        {\ar "0,0" ; "2,1"};
        {\ar_t "-2,1" ; "-1,2"};
        {\ar^s "-1,1" ; "-1,2"};
        {\ar_t "1,1" ; "1,2"};
        {\ar^s "2,1" ; "1,2"};
      \endxy
    \]
  Thus we have that
    \begin{align*}
      & \underbrace{P\Alg(I_2(j, 1), \mathcal{A}) \times_{P\Alg(I_2(j, 0), \mathcal{A})} \dotsb \times_{P\Alg(I_2(j, 0), \mathcal{A})} P\Alg(I_2(j, 1), \mathcal{A})}_k  \\
      & \iso P\Alg(I_2(j, 1)^{\pushk}, \mathcal{A})
    \end{align*}
    and the Segal map $S_{j, k}$ is given by composition with $d^{\pushk}$, as required.
  \end{proof}

  We now use Lemmas~\ref{lem:sk} and \ref{lem:sjk} to prove that the nerve of a Penon weak $2$-category satisfies the Segal condition.  We begin with the Segal maps of the form $S_j$.

  \begin{prop}  \label{prop:soo2}
    Let $\mathcal{A}$ be a Penon weak $2$-category.  For all $j > 0$, the Segal map
      \[
        S_j \colon \nerve\mathcal{A}(j, -) \rightarrow \underbrace{\nerve\mathcal{A}(1, -) \times_{\nerve\mathcal{A}(0, -)} \dotsb \times_{\nerve\mathcal{A}(0, -)} \nerve\mathcal{A}(1, -)}_j
      \]
    is surjective on $0$-cells, i.e. the map
      \[
        (S_j)_0 \colon \nerve\mathcal{A}(j, 0) \rightarrow \underbrace{\nerve\mathcal{A}(1, 0) \times_{\nerve\mathcal{A}(0, 0)} \dotsb \times_{\nerve\mathcal{A}(0, 0)} \nerve\mathcal{A}(1, 0)}_j
      \]
    is surjective.
  \end{prop}

  \begin{proof}
    By Lemma~\ref{lem:sk}, the Segal map $S_j$ is given by
      \[
        S_j = \dotblank \comp d^{\pushj} \colon P\Alg(I_2(j, -), \mathcal{A}) \rightarrow P\Alg(I_2(1, -)^{\pushj}, \mathcal{A}),
      \]
    so we need to show that
      \[
        (S_j)_0 = \dotblank \comp d^{\pushj} \colon P\Alg(I_2(j, 0), \mathcal{A}) \rightarrow P\Alg(I_2(1, 0)^{\pushj}, \mathcal{A})
      \]
    is surjective.  Let $\phi \colon I_2(1, 0)^{\pushj} \rightarrow \mathcal{A}$ be a map of Penon weak $2$-categories.  We must find a map $\psi \colon I_2(j, 0) \rightarrow \mathcal{A}$ such that $(S_k)_0(\psi) = \phi$, i.e. such that the diagram
      \[
        \xymatrix{
          I_2(1, 0)^{\pushj} \ar[rr]^{\phi} \ar@{^(->}[dr]_{d^{\pushj}} && \mathcal{A}  \\
          & I_2(j, 0) \ar[ur]_{\psi}
                 }
      \]
    commutes.

    Write the $P$-algebra $\mathcal{A}$ as
      \[
        \xy
          % POINTS
          (0, 0)*+{PA}="0";
          (16, 0)*+{A}="1";
          % ARROWS
          {\ar^{\theta} "0" ; "1"};
        \endxy
      \]
    so $U_{P}\mathcal{A} = A$.  We define $\psi$ by factoring through the free algebra $F^{P}A$.  Define a map
      \[
        \xy
          % POINTS
          (0, 0)*+{R(j, 0)}="0,0";
          (40, 0)*+{PA}="2,0";
          (0, -16)*+{TX(j, 0)}="0,1";
          (20, -16)*+{T^2A}="1,1";
          (40, -16)*+{TA}="2,1";
          % ARROWS
          {\ar^-{g} "0,0" ; "2,0"};
          {\ar_{\theta_{(j, 0)}} "0,0" ; "0,1"};
          {\ar^{p_A} "2,0" ; "2,1"};
          {\ar_-{Th} "0,1" ; "1,1"};
          {\ar_-{\mu^T_{A}} "1,1" ; "2,1"};
        \endxy
      \]
    in $\Rtwo$ as follows:

    The map $g \colon R(j, 0) \rightarrow PA$ is defined by:
      \begin{itemize}
        \item for all $a_u \in R(j, 0)_0$, $g_0(a_u) = \phi_0(a_u)$;
        \item for $f^0_{uv} \in R(j, 0)_1$ with $v = u + 1$,
          \[
            g_1(f^0_{uv}) = \phi_1(f^0_{uv});
          \]
        \item for $f^0_{uv} \in R(j, 0)_1$ with $v > u + 1$
          \[
            g_1(f^0_{uv}) = \Big(\Big(\dotsb\Big(\phi_1(f^0_{v - 1, v}) \comp \phi_1(f^0_{v - 2, v - 1})\Big) \comp \dotsb \Big) \comp \phi_1(f^0_{u, u + 1})\Big).
          \]
      \end{itemize}
    Note that $R(j, 0)_2 = \emptyset$, so we do not need to define $g$ on $2$-cells.

    The map $h \colon X(k, 0) \rightarrow TA$ is defined by:
      \begin{itemize}
        \item for all $a_u \in X(j, 0)_0$, $h_0(a_u) = \phi_0(a_u)$;
        \item for all $f^0_{u, u + 1} \in X(j, 1)_1$,
          \[
            h_1(f^0_{u, u + 1}) = p_A \comp \phi_1(f^0_{u, u + 1})
          \]
      \end{itemize}
    Note that $X(j, 0)_2 = \emptyset$, so we do not need to define $h$ on $2$-cells.

    This defines a map in $\Rtwo$.  We then take the transpose of this map under the the adjunction
      \[
        \xy
          % POINTS
          (0, 0)*+{\Rtwo}="Rn";
          (16, 0)*+{\Qtwo}="Qn";
          % ARROWS
          {\ar@<1ex>^-{J}_-*!/u1pt/{\labelstyle \bot} "Rn" ; "Qn"};
          {\ar@<1ex>^-{W} "Qn" ; "Rn"};
        \endxy
     \]
  We write $\epsilon \colon JW \Rightarrow 1$ for the counit of this adjunction, and $\epsilon_{\phi_{\vectk}}$ for the component corresponding to
    \[
      \xy
        (0, 0)*+{Q(j, 0)}="0";
        (24, 0)*+{TX(j, 0).}="1";
        {\ar^-{\phi_{(j, 0)}} "0" ; "1"};
      \endxy
    \]
  Then the transpose is given by the composite
    \[
      \epsilon_{\phi_{(j, 0)}} \comp J(g, \mu^T_A \comp Th).
    \]
  Finally, we apply the Eilenberg--Moore comparison functor $K \colon \Qtwo \rightarrow P\Alg$ to this; we write
    \[
      \chi := K(\epsilon_{\phi_{(j, 0)}} \comp J(g, \mu^T_A \comp Th)),
    \]
  and define
    \[
      \psi := \theta \comp \chi \colon I_2(j, 0) \rightarrow \mathcal{A}.
    \]

    We now check commutativity of the diagram
      \[
        \xymatrix{
          I_2(1, 0)^{\pushj} \ar[rr]^{\phi} \ar@{^(->}[dr]_{d^{\pushj}} && \mathcal{A}  \\
          & I_2(j, 0) \ar[ur]_{\psi}
                 }
      \]
    This commutes if the diagram
      \[
        \xy
          % POINTS
          (0, 0)*+{X(j, 0)}="0,0";
          (28, 0)*+{U_P F_P X(j, 0)}="2,0";
          (56, 0)*+{U_P \mathcal{A}}="4,0";
          (10, -16)*+{U_P F_P X(j, 0)}="1,1";
          (48, -16)*+{U_P I_2(j, 0)}="3,1";
          % ARROWS
          {\ar^-{\eta^P_{X(j, 0)}} "0,0" ; "2,0"};
          {\ar_-{\eta^P_{X(j, 0)}} "0,0" ; "1,1"};
          {\ar^-{U_P \phi} "2,0" ; "4,0"};
          {\ar_-{U_P d^{\amalg j}} "1,1" ; "3,1"};
          {\ar_-{U_P \psi} "3,1" ; "4,0"};
        \endxy
      \]
    commutes; we check this using an elementary approach.  Since $X(j, 0)_2 = \emptyset$, we do not have to check commutativity on $2$-cells.  We have
      \begin{itemize}
        \item for $a_u \in X(j, 0)_0$,
          \[
            U_{P}\psi_0 \comp U_P d^{\pushj}_0 \comp \eta^P_{X(j, 0)}(a_u) = U_{P}\psi_0(a_u) = U_P \phi_0 \comp \eta^P_{X(j, 0)}(a_u);
          \]
        \item for $f^z_{u, u + 1} \in X(j, 0)_1$,
          \[
            U_{P}\psi_1 \comp U_P d^{\pushj}_1 \comp \eta^P_{X(j, 0)}(f^z_{u, u + 1}) = U_{P}\psi_1(f^z_{u, u + 1}) = U_P \phi_1 \comp \eta^P_{X(j, 0)}(f^z_{u, u + 1});
          \]
      \end{itemize}
    hence the diagram commutes.  Hence $S_j$ is surjective on $0$-cells.
  \end{proof}

  We now use Lemma~\ref{lem:sk} us to express the fullness and faithfulness part of the Segal condition in terms of colimits of $P$-algebras.  Recall from Definition~\ref{defn:simcontr} that, given a map of simplicial sets $\alpha \colon A \rightarrow B$, we have an induced map $\tilde{\alpha}_1$ in $\Set$, as shown in the diagram below:
      \[
        \xymatrix{
          A_{1} \ar@/^1.5pc/[drrr]^s \ar@/_4pc/[dddr]_t \ar@/_4pc/[ddrr]_{\alpha_{1}} \ar@{-->}[dr]^{\tilde{\alpha}_{1}} \\
          & A_{0} \times_{B_{0}} B_{1} \times_{B_{0}} A_{1} \ar[rr] \ar[dd] \ar[dr] && A_{0} \ar[d]^{\alpha_{0}}  \\
          && B_{1} \ar[r]_s \ar[d]^t & B_{0}  \\
          & A_{0} \ar[r]_{\alpha_{0}} & B_{0}
                 }
      \]
    and that $\alpha$ is full and faithful on $1$-cells if the map $\tilde{\alpha}_1$ is an isomorphism.  We wish to show that, for all $j \geq 0$, the Segal map
      \[
        S_j \colon P\Alg(I_2(j, -), \mathcal{A}) \longrightarrow P\Alg(I_2(1, -)^{\pushj}, \mathcal{A})
      \]
    is full and faithful on $1$-cells.  By the description of fullness and faithfulness above, this happens when the diagram
      \[
        \xymatrix{
          P\Alg(I_2(j, 1), \mathcal{A}) \ar[rr]^s \ar[dr]^{- \comp (d^{\pushj})_1} \ar[dd]_t && P\Alg(I_2(j, 0), \mathcal{A}) \ar[d]^{- \comp (d^{\pushj})_0}  \\
          & P\Alg(I_2(1, 1)^{\pushj}, \mathcal{A}) \ar[r]_s \ar[d]^t & P\Alg(I_2(1, 0)^{\pushj}, \mathcal{A})  \\
          P\Alg(I_2(j, 0), \mathcal{A}) \ar[r]_{- \comp (d^{\pushj})_0} & P\Alg(I_2(1, 0)^{\pushj}, \mathcal{A}).
                 }
      \]
    is a limit cone in $\Set$.  This cone lies in the image of the functor
      \[
        P\Alg(-, \mathcal{A}) \colon P\Alg^{\op} \longrightarrow \Set,
      \]
    and this functor is representable, so it preserves limits \cite[V.6 Theorem 3]{Mac98}.  Hence $S_j$ is full and faithful on $1$-cells if the diagram
      \[
        \xymatrix{
          & I_2(1, 0)^{\amalg 3} \ar[dl]_{(d^{\amalg 3})_0} \ar[dr]^{I_2(d_1, 1)^{\amalg 3}} && I_2(1, 0)^{\amalg 3} \ar[dl]_{I_2(d_0, 1)^{\amalg 3}} \ar[dr]^{(d^{\amalg 3})_0}  \\
          I_2(3, 0) \ar[drr]_{I_2(1, d_1)} && I_2(1, 1)^{\amalg 3} \ar[d]^{(d^{\amalg 3})_1} && I_2(3, 0) \ar[dll]^{I_2(1, d_0)}  \\
          && I_2(3, 1)
                 }
      \]
    is a colimit cocone in $P\Alg$.

  Before proving this, we describe what this means in the case $j = 3$.  The $P$-algebra $I_2(1, 1)^{\amalg 3}$ consists of three $2$-cells composed horizontally:
    \[
      \xy
        % POINTS
        (0, 0)*+{a_0}="0";
        (20, 0)*+{a_1}="1";
        (40, 0)*+{a_2}="2";
        (60, 0)*+{a_3,}="3";
        % ARROWS
        {\ar@/^1.5pc/^{f^0_{01}} "0" ; "1"};
        {\ar@/_1.5pc/_{f^1_{01}} "0" ; "1"};
        {\ar@/^1.5pc/^{f^0_{12}} "1" ; "2"};
        {\ar@/_1.5pc/_{f^1_{12}} "1" ; "2"};
        {\ar@/^1.5pc/^{f^0_{12}} "2" ; "3"};
        {\ar@/_1.5pc/_{f^1_{12}} "2" ; "3"};
        {\ar@{=>}^{\alpha^1_{01}} (10, 4) ; (10, -4)}
        {\ar@{=>}^{\alpha^1_{12}} (30, 4) ; (30, -4)}
        {\ar@{=>}^{\alpha^1_{23}} (50, 4) ; (50, -4)}
      \endxy
    \]
  with the copies of $I_2(1, 0)^{\amalg 3}$ in the diagram giving its source and target strings of $1$-cells.  The $P$-algebra $I_2(3, 0)$ is a tetrahedron whose faces are isomorphism $2$-cells:
  \[
    \xy
      % POINTS
      (-50, 15)*+{a_1}="a1";
      (-30, 15)*+{a_2}="a2";
      (-60, -5)*+{a_0}="a0";
      (-20, -5)*+{a_3}="a3";
      (-5, 5)*+{=};
      (20, 15)*+{a_1}="a1r";
      (40, 15)*+{a_2}="a2r";
      (10, -5)*+{a_0}="a0r";
      (50, -5)*+{a_3.}="a3r";
      % INVISIBLE POINTS
      (-48, 13)*+{}="i0s";
      (-43, 8)*+{}="i0t";
      (38, 13)*+{}="i1s";
      (33, 8)*+{}="i1t";
      (-32, 5)*+{}="i2s";
      (-34, -1)*+{}="i2t";
      (22, 5)*+{}="i3s";
      (24, -1)*+{}="i3t";
      % ARROWS
      {\ar^{f_{01}} "a0" ; "a1"};
      {\ar^{f_{12}} "a1" ; "a2"};
      {\ar^{f_{23}} "a2" ; "a3"};
      {\ar_{f_{03}} "a0" ; "a3"};
      {\ar_{f_{02}} "a0" ; "a2"};
      {\ar^{f_{01}} "a0r" ; "a1r"};
      {\ar^{f_{12}} "a1r" ; "a2r"};
      {\ar^{f_{23}} "a2r" ; "a3r"};
      {\ar_{f_{03}} "a0r" ; "a3r"};
      {\ar_{f_{13}} "a1r" ; "a3r"};
      % 2-CELLS:
      {\ar@{=>}_{\iso} "i0s" ; "i0t"};
      {\ar@{=>}^{\iso} "i1s" ; "i1t"};
      {\ar@{=>}^{\iso} "i2s" ; "i2t"};
      {\ar@{=>}_{\iso} "i3s" ; "i3t"};
    \endxy
  \]
  Taking the colimit of the diagram glues one of these tetrahedra to the string of source $1$-cells of $I_2(1, 1)^{\amalg 3}$, and the other to the string of target $1$-cells.  Thus the fullness and faithfulness condition tells us that $I_2(3, 1)$ can be obtained this way; it is a simplicially weakened version of the cuboidal pasting diagram $(3, 1)$.

  \begin{lemma}  \label{lem:i2j1}
    For all $j > 0$, the diagram
      \[
        \xymatrix{
          & I_2(1, 0)^{\amalg j} \ar[dl]_{(d^{\amalg j})_0} \ar[dr]^{I_2(d_1, 1)^{\amalg j}} && I_2(1, 0)^{\amalg j} \ar[dl]_{I_2(d_0, 1)^{\amalg j}} \ar[dr]^{(d^{\amalg j})_0}  \\
          I_2(j, 0) \ar[drr]_{I_2(1, d_1)} && I_2(1, 1)^{\amalg j} \ar[d]^{(d^{\amalg j})_1} && I_2(j, 0) \ar[dll]^{I_2(1, d_0)}  \\
          && I_2(j, 1)
                 }
      \]
    is a colimit cocone in $P\Alg$.
  \end{lemma}

  To prove Lemma~\ref{lem:i2j1}, we check directly that $I_2(j, 1)$ satisfies the universal property for the colimit.  In order to do this we must specify maps out of $I_2(j, 1)$ and $I_2(j, k)$, which we define dimension by dimension, starting at dimension $0$ and working up.

  Recall from the construction of $I_2(j, k)$ that at each dimension (excluding dimension $0$), we have three types of cell: generating cells (those in $R(j, k)$), contraction cells, and composites.  Since we are defining a map of $P$-algebras, once we have defined the effect of the map on generating cells and contraction cells, the effect on composites is determined by the fact that the map must preserve the $P$-algebras structure (in a way that we will make precise later).  A similar statement is true for some of the contraction cells, but not all of them; due to the fact that (for $j > 1$) $I_2(j, k)$ is not a free $P$-algebra, only certain contraction cells are required to be preserved by the $P$-algebra structure.  We refer to these cells as ``algebraic contraction cells''.

  To see which contraction cells are algebraic contraction cells, suppose we are defining a map $\psi \colon I_2(j, k) \rightarrow \mathcal{A}$.  This consists of a map of $2$-globular sets $\psi \colon U_{P} I_2(j, k) = Q(j, k) \rightarrow A$ such that
      \[
        \xy
          % POINTS
          (0, 0)*+{PQ(j, k)}="0,0";
          (24, 0)*+{P A}="1,0";
          (0, -16)*+{Q(j, k)}="0,1";
          (24, -16)*+{A}="1,1";
          % ARROWS
          {\ar^-{P\psi} "0,0" ; "1,0"};
          {\ar "0,0" ; "0,1"};
          {\ar^{\theta} "1,0" ; "1,1"};
          {\ar_-{\psi} "0,1" ; "1,1"};
        \endxy
      \]
  commutes, where the left-hand map is the algebra action for $I_2(j, k)$.  The commutativity of this diagram is what ensures that the $P$-algebra structure is preserved.  Thus, the contraction cells that must be preserved are precisely those which are recognised as contraction cells by the $P$-algebra structure, i.e. a contraction cell in $Q(j, k)$ is an algebraic contraction cell if it is the image under the algebra action $PQ(j, k) \rightarrow Q(j, k)$ of a contraction cell in $PQ(j, k)$.  Since the only contraction $1$-cells in $I_2(j, k)$ are the identities, all contraction $1$-cells are algebraic.  The algebraic contraction $2$-cells in $I_2(j, k)$ consist of the identities, and any contraction cells that alter the bracketing of a composite, or alter the number of identities that appear in a composite, but do nothing else.  In particular, the source and target of a non-identity algebraic contraction $2$-cell in $I_2(j, k)$ are always composites of cells in $I_2(j, k)$, and these composites feature the same generating cells in the same order.

  Another pivotal fact about $I_2(j, k)$ is that, in the construction, when we apply the coequaliser (in order to ``add contraction $3$-cells''), we identify all parallel $2$-cells.  Thus in $I_2(j, k)$ there are no distinct parallel $2$-cells.  This allows us to write many of the contraction cells as composites of others.

  \begin{proof}[Proof of Lemma~\ref{lem:i2j1}]
    In this proof, we present the case $j = 3$, before moving on to the case of general $j$, since for a fixed value of $j$ we are able to write down all of the cells in $I_2(j, 1)$ (though note that we still omit certain composites).  We use $j = 3$ rather than $j = 2$ (the simplest case of the lemma) because $I_2(2, 1)$ is too small for this case to exhibit all the features of the general case.

    Suppose we have a $P$-algebra $\mathcal{A}$ and a cocone
      \[
        \xymatrix{
          & I_2(1, 0)^{\amalg 3} \ar[dl]_{(d^{\amalg 3})_0} \ar[dr]^{I_2(d_1, 1)^{\amalg 3}} && I_2(1, 0)^{\amalg 3} \ar[dl]_{I_2(d_0, 1)^{\amalg 3}} \ar[dr]^{(d^{\amalg 3})_0}  \\
          I_2(3, 0) \ar[drr]_g && I_2(1, 1)^{\amalg 3} \ar[d]_{\lambda} && I_2(3, 0) \ar[dll]^h  \\
          && \mathcal{A}
                 }
      \]
    in $P\Alg$.  We define a map of $P$-algebras
      \[
        \psi \colon I_2(3, 1) \rightarrow \mathcal{A}
      \]
    such that the diagram
      \[
        \xymatrix{
          & I_2(1, 0)^{\amalg 3} \ar[dl]_{(d^{\amalg 3})_0} \ar[dr]^{I_2(d_1, 1)^{\amalg 3}} && I_2(1, 0)^{\amalg 3} \ar[dl]_{I_2(d_0, 1)^{\amalg 3}} \ar[dr]^{(d^{\amalg 3})_0}  \\
          I_2(3, 0) \ar@/_2pc/[ddrr]_g \ar[drr]_{I_2(1, d_1)} && I_2(1, 1)^{\amalg 3} \ar@/_4pc/[dd]_{\lambda} \ar[d]^{(d^{\amalg 3})_1} && I_2(3, 0) \ar@/^2pc/[ddll]^h \ar[dll]^{I_2(1, d_0)}  \\
          && I_2(3, 1) \ar@{-->}[d]^{\psi} \\
          && \mathcal{A}
                 }
      \]
    commutes.

    To define the map $\psi$, we first list the cells in $I_2(3, 1)$.  We list the cells by dimension, and for dimensions above $0$, we break the list down further, into generating cells, contraction cells, and composites.
      \begin{itemize}
        \item $0$-cells: $a_u$ for all $ 0 \leq u \leq 3$;
        \item $1$-cells:
          \begin{itemize}
            \item Generating cells:
              \[
                f_{uv}^{z} \text{ for all } 0 \leq u < v \leq 3, \; 0 \leq z \leq 1;
              \]
            \item Contraction cells:
              \[
                [a_u, a_u] = \id_{a_u} \text{ for all } 0 \leq u \leq 3;
              \]
            \item Composites:  Although we don't need to define the action of $\psi$ on composites, since this is determined by the fact that $\psi$ preserves the $P$-algebra structure, it is useful to list them here since we need to know what they are in order to write down the contraction $2$-cells.  Note that this list does not include composites involving identities.
                  \[
                    f_{vw}^z \comp f_{uv}^y \text{ for all } 0 \leq u < v < w \leq 3, \; y, z \in \{ 0, 1 \};
                  \]
                  \[
                    (f_{23}^z \comp f_{12}^y) \comp f_{01}^x, \; f_{23}^z \comp (f_{12}^y \comp f_{01}^x) \text{ for all } x, y, z \in \{ 0, 1 \}
                  \]
          \end{itemize}
        \item $2$-cells:
          \begin{itemize}
            \item Generating cells:
              \[
                \alpha_{uv}^{1} \text{ for all } 0 \leq u < v \leq 3;
              \]
            \item Contraction cells:  There are three different types of contraction cell in $I_2(3, 1)$ -- the algebraic contraction cells, the triangular contraction cells corresponding to the cells denoted $\iota^z_{uvw}$ in Leinster nerve construction (see Section~\ref{sect:Leinsternerves}), and those which are composites of cells of the two other types.

                The algebraic contraction cells are those of the form:
                  \[
                    [(f_{23}^z \comp f_{12}^y) \comp f_{01}^x, f_{23}^z \comp (f_{12}^y \comp f_{01}^x)],
                  \]
                  \[
                    [f_{23}^z \comp (f_{12}^y \comp f_{01}^x), (f_{23}^z \comp f_{12}^y) \comp f_{01}^x],
                  \]
                for all $x$, $y$, $z \in \{0, 1\}$, as well as identities on all $1$-cells.  The triangular contraction cells, all of which lie in the image of either $I_2(1, d_1)$ or $I_2(1, d_0)$, are those of the form:
                  \[
                    [f_{uw}^0, f_{vw}^0 \comp f_{uv}^0] = I_2(1, d_1)[f_{uw}^0, f_{vw}^0 \comp f_{uv}^0],
                  \]
                  \[
                    [f_{vw}^0 \comp f_{uv}^0, f_{uw}^0] = I_2(1, d_1)[f_{vw}^0 \comp f_{uv}^0, f_{uw}^0],
                  \]
                  \[
                    [f_{uw}^1, f_{vw}^1 \comp f_{uv}^1] = I_2(1, d_0)[f_{uw}^0, f_{vw}^0 \comp f_{uv}^0],
                  \]
                  \[
                    [f_{vw}^1 \comp f_{uv}^1, f_{uw}^1] = I_2(1, d_0)[f_{vw}^0 \comp f_{uv}^0, f_{uw}^0],
                  \]
                For all $0 \leq u < v < w \leq 3$.  The remaining contraction cells are composites of those above:
                  \[
                    [f_{13}^0 \comp f_{01}^1, (f_{23}^0 \comp f_{12}^0) \comp f_{01}^1] = [f_{13}^0, f_{23}^0 \comp f_{12}^0] * [f_{01}^1, f_{01}^1],
                  \]
                  \[
                    [(f_{23}^0 \comp f_{12}^0) \comp f_{01}^1, f_{13}^0 \comp f_{01}^1] = [f_{23}^0 \comp f_{12}^0, f_{13}^0] * [f_{01}^1, f_{01}^1],
                  \]
                  \[
                    [f_{13}^1 \comp f_{01}^0, (f_{23}^1 \comp f_{12}^1) \comp f_{01}^0] = [f_{13}^1, f_{23}^1 \comp f_{12}^1] * [f_{01}^0, f_{01}^0],
                  \]
                  \[
                    [(f_{23}^1 \comp f_{12}^1) \comp f_{01}^0, f_{13}^1 \comp f_{01}^0] = [f_{23}^1 \comp f_{12}^1, f_{13}^1] * [f_{01}^0, f_{01}^0],
                  \]
                  \[
                    [f_{23}^0 \comp f_{02}^1, f_{23}^0 \comp (f_{12}^1 \comp f_{01}^1)] = [f_{23}^0, f_{23}^0] * [f_{02}^1, f_{12}^1 \comp f_{01}^1],
                  \]
                  \[
                    [f_{23}^0 \comp (f_{12}^1 \comp f_{01}^1), f_{23}^0 \comp f_{02}^1] = [f_{23}^0, f_{23}^0] * [f_{12}^1 \comp f_{01}^1, f_{02}^1],
                  \]
                  \[
                    [f_{23}^1 \comp f_{02}^0, f_{23}^1 \comp (f_{12}^0 \comp f_{01}^0)] = [f_{23}^1, f_{23}^1] * [f_{02}^0, f_{12}^0 \comp f_{01}^0],
                  \]
                  \[
                    [f_{23}^1 \comp (f_{12}^0 \comp f_{01}^0), f_{23}^1 \comp f_{02}^0] = [f_{23}^1, f_{23}^1] * [f_{12}^0 \comp f_{01}^0, f_{02}^0].
                  \]
          \end{itemize}
      \end{itemize}

    We now define the map $\psi \colon I_2(3, 1) \rightarrow \mathcal{A}$:
      \begin{itemize}
        \item On $0$-cells:
          \[
            \psi_0(a_u) := g_0(a_u) = h_0(a_u) = \lambda_0(a_u).
          \]
        \item On $1$-cells:
        \[
          \psi_1(f^z_{uv}) :=
            \left\{
              \begin{array}{rl}
              g_1(f^z_{uv}) & \text{if } z = 0, \\
              h_1(f^{z - 1}_{uv}) & \text{if } z = 1;
              \end{array}
            \right.
        \]
        \[
          \psi_1[a_u, a_u] = \psi(\id_{a_u}) := \lambda_1(\id_{a_u}) = g_1(\id_{a_u}) = h_1(\id_{a_u}).
        \]
      We do not need to define the action of $\psi_1$ on composites explicitly; this is automatic since $\psi$ must preserve the $P$-algebra structure.
        \item On $2$-cells:
          \[
            \psi_2(\alpha^1_{uv}) := \lambda(\alpha^1_{uv});
          \]
          \[
            \psi_2[(f_{23}^z \comp f_{12}^y) \comp f_{01}^x, f_{23}^z \comp (f_{12}^y \comp f_{01}^x)] := [\psi_1 \big( (f_{23}^z \comp f_{12}^y) \comp f_{01}^x \big), \psi_1\big(f_{23}^z \comp (f_{12}^y \comp f_{01}^x)\big)];
          \]
          \[
            \psi_2[f_{23}^z \comp (f_{12}^y \comp f_{01}^x), (f_{23}^z \comp f_{12}^y) \comp f_{01}^x] := [\psi_1\big(f_{23}^z \comp (f_{12}^y \comp f_{01}^x)\big), \psi_1\big((f_{23}^z \comp f_{12}^y) \comp f_{01}^x\big)];
          \]
          \[
            \psi_2[f^0_{uw}, f^0_{vw} \comp f^0_{uv}] := g_2[f^0_{uw}, f^0_{vw} \comp f^0_{uv}];
          \]
          \[
            \psi_2[f^0_{vw} \comp f^0_{uv}, f^0_{uw}] := g_2[f^0_{vw} \comp f^0_{uv}, f^0_{uw}];
          \]
          \[
            \psi_2[f^1_{uw}, f^1_{vw} \comp f^1_{uv}] := h_2[f^0_{uw}, f^0_{vw} \comp f^0_{uv}];
          \]
          \[
            \psi_2[f^1_{vw} \comp f^1_{uv}, f^1_{uw}] := h_2[f^0_{vw} \comp f^0_{uv}, f^0_{uw}].
          \]
        As with $1$-cells, we do not need to define the action of $\psi_2$ on  composites, including those contraction cells that are composites of others, since $\psi$ must preserve the $P$-algebra structure.
      \end{itemize}

    We see by definition of $\psi$ that it is a map of $P$-algebras, and that it makes the required diagram commute.  It is clear that, at each stage of the construction of $\psi$, if we defined the map differently it would not have satisfied these conditions; in the case of the cells on which $\psi$ is defined explicitly, any other definition would fail to make the diagram commute, and in the case of all other cells, any other definition would fail to give a map of $P$-algebras.

    Thus, $\psi$ is the unique map of $P$-algebras making the required diagram commute, so $I_2(3, 1)$ is the colimit in $P\Alg$ of the diagram
      \[
        \xymatrix{
          & I_2(1, 0)^{\amalg 3} \ar[dl]_{(d^{\amalg 3})_0} \ar[dr]^{I_2(d_1, 1)^{\amalg 3}} && I_2(1, 0)^{\amalg 3} \ar[dl]_{I_2(d_0, 1)^{\amalg 3}} \ar[dr]^{(d^{\amalg 3})_0}  \\
          I_2(3, 0) && I_2(1, 1)^{\amalg 3} && I_2(3, 0)
                 }
      \]

    We now prove the lemma for a general value of $j$.  Suppose we have a $P$-algebra $\mathcal{A}$ and a cocone
      \[
        \xymatrix{
          & I_2(1, 0)^{\amalg j} \ar[dl]_{(d^{\amalg j})_0} \ar[dr]^{I_2(d_1, 1)^{\amalg j}} && I_2(1, 0)^{\amalg j} \ar[dl]_{I_2(d_0, 1)^{\amalg j}} \ar[dr]^{(d^{\amalg j})_0}  \\
          I_2(j, 0) \ar[drr]_g && I_2(1, 1)^{\pushj} \ar[d]_{\lambda} && I_2(j, 0) \ar[dll]^h  \\
          && \mathcal{A}
                 }
      \]
    in $P\Alg$.  We define a map of $P$-algebras
      \[
        \psi \colon I_2(j, 1) \rightarrow \mathcal{A}
      \]
    such that the diagram
      \[
        \xymatrix{
          & I_2(1, 0)^{\pushj} \ar[dl]_{(d^{\pushj})_0} \ar[dr]^{I_2(d_1, 1)^{\pushj}} && I_2(1, 0)^{\pushj} \ar[dl]_{I_2(d_0, 1)^{\pushj}} \ar[dr]^{(d^{\pushj})_0}  \\
          I_2(j, 0) \ar@/_2pc/[ddrr]_g \ar[drr]_{I_2(1, d_1)} && I_2(1, 1)^{\pushj} \ar@/_4pc/[dd]_{\lambda} \ar[d]^{(d^{\pushj})_1} && I_2(j, 0) \ar@/^2pc/[ddll]^h \ar[dll]^{I_2(1, d_0)}  \\
          && I_2(j, 1) \ar@{-->}[d]^{\psi} \\
          && \mathcal{A}
                 }
      \]
    commutes.

    To define the map $\psi$, we first list the cells in $I_2(j, 1)$.  As for the case $j = 3$, we list the cells by dimension, and for dimensions above $0$, we list generating cells and contraction cells separately.  Note that in this case we do not list the composites, since the notation would become very unwieldy; the action of $\psi$ on composites is determined by the fact that fact that it must preserve the $P$-algebra structure, so we do not need to list the composites explicitly.
      \begin{itemize}
        \item $0$-cells: $a_u$ for all $ 0 \leq u \leq j$;
        \item $1$-cells:
          \begin{itemize}
            \item Generating cells:
              \[
                f_{uv}^{z} \text{ for all } 0 \leq u < v \leq j, \; 0 \leq z \leq 1;
              \]
            \item Contraction cells:
              \[
                [a_u, a_u] = \id_{a_u} \text{ for all } 0 \leq u \leq j;
              \]
          \end{itemize}
        \item $2$-cells:
          \begin{itemize}
            \item Generating cells:
              \[
                \alpha_{uv}^{1} \text{ for all } 0 \leq u < v \leq j;
              \]
            \item Contraction cells:  As in the case $j = 3$, we have algebraic contraction cells and triangular contraction cells corresponding to the cells $\iota^z_{uvw}$; since all diagrams of contraction $2$-cells commute in $I_2(j, 1)$, all other contraction cells can be expressed as composites of contraction cells of these two types.

                The algebraic contraction cells are those mediating between differently bracketed composites of the same $1$-cells, and also identities on all $1$-cells.  The triangular contraction cells are those of the form:
                  \[
                    [f_{uw}^0, f_{vw}^0 \comp f_{uv}^0] = I_2(1, d_1)[f_{uw}^0, f_{vw}^0 \comp f_{uv}^0],
                  \]
                  \[
                    [f_{vw}^0 \comp f_{uv}^0, f_{uw}^0] = I_2(1, d_1)[f_{vw}^0 \comp f_{uv}^0, f_{uw}^0],
                  \]
                  \[
                    [f_{uw}^1, f_{vw}^1 \comp f_{uv}^1] = I_2(1, d_0)[f_{uw}^0, f_{vw}^0 \comp f_{uv}^0],
                  \]
                  \[
                    [f_{vw}^1 \comp f_{uv}^1, f_{uw}^1] = I_2(1, d_0)[f_{vw}^0 \comp f_{uv}^0, f_{uw}^0],
                  \]
                for all $0 \leq u < v < w \leq j$.  All remaining contraction cells are horizontal composites of those of the form
                  \[
                    [f^z_{v_{m - 1}, v_m} \comp \dotsb \comp f^z_{v_1, v_2} \comp f^z_{v_0, v_1}, f^z_{u_{l - 1}, u_l} \comp \dotsb \comp f^z_{u_1, u_2} \comp f^z_{u_0, u_1}],
                  \]
                for all $l$, $m \geq 2$, $0 \leq u_0 < u_1 < \dotsb < u_l \leq j$, $u_0 = v_0 < v_1 < \dotsb < v_m = u_l$, $0 \leq z \leq 1$.  Note that we omit the choice of bracketing in the contraction cell above; there is one such cell for each choice of bracketing of the source and target.  Each of these contraction cells can be written as a composite of algebraic contraction cells and the triangular contraction cells above.
          \end{itemize}
      \end{itemize}

    We now define the map $\psi \colon I_2(j, 1) \rightarrow \mathcal{A}$:
      \begin{itemize}
        \item On $0$-cells:
          \[
            \psi_0(a_u) := g_0(a_u) = h_0(a_u) = \lambda_0(a_u).
          \]
        \item On $1$-cells:
        \[
          \psi_1(f^z_{uv}) :=
            \left\{
              \begin{array}{rl}
              g_1(f^z_{uv}) & \text{if } z = 0, \\
              h_1(f^{z - 1}_{uv}) & \text{if } z = 1;
              \end{array}
            \right.
        \]
        \[
          \psi_1[a_u, a_u] = \psi(\id_{a_u}) := \lambda_1(\id_{a_u}) = g_1(\id_{a_u}) = h_1(\id_{a_u}).
        \]
      As in the case $j = 3$, we do not need to describe the action of $\psi$ on composites explicitly, since it must preserve the $P$-algebra structure.
        \item On $2$-cells:
          \[
            \psi_2(\alpha^1_{uv}) := \lambda(\alpha^1_{uv});
          \]
          \[
            \psi_2[f^0_{uw}, f^0_{vw} \comp f^0_{uv}] := g_2[f^0_{uw}, f^0_{vw} \comp f^0_{uv}];
          \]
          \[
            \psi_2[f^0_{vw} \comp f^0_{uv}, f^0_{uw}] := g_2[f^0_{vw} \comp f^0_{uv}, f^0_{uw}];
          \]
          \[
            \psi_2[f^1_{uw}, f^1_{vw} \comp f^1_{uv}] := h_2[f^0_{uw}, f^0_{vw} \comp f^0_{uv}];
          \]
          \[
            \psi_2[f^1_{vw} \comp f^1_{uv}, f^1_{uw}] := h_2[f^0_{vw} \comp f^0_{uv}, f^0_{uw}].
          \]
        As in the case $j = 3$, we do not need to describe the action of $\psi$ on the remaining $2$-cells explicitly, since they are either algebraic contraction cells, or composites involving the algebraic contraction cells and those above.
      \end{itemize}

    We see by definition of $\psi$ that it is a map of $P$-algebras, and that it makes the required diagram commute.  It is clear that, at each stage of the construction of $\psi$, if we defined the map differently it would not have satisfied these conditions; in the case of the cells on which $\psi$ is defined explicitly, any other definition would fail to make the diagram commute, and in the case of all other cells, any other definition would fail to give a map of $P$-algebras.

    Thus, $\psi$ is the unique map of $P$-algebras making the required diagram commute, so $I_2(j, 1)$ is the colimit in $P\Alg$ of the diagram
      \[
        \xymatrix{
          & I_2(1, 0)^{\amalg j} \ar[dl]_{d^{\amalg j}_0} \ar[dr]^{I_2(d_1, 1)^{\amalg j}} && I_2(1, 0)^{\amalg j} \ar[dl]_{I_2(d_0, 1)^{\amalg j}} \ar[dr]^{d^{\amalg j}_0}  \\
          I_2(j, 0) && I_2(1, 1)^{\amalg j} && I_2(j, 0),
                 }
      \]
    as required.
  \end{proof}

  The following is now an immediate corollary of Lemma~\ref{lem:i2j1}, via our characterisation of fullness and faithfulness of the Segal maps in terms of colimits in $P\Alg$.

  \begin{corol}  \label{prop:Sjfandf}
    Let $\mathcal{A}$ be a Penon weak $2$ category.  For all $j > 0$, the Segal map
      \[
        S_j \colon P\Alg(I_2(j, -), \mathcal{A}) \longrightarrow P\Alg(I_2(1, -)^{\pushj}, \mathcal{A})
      \]
    is full and faithful on $1$-cells.
  \end{corol}

  We now apply a similar argument to the Segal maps $S_{j, k}$, and reformulate the remaining part of the Segal condition in terms of colimits of $P$-algebras, as we did for $S_j$.  By Lemma~\ref{lem:sk}, $S_{j, k}$ is given by
      \[
        \underbrace{\nerve\mathcal{A}(j, 1) \times_{\nerve\mathcal{A}(j, 0)} \dotsb \times_{\nerve\mathcal{A}(j, 0)} \nerve\mathcal{A}(j, 1)}_k \iso P\Alg(I_2(j, 1)^{\pushk}, \mathcal{A}).
      \]
  This is a bijection if $I_2(j, 1)^{\pushk} = I_2(j, k)$, and the map
      \[
        d^{\pushk} \colon I_2(j, 1)^{\pushk} \rightarrow I_2(j, k)
      \]
  is the identity.  This tells us that $I_2(j, k)$ can be obtained by gluing $k$ copies of $I_2(j, 1)$ along their boundary copies of $I_2(j, 0)$.  Thus, the Segal map $S_{j, k}$ is a bijection if the following lemma holds:

  \begin{lemma} \label{lem:i2jk}
    For all $j \geq 0$, $k > 0$, the diagram
      \[
        \xy
          % POINTS
          (-32, 0)*+{I_2(j, 0)}="-1,0";
          (-48,-16)*+{I_2(j, 1)}="-2,1";
          (-16,-16)*+{I_2(j, 1)}="-1,1";
          (32, 0)*+{I_2(j, 0)}="1,0";
          (48,-16)*+{I_2(j, 1)}="2,1";
          (16,-16)*+{I_2(j, 1)}="1,1";
          (0, -16)*+{\dotsc};
          (0, -32)*+{I_2(j, k)}="0,2";
          % ARROWS
          {\ar_{I_2(1, d_0)} "-1,0" ; "-2,1"};
          {\ar^{I_2(1, d_1)} "-1,0" ; "-1,1"};
          {\ar_{I_2(1, d_0)} "1,0" ; "1,1"};
          {\ar^{I_2(1, d_1)} "1,0" ; "2,1"};
          {\ar@/_1pc/_{I_2(1, \iota_1)} "-2,1" ; "0,2"};
          {\ar_{I_2(1, \iota_2)} "-1,1" ; "0,2"};
          {\ar^{I_2(1, \iota_{k - 1})} "1,1" ; "0,2"};
          {\ar@/^1pc/^{I_2(1, \iota_k)} "2,1" ; "0,2"};
        \endxy
      \]
    is a colimit cocone in $P\Alg$.
  \end{lemma}

  \begin{proof}
    Let $\mathcal{A}$ be a Penon weak $2$-category, and suppose we have a cocone
      \[
        \xy
          % POINTS
          (-32, 0)*+{I_2(j, 0)}="-1,0";
          (-48,-16)*+{I_2(j, 1)}="-2,1";
          (-16,-16)*+{I_2(j, 1)}="-1,1";
          (32, 0)*+{I_2(j, 0)}="1,0";
          (48,-16)*+{I_2(j, 1)}="2,1";
          (16,-16)*+{I_2(j, 1)}="1,1";
          (0, -16)*+{\dotsc};
          (0, -32)*+{\mathcal{A}}="0,3";
          % ARROWS
          {\ar_{I_2(1, d_0)} "-1,0" ; "-2,1"};
          {\ar^{I_2(1, d_1)} "-1,0" ; "-1,1"};
          {\ar_{I_2(1, d_0)} "1,0" ; "1,1"};
          {\ar^{I_2(1, d_1)} "1,0" ; "2,1"};
          {\ar@/_1pc/_{g^{(1)}} "-2,1" ; "0,3"};
          {\ar_{g^{(2)}} "-1,1" ; "0,3"};
          {\ar^{g^{(k - 1)}} "1,1" ; "0,3"};
          {\ar@/^1pc/^{g^{(k)}} "2,1" ; "0,3"};
        \endxy
      \]
    in $P\Alg$.  We define a map of $P$-algebras
      \[
        \psi \colon I_2(j, k) \longrightarrow \mathcal{A}
      \]
    such that the diagram
      \[
        \xy
          % POINTS
          (-32, 0)*+{I_2(j, 0)}="-1,0";
          (-48,-16)*+{I_2(j, 1)}="-2,1";
          (-16,-16)*+{I_2(j, 1)}="-1,1";
          (32, 0)*+{I_2(j, 0)}="1,0";
          (48,-16)*+{I_2(j, 1)}="2,1";
          (16,-16)*+{I_2(j, 1)}="1,1";
          (0, -16)*+{\dotsc};
          (0, -32)*+{I_2(j, k)}="0,2";
          (0, -48)*+{\mathcal{A}}="0,3";
          % ARROWS
          {\ar_{I_2(1, d_0)} "-1,0" ; "-2,1"};
          {\ar^{I_2(1, d_1)} "-1,0" ; "-1,1"};
          {\ar_{I_2(1, d_0)} "1,0" ; "1,1"};
          {\ar^{I_2(1, d_1)} "1,0" ; "2,1"};
          {\ar@/_1pc/ "-2,1" ; "0,2"};
          {\ar "-1,1" ; "0,2"};
          {\ar "1,1" ; "0,2"};
          {\ar@/^1pc/ "2,1" ; "0,2"};
          {\ar@/_2pc/_{g^{(1)}} "-2,1" ; "0,3"};
          {\ar@/_1pc/_(0.65){g^{(2)}} "-1,1" ; "0,3"};
          {\ar@/^1pc/^(0.65){g^{(k - 1)}} "1,1" ; "0,3"};
          {\ar@/^2pc/^{g^{(k)}} "2,1" ; "0,3"};
          {\ar@{-->}^{\psi} "0,2" ; "0,3"};
        \endxy
      \]
    commutes, and show that this is the unique such map of $P$-algebras.  We take the same approach as in the proof of Lemma~\ref{lem:i2j1}, defining the map by an elementary approach, and using the fact that it must preserve the $P$-algebra structure to avoid having to define it explicitly on every cell of $I_2(j, k)$.  To do so we now list the cells of $I_2(j, k)$; note that, as before, we do not list composites or algebraic contraction cells.
      \begin{itemize}
        \item $0$-cells: $a_u$ for all $ 0 \leq u \leq j$;
        \item $1$-cells:
          \begin{itemize}
            \item Generating cells:
              \[
                f_{uv}^{z} \text{ for all } 0 \leq u < v \leq j, \; 0 \leq z \leq k;
              \]
            \item Contraction cells:
              \[
                [a_u, a_u] = \id_{a_u} \text{ for all } 0 \leq u \leq j;
              \]
          \end{itemize}
        \item $2$-cells:
          \begin{itemize}
            \item Generating cells:
              \[
                \alpha_{uv}^{z} \text{ for all } 0 \leq u < v \leq j, \; 1 \leq z \leq k;
              \]
            \item Contraction cells:  As in Lemma~\ref{lem:i2j1}, we have algebraic contraction cells and triangular contraction cells corresponding to the cells $\iota^z_{uvw}$ from Leinster's nerve construction for bicategories (Section~\ref{sect:Leinsternerves}); since all diagrams of contraction $2$-cells commute in $I_2(j, 1)$, all other contraction cells can be expressed as composites of contraction cells of these two types.

                The algebraic contraction cells are those mediating between differently bracketed composites of the same $1$-cells, and also identities on all $1$-cells.  The triangular contraction cells are those of the form:
                  \[
                    [f_{uw}^z, f_{vw}^z \comp f_{uv}^z],
                  \]
                and
                  \[
                    [f_{vw}^z \comp f_{uv}^z, f_{uw}^z],
                  \]
                for all $0 \leq u < v < w \leq j$, $0 \leq z \leq k$.  As in Lemma~\ref{lem:i2j1}, all remaining contraction cells are composites of those above.
          \end{itemize}
      \end{itemize}

    We now define the map $\psi \colon I_2(j, k) \rightarrow \mathcal{A}$:
      \begin{itemize}
        \item On $0$-cells:
          \[
            \psi_0(a_u) := g^{(1)}_0(a_u).
          \]
        \item On $1$-cells:
        \[
          \psi_1(f^z_{uv}) :=
            \left\{
              \begin{array}{rl}
              g^{(0)}_1(f^0_{uv}) & \text{if } z = 0, \\
              g^{(z)}_1(f^1_{uv}) & \text{otherwise};
              \end{array}
            \right.
        \]
        \[
          \psi_1[a_u, a_u] = \psi(\id_{a_u}) := g^{(1)}_1(\id_{a_u}).
        \]
      As in Lemma~\ref{lem:i2j1}, we do not need to describe the action of $\psi$ on composites explicitly, since it must preserve the $P$-algebra structure.
        \item On $2$-cells:
          \[
            \psi_2(\alpha^z_{uv}) := g^{(z)}_2(\alpha^1_{uv});
          \]
          \[
            \psi_2[f^0_{uw}, f^0_{vw} \comp f^0_{uv}] := g^{(1)}_2[f^0_{uw}, f^0_{vw} \comp f^0_{uv}];
          \]
          \[
            \psi_2[f^0_{vw} \comp f^0_{uv}, f^0_{uw}] := g^{(1)}_2[f^0_{vw} \comp f^0_{uv}, f^0_{uw}];
          \]
        and for $1 \leq z \leq k$,
          \[
            \psi_2[f^z_{uw}, f^z_{vw} \comp f^z_{uv}] := g^{(z)}_2[f^1_{uw}, f^1_{vw} \comp f^1_{uv}];
          \]
          \[
            \psi_2[f^z_{vw} \comp f^z_{uv}, f^z_{uw}] := g^{(z)}_2[f^1_{vw} \comp f^1_{uv}, f^1_{uw}].
          \]
        As in Lemma~\ref{lem:i2j1}, we do not need to describe the action of $\psi$ on the remaining $2$-cells explicitly, since they are either algebraic contraction cells, or composites involving the algebraic contraction cells and those above.
      \end{itemize}

    We see by definition of $\psi$ that it is a map of $P$-algebras, and that it makes the required diagram commute.  It is clear that, at each stage of the construction of $\psi$, if we defined the map differently it would not have satisfied these conditions; in the case of the cells on which $\psi$ is defined explicitly, any other definition would fail to make the diagram commute, and in the case of all other cells, any other definition would fail to give a map of $P$-algebras.

    Thus, $\psi$ is the unique map of $P$-algebras making the required diagram commute, so $I_2(j, k)$ is the colimit in $P\Alg$ of the diagram
      \[
        \underbrace{
        \xy
          % POINTS
          (-32, 0)*+{I_2(j, 0)}="-1,0";
          (-48,-16)*+{I_2(j, 1)}="-2,1";
          (-16,-16)*+{I_2(j, 1)}="-1,1";
          (32, 0)*+{I_2(j, 0)}="1,0";
          (48,-16)*+{I_2(j, 1)}="2,1";
          (16,-16)*+{I_2(j, 1)}="1,1";
          (0, -16)*+{\dotsc};
          % ARROWS
          {\ar_{I_2(1, d_0)} "-1,0" ; "-2,1"};
          {\ar^{I_2(1, d_1)} "-1,0" ; "-1,1"};
          {\ar_{I_2(1, d_0)} "1,0" ; "1,1"};
          {\ar^{I_2(1, d_1)} "1,0" ; "2,1"};
        \endxy}_k
      \]
    as required.
  \end{proof}

  The following is now an immediate corollary of Lemma~\ref{lem:i2jk}:

  \begin{corol}  \label{prop:Sjkbij}
    Let $\mathcal{A}$ be a Penon weak $2$-category.  For each $j$, $k > 0$, the Segal map
      \[
        S_{j, k} \colon \nerve\mathcal{A}(j, k) \rightarrow \underbrace{\nerve\mathcal{A}(j, 1) \times_{\nerve\mathcal{A}(j, 0)} \dotsb \times_{\nerve\mathcal{A}(j, 0)} \nerve\mathcal{A}(j, 1)}_k
      \]
    is a bijection.
  \end{corol}

  We now have all the results we need to show that the nerve of a Penon weak $2$-category is a Tamsamani--Simpson weak $2$-category.

  \begin{thm}  \label{thm:2nerveSegalcond}
    Let $\mathcal{A}$ be a Penon weak $2$-category.  Then the nerve $\nerve\mathcal{A}$ satisfies the Segal condition, and is thus a Tamsamani--Simpson weak $2$-category.
  \end{thm}

  \begin{proof}
    Let $\mathcal{A}$ be a Penon weak $2$-category, and consider its nerve $\nerve \mathcal{A}$.  For all $j \geq 0$, the Segal map
          \[
            S_j: \nerve\mathcal{A}(j,-) \longrightarrow \underbrace{\nerve\mathcal{A}(1,-) \times_{\nerve\mathcal{A}(0,1)} \dotsb \times_{\nerve\mathcal{A}(0,1)} \nerve\mathcal{A}(1,-)}_j
          \]
    is surjective on objects by Proposition~\ref{prop:soo2} and full and faithful on $1$-cells by Corollary~\ref{prop:Sjfandf}; hence $S_j$ is contractible.  Note that the proposition and corollary are valid only for $j > 0$, but for $j = 0$ the result holds trivially.

    For all $j$, $k \geq 0$, the Segal map
          \[
            S_{j,k}: \nerve\mathcal{A}(j,k) \longrightarrow \underbrace{\nerve\mathcal{A}(j,1) \times_{\nerve\mathcal{A}(j,0)} \dotsb \times_{\nerve\mathcal{A}(j,0)} \nerve\mathcal{A}(j,1)}_k
          \]
    is a bijection by Corollary~\ref{prop:Sjkbij}.  As above, this corollary is only valid for $k > 0$, but for $k = 0$ the result holds trivially.

    Hence $\nerve \mathcal{A}$ satisfies the Segal condition, so it is a Tamsamani--Simpson weak $2$-category.
  \end{proof}

\chapter{Nerves of Penon weak $n$-categories}  \label{chap:nerveconstrn}

  In this chapter we generalise the nerve construction for Penon weak $2$-categories from Section~\ref{sect:2nerveconstr} to a nerve construction for Penon weak $n$-categories for all $n \in \mathbb{N}$.  We then discuss various questions that this nerve construction raises, and in particular we conjecture that the nerve of a Penon weak $n$-category is a Tamsamani--Simpson weak $n$-category.

\section{The nerve construction for general $n$}  \label{sect:nnerveconstr}

  We now give the construction of the nerve functor for Penon weak $n$-categories.  This construction proceeds analogously to that for $n = 2$.  Since we are potentially working with a greater number of dimensions in the general case, we have to weaken composition in each cuboidal $n$-pasting diagram at every dimension (apart from dimensions $0$ and $n$).  The greater number of dimensions entails that the notation for the cells of the $P$-algebras we construct necessarily becomes more complicated and unwieldy.

  In analogy with the case $n = 2$, when defining the nerve functor for Penon weak $n$-categories, we first define a functor $I_n \colon \Theta^n \rightarrow P\Alg$ which gives us, for each object of $\Theta^n$, the corresponding cuboidal $n$-pasting diagram expressed as a freely generated Penon weak $n$-category.  We obtain the functor $I_n$ by defining a functor $E_n \colon \Theta^n \rightarrow \mathcal{Q}$, then composing this with the Eilenberg--Moore comparison functor $K \colon \mathcal{Q} \rightarrow P\Alg$ for the adjunction $F \ladj U$ defining the monad $P$.

  As in the $2$-dimensional case, for each object $\vectj = (j_1, j_2, \dotsc, j_n)$ of $\Theta^n$, we define two $n$-globular sets, $X(\vectj)$ and $R(\vectj)$; $X(\vectj)$ is the associated $n$-globular set of the cuboidal pasting diagram $\vectj$, while $R(\vectj)$ also contains extra cells to weaken the composition structure on certain simplicial shapes of composite.  We then define an object of $\Rn$
    \[
      \xy
        (0, 0)*+{R(\vectj)}="0";
        (24, 0)*+{TX(\vectk),}="1";
        {\ar^-{\theta_{\vectj}} "0" ; "1"};
      \endxy
    \]
  and define $E_n(\vectj)$ to be the image of this under the functor $J \colon \Rn \rightarrow \Qn$; that is, the left adjoint to the forgetful functor $W \colon \Qn \rightarrow \Rn$, as described in Section~\ref{sect:ladj}.

  Before giving the construction, once again we discuss the notation we will use.  We will use a coordinate system similar to that used in the $2$-dimensional construction.  The difference is that, since higher dimensional cells require a greater number of coordinates, instead of using subscripts and superscripts, the coordinates of a cell will be written as a string in brackets.  Thus, the $m$-cell
    \[
      \alpha^m(u_0, v_0; u_1, v_1; \dotsc; u_{m - 1}, v_{m - 1}; z)
    \]
  has source $(m - 1)$-cell with coordinates $(u_0, v_0; \dotsc; u_{m - 2}, v_{m - 2}; u_{m - 1})$ and target $(k - 1)$-cell with coordinates $(u_0, v_0; \dotsc; u_{m - 2}, v_{m - 2}; v_{m - 1})$.  The $z$-coordinate indicates the position of this cell in relation to the other $m$-cells parallel to it, and the superscript $m$ indicates the dimension of the cell.  As in the $2$-dimensional construction, each $n$-cell has the same coordinates as its target $(n - 1)$-cell.

  Recall that an object of $\Theta^n$ is an equivalence class of objects of $\Delta^n$.  An object of $\Delta^n$ is in an equivalence class with more than one member if and only if it has a $0$ in the $k$th position for some $k < n$.  Thus, for the purposes of the following definition we treat the equivalence class of $(l_1, \dotsc, l_{m - 1}, 0, l_{m + 1}, \dotsc, l_n)$, with $m < n$, as the object $(l_1, \dotsc, l_{m - 1}, 0, 0, \dotsc, 0)$ of $\Delta^n$; all other equivalence classes are treated as their sole member.

  Let $\vectj \in \Theta^n$ and define $n$-globular sets $X(\vectj)$ and $R(\vectj)$ as follows: $X(\vectj)$ is defined by
    \begin{itemize}
      \item $X(\vectj)_0 = \{ a_u \gt u \in \mathbb{N}, 0 \leq u \leq j_1 \}$;
      \item for $0 < m < n$,
        \begin{align*}
          X(\vectj)_m = & \; \{ \alpha^m(u_1, u_1 + 1; u_2, u_2 + 1; \dotsc ; u_m, u_m + 1; z)  \\
          & \gt 0 \leq u_l < j_l \text{ for all } 1 \leq l \leq m, 0 \leq z \leq j_{m + 1} \};
        \end{align*}
      \item for $m = n$,
        \begin{align*}
          X(\vectj)_n = & \; \{ \alpha^n(u_1, u_1 + 1; u_2, u_2 + 1; \dotsc ; u_{n - 1}, u_{n - 1} + 1; z)  \\
          & \gt 0 \leq u_l < j_l \text{ for all } 1 \leq l \leq n - 1, 1 \leq z \leq j_{n} \};
        \end{align*}
    \end{itemize}
  and $R(\vectj)$ is defined by
    \begin{itemize}
      \item $R(\vectj)_0 = \{ a_u \gt u \in \mathbb{N}, 0 \leq u \leq j_1 \}$;
      \item for $0 < m < n$,
        \begin{align*}
          R(\vectj)_m = & \; \{ \alpha^m(u_1, v_1; u_2, v_2; \dotsc ; u_m, v_m; z)  \\
          & \gt 0 \leq u_l < v_l \leq j_l \text{ for all } 1 \leq l \leq m, 0 \leq z \leq j_{m + 1} \};
        \end{align*}
      \item for $m = n$,
        \begin{align*}
          R(\vectj)_n = & \; \{ \alpha^n(u_1, v_1; u_2, v_2; \dotsc ; u_{n - 1}, v_{n - 1}; z)  \\
          & \gt 0 \leq u_l < v_l \leq j_l \text{ for all } 1 \leq l \leq n - 1, 1 \leq z \leq j_{n} \};
        \end{align*}
    \end{itemize}
  for both $X(\vectj)$ and $R(\vectj)$, the source and target maps are defined by:
    \begin{itemize}
      \item for all $1$-cells $\alpha^1(u_1, v_1; z)$,
        \[
          s(\alpha^1(u_1, v_1; z)) = a_{u_1}, \; t(\alpha^1(u_1, v_1; z)) = a_{v_1};
        \]
      \item for all $1 < m < n$, and for all $m$-cells $\alpha^m(u_1, v_1; u_2, v_2; \dotsc ; u_m, v_m; z)$,
        \begin{align*}
          & s(\alpha^m(u_1, v_1; u_2, v_2; \dotsc ; u_m, v_m; z))  \\
          = \; & \alpha^{m - 1}(u_1, v_1; u_2, v_2; \dotsc ; u_{m - 1}, v_{m - 1}; u_m),
        \end{align*}
        and
        \begin{align*}
          & t(\alpha^m(u_1, v_1; u_2, v_2; \dotsc ; u_m, v_m; z))  \\
          = \; & \alpha^{m - 1}(u_1, v_1; u_2, v_2; \dotsc ; u_{m - 1}, v_{m - 1}; v_m),
        \end{align*}
      \item for all $n$-cells $\alpha^n(u_1, v_1; u_2, v_2; \dotsc ; u_{n - 1}$,
        \begin{align*}
          & s(\alpha^n(u_1, v_1; u_2, v_2; \dotsc ; u_{n - 1}, v_{n - 1}; z))  \\
          = \; & \alpha^{n - 1}(u_1, v_1; u_2, v_2; \dotsc ; u_{n - 1}, v_{n - 1}; z - 1),
        \end{align*}
        and
        \begin{align*}
          & t(\alpha^n(u_1, v_1; u_2, v_2; \dotsc ; u_{n - 1}, v_{n - 1}; z))  \\
          = \; & \alpha^{n - 1}(u_1, v_1; u_2, v_2; \dotsc ; u_{n - 1}, v_{n - 1}; z).
        \end{align*}
    \end{itemize}
  Once again we note that, in spite of the notation, this does not define functors $R$ and $X$ into $\nGSet$.

  We now wish to construct, for each $\vectj \in \Theta^n$, an object of $\Rn$ which will consist of a map from $R(\vectj)$ into the free strict $n$-category on $X(\vectj)$.  Before doing so, we must first establish notation for the freely generated composite cells in $TX(\vectj)$.  Following Penon's notation for composition in an $n$-magma (see Definition~\ref{defn:magma}), given $m$-cells $\alpha_1$, $\alpha_2$ and $0 \leq p < m$, where the target $p$-cell of $\alpha_1$ coincides with the source $p$-cell of $\alpha_2$, we write $\alpha_2 \comp^m_p \alpha_1$ for their composite along boundary $p$-cells.  For composites involving greater numbers of cells we extend this to summation-style notation; for $m$-cells $\alpha_i$, $1 \leq i \leq k$ for some $k$, satisfying the appropriate source and target conditions to be composable, we write
    \[
      {\sumcomp{m}{p}_{1 \leq i \leq k}} \alpha_i := \alpha_k \comp^m_p \alpha_{k - 1} \comp^m_p \dotsb \comp^m_p \alpha_2 \comp^m_p \alpha_1.
    \]
  We now define $\theta_{\vectj} \colon R(\vectj) \rightarrow TX(\vectj)$ by:
    \begin{itemize}
      \item for $a_u \in R(\vectj)_0$, $(\theta_{\vectj})_0(a_u) = a_u$;
      \item for $0 < m < n$, $(\theta_{\vectj})_m(\alpha^m(u_1, v_1; u_2, v_2; \dotsc ; u_m, v_m; z)) =$
        \[
           \sumcomp{m}{m - 1}_{u_m \leq w_m < v_m} \dotsb \sumcomp{m}{0}_{u_1 \leq w_1 < v_1} \alpha^m(w_1, w_1 + 1; w_2, w_2 + 1; \dotsc ; w_m, w_m + 1; z)
        \]
      \item for $m = n$, $(\theta_{\vectj})_n(\alpha^n(u_1, v_1; u_2, v_2; \dotsc ; u_{n - 1}, v_{n - 1}; z)) = $
    \end{itemize}
        \[
           \sumcomp{n}{n - 2}_{u_{n - 1} \leq w_{n - 1} < v_{n - 1}} \dotsb \sumcomp{n}{0}_{u_1 \leq w_1 < v_1} \alpha^m(w_1, w_1 + 1; w_2, w_2 + 1; \dotsc ; w_{n - 1}, w_{n - 1} + 1; z)
        \]
  Similar to the $2$-dimensional case, $\theta$ coincides with $\eta^T_{X(\vectj)}$ whenever $v_l = u_l + 1$ for all $0 \leq l \leq m - 1$.

  To complete the construction of the action of the functor $E_n \colon \Theta^n \rightarrow \Qn$ on objects, we apply the functor $J \colon \Rn \rightarrow \Qn$ to $\theta_{\vectj} \colon R(\vectj) \rightarrow TX(\vectj)$.  This adds to $R(\vectj)$ all the required composites and contraction cells, including those which ensure that the weakened composites (those cells in $R(\vectj)$ but not in $X(\vectj)$) are coherently equivalent to the corresponding freely generated composites at the same level in the pasting diagram. We denote the resulting object of $\Qn$ by
    \[
      \xy
        (0, 0)*+{Q(\vectj)}="0";
        (24, 0)*+{TX(\vectj).}="1";
        {\ar^-{\phi_{\vectj}} "0" ; "1"};
      \endxy
    \]

  We now define the action of the functor $E_n: \Theta^n \rightarrow \Qn$ on morphisms.  As in the $2$-dimensional case, to do so we first define a morphism in $\Rn$, then take its transpose under the adjunction
      \[
        \xy
          % POINTS
          (0, 0)*+{\Rn}="Rn";
          (16, 0)*+{\Qn}="Qn";
          % ARROWS
          {\ar@<1ex>^-{J}_-*!/u1pt/{\labelstyle \bot} "Rn" ; "Qn"};
          {\ar@<1ex>^-{W} "Qn" ; "Rn"};
        \endxy
     \]
  to obtain a morphism in $\Qn$.

  Let $\vectp: \vectj \rightarrow \vectk$ be a morphism in $\Theta^n$.  We define the strict $n$-category part of the morphism of $\Rn$ first.  Define a map of $2$-globular sets $x(\vectp) \colon X(\vectj) \rightarrow TX(\vectk)$ as follows:
    \begin{itemize}
      \item for $a_u \in X(\vectj)_0$, $x(\vectp)_0(a_u) = a_{p_1(u)}$;
      \item for $0 < m < n$, $\alpha^m(u_1, u_1 + 1; \dotsc; u_m, u_m + 1; z) \in X(\vectj)_m$, if for all $1 \leq l \leq m$ we have $p_l(u_l) < p_l(v_l)$, then
        \begin{align*}
          & x(\vectp)_m(\alpha^m(u_1, u_1 + 1; \dotsc; u_m, u_m + 1; z)) = \\
          & \sumcomp{m}{m - 1}_{\substack{p_m(u_m) \leq w_m \\ < p_m(u_m + 1)}} \dotsb \sumcomp{m}{0}_{\substack{p_1(u_1) \leq w_1 \\ < p_1(u_1 + 1)}} \alpha^m(w_1, w_1 + 1; \dotsc ; w_m, w_m + 1; p_{m + 1}(z));
        \end{align*}
        otherwise, for the smallest $l$ such that $p_l(u_l) = p_l(v_l)$ we define
          \[
            x(\vectp)_m(\alpha^m(u_1, u_1 + 1; \dotsc; u_m, u_m + 1; z))
          \]
        to be the identity $m$-cell on the $(l - 1)$-cell
          \[
            \sumcomp{l - 1}{l - 2}_{\substack{p_{l - 1}(u_{l - 1}) \leq w_{l - 1} \\ < p_{l - 1}(u_{l - 1} + 1)}} \dotsb \sumcomp{l - 1}{0}_{\substack{p_1(u_1) \leq w_1 \\ < p_1(u_1 + 1)}} \alpha^m(w_1, w_1 + 1; \dotsc ; w_{l - 1}, w_{l - 1} + 1; p_{l}(u_l));
          \]
      \item for $\alpha^n(u_1, u_1 + 1; \dotsc; u_{n - 1}, u_{n - 1} + 1; z) \in X(\vectj)_n$, if for all $1 \leq l \leq m$ we have $p_l(u_l) < p_l(v_l)$, and $p_n(z - 1) < p_n(z)$, then
        \begin{align*}
          & x(\vectp)_n(\alpha^n(u_1, u_1 + 1; \dotsc; u_{n - 1}, u_{n - 1} + 1; z)) = \\
          & \sumcomp{n}{n - 2}_{\substack{p_{n - 1}(u_{n - 1}) \leq w_{n - 1} \\ < p_{n - 1}(u_{n - 1} + 1)}} \dotsb \sumcomp{n}{0}_{\substack{p_1(u_1) \leq w_1 \\ < p_1(u_1 + 1)}} \alpha^m(w_1, w_1 + 1; \dotsc ; w_{n - 1}, w_{n - 1} + 1; p_{n}(z));
        \end{align*}
        if for all $1 \leq l \leq m$ we have $p_l(u_l) < p_l(v_l)$, and $p_n(z - 1) = p_n(z)$, then we define
          \[
            x(\vectp)_n(\alpha^n(u_1, u_1 + 1; \dotsc; u_{n - 1}, u_{n - 1} + 1; z))
          \]
        to be the identity $n$-cell on the $(n - 1)$-cell
          \[
            \sumcomp{n - 1}{n - 2}_{\substack{p_{n - 1}(u_{n - 1}) \leq w_{n - 1} \\ < p_{n - 1}(u_{n - 1} + 1)}} \dotsb \sumcomp{n}{0}_{\substack{p_1(u_1) \leq w_1 \\ < p_1(u_1 + 1)}} \alpha^m(w_1, w_1 + 1; \dotsc ; w_{n - 1}, w_{n - 1} + 1; p_{n}(z));
          \]
        otherwise, for the smallest $l$ such that $p_l(u_l) = p_l(v_l)$, we define
          \[
            x(\vectp)_n(\alpha^n(u_1, u_1 + 1; \dotsc; u_{n - 1}, u_{n - 1} + 1; z))
          \]
         to be the identity $m$-cell on the $(l - 1)$-cell
          \[
            \sumcomp{l - 1}{l - 2}_{\substack{p_{l - 1}(u_{l - 1}) \leq w_{l - 1} \\ < p_{l - 1}(u_{l - 1} + 1)}} \dotsb \sumcomp{n}{0}_{\substack{p_1(u_1) \leq w_1 \\ < p_1(u_1 + 1)}} \alpha^m(w_1, w_1 + 1; \dotsc ; w_{l - 1}, w_{l - 1} + 1; p_{l}(u_l)).
          \]
    \end{itemize}

  To obtain a map $TX(\vectj) \rightarrow TX(\vectk)$ we apply $T$ and compose this with the multiplication for $T$, giving
    \[
      \xy
        % POINTS
        (0, 0)*+{TX(\vectj)}="0";
        (28, 0)*+{T^2X(\vectk)}="1";
        (56, 0)*+{TX(\vectk)}="2";
        % ARROWS
        {\ar^-{Tx(\vectp)} "0" ; "1"};
        {\ar^-{\mu^T_{X(\vectk)}} "1" ; "2"};
      \endxy
    \]
  We now define a map
    \[
      \xy
        % POINTS
        (0, 0)*+{R(\vectj)}="0,0";
        (56, 0)*+{Q(\vectk)}="2,0";
        (0, -20)*+{TX(\vectj)}="0,1";
        (28, -20)*+{T^2X(\vectk)}="1,1";
        (56, -20)*+{TX(\vectk),}="2,1";
        % ARROWS
        {\ar^{r(\vectp)} "0,0" ; "2,0"};
        {\ar_{\theta_{\vectj}} "0,0" ; "0,1"};
        {\ar^{\phi_{\vectk}} "2,0" ; "2,1"};
        {\ar_-{Tx(\vectp)} "0,1" ; "1,1"};
        {\ar_-{\mu^T_{X(\vectk)}} "1,1" ; "2,1"};
      \endxy
    \]
  where the map $r(\vectp)$ is defined as follows:
    \begin{itemize}
      \item for $a_u \in R(\vectj)_0$, $r(\vectp)_0(a_u) = a_{p_1(u)}$;
      \item for $0 < m < n$, $\alpha^m(u_1, v_1; \dotsc; u_m, v_m; z) \in R(\vectj)_m$, if for all $1 \leq l \leq m$ we have $p_l(u_l) < p_l(v_l)$, then
        \begin{align*}
          & r(\vectp)_m(\alpha^m(u_1, v_1; \dotsc; u_m, v_m; z)) = \\ & \alpha^m(p_1(u_1), p_1(v_1); \dotsc; p_m(u_m), p_m(v_m); p_{m + 1}(z));
        \end{align*}
        otherwise, for the smallest $l$ such that $p_l(u_l) = p_l(v_l)$, we define
          \[
            r(\vectp)_m(\alpha^m(u_1, v_1; \dotsc; u_m, v_m; z))
          \]
         to be the identity $m$-cell on the $(l - 1)$-cell
          \[
            \alpha^{l - 1}(p_1(u_1), p_1(v_1); \dotsc; p_{l - 1}(u_{l - 1}), p_{l - 1}(v_{l - 1}); p_{l}(u_l));
          \]
      \item for $\alpha^n(u_1, v_1; \dotsc; u_{n - 1}, v_{n - 1}; z) \in R(\vectj)_n$, if for all $1 \leq l \leq n - 1$ we have $p_l(u_l) < p_l(v_l)$, and $p_n(z - 1) < p_n(z)$, then
        \begin{align*}
          & r(\vectp)_n(\alpha^m(u_1, v_1; \dotsc; u_{n - 1}, v_{n - 1}; z)) = \\ & \alpha^m(p_1(u_1), p_1(v_1); \dotsc; p_{n - 1}(u_{n - 1}), p_{n - 1}(v_{n - 1}); p_n(z));
        \end{align*}
        if for all $1 \leq l \leq n - 1$ we have $p_l(u_l) < p_l(v_l)$, and $p_n(z - 1) = p_n(z)$, then we define
          \[
            r(\vectp)_n(\alpha^m(u_1, v_1; \dotsc; u_{n - 1}, v_{n - 1}; z))
          \]
         to be the identity $n$-cell on the $(n - 1)$-cell
          \[
            \alpha^{n - 1}(p_1(u_1), p_1(v_1); \dotsc; p_{l - 1}(u_{l - 1}), p_{n - 1}(v_{n - 1}); p_{l}(z));
          \]
        otherwise, for the smallest $l$ such that $p_l(u_l) = p_l(v_l)$, we define
          \[
            r(\vectp)_n(\alpha^m(u_1, v_1; \dotsc; u_{n - 1}, v_{n - 1}; z))
          \]
         to be the identity $m$-cell on the $(l - 1)$-cell
          \[
            \alpha^{l - 1}(p_1(u_1), p_1(v_1); \dotsc; p_{l - 1}(u_{l - 1}), p_{l - 1}(v_{l - 1}); p_{l}(u_l)).
          \]
    \end{itemize}

  Finally, we take the transpose of this map under the adjunction
      \[
        \xy
          % POINTS
          (0, 0)*+{\Rtwo}="Rn";
          (16, 0)*+{\Qtwo.}="Qn";
          % ARROWS
          {\ar@<1ex>^-{J}_-*!/u1pt/{\labelstyle \bot} "Rn" ; "Qn"};
          {\ar@<1ex>^-{W} "Qn" ; "Rn"};
        \endxy
     \]
  We write $\epsilon \colon JW \Rightarrow 1$ for the counit of this adjunction, and $\epsilon_{\phi_{\vectk}}$ for the component corresponding to
    \[
      \xy
        (0, 0)*+{Q(\vectk)}="0";
        (24, 0)*+{TX(\vectk).}="1";
        {\ar^-{\phi_{\vectk}} "0" ; "1"};
      \endxy
    \]
  Then the transpose is given by the composite
    \[
      \epsilon_{\phi_{\vectk}} \comp J\big( r(\vectp), \mu^T_{X(\vectk)} \comp Tx(\vectp) \big).
    \]
  This allows us to define the functors $E_n \colon \Theta^n \rightarrow \Qn$ and $I_n \colon \Theta^n \rightarrow P\Alg$.

  \begin{defn}
    Define a functor $E_n \colon \Theta^n \rightarrow \Qn$ as follows:
      \begin{itemize}
        \item given an object $\vectj \in \Theta^n$, $E_n(\vectj)$ is defined to be the object
      \[
      \xy
        (0, 0)*+{Q(\vectj)}="0";
        (24, 0)*+{TX(\vectj).}="1";
        {\ar^-{\phi_{(\vectj)}} "0" ; "1"};
      \endxy
      \]
    of $\Qn$;
        \item given a morphism $\vectp \colon \vectj \rightarrow \vectk$ in $\Theta^n$, $E_n(\vectp)$ is defined to be the map
    \[
      \epsilon_{\phi_{\vectk}} \comp J\big( r(\vectp), \mu^T_{X(\vectk)} \comp Tx(\vectp) \big).
    \]
      \end{itemize}
    Write $K \colon \Qn \rightarrow P\Alg$ for the Eilenberg--Moore comparison functor for the adjunction
      \[
        \xy
          % POINTS
          (0, 0)*+{\nGSet}="Rn";
          (24, 0)*+{\Qtwo.}="Qn";
          % ARROWS
          {\ar@<1ex>^-{F}_-*!/u1pt/{\labelstyle \bot} "Rn" ; "Qn"};
          {\ar@<1ex>^-{U} "Qn" ; "Rn"};
        \endxy
     \]
    We define a functor $I_n := K \comp E_n : \Theta^n \rightarrow P\Alg$.
  \end{defn}

  We can now define the nerve functor for Penon weak $n$-categories.

  \begin{defn}
    The \emph{nerve functor} $\nerve$ for Penon weak $n$-categories is defined by
      \[
        \xymatrix@C=0pt@R=2pt{
          \nerve \colon P\Alg & \longrightarrow & [(\Theta^n)^{\op}, \Set]  \\
          \mathcal{A} \ar[dddd]_f && P\Alg(I_n(-), \mathcal{A}) \ar[dddd]^{f \comp -}  \\
          \\
          & \longmapsto & \\
          \\
          \mathcal{B} && P\Alg(I_n(-), \mathcal{B}).
                 }
      \]
    For a $P$-algebra $\mathcal{A}$, the presheaf $\nerve \mathcal{A} = P\Alg(I_n(-), \mathcal{A})$ is called the \emph{nerve of $\mathcal{A}$}.
  \end{defn}

  \section{Directions for further investigation}  \label{sect:properties}

  In this section we discuss the questions that arise from this nerve construction, and what further results need to be proved in order to make a more complete comparison between Penon weak $n$-categories and Tamsamani--Simpson weak $n$-categories.  The central question is whether the following conjecture holds:

  \begin{conjecture}  \label{conj:Segal}
    Let $\mathcal{A}$ be a Penon weak $n$-category.  Then the nerve $\nerve\mathcal{A}$ satisfies the Segal condition, and is thus a Tamsamani--Simpson weak $n$-category.
  \end{conjecture}

  We have proved this only in the case $n = 2$ (Theorem~\ref{thm:2nerveSegalcond}).  As in the $2$-dimensional case, for general $n$ we can express the Segal maps in terms of composition with wide pushouts of face maps, allowing us to rephrase some parts of the Segal condition in terms of colimits of $P$-algebras in the image of the functor $I_n \colon \Theta^n \rightarrow P\Alg$ (for the 2-dimensional version, see Lemmas~\ref{lem:sk} and \ref{lem:sjk}).  However, it is not practical to generalise the proofs from the $2$-dimensional case to the general case by hand, due to their elementary approach.  The use of computers in mathematical proofs has become more prevalent in recent years, and it may be possible to generalise these elementary proofs for low values of $n$, by using a computer to perform the calculations of the cells in the $P$-algebras $I_n(\vectj)$.  To prove Conjecture~\ref{conj:Segal} in general we would need a more abstract approach.  We believe that this would require a deeper understanding of the ``partially free'' $P$-algebras used in the nerve construction;  colimits of free $P$-algebras are easy to work with, since the free $P$-algebra functor preserves colimits, but this is not true for ``partially free'' $P$-algebras.  The coherence theorems of Section~\ref{sect:globopcoh} would likely play a key role in this, and we believe that those that apply only to free algebras can be extended to ``partially free'' algebras using the contractions in $\Qn$, though we have not yet made this precise.

  Another natural question to ask is whether the nerve functor for Penon weak $n$-categories is full and faithful.  We now prove that it is faithful, then argue that it is not full and explain why this is the case.

  \begin{prop}
    The nerve functor $\nerve \colon P\Alg \rightarrow [(\Theta^n)^{\op}, \Set]$ is faithful.
  \end{prop}

  \begin{proof}
    The idea of the proof is as follows:  every presheaf $(\Theta^n)^{\op} \rightarrow \Set$ has an underlying $n$-globular set, and in the case of the nerve of a Penon weak $n$-category, this is isomorphic to the underlying $n$-globular set of the original $P$-algebra.  A map of $P$-algebras is a map of the underlying $n$-globular sets satisfying a certain commutativity condition, and when we apply the nerve functor to such a map the action on underlying $n$-globular sets remains unchanged.

    For all $0 \leq k \leq n$, write
      \[
        (\vectone_k, \vectzero) := (\underbrace{1, 1, \dotsc, 1}_{k \text{ times}}, 0, 0, \dotsc, 0) \in \Theta^n.
      \]
    Observe that $R(\vectone_k, \vectzero) = X(\vectone_k, \vectzero)$, so $I_n(\vectone_k, \vectzero) = F_{P}X(\vectone_k, \vectzero)$, where
      \[
        F_{P} \colon \nGSet \longrightarrow P\Alg
      \]
    is the free $P$-algebra functor.  Furthermore, for $k \in \mathbb{G}_n$,
      \[
        X(\vectone_k, \vectzero) \iso H_k = \mathbb{G}_n(-, k),
      \]
    i.e. $X(\vectone_k, \vectzero)$ is a representable functor.  Thus, by the Yoneda lemma, for any $A \in \nGSet$,
      \[
        A_k \iso \nGSet(H_k, A) \iso \nGSet(X(\vectone_k, \vectzero, A)),
      \]
    naturally in $A$ and $k$.  Let $\mathcal{A} = (\theta_A \colon P A \rightarrow A)$ be a $P$-algebra.  Then, by the adjunction $F_{P} \ladj U_{P}$,
      \[
        \nGSet(X(\vectone_k, \vectzero), A) \iso P\Alg(I_n(\vectone_k, \vectzero), \mathcal{A}),
      \]
    naturally in $\mathcal{A}$.

    Now suppose we have $P$-algebras $\mathcal{A} = (\theta_A \colon P A \rightarrow A)$, $\mathcal{B} = (\theta_B \colon P B \rightarrow B)$, and maps of $P$-algebras $u$, $v \colon \mathcal{A} \rightarrow \mathcal{B}$ such that $\nerve u = \nerve v$.  Thus, for each $0 \leq k \leq n$ we have
      \[
        u \comp - = v \comp - \colon P\Alg(I_n(\vectone_k, \vectzero), \mathcal{A}) \rightarrow P\Alg(I_n(\vectone_k, \vectzero), \mathcal{B}).
      \]
    We can write $u_k$ as the composite shown in the diagram below:
      \[
        \xy
          % POINTS
          (-50, 16)*+{A_k}="Ak";
          (0, 16)*+{B_k}="Bk";
          (-50, 0)*+{\nGSet(H_k, A)}="HkA";
          (0, 0)*+{\nGSet(H_k, B)}="HkB";
          (-50, -16)*+{\nGSet(X(\vectone_k, \vectzero), A)}="XA";
          (0, -16)*+{\nGSet(X(\vectone_k, \vectzero), B)}="XB";
          (-50, -32)*+{P\Alg(I_n(\vectone_k, \vectzero), \mathcal{A})}="InA";
          (0, -32)*+{P\Alg(I_n(\vectone_k, \vectzero), \mathcal{B})}="InB";
          % ARROWS
          {\ar^{u_k} "Ak" ; "Bk"};
          {\ar_{u \comp -} "HkA" ; "HkB"};
          {\ar_{u \comp -} "XA" ; "XB"};
          {\ar_{u \comp -} "InA" ; "InB"};
          {\ar_{\iso} "Ak" ; "HkA"};
          {\ar_{\iso} "HkA" ; "XA"};
          {\ar_{\iso} "XA" ; "InA"};
          {\ar_{\iso} "InB" ; "XB"};
          {\ar_{\iso} "XB" ; "HkB"};
          {\ar_{\iso} "HkB" ; "Bk"};
        \endxy
      \]
    and similarly, we can write $v_k$ as:
      \[
        \xy
          % POINTS
          (-50, 16)*+{A_k}="Ak";
          (0, 16)*+{B_k}="Bk";
          (-50, 0)*+{\nGSet(H_k, A)}="HkA";
          (0, 0)*+{\nGSet(H_k, B)}="HkB";
          (-50, -16)*+{\nGSet(X(\vectone_k, \vectzero), A)}="XA";
          (0, -16)*+{\nGSet(X(\vectone_k, \vectzero), B)}="XB";
          (-50, -32)*+{P\Alg(I_n(\vectone_k, \vectzero), \mathcal{A})}="InA";
          (0, -32)*+{P\Alg(I_n(\vectone_k, \vectzero), \mathcal{B}).}="InB";
          % ARROWS
          {\ar^{v_k} "Ak" ; "Bk"};
          {\ar_{v \comp -} "HkA" ; "HkB"};
          {\ar_{v \comp -} "XA" ; "XB"};
          {\ar_{v \comp -} "InA" ; "InB"};
          {\ar_{\iso} "Ak" ; "HkA"};
          {\ar_{\iso} "HkA" ; "XA"};
          {\ar_{\iso} "XA" ; "InA"};
          {\ar_{\iso} "InB" ; "XB"};
          {\ar_{\iso} "XB" ; "HkB"};
          {\ar_{\iso} "HkB" ; "Bk"};
        \endxy
      \]
    Since $u \comp - = v \comp -$, these diagrams give us that $u_k = v_k$ for all $0 \leq k \leq n$, so $u = v$.  Hence the nerve functor $\nerve \colon P\Alg \rightarrow [(\Theta^n)^{\op}, \Set]$ is faithful.
  \end{proof}

  To see that the nerve functor is not full, consider the $P$-algebra illustrated below:
    \[
      \xymatrix{
        & \bullet \ar[ddr]^g  \\
        \\
        \bullet \ar[uur]^f \ar[rr]^{h}_{\rotatebox[origin=c]{270}{ $\iso$}} \ar@/_2pc/[rr]_k && \bullet
               }
    \]
  where $g \comp f = h$.  Any endomorphism of this $P$-algebra that sends $f$ to $f$ and $g$ to $g$ must also send $h$ to $h$, since maps of $P$-algebras preserve composition, and $h = g \comp f$.  However, when we consider endomorphisms of the nerve of this $P$-algebra, we see that there are endomorphisms sending $f$ to $f$ and $g$ to $g$ that send $h$ to $k$; such endomorphisms are not in the image of the nerve functor.

  This illustrates a key difference between algebraic and non-algebraic definitions of weak $n$-category:  in the algebraic case the natural notion of map preserves the composition structure, but in the non-algebraic case there is no specified composition structure to preserve.  In the example above, once we have applied the nerve functor we no longer remember which cell was $g \comp f$, and morphisms can now map $h$ to any legitimate choice of composite.

  Note that maps of nerves are still required to preserve identities, however, since these are specified by degeneracy maps.  This means that maps of Tamsamani--Simpson weak $n$-categories behave like normalised maps, i.e.~those that preserve identities strictly, but are only required to preserve composition weakly.  This has been formalised in the $2$-dimensional case~\cite{LP08}.  There is currently no definition of normalised maps of Penon weak $n$-categories, and we believe that such a definition would be necessary to adapt our nerve construction to give a a full nerve functor for Penon weak $n$-categories.

  One final question raised by this work is whether every Tamsamani--Simpson weak $n$-category arises as the nerve of a Penon weak $n$-category.  To answer this question we would need to construct a Penon weak $n$-category from a Tamsamani--Simpson weak $n$-category.  Note that there will be no canonical way to do this, since it would involve making choices of composites.

  This nerve construction is a considerable first step towards understanding the relationships between algebraic and non-algebraic definitions of weak $n$-categories.  We have made a connection between the algebraic definition of Penon weak $n$-categories and the non-algebraic setting in which Tamsamani--Simpson weak $n$-categories are defined, allowing for the relationship between these definitions to be studied.  Our nerve construction is the first to allow for such a comparison, and we believe that it should pave the way for more connections to be made between algebraic and non-algebraic definitions of weak $n$-category.

\backmatter
\bibliography{biblio}
\bibliographystyle{halpha}

\end{document}